\documentclass[10pt,reqno]{amsart}
\usepackage[utf8]{inputenc}
\usepackage{geometry}
\numberwithin{equation}{section}
\usepackage[pdfencoding=auto]{hyperref}
\usepackage{graphicx}
\usepackage{amssymb,stmaryrd}
\usepackage{epstopdf}
\usepackage{mathtools}
\usepackage[usenames,dvipsnames,svgnames,table]{xcolor}
\usepackage{cancel}
\usepackage{soul}

\newcommand\FirstN[1]{\{1,\dots,#1\}}

\def \NNsymbol {\mathcal{NN}}
\def \SNNsymbol {\mathcal{SNN}}

\def \SNNreal {\mathtt{SNN}}
\def \NNreal {\mathtt{NN}}
\def \Realization {\mathtt{R}}
\def \OutDim {{d_{\mathrm{out}}}}
\def \InDim {{d_{\mathrm{in}}}}

\def \AppSet {\Sigma}
\def \AppErr {E}
\def \PPoly {\mathtt{PPoly}}
\def \Spline {\mathtt{Spline}}

\usepackage{mathrsfs}
\usepackage{todonotes}
\usepackage{wrapfig}
\usepackage{enumitem}
\usepackage{microtype}
\usepackage{dsfont}
\usepackage{xparse}

\NewDocumentCommand \ReLUApproxSpace {O{\alpha} O{q} O{} O{\mathscr{L}}}{A_{#2}^{#1} (X,\varrho_{#3},#4)}

\NewDocumentCommand \GenApproxSpace {O{\alpha} O{q}}{A_{#2}^{#1} (X,\AppSet)} 

\NewDocumentCommand \GenApproxSpaceSet {O{\Sigma} O{\alpha} O{q}}{A_{#3}^{#2} (X,#1)} 

\newcommand{\Cr}{\mathrm{Cr}}

\newcommand{\eps}{\varepsilon}
\newcommand{\Z}{\mathbb{Z}}
\newcommand{\N}{\mathbb{N}}

\newcommand{\R}{\mathbb{R}}

\newcommand{\CalT}{\mathcal{T}}
\newcommand{\CalG}{\mathcal{G}}

\newcommand{\CalD}{\mathcal{D}}

\newcommand{\vertiii}[1]{{\left\vert\kern-0.25ex\left\vert\kern-0.25ex\left\vert #1
    \right\vert\kern-0.25ex\right\vert\kern-0.25ex\right\vert}}

\NewDocumentCommand \StandardSigma {O{n}}{\AppSet_{#1} (X,\varrho,\mathscr{L})}
\NewDocumentCommand \StandardSigmaStrict {O{n}}{\AppSet_{#1} ' (X,\varrho,\mathscr{L})}
\NewDocumentCommand \StandardXSpace {O{k} O{p}} {X_{#2}^{#1}}

\newcommand{\WeightClassSymbol}{\mathtt{W}}
\newcommand{\NeuronClassSymbol}{\mathtt{N}}

\NewDocumentCommand \StandardSigmaW {O{n}}{\WeightClassSymbol_{#1} (X,\varrho,\mathscr{L})}
\NewDocumentCommand \StandardSigmaN {O{n}}{\NeuronClassSymbol_{#1} (X,\varrho,\mathscr{L})}

\NewDocumentCommand \StandardSigmaWS {O{n}}{\mathtt{SW}_{#1} (X,\varrho,\mathscr{L})}
\NewDocumentCommand \StandardSigmaNS {O{n}}{\mathtt{SN}_{#1} (X,\varrho,\mathscr{L})}

\NewDocumentCommand \ASpace {O{\StandardXSpace (\Omega)} O{\varrho} O{q} O{\alpha} O{\mathscr{L}}}{A^{#4}_{#3}(#1,#2,#5)}
\NewDocumentCommand \NSpace {O{\StandardXSpace (\Omega)} O{\varrho} O{q} O{\alpha} O{\mathscr{L}}}{AN^{#4}_{#3}(#1,#2,#5)}
\NewDocumentCommand \WSpace {O{\StandardXSpace (\Omega)} O{\varrho} O{q} O{\alpha} O{\mathscr{L}}}{AW^{#4}_{#3}(#1,#2,#5)}

\NewDocumentCommand \WASpace {O{\StandardXSpace (\Omega)} O{\varrho} O{q} O{\alpha} O{\mathscr{L}}}{W^{#4}_{#3}(#1,#2,#5)}
\NewDocumentCommand \NASpace {O{\StandardXSpace (\Omega)} O{\varrho} O{q} O{\alpha} O{\mathscr{L}}}{N^{#4}_{#3}(#1,#2,#5)}

\NewDocumentCommand \SWASpace {O{\StandardXSpace (\Omega)} O{\varrho} O{q} O{\alpha} O{\mathscr{L}}}{SW^{#4}_{#3}(#1,#2,#5)}
\NewDocumentCommand \SNASpace {O{\StandardXSpace (\Omega)} O{\varrho} O{q} O{\alpha} O{\mathscr{L}}}{SN^{#4}_{#3}(#1,#2,#5)}

\let\emptyset\varnothing

\newcommand{\identity}{\mathrm{id}}
\newcommand{\Indicator}{{\mathds{1}}}

\newcommand{\diag}{{\operatorname{diag}}}

\newcommand{\supp}{{\operatorname{supp}}}
\newcommand{\hess}{{\operatorname{Hess}}}
\newcommand{\GL}{\mathrm{GL}}

\theoremstyle{remark}
\newtheorem*{rem*}{Remark}
\theoremstyle{definition}
\newtheorem{defn}{Definition}[section]
\theoremstyle{plain}
\newtheorem{thm}[defn]{Theorem}
\theoremstyle{plain}
\newtheorem{lem}[defn]{Lemma}
\theoremstyle{remark}
\newtheorem{rem}[defn]{Remark}
\theoremstyle{plain}
\newtheorem{prop}[defn]{Proposition}
\newtheorem{cor}[defn]{Corollary}
\newtheorem{example}[defn]{Example}

\usepackage[ddmmyyyy,hhmmss]{datetime}
\newdateformat{daymonthyeardate}{%
  \THEDAY.\THEMONTH.\THEYEAR}
\usepackage{fancyhdr}

\theoremstyle{remark}
\newtheorem{remx}[defn]{Remark}
\renewenvironment{rem}
  {\pushQED{\qed}\remx}
  {\popQED\endremx}

\theoremstyle{remark}
\newtheorem*{remxx}{Remark}
\renewenvironment{rem*}
  {\pushQED{\qed}\remxx}
  {\popQED\endremxx}

\theoremstyle{definition}
\newtheorem{defnx}[defn]{Definition}
\renewenvironment{defn}
  {\pushQED{\qed}\defnx}
  {\popQED\enddefnx}

\theoremstyle{plain}
\newtheorem{lemx}[defn]{Lemma}
\renewenvironment{lem}
  {\pushQED{\qed}\lemx}
  {\popQED\endlemx}

\theoremstyle{plain}
\newtheorem{corx}[defn]{Corollary}
\renewenvironment{cor}
  {\pushQED{\qed}\corx}
  {\popQED\endcorx}

\theoremstyle{plain}
\newtheorem{propx}[defn]{Proposition}
\renewenvironment{prop}
  {\pushQED{\qed}\propx}
  {\popQED\endpropx}

\theoremstyle{definition}
\newtheorem{thmx}[defn]{Theorem}
\renewenvironment{thm}
  {\pushQED{\qed}\thmx}
  {\popQED\endthmx}

\usepackage{etoolbox}
\makeatletter
\patchcmd{\@maketitle}
  {\ifx\@empty\@dedicatory}
  {\ifx\@empty\@date \else {\vskip3ex \centering\footnotesize\@date\par\vskip1ex}\fi
   \ifx\@empty\@dedicatory}
  {}{}
\patchcmd{\@adminfootnotes}
  {\ifx\@empty\@date\else \@footnotetext{\@setdate}\fi}
  {}{}{}
\makeatother

\usepackage{ifthen} 
\newif\ifarxiv
\arxivtrue 

\ifarxiv
    \newcommand{\ocite}[2][]{\cite[#1]{#2}} 
\else
    \newcommand{\ocite}[2][]{\unskip\unpenalty} 
\fi

\title{Approximation spaces of Deep Neural Networks}
\author{R{\'e}mi Gribonval}
\address{Univ Lyon, Inria, CNRS, ENS de Lyon, UCB Lyon 1, LIP UMR 5668\protect\\ F-69342, Lyon, France}
\thanks{This work was conducted while R.G. was with Univ Rennes, Inria, CNRS, IRISA}
\email{remi.gribonval@inria.fr}
\author{Gitta Kutyniok}
\address{Institut für Mathematik, Technische Universität Berlin, Germany}
\thanks{G.K. acknowledges partial support by the Bundesministerium fur Bildung und Forschung (BMBF) through the Berliner Zentrum for Machine Learning (BZML), Project AP4, RTG DAEDALUS (RTG 2433), Projects P1 and P3, RTG BIOQIC (RTG 2260), Projects P4 and P9, and by the Berlin Mathematics Research Center MATH+, Projects EF1-1 and EF1-4.}
\email{kutyniok@math.tu-berlin.de}

\author{Morten Nielsen}
\address{Department of Mathematical Sciences, Aalborg University, Denmark }
\email{mnielsen@math.aau.dk}
\author{Felix Voigtlaender}
\address{Department of Scientific Computing,
Katholische Universität Eichstätt-Ingolstadt}
\email{felix@voigtlaender.xyz}
\thanks{G.K. and F.V. acknowledge support by the European Commission-Project DEDALE (contract
no. 665044) within the H2020 Framework.}

\begin{document}

\begin{abstract}
We study the expressivity of 
deep neural networks.
Measuring a network's complexity by its number of connections or by its number of neurons,
we consider the class of functions for which the error of best approximation
with networks of a given complexity decays at a certain rate when increasing the complexity budget.
Using results from classical approximation theory, we show that this class can be endowed
with a (quasi)-norm that makes it a linear function space, called approximation space.
We establish that allowing the networks to have certain types of ``skip connections''
does not change the resulting approximation spaces.
We also discuss the role of the network's nonlinearity
(also known as activation function) on the resulting spaces, as well as the role of depth.
For the popular ReLU nonlinearity and its powers,
we relate the newly constructed spaces to classical Besov spaces.
The established embeddings highlight that some functions of very low Besov smoothness
can nevertheless be well approximated by neural networks, if these networks are sufficiently deep.
\end{abstract}



\maketitle

%
%
%
%
%
%
\vspace{-0.5cm}

\section{Introduction}

Today, we witness a worldwide triumphant march of deep neural networks, impacting not
only various application fields, but also areas in mathematics such as inverse problems.
Originally, neural networks were developed by McCulloch and Pitts \cite{McCullochPitts} in 1943
to introduce a theoretical framework for artificial intelligence.
At that time, however, the limited amount of data and the lack of sufficient
computational power only allowed the training of shallow networks, that is, networks with
only few layers of neurons, which did not lead to the anticipated results.
The current age of big data and the significantly increased computer performance now
make the application of deep learning algorithms feasible, leading to the successful training
of very deep neural networks.
For this reason, neural networks have seen an impressive comeback.
The list of important applications in public life
ranges from speech recognition systems on cell phones over self-driving cars
to automatic diagnoses in healthcare.
For applications in science, one can witness a similarly strong impact
of deep learning methods in research areas such as quantum chemistry \cite{TACMT17}
and molecular dynamics \cite{MPWN18}, often allowing to resolve problems
which were deemed unreachable before.
This phenomenon is manifested similarly in certain fields of mathematics, foremost in inverse
problems \cite{JO17,Bubb18}, but lately also, for instance, in numerical analysis of
partial differential equations \cite{BBGJJ18}.

Yet, most of the existing research related
to deep learning is empirically driven and a profound and comprehensive
mathematical foundation is still missing, in particular for the previously mentioned applications.
This poses a significant challenge not only for mathematics itself,
but in general for the ``safe'' applicability of deep neural networks 
\cite{Elad17}.

A {\em deep neural network} in mathematical terms is a tuple
\begin{equation}
  \Phi = \big( (T_1, \alpha_1), \dots,  (T_L, \alpha_L) \big)
  \label{eq:intro1}
\end{equation}
consisting of affine-linear maps $T_\ell : \R^{N_{\ell - 1}} \to \R^{N_\ell}$
(hence $T_\ell(x) = A_\ell \, x + b_\ell$ for appropriate matrices
 $A_\ell$ and vectors $b_\ell$, often with a convolutional or Toeplitz structure)
and of nonlinearities $\alpha_{\ell}: \R^{N_{\ell}} \to \R^{N_{\ell}}$
that typically encompass componentwise rectification, possibly followed by a pooling operation.

The tuple in \eqref{eq:intro1} encodes the architectural components of the neural network,
where $L$ denotes the {\em number of layers} of the network,
while $L-1$ is the number of {\em hidden layers}.
The highly structured function $\Realization(\Phi)$ implemented by such a network $\Phi$
is then defined by applying the different maps in an iterative (layer-wise) manner; precisely,
\[
  \Realization (\Phi) : \R^{N_0} \to \R^{N_L}, \quad \text{with} \quad
  \Realization(\Phi) := \alpha_L \circ T_L \circ \cdots \circ \alpha_1 \circ T_1 \, .
\]
We call this function the {\em realization} of the deep neural network $\Phi$.
It is worth pointing out that most of the literature calls this function itself the neural network;
one can however---depending on the choice of the activation functions---imagine
the same function being realized by different architectural components,
so that it would not make sense, for instance, to speak of the number of layers of $\Realization(\Phi)$;
this is only well-defined when we talk about $\Phi$ itself.
The {\em complexity} of a neural network can be captured by various numbers such as the
{\em depth} $L$, the \emph{number of hidden neurons} $N(\Phi) = \sum_{\ell=1}^{L-1} N_{\ell}$,
or the {\em number of connections} (also called the \emph{connectivity}, or the \emph{number of weights})
given by $W(\Phi) = \sum_{\ell = 1}^{L} \| A_\ell \|_{\ell^0}$, where $\| A_\ell \|_{\ell^0}$
denotes the number of non-zero entries of the matrix $A_\ell$.

From a mathematical perspective, the central task of a deep neural network is to approximate a function
$f : \R^{N_0} \to \R^{N_L}$, which for instance encodes a classification problem.
Given a training data set $\big( x_i,f(x_i) \big)_{i=1}^m$ a loss function
$\mathcal{L} : \R^{N_L} \times \R^{N_L} \to \R$,
and a regularizer $\mathcal{P}$, which
imposes, for instance, sparsity conditions on the weights of the neural network $\Phi$,
solving the optimization problem
\begin{equation}\label{eq:DefERM}
  \min_{\Phi}
    \sum_{i=1}^m
      \mathcal{L} \big( \Realization (\Phi)(x_i),f(x_i) \big)
    + \lambda \mathcal{P}(\Phi)
\end{equation}
typically through a variant of stochastic gradient descent,
yields a learned neural network $\widehat{\Phi}$.
The objective is to achieve 
$\Realization (\widehat{\Phi}) \approx f$,
which is only possible if the function $f$ can indeed be well approximated by (the realization of)
a network with the prescribed architecture.
Various theoretical results have already been published to establish the ability of
neural networks---often with specific architectural constraints---to approximate
functions from certain function classes; this is referred to as analyzing the
{\em expressivity} of neural networks.
However, the fundamental question asking which function spaces are truly natural for deep neural
networks has never been comprehensively addressed.
Such an approach may open the door to a novel viewpoint and lead to a
refined understanding of the expressive power of deep neural networks.

In this paper we introduce approximation spaces associated to neural networks.
This leads to an extensive theoretical framework for studying the expressivity
of deep neural networks, allowing us also to address questions such as the impact of the depth
and of the activation function, or of so-called (and widely used) skip connections
on the approximation power of deep neural networks.


\subsection{Expressivity of Deep Neural Networks}

The first theoretical results concerning the expressivity of neural networks date back to
the early 90s, at that time focusing on shallow networks, mainly in the context of
the 
\emph{universal approximation theorem}
\cite{PinkusUniversalApproximation,Hornik1989universalApprox,Cybenko1989,Hornik1991251}.
The breakthrough-result of the ImageNet competition in 2012 \cite{AlexNet},
and the ensuing worldwide success story of neural networks
has brought renewed interest to the study of neural networks,
now with an emphasis on \emph{deep} networks.
The surprising effectiveness of such networks in applications has motivated the
study of the effect of depth on the expressivity of these networks.
Questions related to the learning phase are of a different
nature, focusing on aspects of statistical learning and optimization,
and hence constitute a different research field.

Let us recall some of the key contributions in the area of expressivity, in order to put our results
into perspective.
The universal approximation theorems by  Hornik \cite{Hornik1991251} and Cybenko \cite{Cybenko1989}
can be counted as a first highlight, stating that neural networks with only one hidden layer
can approximate continuous functions on compact sets arbitrarily well.
Examples of further work in this early stage, hence focusing on networks with a single hidden
layer, are approximation error bounds in terms of the number of neurons for functions with bounded
first Fourier moments \cite{Barron1993,Barron1994},
the failure of those networks to provide localized
approximations \cite{ChuXM1994networksforlocApprox},
a fundamental lower bound on approximation rates \cite{DeVore1997approxfeedforward, CandesDiss},
and the approximation of smooth/analytic functions \cite{Mhaskar1996NNapprox,Mhaskar1995151}.
Some of the early contributions already study networks with multiple hidden layers,
such as \cite{Funahashi1989183} for approximating continuous functions,
and \cite{NguyenThien1999687} for approximating functions together with their derivatives.
Also \cite{ChuXM1994networksforlocApprox}, which shows in certain instances
that deep networks can perform better than single-hidden-layer networks
can be counted towards this line of research.
For a survey of those early results, we refer to \cite{ellacott1994aspects,pinkus1999approximation}.

More recent work focuses predominantly on the analysis of the effect of depth.
Some examples---again without any claim of completeness---are \cite{Eldan2016PowerofDepth},
in which a function is constructed which cannot be expressed by a small two-layer network,
but which is implemented by a three-layer network of low complexity,
or \cite{Mhaskar2016DeepVSShallow} which considers so-called compositional functions,
showing that such functions can be approximated by neural networks without suffering
from the curse of dimensionality.
A still different viewpoint is taken in \cite{cohen2016expressive, cohen2016convolutional},
which focus on a similar problem as \cite{Mhaskar2016DeepVSShallow} but attacking it
by utilizing results on tensor decompositions.
Another line of research aims to study the approximation rate when approximating
certain function classes by neural networks with growing complexity
\cite{ShaCC2015provableAppDNN,bolcskei2017optimal,PetersenVoigtlaenderReLU,yarotsky2017error,Mhaskar93}.

\subsection{The classical notion of approximation spaces}

In classical approximation theory, the notion of approximation spaces refers to
(quasi)-normed spaces that are defined by their elements satisfying a specific decay
of a certain approximation error;
see for instance \cite{ConstructiveApproximation} 
In this introduction, we will merely sketch the key construction and
properties; we refer to Section~\ref{sec:ApproximationSpaces} for more details.

Let $X$ be a 
quasi-Banach space equipped with the quasi-norm $\|\cdot\|_{X}$.
Furthermore, here, as in the rest of the paper, let us denote by $\N = \{1,2,\ldots\}$
the set of natural numbers, and write $\N_{0} = \{0\} \cup \N$, $\N_{\geq m} = \{n \in \N, n \geq m\}$.
For a prescribed family $\AppSet = (\AppSet_{n})_{n \in \N_{0}}$
of subsets $\AppSet_{n} \subset X$, one aims to classify functions $f \in X$
by the decay (as $n \to \infty$) of the error of best approximation
by elements from $\AppSet_{n}$, given by
$\AppErr (f,\AppSet_{n})_X :=\inf_{g \in \AppSet_{n}} \| f - g \|_{X}$.
The desired rate of decay of this error is prescribed by a discrete weighted $\ell^q$-norm,
where the weight depends on the parameter $\alpha > 0$.
For $q = \infty$, this leads to the class
\[
  A_\infty^\alpha(X,\Sigma)
  := \Big\{
       f \in X
       \, : \,
       \sup_{n \geq 1} \,\,
         [n^{\alpha} \cdot \AppErr (f,\AppSet_{n-1})_{X}] < \infty
     \Big\} \, .
\]
Thus, intuitively speaking, this class consists of those elements of $X$ for which the error of
best approximation by elements of $\AppSet_{n}$ decays at least as $\mathcal{O}(n^{-\alpha})$
for $n \to \infty$.
This general philosophy also holds for the more general classes $\GenApproxSpace$, $q > 0$.

If the initial family $\AppSet$ of subsets of $X$ satisfies some quite natural conditions,
more precisely $\AppSet_0 = \{0\}$, each $\Sigma_n$ is invariant to scaling,
$\Sigma_n \subset \Sigma_{n+1}$,
and the union $\bigcup_{n \in \N_0} \Sigma_n$ is dense in $X$,
as well as the slightly more involved condition that
$\AppSet_n + \AppSet_n \subset \AppSet_{cn}$ for some fixed $c \in \N$,
then an abundance of results are available for the approximation classes $\GenApproxSpace$.
In particular, $\GenApproxSpace$ turns out to be a proper {\em linear function space},
equipped with a natural (quasi)-norm.
Particular highlights of the theory are various embedding and interpolation results
between the different approximation spaces.

\subsection{Our Contribution}

We introduce a novel perspective on the study of expressivity of deep neural networks
by introducing the associated approximation spaces and investigating their properties.
This is in contrast with the usual approach of studying the approximation fidelity of
neural networks on \emph{classical} spaces.
We utilize this new viewpoint for deriving novel results on, for instance,
the impact of the choice of activation functions and the depth of the networks.

Given a so-called (non-linear) activation function $\varrho :
\R \to \R$, a classical setting is to consider nonlinearities $\alpha_\ell$
in \eqref{eq:intro1}  corresponding to a componentwise application of the activation function
for each hidden layer $1 \leq \ell < L$, and $\alpha_{L}$ being the identity.
We refer to networks of this form as {\em strict $\varrho$-networks}.
To introduce a framework of sufficient flexibility, we also consider nonlinearities
where for each component either $\varrho$ or the identity is applied.
We refer to such networks as \emph{generalized $\varrho$-networks};
the realizations of such generalized networks include
various function classes such as multilayer sparse linear transforms
\cite{lemagoarou:hal-01167948}, networks with skip-connections \cite{SkipConnections},
ResNets \cite{ResidualNetworks,WiderOrDeeper} or U-nets \cite{Ronneberger:2015gk}.



Let us now explain how we utilize this framework of approximation spaces.
Our focus will be on approximation rates in terms of growing complexity of neural networks,
which we primarily measure by their {\em connectivity},
since this connectivity is closely linked to the number of bytes needed to describe the network,
and also to the number of floating point operations needed to apply the corresponding function
to a given input.
This is in line with recent results \cite{bolcskei2017optimal,PetersenVoigtlaenderReLU,yarotsky2017error}
which explicitly construct neural networks that reach an optimal approximation rate
for very specific function classes, and in contrast to most of the existing literature
focusing on complexity measured by the number of neurons.
We also consider the approximation spaces for which the complexity of the
networks is measured by the number of neurons.

In addition to letting the number of connections or neurons tend to infinity
while keeping the depth of the networks fixed,
we also allow the depth to evolve with the number of connections or neurons.
To achieve this, we link both by a non-decreasing \emph{depth-growth function}
$\mathscr{L}: \N \to \N \cup \{ \infty \}$,
where we allow the possibility of not restricting the number of layers
when $\mathscr{L}(n) = \infty$.
We then consider the function families 
$\WeightClassSymbol_n(\Omega\to\R^{k},\varrho, \mathscr{L})$
(resp.~$\NeuronClassSymbol_n(\Omega\to\R^{k},\varrho, \mathscr{L})$)
made of all restrictions to a given 
 subset $\Omega \subseteq \R^{d}$
of functions which can be represented by (generalized)
$\varrho$-networks with input/output dimensions $d$ and $k$,
at most $n$ nonzero connection weights (resp.~at most $n$ hidden neurons),
and at most $\mathscr{L}(n)$ layers.
Finally, given a space $X$ of functions $\Omega \to \R^{k}$, we will use the sets
$\Sigma_n = \StandardSigmaW := \WeightClassSymbol_n(\Omega\to\R^{k},\varrho, \mathscr{L}) \cap X$
(resp.~$\Sigma_n = \StandardSigmaN := \NeuronClassSymbol_n (\Omega\to\R^{k},\varrho, \mathscr{L}) \cap X$)
to define the associated approximation spaces.
Typical choices for $X$ are
\begin{equation}
  \StandardXSpace(\Omega) := L_p(\Omega; \R^k) \; \mbox{for } 0<p<\infty
  \quad \mbox{or} \quad
  \StandardXSpace[k][\infty](\Omega),
  \label{eq:StandardXSpace}
\end{equation}
with $\StandardXSpace[k][\infty](\Omega)$ the space of uniformly continuous functions on $\Omega$
that vanish at infinity, equipped with the supremum norm.
For ease of notation, we will sometimes also write
$\StandardXSpace[](\Omega) := \StandardXSpace[1](\Omega)$, and
$\StandardXSpace := \StandardXSpace(\Omega)$ (resp.~$\StandardXSpace[] := \StandardXSpace[](\Omega)$).

Let us now give a coarse overview of our main results, which we are able to derive
with our choice of approximation spaces based on $\StandardSigmaW$ or $\StandardSigmaN$.

\subsubsection{Core properties of the novel approximation spaces.}

We first prove that each of these two families $\Sigma = (\Sigma_n)_{n \in \N_0}$
satisfies the necessary requirements for the associated approximation spaces
$\GenApproxSpace$---which we denote by $\WASpace[X]$ and $\NASpace[X]$, respectively---to be amenable
to various results from approximation theory.
Under certain conditions on $\varrho$ and $\mathscr{L}$,
Theorem~\ref{th:DNNApproxSpaceWellDefined} shows that these
approximation spaces are even equipped with a convenient (quasi-)Banach \emph{spaces} structure.
The spaces $\WASpace[X]$ and $\NASpace[X]$ are nested (Lemma~\ref{lem:weightsvsneurons})
and do not generally coincide (Lemma~\ref{lem:WeightAndNeuronSpacesDistinct}).

To prepare the ground for the analysis of the impact of depth, we then prove nestedness with
respect to the depth growth function. In slightly more detail, we identify a partial order
$\preceq$ and an equivalence relation $\sim$ on depth growth functions such that the following
holds (Lem.~\ref{lem:RoleGrowthFunctionRate} and Thm.~\ref{thm:RoleGrowthFunction}):
\begin{enumerate}
  \item If $\mathscr{L}_{1} \preceq \mathscr{L}_{2}$, then
        \(
          \WASpace[X][\varrho][q][\alpha][\mathscr{L}_{1}]
          \subset \WASpace[X][\varrho][q][\alpha][\mathscr{L}_{2}]
        \)
        for any $\alpha$, $q$, $X$ and $\varrho$; and
  \item if $\mathscr{L}_{1} \sim \mathscr{L}_{2}$, then
        \(
          \WASpace[X][\varrho][q][\alpha][\mathscr{L}_{1}]
          = \WASpace[X][\varrho][q][\alpha][\mathscr{L}_{2}]
        \)
        for any $\alpha$, $q$, $X$ and $\varrho$.
\end{enumerate}
The same nestedness results hold for the spaces $\NASpace[X]$.
Slightly surprising and already insightful might be that under mild conditions on the
activation function $\varrho$,  the approximation classes for strict and generalized
$\varrho$-networks are in fact identical, allowing to derive the conclusion that their
expressivities coincide (see Theorem~\ref{th:DNNApproxSpaceWellDefinedStrict}).


\subsubsection{ Approximation spaces associated with ReLU-networks.}

The rectified linear unit (ReLU) and its powers of exponent $r \in \N$---in spline theory better-known
under the name of \emph{truncated powers} \cite[Chapter 5, Equation (1.1)]{ConstructiveApproximation}---%
are defined by
\[
  \varrho_r : \R \to \R, x \mapsto (x_{+})^{r},
\]
where $x_{+} = \max \{ 0,x \} = \varrho_{1}(x)$, with the ReLU activation function being $\varrho_{1}$.
Considering these activation functions is motivated practically by the wide use
of the ReLU \cite{LeCunDeepLearningNaturePaper}, as well as theoretically
by the existence \cite[Theorem 4]{PinkusLowerBoundsForMLPApproximation} of pathological
activation functions giving rise to trivial---too rich---approximation spaces that satisfy
$\WASpace[\StandardXSpace]=\NASpace[\StandardXSpace] = \StandardXSpace$, for all $\alpha,q$.
In contrast, the classes associated to $\varrho_{r}$-networks are nontrivial for $p \in (0,\infty]$
(Theorem~\ref{thm:ApproximationSpacesNonTrivial}).
Moreover, 
strict and generalized $\varrho_{r}$-networks yield
identical approximation classes for any 
subset $\Omega \subseteq \R^d$ of nonzero measure (even unbounded), for any
$p \in (0,\infty]$ (Theorem~\ref{th:ReLUDNNApproxSpaceWellDefined}).
Furthermore, for any $r \in \N$, these approximation classes are (quasi-)Banach spaces
(Theorem~\ref{th:ReLUDNNApproxSpaceWellDefined}), as soon as
\[
  L := \sup_{n \in \N} \mathscr{L}(n)
  \geq \begin{cases}
          2, & \text{if}\ \Omega\ \text{is bounded \emph{or}}\ d=1, \\
          3, & \text{otherwise.}
        \end{cases}.
\]
The expressivity of networks with more general activation functions can be related to that of
$\varrho_{r}$-networks (see Theorem~\ref{thm:ReLUPowersApproxSpaces}) in the following sense:
If $\varrho$ is continuous and piecewise polynomial of degree at most $r$, then its
approximation spaces are contained in those of $\varrho_{r}$-networks. In particular,
if $\Omega$ is bounded or if $\mathscr{L}$ satisfies a certain growth condition,
then for $s,r \in \N$ such that $1 \leq s \leq r$
       \[
         \WASpace[X][\varrho_s]  \subset \WASpace[X][\varrho_r]
         \qquad \text{and} \qquad
         \NASpace[X][\varrho_s]  \subset \NASpace[X][\varrho_r].
       \]
Also, if $\varrho$ is a spline of degree $r$ and not a polynomial, then its approximation
spaces match those of $\varrho_{r}$ on bounded $\Omega$.
In particular, on a bounded domain $\Omega$,  the spaces associated to the leaky-ReLU \cite{Maas:tn},
the parametric ReLU \cite{He:2015:DDR:2919332.2919814},  the absolute value
(as, e.g, in scattering transforms \cite{Mallat:2016jr}) and the
soft-thresholding activation function \cite{Gregor2010} are all identical to the spaces associated to the ReLU.

Studying the relation of approximation spaces of $\varrho_{r}$-networks for different
$r$, we derive the following statement as a corollary (Corollary~\ref{cor:SaturationProp})
of Theorem~\ref{thm:ReLUPowersApproxSpaces}: Approximation spaces of $\varrho_{2}$-networks
and $\varrho_{r}$-networks are equal for $r \geq 2$ when $\mathscr{L}$ satisfies a certain
growth condition, showing a saturation from degree $2$ on.
Given this growth condition, for any $r \geq 2$, we obtain the following diagram:
\[
  \begin{array}{ccccc}
    \WASpace[X][\varrho_{1}]
      & \subset &
    \WASpace[X][\varrho_{2}]
    &=&
    \WASpace[X][\varrho_{r}],\\
    \cap & & \cap & & \\
    \NASpace[X][\varrho_{1}]
            & \subset &
    \NASpace[X][\varrho_{2}]
    &=&
    \NASpace[X][\varrho_{r}].
  \end{array}
\]


\subsubsection{Relation to classical function spaces.}

Focusing still on ReLU-networks,
we show that ReLU-networks of bounded depth approximate $C_{c}^{3}(\Omega)$ functions
at bounded rates (Theorem~\ref{thm:UniversalLowerBound}) in the sense that,
for open $\Omega \subset \R^d$ and $L := \sup_{n} \mathscr{L}(n) < \infty$, we prove
\[
  \NASpace[X][\varrho_{1}] \cap C_{c}^{3}(\Omega) = \{0\} \text{ if } \alpha > 2 \cdot (L-1),
  \qquad\!\!\!
  \text{and}
  \qquad\!\!\!
  \WASpace[X][\varrho_{1}] \cap C_{c}^{3}(\Omega) = \{0\} \text{ if } \alpha > 2 \cdot \lfloor L/2 \rfloor.
\]
As classical function spaces (e.g. Sobolev, Besov) intersect $C^{3}_{c}(\Omega)$ nontrivially,
they can only embed into $\WASpace[X][\varrho_{1}]$ or $\NASpace[X][\varrho_{1}]$
if the networks are somewhat deep ($L \geq 1 + \alpha/2$ or $\lfloor L/2 \rfloor \geq \alpha/2$,
respectively), giving some insight about the impact of depth on the expressivity of neural networks.

We then study relations to the classical Besov spaces
${B^{s}_{\sigma,\tau} (\Omega) := B^s_{\tau}(L_\sigma(\Omega;\R))}$.
We establish both \emph{direct estimates}---that is, embeddings of certain Besov spaces
into approximation spaces of $\varrho_{r}$-networks---and \emph{inverse estimates}---%
that is, embeddings of the approximation spaces into certain Besov spaces.

The main result in the regime of direct estimates is Theorem~\ref{thm:better_direct} showing that
if $\Omega \subset \R^d$ is a bounded Lipschitz domain, if
$r \geq 2$, and if $L := \sup_{n \in \N} \mathscr{L}(n)$
satisfies $L \geq 2 + 2 \lceil \log_2 d \rceil$, then
\begin{equation}
  B^{d \alpha}_{p,q} (\Omega) \hookrightarrow \WASpace[X_p(\Omega)][\varrho_r]
  \quad \forall \, p,q \in (0,\infty] \text{ and } 0 < \alpha < \frac{r + \min\{1,p^{-1}\}}{d} .
  \label{eq:IntroductionGeneralDirectEmbedding}
\end{equation}
For large input dimensions $d$, however, the condition $L \geq 2 + 2 \lceil \log_2 d \rceil$
is only satisfied for quite deep networks.
In the case of more shallow networks with $L \geq 3$, the embedding
\eqref{eq:IntroductionGeneralDirectEmbedding} still holds (for any $r \in \N$),
but is only established for
$0 < \alpha < \tfrac{\min \{ 1, p^{-1} \}}{d}$.
Finally, in case of $d = 1$, the embedding \eqref{eq:IntroductionGeneralDirectEmbedding} is
valid as soon as $L \geq 2$ and $r \geq 1$.


Regarding 
 inverse estimates, we first establish limits on possible embeddings
(Theorem~\ref{thm:LimitInverseBesov}).
Precisely, for $\Omega = (0,1)^{d}$ and any $r \in \N$, $\alpha,s \in (0,\infty)$,
and $\sigma, \tau \in (0,\infty]$ we have, with $L := \sup_{n} \mathscr{L}(n) \geq 2$:
\begin{itemize}[leftmargin=0.6cm]
%
  \item if $\alpha < \lfloor L/2\rfloor \cdot \min \{ s, 2 \}$
        then $\WASpace[L_p][\varrho_{r}]$ does \emph{not} embed into $B^{s}_{\sigma,\tau} (\Omega)$;

  \item if $\alpha < (L-1) \cdot \min \{ s, 2 \}$ then $\NASpace[L_p][\varrho_{r}]$
        does \emph{not} embed into $B^{s}_{\sigma,\tau} (\Omega)$.
\end{itemize}
A particular consequence is that for unbounded depth $L = \infty$,
none of the spaces $\WASpace[X][\varrho_{r}]$, $\NASpace[X][\varrho_{r}]$ can embed
into {\em any} Besov space of strictly positive smoothness $s>0$.

For scalar input dimension $d=1$, an embedding into a Besov space with the relation
$\alpha = \lfloor L/2 \rfloor \cdot s$ (respectively $\alpha = (L-1) \cdot s$) is indeed achieved
for $X = L_{p}( (0,1) )$, $0 < p < \infty$, $r \in \N$,
(Theorem~\ref{thm:ApproxSpaceIntoBesovForBoundedDepth}):
\begin{alignat*}{5}
  \WASpace[L_p][\varrho_{r}] \subset B^{s}_{\sigma,\sigma} (\Omega),
  & \quad \text{for each}\ 0<s < r+1,s
  && \quad \alpha := \lfloor L/2\rfloor \cdot s,
  && \quad \sigma := (s+1/p)^{-1},\\
  \NASpace[L_p][\varrho_{r}] \subset B^{s}_{\sigma,\sigma} (\Omega),
  & \quad \text{for each}\ 0< s < r+1,
  && \quad \alpha := (L-1) \cdot s,
  && \quad \sigma := (s+1/p)^{-1}.
\end{alignat*}



\subsection{Expected Impact and Future Directions}

We anticipate our results to have an impact in a number of areas
that we now describe together with possible future directions:

\begin{itemize}[leftmargin=0.6cm]
  \item {\em Theory of Expressivity.}
        We introduce a general framework to study approximation properties
        of deep neural networks from an approximation space viewpoint.
        This opens the door to transfer various results from this part of
        approximation theory to deep neural networks.
        We believe that this conceptually new approach in the theory of expressivity
        will lead to further insight.
        One interesting topic for future investigation is, for instance, to derive
        a finer characterization of the spaces $\WASpace[\StandardXSpace[]][\varrho_{r}]$,
        $\NASpace[\StandardXSpace[]][\varrho_{r}]$, for $r \in \{1,2\}$ (with some assumptions
        on $\mathscr{L}$).

        Our framework is amenable to various extensions; 
        for example the restriction to convolutional weights would allow a study of approximation spaces of convolutional neural networks.

  \item {\em Statistical Analysis of Deep Learning.}
        Approximation spaces characterize fundamental tradeoffs between the complexity of a network
        architecture and its ability to approximate (with proper choices of parameter values)
        a given function $f$.
        In statistical learning, a related question is to characterize which generalization bounds
        (also known as excess risk guarantees) can be achieved when fitting network parameters
        using $m$ independent training samples.
        Some ``oracle inequalities'' \cite{SchmidtHieber:2017vn} of this type
        have been recently established for idealized training algorithms
        minimizing the empirical risk~\eqref{eq:DefERM}.
        Our framework, in combination with existing results on the VC-dimension of neural networks
        \cite{Bartlett:2017ux} is expected to shed new light on such generalization guarantees
        through a generic approach encompassing various types of constraints
        on the considered architecture.

  \item {\em Design of Deep Neural Networks---Architectural Guidelines.}
        Our results reveal how the expressive power of a network architecture may be impacted
        by certain choices such as the presence of certain types of skip connections
        or the selected activation functions.
        Thus, our results provide indications on how a network architecture may be adapted
        without hurting its expressivity, in order to get additional degrees of freedom
        to ease the task of optimization-based learning algorithms and improve their performance.
        For instance, while we show that generalized and strict networks have
        (under mild assumptions on the activation function) the same expressivity,
        we have not yet considered so-called ResNet architectures.
        Yet, the empirical observation that a ResNet architecture makes it easier
        to train deep networks \cite{ResidualNetworks} calls for a better understanding
        of the relations between the corresponding approximations classes.
\end{itemize}


\subsection{Outline}

The paper is organized as follows.

Section~\ref{sec:NetworkCalculus} introduces our notations regarding neural networks and
provides basic lemmata concerning the ``calculus'' of neural networks. The classical notion of
approximation spaces is reviewed in Section~\ref{sec:ApproximationSpaces}, and therein also
specialized to the setting of approximation spaces of networks, with a focus on approximation
in $L_p$ spaces. This is followed by Section~\ref{sec:AppSpacesReLU}, which concentrates on
$\varrho$-networks with $\varrho$ the so-called ReLU or one of its powers.
Finally, Section~\ref{sec:directinverse} studies embeddings between $\WASpace[X][\varrho_{r}]$
(resp.  $\NASpace[X][\varrho_{r}]$) and classical Besov spaces,
with $X = \StandardXSpace[](\Omega)$.

\section{Neural networks and their elementary properties}
\label{sec:NetworkCalculus}




In this section, we formally introduce the definition of neural networks used throughout this
paper, and discuss the elementary properties of the corresponding sets of
functions.

\subsection{Neural networks and their main characteristics}\label{sec:gnndefs}

\begin{defn}[Neural network]\label{defn:NeuralNetworks}

  Let $\varrho : \R \to \R$.
  A \emph{(generalized) neural network with activation function $\varrho$}
  (in short: a \emph{$\varrho$-network}) is a tuple
  $\big( (T_1, \alpha_1), \dots, (T_L, \alpha_L) \big)$, where each
  $T_\ell : \R^{N_{\ell - 1}} \to \R^{N_\ell}$ is an affine-linear map,
  $\alpha_L = \identity_{\R^{N_L}}$, and each function
  $\alpha_\ell : \R^{N_\ell}\to \R^{N_\ell}$ for $1 \leq \ell < L$ is of
  the form $\alpha_\ell = \bigotimes_{j = 1}^{N_\ell} \varrho_j^{(\ell)}$
  for certain $\varrho_j^{(\ell)} \in \{\identity_{\R}, \varrho\}$.
  Here, we use the notation
  \[
    \bigotimes_{j=1}^n \theta_j :
    X_1 \times \cdots \times X_n \to Y_1 \times \cdots \times Y_n,
    (x_1,\dots,x_n) \mapsto \big( \theta_1 (x_1), \dots, \theta_n (x_n) \big)
    \, \text{ for } \, \theta_j : X_j \to Y_j \, .
    \qedhere
  \]
  \end{defn}
  %
  %

\begin{defn}
  A $\varrho$-network as above is called \emph{strict} if
  $\varrho_j^{(\ell)} = \varrho$ for all $1 \leq \ell < L$ and
  $1 \leq j \leq N_\ell$.
\end{defn}

\begin{defn}[Realization of a network]
    The \emph{realization} $\Realization(\Phi)$ of a network
    ${\Phi = \big( (T_1, \alpha_1), \dots, (T_L, \alpha_L) \big)}$ as above is the function
    \[
      \Realization (\Phi)
      : \R^{N_0} \to \R^{N_L} ,
      \quad \text{with} \quad
      \Realization(\Phi)
      := \alpha_L \circ T_L \circ \cdots \circ \alpha_1 \circ T_1 \, .
      \qedhere
    \]
\end{defn}

The {\em complexity} of a neural network is characterized by several features.
\begin{defn}[Depth, number of hidden neurons, number of connections]\label{def:complexity}
    Consider a neural network $\Phi = \big( (T_1, \alpha_1), \dots, (T_L, \alpha_L) \big)$
    with $T_\ell : \R^{N_{\ell - 1}} \to \R^{N_\ell}$ for $1 \leq \ell \leq L$.
    \begin{itemize}[leftmargin=0.6cm]
        \item The \emph{input-dimension} of $\Phi$ is $\InDim (\Phi) := N_0 \in \N$,
              its \emph{output-dimension} is $\OutDim (\Phi) := N_L \in \N$.

        \item The \emph{depth} of $\Phi$ is $L(\Phi) := L \in \N$, corresponding to the
              \emph{number of (affine) layers} of $\Phi$.\\
              We remark that with these notations, the number of \emph{hidden} layers is $L-1$.

        \item The \emph{number of hidden neurons} of $\Phi$ is
              $N(\Phi) := \sum_{\ell = 1}^{L-1} N_\ell \in \N_0$;

        \item The {\em number of connections} (or \emph{number of weights}) of $\Phi$ is
              $W(\Phi) := \sum_{\ell = 1}^{L} \| T_\ell \|_{\ell^0} \in \N_0$,
              with  ${\|T\|_{\ell^0} := \|A\|_{\ell^{0}}}$ for an affine map
              $T: x \mapsto A x + b$ with $A$ some matrix and $b$ some vector;
              here, $\|\cdot\|_{\ell^{0}}$ counts the number of nonzero entries
              in a vector or a matrix.
              \qedhere
    \end{itemize}
    \end{defn}
\begin{rem}
If $W(\Phi) = 0$ then $\Realization(\Phi)$ is constant (but not necessarily zero), and if
  $N(\Phi)=0$, then $\Realization (\Phi)$ is affine-linear (but not necessarily zero or constant).
\end{rem}

Unlike the notation used in \cite{bolcskei2017optimal,PetersenVoigtlaenderReLU}, which considers
$W_0 (\Phi) := \sum_{\ell = 1}^L (\|A^{(\ell)}\|_{\ell^0} + \|b^{(\ell)}\|_{\ell^0})$
where ${T_\ell \, x = A^{(\ell)} x + b^{(\ell)}}$, Definition~\ref{def:complexity} only
counts the nonzero entries \emph{of the linear part} of each
$T_\ell$, so that $W(\Phi) \leq W_{0}(\Phi)$.
Yet, as shown with the following lemma, both definitions are in fact
equivalent up to constant factors if one is only interested in the represented
functions. The proof is in Appendix~\ref{app:PfBiasWeightsDontMatter}.

\begin{lem}\label{lem:BiasWeightsDontMatter}
For any network $\Phi$ there is a ``compressed'' network $\widetilde{\Phi}$
with $\Realization(\, \widetilde{\Phi} \,) = \Realization(\Phi)$ such that
$L(\, \widetilde{\Phi} \,) \leq L(\Phi)$, $N(\, \widetilde{\Phi} \,) \leq N(\Phi)$, and
\[
  W(\widetilde{\Phi})
  \leq W_0 (\, \widetilde{\Phi} \,)
  \leq \OutDim(\Phi) + 2 \cdot W(\Phi) \, .
\]
The network $\widetilde{\Phi}$ can be chosen to be strict if $\Phi$ is strict.
\end{lem}

\begin{rem}
The reason for distinguishing between a neural network and its
associated realization is that for a given function $f : \R^d\to\R^k$,
there might be many different neural networks $\Phi$ with
$f = \Realization (\Phi)$, so that talking about the number of
layers, neurons, or weights of the \emph{function} $f$ is not
well-defined, whereas these notions certainly make sense for neural
networks as defined above. A possible alternative would be to define for example
\[
  L(f)
  := \min
  \big\{
      L(\Phi)
      \, : \,
      \Phi \text{ neural network with }
      \Realization(\Phi) = f
  \big\},
\]
and analogously for $N(f)$ and $W(f)$; but this has the considerable
drawback that it is not clear whether there is a neural network
$\Phi$ that \emph{simultaneously} satisfies e.g.\@ $L(\Phi) = L(f)$
and $W(\Phi) = W(f)$. Because of these issues, we prefer to properly
distinguish between a neural network and its realization.
\end{rem}

\begin{rem}\label{rem:NeuralNetworkDefinition}
  Some of the conventions in the above definitions might appear
  unnecessarily complicated at first sight, but they have been chosen after
  careful thought. In particular:

  \begin{itemize}[leftmargin=0.6cm]
    \item Many neural network architectures used in practice use the same
          activation function for all neurons in a common layer.
          If this choice of activation function even stays the same across all
          layers---except for the last one---one obtains a \emph{strict}
          neural network.

    \item In applications, network architectures very similar to our ``generalized''
          neural networks are used; examples include \emph{residual networks}
          (also called ``ResNets'', see \cite{ResidualNetworks,WiderOrDeeper}),
          and \emph{networks with skip connections} \cite{SkipConnections}.

    \item As expressed in Section~\ref{sec:gnnclosed}, the class of
          realizations of \emph{generalized} neural networks admits nice closure
          properties under linear combinations and compositions of functions.
          Similar closure properties do in general \emph{not} hold for the class
          of strict networks.

    \item The introduction of generalized networks will be justified
          in Section \ref{sub:StrictNetworks}, where we show that if one
          is only interested in approximation theoretic properties of
          the respective function class, then---at least on bounded domains
          $\Omega \subset \R^d$ for ``generic'' $\varrho$,
          but also on unbounded domains for the ReLU activation function and its
          powers---generalized networks and strict networks have identical properties.
          \qedhere
%
%
%
%
  \end{itemize}
\end{rem}

\subsection{Relations between depth, number of neurons, and number of connections}

We now investigate the relationships between the quantities describing the complexity
of a neural network $\Phi = \big( (T_1,\alpha_1), \dots, (T_L, \alpha_L) \big)$
with $T_\ell : \R^{N_{\ell - 1}} \to \R^{N_\ell}$.

Given the number of (hidden) neurons of the network, the other quantities can be bounded.
Indeed, by definition we have $N_\ell \geq 1$ for all $1 \leq \ell \leq L-1$;
therefore, the number of layers satisfies
\begin{equation}
  L(\Phi)
  = 1+\sum_{\ell = 1}^{L-1} 1
  \leq 1+ \sum_{\ell = 1}^{L-1} N_\ell
  = 1+ N(\Phi) \, .
  \label{eq:LayersBoundedByNeurons}
\end{equation}
Similarly, as $\|T_{\ell}\|_{\ell^{0}} \leq N_{\ell-1} N_{\ell}$ for each $1 \leq \ell < L$, we have
\begin{equation}
  W(\Phi)
  = \sum_{\ell = 1}^{L} \|T_{\ell}\|_{\ell^{0}}
  \leq \sum_{\ell = 1}^{L} N_{\ell-1} N_{\ell}
  \leq \sum_{\ell'=0}^{L-1}\sum_{\ell=1}^{L} N_{\ell'}N_{\ell}
  = (d_{\mathrm{in}} (\Phi) +N(\Phi)) (N(\Phi) + d_{\mathrm{out}}(\Phi))\, ,
  \label{eq:WeightsBoundedByNeurons}
\end{equation}
showing that $W(\Phi) = \mathcal{O}([N(\Phi)]^2 + d k)$ for fixed input and output dimensions $d,k$.
When $L(\Phi)=2$ we have in fact
\(
  W(\Phi)
  = \|T_{1}\|_{\ell^{0}}+\|T_{2}\|_{\ell^{0}}\leq N_{0}N_{1}+N_{1}N_{2}
  = (N_{0}+N_{2})N_{1}
  = (d_{\mathrm{in}} (\Phi) + d_{\mathrm{out}}(\Phi)) \cdot N(\Phi)
\).

In general, one \emph{cannot} bound the number of layers or of hidden neurons
by the number of nonzero weights,
as one can build arbitrarily large networks with many ``dead neurons''.
Yet, such a bound is true if one is willing to switch
to a potentially different network \emph{which has the same realization as the original network}.
To show this, we begin with the case of networks with zero connections.

\begin{lem}\label{lem:BoundingConnectionsWithLayers}
  Let $\Phi = \big( (T_1,\alpha_1), \dots, (T_L, \alpha_L) \big)$ be a neural network.
  If there exists some $\ell \in \FirstN{L}$ such that
  $\|T_{\ell}\|_{\ell^{0}} = 0$,
  then $\Realization(\Phi) \equiv c$ for some $c \in \R^{k}$ where $k = d_{\mathrm{out}}(\Phi)$.
\end{lem}

\begin{proof}
  As $\|T_{\ell}\|_{\ell^{0}} = 0$, the affine map $T_{\ell}$ is a constant map
  $\R^{N_{\ell - 1}} \ni y \mapsto b^{(\ell)} \in \R^{N_{\ell}}$.
  Therefore, $f_{\ell} = \alpha_{\ell} \circ T_{\ell}: \R^{N_{\ell-1}} \to \R^{N_{\ell}}$
  is a constant map, so that also
  $\Realization(\Phi) = f_{L} \circ \cdots \circ f_{\ell} \circ \cdots \circ f_{1}$ is constant.
\end{proof}

\begin{cor}\label{cor:BoundingConnectionsWithLayers2}
  If $W(\Phi) < L(\Phi)$ then
 $\Realization(\Phi) \equiv c$ for some $c \in \R^{k}$ where $k = d_{\mathrm{out}}(\Phi)$.
\end{cor}

\begin{proof}
  Let $\Phi = \big( (T_1,\alpha_1), \dots, (T_L,\alpha_L) \big)$ and observe that if
  $\sum_{\ell=1}^L \|T_{\ell}\|_{\ell^{0}} = W(\Phi)
  < L(\Phi) = \sum_{\ell=1}^{L} 1$
  then there must exist $\ell \in \FirstN{L}$ such that $\|T_{\ell}\|_{\ell^{0}}=0$,
  so that we can apply Lemma~\ref{lem:BoundingConnectionsWithLayers}.
\end{proof}

Indeed, constant maps play a special role as they are exactly the set of
realizations of neural networks with no (nonzero) connections.
Before formally stating this result, we introduce notations for families of neural networks
of constrained complexity, which can have a variety of shapes as illustrated on Figure~\ref{fig:NetworkVariety}.
\begin{defn}
\label{def:NNsymbol}
  Consider $L \in \N \cup \{\infty\}$, $W,N \in \N_0 \cup \{\infty\}$, and $\Omega \subseteq \R^{d}$
  a non-empty set.
  \begin{itemize}[leftmargin=0.6cm]
    \item $\NNsymbol_{W,L,N}^{\varrho,d,k}$ denotes the set of all
          generalized $\varrho$-networks $\Phi$
          with input dimension $d$, output dimension $k$,
          and with $W(\Phi) \leq W$, $L(\Phi) \leq L$, and $N(\Phi) \leq N$.

    \item $\SNNsymbol_{W,L,N}^{\varrho,d,k}$ denotes the subset of networks
          $\Phi \in \NNsymbol_{W,L,N}^{\varrho,d,k}$ which are strict.

    \item The class of all functions $f : \R^d \to \R^k$ that can be represented
          by (generalized) $\varrho$-networks with at most $W$ weights,
          $L$ layers, and $N$ neurons is
          \[
            \NNreal_{W,L,N}^{\varrho,d,k}
            := \big\{
                \Realization (\Phi) \,:\, \Phi \in \NNsymbol_{W,L,N}^{\varrho,d,k}
               \big\} .
          \]
          The set of all restrictions of such functions to $\Omega$
          is denoted $\NNreal_{W,L,N}^{\varrho,d,k}(\Omega)$.

  \item Similarly
        \[
          \SNNreal_{W,L,N}^{\varrho,d,k}
          := \big\{
                  \Realization (\Phi)
                  \,:\,
                  \Phi \in \SNNsymbol_{W,L,N}^{\varrho, d, k}
             \big\} .
        \]
        The set of all restrictions of such functions to $\Omega$
        is denoted $\SNNreal_{W,L,N}^{\varrho,d,k}(\Omega)$.
\end{itemize}
Finally, we define
$\NNreal_{W,L}^{\varrho,d,k} := \NNreal^{\varrho,d,k}_{W,L,\infty}$
and $\NNreal_{W}^{\varrho,d,k} := \NNreal^{\varrho,d,k}_{W,\infty,\infty}$,
as well as ${\NNreal^{\varrho,d,k} := \NNreal^{\varrho,d,k}_{\infty,\infty,\infty}}$.
We will use similar notations for $\SNNreal$, $\NNsymbol$, and $\SNNsymbol$.
\end{defn}

\begin{rem}
  If the dimensions $d,k$ and/or the activation function $\varrho$ are implied
  by the context, we will sometimes omit them from the notation.
\end{rem}

\begin{figure}[htbp]
\begin{center}
\includegraphics[width=0.65\textwidth]{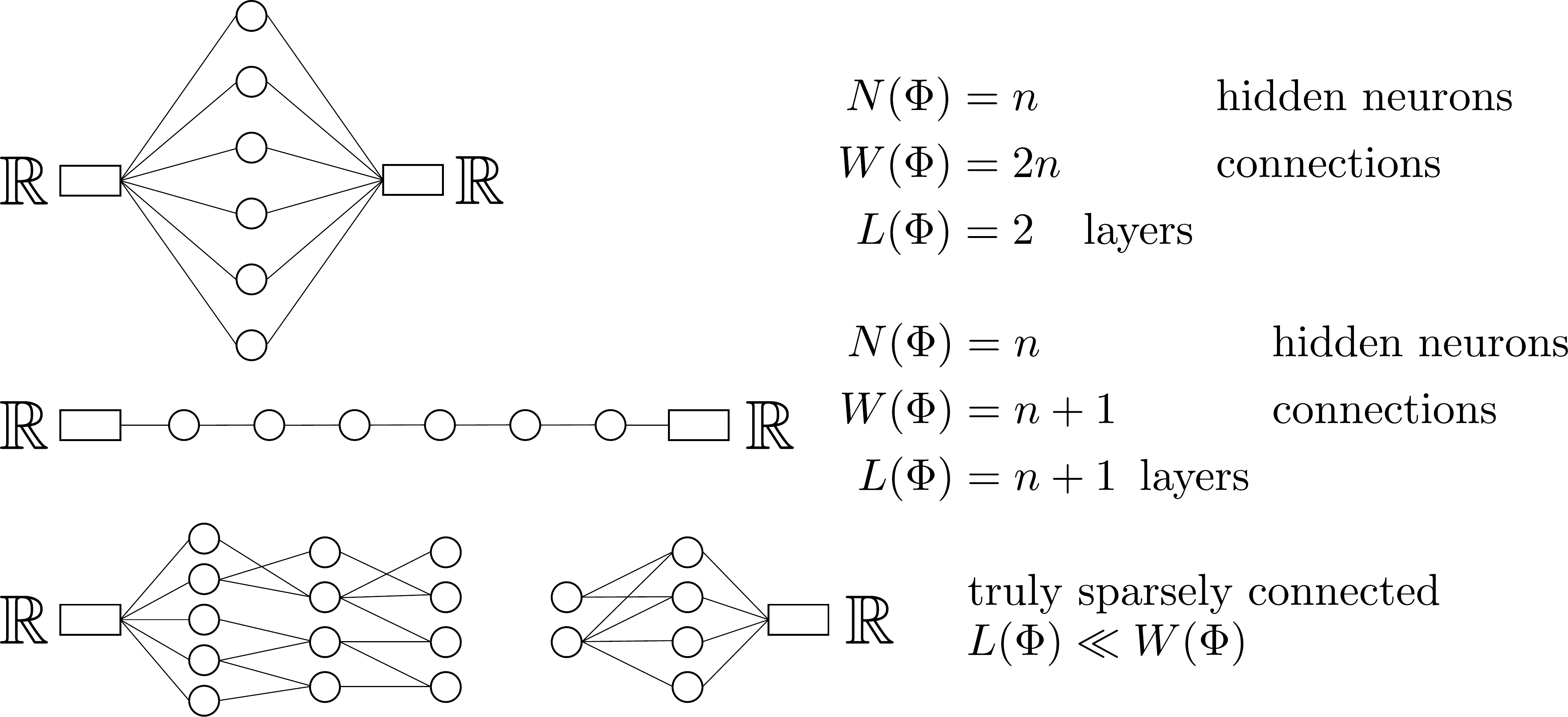}
\caption{\label{fig:NetworkVariety}
The considered network classes include a variety of networks such as:
(top) shallow networks with a single hidden layer, where the number of neurons
is of the same order as the number of possible connections;
(middle) ``narrow and deep'' networks, e.g.~with a single neuron per layer,
where the same holds;
(bottom) ``truly'' sparse networks that have much fewer nonzero weights than potential connections.}
\end{center}
\end{figure}


\begin{lem}\label{lem:ConstantMaps}
  Let  $\varrho: \R \to \R$, and let $d,k \in \N$, $N \in \N_{0} \cup \{\infty\}$,
  and $L \in \N \cup \{\infty\}$ be arbitrary. Then
  \[
    \NNreal_{0,L,N}^{\varrho,d,k}
    = \SNNreal_{0,L,N}^{\varrho,d,k}
    = \NNreal_{0,1,0}^{\varrho,d,k}
    = \SNNreal_{0,1,0}^{\varrho,d,k}
    = \{f: \R^{d} \to \R^{k} \,\mid\, \exists c \in \R^{k}: f \equiv c\} \, .
    \qedhere
  \]
\end{lem}

\begin{proof}
  If $f \equiv c$ where $c \in \R^{k}$ then the affine map
  $T : \R^d \to \R^{k}, x \mapsto c$ satisfies $\| T \|_{\ell^0} = 0$
  and the (strict) network $\Phi := \big( (T, \identity_{\R^{k}}) \big)$
  satisfies $\Realization (\Phi) \equiv c = f$, $W(\Phi) = 0$, $N(\Phi) = 0$ and
  $L(\Phi) = 1$.
  By Definition~\ref{def:NNsymbol}, we have $\Phi \in \SNNsymbol_{0,1,0}^{\varrho,d,k}$
  whence $f \in \SNNreal_{0,1,0}^{\varrho,d,k}$.
  The inclusions $\SNNreal_{0,1,0}^{\varrho,d,k}
  \subset \NNreal_{0,1,0}^{\varrho,d,k} \subset \NNreal_{0,L,N}^{\varrho,d,k}$
  and $\SNNreal_{0,1,0}^{\varrho,d,k} \subset \SNNreal_{0,L,N}^{\varrho,d,k}
  \subset \NNreal_{0,L,N}^{\varrho,d,k}$ are trivial by definition of these sets.
  If $f \in \NNreal_{0,L,N}^{\varrho,d,k}$ then there is
  $\Phi \in \NNsymbol_{0,L,N}^{\varrho,d,k}$ such that $f = \Realization(\Phi)$.
  As $W(\Phi) = 0 < 1 \leq L(\Phi)$,
  Corollary~\ref{cor:BoundingConnectionsWithLayers2} yields
  $f = \Realization(\Phi) \equiv c$.
\end{proof}

Our final result in this subsection shows that any realization
of a network with at most $W \geq 1$ connections can also be obtained by a network
with $W$ connections but which additionally has
at most $L \leq W$ layers and at most $N \leq W$ hidden neurons.
The proof is postponed to Appendix~\ref{app:PfBoundingLayersAndNeuronsByWeights}.
\begin{lem}\label{lem:BoundingLayersAndNeuronsByWeights}
  Let $\varrho : \R \to \R$, $d,k \in \N$, $L \in \N \cup \{\infty\}$,
  and $W \in \N$ be arbitrary. Then we have
  \[
    \NNreal_{W,L,\infty}^{\varrho,d,k}
    = \NNreal_{W, L,W}^{\varrho,d,k}
    \subset \NNreal_{W,W,W}^{\varrho,d,k} \, . 
  \]
  The inclusion is an equality for $L \geq W$. In particular,
  $\NNreal_{W}^{\varrho,d,k} = \NNreal_{W,\infty,W}^{\varrho,d,k}
  = \NNreal_{W,W,W}^{\varrho,d,k}$.
  The same claims are valid for \emph{strict} networks, replacing the
  symbol $\NNreal$ by $\SNNreal$ everywhere.
\end{lem}

To summarize, for given input and output dimensions $d,k$,
when combining~\eqref{eq:WeightsBoundedByNeurons} with the above lemma,
we obtain that for any network $\Phi$ there exists a network $\Psi$
with $\Realization(\Psi) = \Realization(\Phi)$ and $L(\Psi) \leq L(\Phi)$, and such that
\begin{equation}
  N(\Psi) \leq W(\Psi) \leq W(\Phi) \leq N^{2}(\Phi) + (d+k) N(\Phi) + dk.
  \label{eq:WeightsVsNeurons}
\end{equation}
When $L(\Phi) = 2$ we have in fact $N(\Psi) \leq W(\Psi) \leq W(\Phi) \leq (d+k) N(\Phi)$;
see the discussion after \eqref{eq:WeightsBoundedByNeurons}.

\begin{rem}{(\em Connectivity, flops and bits.)}
  A motivation for measuring a network's complexity by its connectivity is
  that the number of connections is directly related to several practical quantities
  of interest such as the number of floating point operations needed to compute the output
  given the input, or the number of bits needed to store a (quantized) description of the network
  in a computer file.
  This is not the case for complexity measured in terms of the number of neurons.
\end{rem}


\subsection{Calculus with generalized neural networks}
\label{sec:gnnclosed}

In this section, we show as a consequence of
Lemma~\ref{lem:BoundingLayersAndNeuronsByWeights} that the
class of realizations of generalized neural networks of a given \emph{complexity}---as measured
by the number of connections $W(\Phi)$---is closed under addition and composition, as long as
one is willing to increase the complexity by a constant factor.
To this end, we first show that \emph{one can increase the depth of
generalized neural networks with controlled increase of the required complexity}.

\begin{lem}
 \label{lem:DeepeningLemma}
  Given $\varrho:\R\to\R$, ${d,k \in \N}$, $c := \min \{d,k\}$, $\Phi \in \NNsymbol^{\varrho,d,k}$,
  and $L_0 \in \N_{0}$, there exists $\Psi \in \NNsymbol^{\varrho,d,k}$ such that
  $\Realization(\Psi) = \Realization(\Phi)$,
  $L(\Psi) = L(\Phi) + L_0$, $W(\Psi) = W(\Phi) + c L_0$, $N(\Psi) = N(\Phi) + c L_0$.
\end{lem}

\begin{figure}[htbp]
\begin{center}
\begin{minipage}{.55\textwidth}
\includegraphics[width=\textwidth]{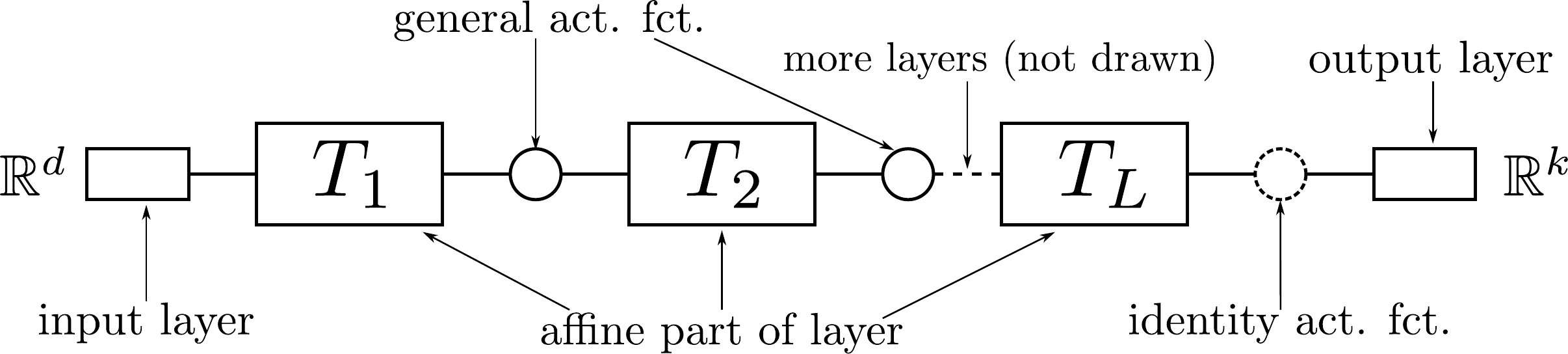}
\end{minipage}
\,\quad\,
\begin{minipage}{.4\textwidth}
\includegraphics[width=\textwidth]{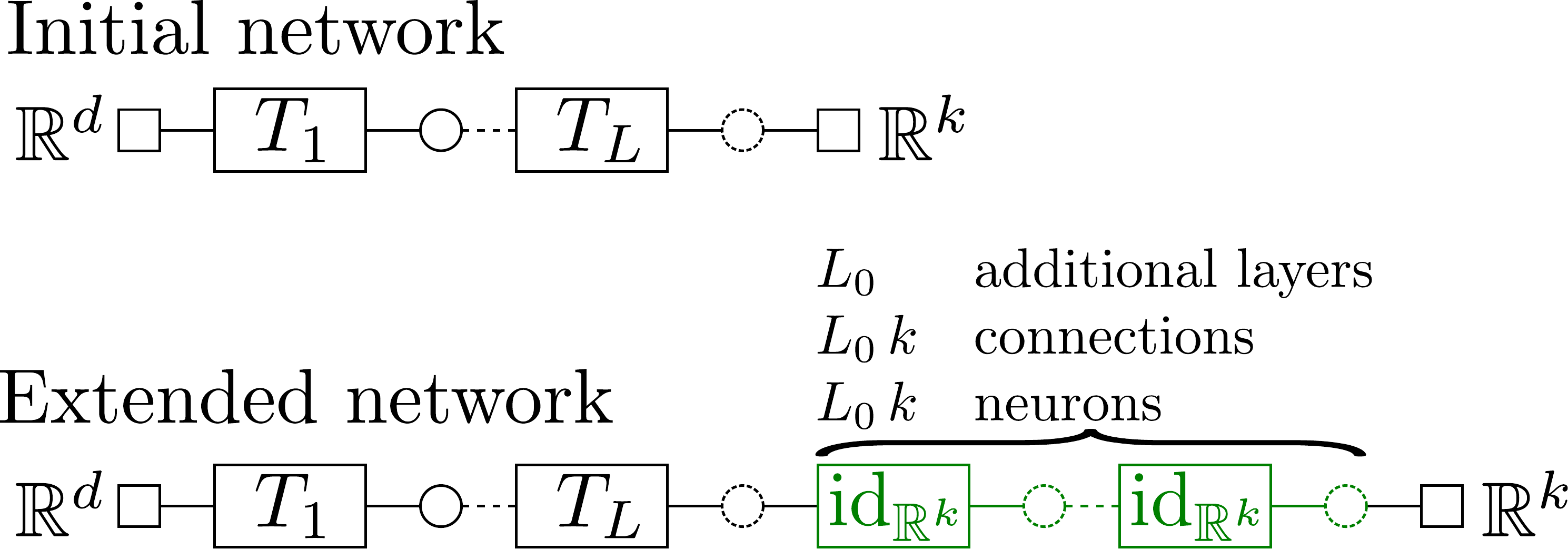}
\end{minipage}
\caption{\label{fig:sync}
(left) Graphical convention for drawing neural networks; this convention is used everywhere
except in Figure~\ref{fig:NetworkVariety}.
(right) Depth synchronization of Lemma~\ref{lem:DeepeningLemma},
identity layers are added at the output if $k < d$; in case of $d < k$ they are added at the input.}
\end{center}
\end{figure}
\ifarxiv
This fact appears without proof in \cite[Section 5.1]{SchmidtHieber:2017vn}
under the name of {\em depth synchronization} for strict networks with the ReLU activation function,
with $c = d$.
We refine it to $c = \min\{d,k\}$ and give a proof for {\em generalized networks}
with {\em arbitrary} activation function in Appendix~\ref{app:PfDeepeningLemma}.
\fi
The underlying proof idea is illustrated in Figure~\ref{fig:sync}.

\begin{figure}[htbp]
\begin{center}
\begin{minipage}{\textwidth}
\includegraphics[width=\textwidth]{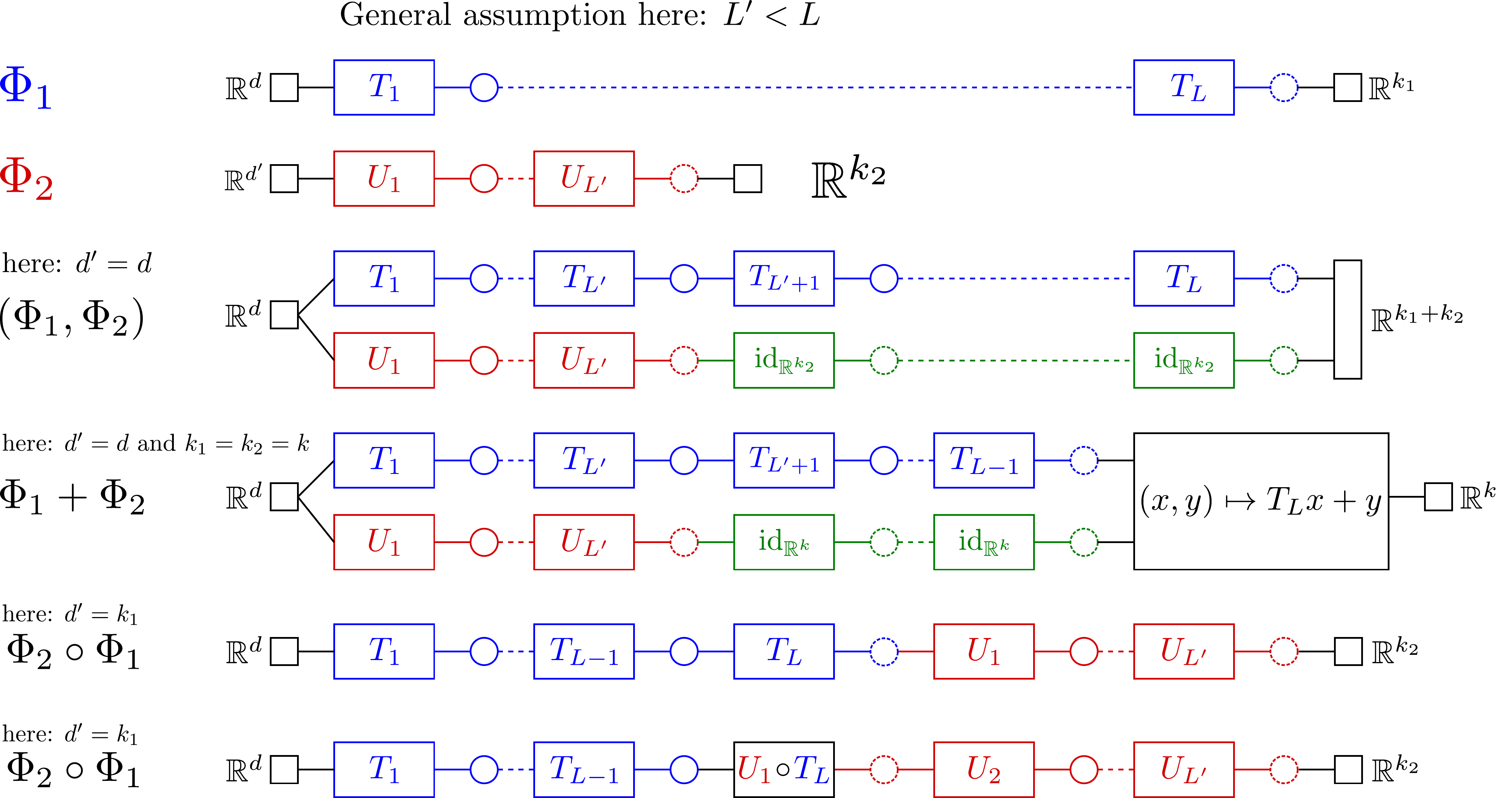}
\end{minipage}
\caption{\label{fig:calculus}
Illustration of the networks constructed in the proofs of Lemmas~\ref{lem:SummationLemma}
and \ref{lem:NetworkCalculus}.
(top) Implementation of Cartesian products;
(middle) Implementation of addition;
(bottom) Implementation of composition.}
\end{center}
\end{figure}
A consequence of the depth synchronization property is that the class of generalized networks
is closed under linear combinations and Cartesian products.
The proof idea behind the following lemma,
\ifarxiv whose proof is in Appendix~\ref{app:PfSummationLemma}
\fi
is illustrated in Figure~\ref{fig:calculus} (top and middle).
%
%

\begin{lem} \label{lem:SummationLemma}
  Consider arbitrary $d,k,n \in \N$, $c \in \R$, $\varrho:\R \to \R$,
  and $k_i \in \N$ for $i \in \FirstN{n}$.
  \begin{enumerate}[leftmargin=0.6cm]
    \item \label{enu:ScalMult}
          If $\Phi \in\NNsymbol^{\varrho,d,k}$ then $c \cdot \Realization(\Phi) = \Realization(\Psi)$
          where $\Psi \in \NNsymbol^{\varrho,d,k}$ satisfies $W(\Psi) \leq W(\Phi)$
          (with equality if $c \neq 0$), $L(\Psi) = L(\Phi)$, $N(\Psi) = N(\Phi)$.
%
          The same holds with $\SNNsymbol$ instead of $\NNsymbol$.

    \item \label{enu:Cartesian}
          If $\Phi_i \in \NNsymbol^{\varrho, d, k_i}$ for $i \in \FirstN{n}$, then
          $(\Realization(\Phi_{1}),\ldots,\Realization(\Phi_{n})) = \Realization(\Psi)$
          with $\Psi \in\NNsymbol^{\varrho,d,K}$, where
          \begin{equation*}
            L(\Psi) = \max_{i = 1,\dots,n} L(\Phi_{i}),
            \quad
            W(\Psi) \leq \delta+\sum_{i=1}^{n} W(\Phi_{i}),
            \quad
            N(\Psi) \leq \delta+\sum_{i=1}^{n} N(\Phi_{i}),
            \quad \text{and} \quad
            K:=\sum_{i=1}^{n} k_{i},
          \end{equation*}
          with $\delta := c \cdot \big(\max_{i=1,\dots,n} L(\Phi_{i})-\min_{i} L(\Phi_{i})\big)$
          and $c := \min \{d, K-1 \}$.

   \item \label{enu:LinComb}
         If $\Phi_1,\dots,\Phi_n \in \NNsymbol^{\varrho, d, k}$, then
         $\sum_{i=1}^n \Realization(\Phi_{i}) = \Realization(\Psi)$
         with $\Psi \in\NNsymbol^{\varrho,d,k}$, where
         \begin{equation*}
           L(\Psi) = \max_{i} L(\Phi_{i}),
           \quad
           W(\Psi) \leq \delta + \sum_{i=1}^{n} W(\Phi_{i}),
           \qquad \text{and} \quad
           N(\Psi) \leq \delta + \sum_{i=1}^{n} N(\Phi_{i}),
         \end{equation*}
         with $\delta := c \left(\max_{i} L(\Phi_{i})-\min_{i} L(\Phi_{i})\right)$
         and $c := \min \{ d,k \}$.
         \qedhere
  \end{enumerate}
\end{lem}
%
One can also control the complexity of certain networks resulting from compositions in an intuitive way.
To state and prove this, we introduce a convenient notation:
For a matrix $A \in \R^{n \times d}$, we denote
\begin{equation}
  \| A \|_{\ell^{0,\infty}}
  := \max_{i \in \FirstN{d}} \| A \, e_i \|_{\ell^0}
  \quad \text{and} \quad
  \| A \|_{\ell^{0,\infty}_\ast}
  := \| A^T \|_{\ell^{0,\infty}}
  =  \max_{i \in \FirstN{n}} \| e_i^T A \|_{\ell^0},
  \label{eq:DefL0MixedNorms}
\end{equation}
where $e_1,\dots, e_n$ is the standard basis of $\R^n$.
Likewise, for an affine-linear map $T = A \bullet + b$, we denote
$\| T \|_{\ell^{0,\infty}} := \| A \|_{\ell^{0,\infty}}$ and
$\| T \|_{\ell^{0,\infty}_{\ast}} := \| A \|_{\ell^{0,\infty}_{\ast}}$.
\begin{lem}\label{lem:NetworkCalculus}
  Consider arbitrary
  $d,d_1,d_2,k,k_1 \in \N$ 
  and
  $\varrho:\R \to \R$.
  \begin{enumerate}[leftmargin=0.6cm]
    \item \label{enu:PrePostAffine} If $\Phi \in \NNsymbol^{\varrho,d,k}$
          and $P: \R^{d_{1}} \to \R^{d}$, $Q:\R^{k} \to \R^{k_{1}}$
          are two affine maps then $Q \circ \Realization(\Phi) \circ P = \Realization(\Psi)$
          where $\Psi \in \NNsymbol^{\varrho,d_{1},k_{1}}$ with $L(\Psi)= L(\Phi)$,
          $N(\Psi)= N(\Phi)$ and
          \[
            W(\Psi)
            \leq \| Q \|_{\ell^{0,\infty}} \cdot W(\Phi) \cdot \| P \|_{\ell^{0,\infty}_{\ast}}\, .
          \]
          The same holds with $\SNNsymbol$ instead of $\NNsymbol$.

    \item \label{enu:Compos}
          If $\Phi_1 \in \NNsymbol^{\varrho, d, d_1}$ and $\Phi_2 \in \NNsymbol^{\varrho, d_1, d_2}$ then
          $\Realization(\Phi_2) \circ \Realization(\Phi_1) = \Realization(\Psi)$
          where $\Psi \in \NNsymbol^{\varrho, d, d_2}$ and
          \[
            W(\Psi) = W(\Phi_{1})+W(\Phi_{2}),
            \quad
            L(\Psi) = L(\Phi_1)+L(\Phi_2),
            \quad
            N(\Psi) = N(\Phi_1)+N(\Phi_2)+d_1.
          \]
%
    \item \label{enu:ComposLessDepth}
          Under the assumptions of Part~\eqref{enu:Compos},
          there is also $\Psi' \in \NNsymbol^{\varrho, d, d_2}$
          such that $\Realization(\Phi_2) \circ \Realization(\Phi_1) = \Realization(\Psi')$ and
          \[
            \quad \,\,\,\,
            W(\Psi') \leq W(\Phi_{1}) + \max \{ N(\Phi_{1}),d \} \, W(\Phi_{2}),
            \quad
            L(\Psi') = \! L(\Phi_1) \!+\! L(\Phi_2) \!-\! 1,
            \quad
            N(\Psi') = N(\Phi_1) \!+\! N(\Phi_2).
          \]
          In this case, the same holds for $\SNNsymbol$ instead of $\NNsymbol$.
          \qedhere
  \end{enumerate}
\end{lem}

\noindent
The proof idea of Lemma~\ref{lem:NetworkCalculus} is illustrated in Figure~\ref{fig:calculus} (bottom).
\ifarxiv The formal proof is in Appendix~\ref{app:PfTechnicalA}.
\fi
A direct consequence of Lemma~\ref{lem:NetworkCalculus}-(\ref{enu:PrePostAffine}) that we will use
in several places is that ${x \mapsto a_2 \, g(a_{1}x+b_{1})+b_{2} \in \NNreal^{\varrho,d,k}_{W,L,N}}$
whenever $g \in \NNreal^{\varrho,d,k}_{W,L,N}$, $a_{1},a_{2} \in \R$, $b_{1} \in \R^{d}$, $b_{2} \in \R^{k}$.

Our next result shows that if $\sigma$ can be expressed as the realization of a $\varrho$-network
then realizations of $\sigma$-networks can be re-expanded into realizations of $\varrho$-networks
of controlled complexity.

\begin{lem}\label{lem:RecursiveNNsets}
  Consider two activation functions $\varrho,\sigma$ such that
  $\sigma = \Realization(\Psi_{\sigma})$ for some
  \(
    \Psi_{\sigma} \in \NNsymbol^{\varrho,1,1}_{w,\ell,m}
  \)
  with $L(\Psi_{\sigma}) = \ell \in \N$, $w \in \N_{0}$, $m \in \N$.
  Furthermore, assume that $\sigma \not\equiv \mathrm{const}$.

  Then the following hold:
  \begin{enumerate}[leftmargin=0.6cm]
    \item if $\ell=2$ then for any $W,N,L,d,k$ we have
          \(
            \NNreal_{W,L,N}^{\sigma,d,k} \subset \NNreal_{Wm^{2},L,Nm}^{\varrho,d,k}
          \)

    \item for any $\ell,W,N,L,d,k$ we have
          \(
            \NNreal_{W,L,N}^{\sigma,d,k}
            \subset \NNreal_{mW + w N, 1 + (L-1)\ell, N(1+m)}^{\varrho,d,k} \, .
          \)
          \qedhere
  \end{enumerate}
\end{lem}


The proof of Lemma~\ref{lem:RecursiveNNsets} is in Appendix~\ref{app:RecursiveNNsets}.
In the case when $\sigma$ is simply an $s$-fold composition of $\varrho$,
we have the following improvement of Lemma~\ref{lem:RecursiveNNsets}.

\begin{lem}\label{lem:NestednessBasic}
  Let $s \in \N$.
  Consider an activation function $\varrho : \R \to \R$,
  and let $\sigma := \varrho \circ \cdots \circ \varrho$,
  where the composition has $s$ ``factors''.
  We have
  \[
    \NNreal^{\sigma,d,k}_{W,L,N} \subset \NNreal^{\varrho,d,k}_{W+(s-1)N,1+s(L-1),sN}
    \quad \forall \, W, N \in \N_0 \cup \{\infty\} \text{ and } L \in \N \cup \{\infty\}.
  \]
  The same holds for strict networks, replacing $\NNreal$ by $\SNNreal$ everywhere.
\end{lem}

The proof is in Appendix~\ref{app:NestednessBasic}.
In our next result, we consider the case where $\sigma$ cannot be exactly implemented
by $\varrho$-networks, but only approximated arbitrarily well by such networks of
uniformly bounded complexity.

\begin{lem}\label{lem:Recursivity}
  Consider two activation functions $\varrho, \sigma : \R \to \R$.
  Assume that $\sigma$ is continuous and that there are $w,m \in \N_{0}$, $\ell \in \N$ and a family
  $\Psi_{h} \in \NNsymbol_{w,\ell,m}^{\varrho,1,1}$ parameterized by $h \in \R$, with $L(\Psi_{h}) = \ell$,
  such that $\sigma_{h} := \Realization(\Psi_{h}) \xrightarrow[h \to 0]{} \sigma$
  locally uniformly on $\R$.
  For any $d,k \in \N$, $W,N \in \N_{0}$, $L \in \N$ we have
  \begin{equation}
    \NNreal^{\sigma,d,k}_{W,L,N}
    \subset
    \begin{cases}
      \overline{\NNreal^{\varrho,d,k}_{Wm^{2},L,Nm}},                   & \text{if } \ell=2; \\
      \overline{\NNreal^{\varrho,d,k}_{mW+wN,1+(L-1)\ell,N(1+m)}}, & \text{for any } \ell,
    \end{cases}
  \end{equation}
  where the closure is with respect to locally uniform convergence.
\end{lem}

The proof is in Appendix~\ref{app:Recursivity}.
In the next lemma, we establish a relation between the approximation capabilities
of strict and generalized networks.
The proof is given in Appendix~\ref{app:PfGeneralizedVSStrict}.

\begin{lem}\label{lem:GeneralizedVSStrict}
  Let $\varrho : \R \to \R$ be \emph{continuous} and assume that $\varrho$ is differentiable
  at some $x_0 \in \R$ with $\varrho' (x_0) \neq 0$.
  For any $d,k \in \N$, $L \in \N \cup \{\infty\}$, $N \in \N_0 \cup \{\infty\}$,
  and $W \in \N_0$ we have
  \[
    \NNreal_{W,L,N}^{\varrho,d,k} \subset \overline{\SNNreal_{4W,L,2N}^{\varrho,d,k}},
  \]
  where the closure is with respect to locally uniform convergence.
\end{lem}

\subsection{Networks with activation functions that can represent the identity}

The convergence in Lemma~\ref{lem:GeneralizedVSStrict} is only locally uniformly,
which is not strong enough to ensure equality of the associated approximation spaces
on \emph{unbounded} domains.
In this subsection we introduce a certain condition on the activation functions
which ensures that strict and generalized networks yield the same approximation spaces
also on unbounded domains.

\begin{defn}\label{defn:IdentityRepresentationDefinition}
  We say that a function $\varrho : \R\to\R$ \emph{can represent $f: \R \to \R$ with $n$ terms}
  (where $n \in \N$) if $f \in \SNNreal^{\varrho,1,1}_{\infty,2,n}$;
  that is, if there are $a_i, b_i, c_i \in \R$ for
  $i \in \FirstN{n}$, and some $c \in \R$ satisfying
  \[
   f(x) = c + \sum_{i=1}^n a_i \cdot \varrho (b_i \, x + c_i)
    \qquad \forall \, x \in \R \, .
  \]
  A particular case of interest is when $\varrho$
  \emph{can represent the identity $\identity: \R \to \R$ with $n$ terms}.
\end{defn}

\noindent As shown in Appendix~\ref{app:ReLUPowerRepresentsIdentity},
primary examples are the ReLU activation function and its powers.

\begin{lem}\label{lem:ReLUPowerRepresentsIdentity}
  For any $r \in \N$,
   $\varrho_r$ can represent any polynomial of degree $\leq r$ with $2r + 2$ terms.
\end{lem}

\begin{lem}\label{lem:IdentityRectifierForFree}
  Assume that $\varrho : \R \to \R$ can represent the identity with $n$ terms.
  Let $d,k \in\N$, $W, N \in \N_0$, and $L \in \N \cup \{\infty\}$ be arbitrary.
  Then
  \(
      \NNreal_{W,L,N}^{\varrho,d,k}
      \subset \SNNreal_{n^2 \cdot W, L, n \cdot N}^{\varrho,d,k}
  \).
  \qedhere
\end{lem}

\noindent The proof of Lemma~\ref{lem:IdentityRectifierForFree} is in
Appendix~\ref{app:PfGeneralizedVSStrict}.
The next lemma is proved in Appendix~\ref{app:MultNetwork}.

\begin{lem}\label{lem:MultNetwork}
  If  $\varrho: \R \to \R$ can represent all polynomials of degree two with $n$ terms, then:
  \begin{enumerate}[leftmargin=0.6cm]
    \item For $d \in \N_{\geq 2}$ the multiplication function
          $M_{d}: \R^{d} \to \R, x \mapsto \prod_{i=1}^{d} x_{i}$ satisfies
          \[
            M_{d} \in \NNreal^{\varrho,d,1}_{6n(2^{j}-1),2j,(2n+1)(2^{j}-1)-1}
            \quad \text{with}\quad
            j = \lceil \log_{2} d \rceil.
          \]
          In particular, for $d =2$ we have $M_2 \in \SNNreal^{\varrho,d,1}_{6n,2,2n}$.

    \item For  $k \in \N$  the multiplication map
          $m : \R \times \R^k \to \R^k, (x,y) \mapsto x \cdot y$ satisfies
          \(
            m \in \NNreal^{\varrho, 1+k, k}_{6kn, 2,2kn}
          \).
          \qedhere
  \end{enumerate}
\end{lem}

\section{Neural network approximation spaces}
\label{sec:ApproximationSpaces}

The overall goal of this paper is to study \emph{approximation spaces} associated to the sequence
of sets $\AppSet_{n}$ of realizations of networks with at most $n$ connections
(resp.~at most $n$ neurons), $n \in \N_{0}$, either for fixed network depth $L \in \N$,
or for unbounded depth $L = \infty$, or even for varying depth $L = \mathscr{L}(n)$.

In this section, we first formally introduce these approximation spaces,
following the theory from \cite[Chapter 7, Section 9]{ConstructiveApproximation},
and then specialize these spaces to the context of neural networks.
The next sections will be devoted to establishing embeddings between classical functions spaces
and neural network approximation spaces, as well as nesting properties between such spaces.

\subsection{Generic tools from approximation theory}
\label{sub:ApproxTheoryTools}


Consider a quasi-Banach
\footnote{See e.g.
 \cite[Section 3]{MR2029742}
 for reminders
on quasi-norms and quasi-Banach spaces.}
space $X$ equipped with the quasi-norm $\|\cdot\|_{X}$, and let $f \in X$.
The \emph{error of best approximation} of $f$ from a nonempty set $\Gamma \subset X$ is
\begin{equation}
  \AppErr (f,\Gamma)_X := \inf_{g \in \Gamma} \| f - g \|_{X} \in [0,\infty) \, .
  \label{eq:ApproximationErrorDefinition}
\end{equation}
In case of $X = \StandardXSpace (\Omega)$ (as in Equation~\eqref{eq:StandardXSpace})
with $\Omega \subseteq \R^{d}$ a set of nonzero measure,
the corresponding approximation error will be denoted by $\AppErr(f,\Gamma)_{p}$.
As in \cite[Chapter 7, Section 9]{ConstructiveApproximation},
we consider an arbitrary family $\AppSet = (\AppSet_{n})_{n \in \N_{0}}$
of subsets $\AppSet_{n} \subset X$ and define for $f \in X$, $\alpha \in (0,\infty)$,
and $q \in (0,\infty]$ the following quantity
(which will turn out to be a quasi-norm under mild assumptions on the family $\Sigma$):
\[
  \| f \|_{\GenApproxSpace}
  := \begin{cases}
      \displaystyle{\left(
                      \sum_{n=1}^\infty
                        [n^\alpha \cdot \AppErr (f,\AppSet_{n-1})_{X}]^q
                        \frac{1}{n}
                    \right)^{1/q} \in [0,\infty]},
      & \text{if } 0 < q < \infty, \\[0.5cm]
      \,\,\, \displaystyle{\sup_{n \geq 1} \,\,
                             [n^{\alpha} \cdot \AppErr (f,\AppSet_{n-1})_{X}] \in [0,\infty]},
      & \text{if } q = \infty.
     \end{cases}
\]
As expected, the associated \emph{approximation class} is simply
\[
  \GenApproxSpace := \big\{ f \in X \, : \, \| f \|_{\GenApproxSpace} < \infty \big\} \, .
\]
For $q = \infty$, this class is precisely the subset of elements $f \in X$
such that $\AppErr(f,\AppSet_{n})_X = \mathcal{O}(n^{-\alpha})$,
and the classes associated to $0<q<\infty$ correspond to subtle variants of this subset.
If we assume that $\Sigma_n \subset \Sigma_{n+1}$ for all $n \in \N_0$,
then the following ``embeddings'' can be derived directly from the definition;
see \cite[Chapter 7, Equation (9.2)]{ConstructiveApproximation}:
\begin{equation}
  \GenApproxSpace[\alpha][q] \hookrightarrow \GenApproxSpace[\beta][s],
  \quad \text{if} \quad \alpha > \beta
  \quad \text{or if} \quad \alpha = \beta \quad \text{and} \quad q \leq s .
  \qedhere
  \label{eq:ApproximationSpaceEmbedding}
\end{equation}
Note that we do not yet know that the approximation classes
$\GenApproxSpace[\alpha][q]$ are (quasi)-Banach spaces.
Therefore, the notation $X_{1}\hookrightarrow X_{2}$---where for $i \in \{1,2\}$
we consider the class $X_{i} := \{x \in X: \|x\|_{X_{i}}<\infty\}$ associated to some
``proto''-quasi-norm $\|\cdot\|_{X_{i}}$---simply means that $X_{1} \subset X_{2}$ and
$\|\cdot\|_{X_{2}} \leq C \cdot \|\cdot\|_{X_{1}}$, even though $\|\cdot\|_{X_{i}}$
might not be proper (quasi)-norms and $X_{i}$ might not be (quasi)-Banach spaces.
When the classes are indeed (quasi)-Banach spaces (see below),
this corresponds to the standard notion of a continuous embedding.

As a direct consequence of the definitions,
we get the following result on the relation between approximation classes using
\emph{different families} of subsets.

\begin{lem}\label{lem:ApproximationSpaceElementaryNesting}
  Let $X$ be a quasi-Banach space,
  and let $\Sigma = (\Sigma_n)_{n \in \N_0}$ and $\Sigma' = (\Sigma_n')_{n \in \N_0}$
  be two families of subsets $\Sigma_n, \Sigma_n' \subset X$ satisfying the following properties:
  \begin{enumerate}
    \item $\Sigma_0 = \{0\} = \Sigma_0'$;

    \item $\Sigma_n \subset \Sigma_{n+1}$ and $\Sigma_n' \subset \Sigma_{n+1}'$
          for all $n \in \N_0$; and

    \item there are $c \in \N$ and $C > 0$ such that
          $\AppErr(f,\Sigma_{c m})_{X} \leq C \cdot \AppErr(f,\Sigma'_{m})_{X}$
          for all $f \in X, m \in \N$.
  \end{enumerate}
  Then $\GenApproxSpaceSet[\Sigma'] \hookrightarrow \GenApproxSpaceSet[\Sigma]$
  holds for arbitrary $q \in (0,\infty]$ and $\alpha > 0$.
  More precisely, there is a constant $K = K(\alpha,q,c,C) > 0$ satisfying
  \[
    \|f\|_{\GenApproxSpaceSet[\Sigma]}
    \leq K \cdot \|f\|_{\GenApproxSpaceSet[\Sigma']}
    \qquad \forall \, f \in \GenApproxSpaceSet[\Sigma'] \, .
    \qedhere
  \]
\end{lem}

\begin{rem*}
  One can alternatively assume that $\AppErr (f, \Sigma_{c m})_X \leq C \cdot \AppErr (f, \Sigma_m ')_X$
  only holds for $m \geq m_0 \in \N$.
  Indeed, if this is satisfied and if we set $c' := m_0 \, c$, then we see for arbitrary
  $m \in \N$ that $m_0 m \geq m_0$, so that
  \[
    \AppErr (f, \Sigma_{c' m})_X
    = \AppErr(f, \Sigma_{c \cdot m_0 \, m})_X
    \leq C \cdot \AppErr(f, \Sigma_{m_0 m} ')_X
    \leq C \cdot \AppErr (f, \Sigma_m ')_X \, .
  \]
  Here, the last step used that $m_0 \, m \geq m$, so that $\Sigma_m ' \subset \Sigma_{m_0 \, m}'$.
\end{rem*}

The proof of Lemma~\ref{lem:ApproximationSpaceElementaryNesting} can be found in
Appendix~\ref{sub:ApproximationSpaceElementaryNestingProof}.

\medskip{}

In \cite[Chapter 7, Section 9]{ConstructiveApproximation}, the authors develop
a general theory regarding approximation classes of this type.
To apply this theory, we merely have to verify that
$\AppSet = (\AppSet_n)_{n \in \N_0}$ satisfies the following list of axioms,
which is identical to \cite[Chapter 7, Equation (5.2)]{ConstructiveApproximation}:
\begin{enumerate}[label={(P\arabic*)}]
  \item \label{enu:GammaContainsZero}
        $\AppSet_0 = \{0\}$;

  \item \label{enu:GammaIncreasing}
        $\AppSet_n \subset \AppSet_{n+1}$ for all $n \in \N_0$;

  \item \label{enu:GammaScaling}
        $a \cdot \AppSet_n = \AppSet_n$ for all $a \in \R \setminus \{0\}$ and $n \in \N_0$;

  \item \label{enu:GammaAdditive}
        There is a fixed constant $c \in \N$ with $\AppSet_n + \AppSet_n \subset \AppSet_{cn}$
        for all $n \in \N_0$;

  \item \label{enu:GammaDense}
        $\AppSet_{\infty} := \bigcup_{j \in \N_0} \AppSet_j$ is dense in $X$;

  \item \label{enu:GammaBestApproximationExists}
        for any $n \in \N_0$, each $f \in X$ has a best approximation from $\AppSet_n$.
\end{enumerate}

As we will show in Theorem~\ref{th:DNNApproxSpaceWellDefined} below,
Properties~\ref{enu:GammaContainsZero}--\ref{enu:GammaDense}
hold in $X = \StandardXSpace (\Omega)$ for an appropriately defined family $\Sigma$
related to neural networks of fixed or varying network depth $L \in \N \cup \{\infty\}$.

Property~\ref{enu:GammaBestApproximationExists}, however, can fail in this setting
even for the simple case of the ReLU activation function;
indeed, a combination of Lemmas~\ref{lem:ApproximationOfIndicatorCube} and
\ref{lem:ExactSquashingReLUPower} below shows that ReLU networks of bounded complexity can approximate
the \emph{discontinuous} function $\Indicator_{[a,b]}$ arbitrarily well.
Yet, since realizations of ReLU networks are always continuous, $\Indicator_{[a,b]}$ is not
implemented exactly by such a network; hence, no best approximation exists.
Fortunately, 
Property~\ref{enu:GammaBestApproximationExists} is not
essential for the theory from \cite{ConstructiveApproximation} to be applicable:
by the arguments given in
\cite[Chapter 7, discussion around Equation (9.2)]{ConstructiveApproximation}
(see also \cite[Proposition~3.8 and Theorem~3.12]{MR2029742})
we get the following properties of the approximation classes
$\GenApproxSpace$ that turn out to be approximation {\em spaces}, i.e., quasi-Banach spaces.

\begin{prop}\label{prop:ApproxSpaceWellDefined}
If Properties~\ref{enu:GammaContainsZero}--\ref{enu:GammaDense} hold,
then the classes $(\GenApproxSpace, \|\cdot\|_{\GenApproxSpace})$ are quasi-Banach \emph{spaces}
satisfying the continuous embeddings~\eqref{eq:ApproximationSpaceEmbedding}
and $\GenApproxSpace \hookrightarrow X$.
\end{prop}

\begin{rem*}
  Note that $\|\cdot\|_{\GenApproxSpace}$ is in general only a quasi-norm, even if
  $X$ is a Banach space and $q \in [1,\infty]$.
  Only if one additionally knows that all the sets $\Sigma_n$ are vector spaces
  (that is, one can choose $c = 1$ in Property~\ref{enu:GammaAdditive}),
  one knows for sure that $\|\cdot\|_{\GenApproxSpace}$ is a norm.
\end{rem*}

\begin{proof}
Everything except for the completeness and the embedding
$\GenApproxSpace \hookrightarrow X$ is shown in
\cite[Chapter 7, Discussion around Equation (9.2)]{ConstructiveApproximation}.
In \cite[Chapter 7, Discussion around Equation (9.2)]{ConstructiveApproximation}
it was shown that the embedding \eqref{eq:ApproximationSpaceEmbedding} holds.
All other properties claimed in Proposition~\ref{prop:ApproxSpaceWellDefined}
follow by combining Remark~3.5, Proposition~3.8, and Theorem~3.12 in \cite{MR2029742}.
\end{proof}

\subsection{Approximation classes of generalized networks}
\label{sec:AppSpacesMainAssumptions}

We now specialize to the setting of neural networks and consider
$d,k \in \N$, an activation function $\varrho: \R \to \R$, and a non-empty set $\Omega \subseteq \R^d$.

Our goal is to define a family of sets of (realizations of) $\varrho$-networks of ``complexity'' $n \in \N_{0}$.
The complexity will be measured in terms of the number of connections $W \leq n$
or the number of neurons $N \leq n$, possibly with a control on how the depth $L$ evolves with $n$.

\begin{defn}[Depth growth function]\label{def:DepthGrowthFunction}
A  {\em depth growth function} is a non-decreasing function
\[
  \mathscr{L}: \N \to \N \cup \{\infty\}, n \mapsto \mathscr{L}(n) .
  \qedhere
\]
\end{defn}
\begin{defn}[Approximation family, approximation spaces]
\label{def:GammaDefinition}
Given an activation function $\varrho$, a depth growth function $\mathscr{L}$,
a subset $\Omega \subseteq \R^d$, and a quasi-Banach space $X$ whose elements are
(equivalence classes of) functions $f : \Omega \to \R^k$,
we define $\StandardSigmaN[0] = \StandardSigmaW[0] := \{0\}$, and
\begin{align}
   \StandardSigmaW
  &:=
\NNreal_{n,\mathscr{L}(n),\infty}^{\varrho,d,k}(\Omega) \cap X,
  \qquad (n \in \N) \, ,
  \label{eq:WSigmaDefinition}\\
  \StandardSigmaN
  &:=
\NNreal_{\infty,\mathscr{L}(n),n}^{\varrho,d,k}(\Omega) \cap X,
  \qquad (n \in \N) \, .
  \label{eq:NSigmaDefinition}
  \end{align}
To highlight the role of the activation function $\varrho$
and the depth growth function $\mathscr{L}$ in the definition
of the corresponding approximation classes, we introduce the specific notation
\begin{align}
  \WASpace[X] &= \GenApproxSpace[\alpha][q]
  \quad \text{where}\quad
  \Sigma = (\Sigma_n)_{n \in \N_0}, \text{ with }
  \AppSet_{n} := \StandardSigmaW,
  \label{eq:DefWASpace} \\
  \NASpace[X] &= \GenApproxSpace[\alpha][q]
  \quad \text{where}\quad
  \Sigma = (\Sigma_n)_{n \in \N_0}, \text{ with }
  \AppSet_{n} := \StandardSigmaN.
  \label{eq:DefNASpace}
\end{align}
The quantities $\|\cdot\|_{\WASpace[X]}$ and  $\|\cdot\|_{\NASpace[X]}$ are defined similarly.
Notice that the input and output dimensions $d,k$ as well as the set $\Omega$ are
implicitly described by the space $X$.
Finally, if the depth growth function is constant ($\mathscr{L}\equiv L$ for some $L \in \N$),
we write $\WeightClassSymbol_n(X, \varrho, L)$, etc.
\end{defn}

\begin{rem}
  By convention, $\StandardSigmaW[0] = \StandardSigmaN[0] = \{0\}$, while
  $\NNreal_{0,L}^{\varrho,d,k}$ is the set
  of constant functions $f \equiv c$, where $c \in \R^{k}$ is arbitrary (Lemma~\ref{lem:ConstantMaps}),
  and $\NNreal_{\infty,L,0}^{\varrho,d,k}$ is the set of affine functions.
\end{rem}

\begin{rem}\label{rem:InfiniteDepthOrN}
  Lemma~\ref{lem:BoundingLayersAndNeuronsByWeights} shows that
  $\NNreal^{\varrho,d,k}_{W,L} = \NNreal^{\varrho,d,k}_{W,W}$ if $L \geq W \geq 1$;
  hence the approximation family $\StandardSigmaW$ associated to any depth growth function
  $\mathscr{L}$ is also generated by the modified depth growth function
  $\mathscr{L}' (n) := \min \{n, \mathscr{L}(n) \}$,
  which satisfies $\mathscr{L}' (n) \in \{1, \dots, n \}$ for all $n \in \N$.

  In light of Equation~\eqref{eq:LayersBoundedByNeurons},
  a similar observation holds for $\StandardSigmaN$
  with $\mathscr{L}' (n) := \min \{n+1, \mathscr{L}(n) \}$.\\
  It will be convenient, however, to explicitly specify
  unbounded depth as $\mathscr{L} \equiv +\infty$ rather than the equivalent form
  $\mathscr{L}(n) = n$ (resp.~rather than $\mathscr{L}(n) = n+1$).
\end{rem}

We will further discuss the role of the depth growth function in
Section~\ref{sec:RoleOfDepthGeneric}.
Before that, we compare approximation with generalized and strict networks.

\subsection{Approximation with generalized \texorpdfstring{{\em vs}}{vs} strict networks}
\label{sub:StrictNetworks}

In this subsection, we show that \emph{if one only considers the approximation theoretic properties}
of the resulting function classes, then---under extremely
mild assumptions on the activation function $\varrho$---it does not matter
whether we consider strict or generalized networks,
at least on \emph{bounded} domains $\Omega \subset \R^d$.
Here, instead of the approximating sets 
for \emph{generalized} neural networks defined in~\eqref{eq:WSigmaDefinition}-\eqref{eq:NSigmaDefinition}
we wish to consider the corresponding sets for \emph{strict} neural networks,
given by $\StandardSigmaWS[0] := \StandardSigmaNS[0] := \{0\}$, and
\begin{align*}
   \StandardSigmaWS
  &:=
\SNNreal_{n,\mathscr{L}(n),\infty}^{\varrho,d,k}(\Omega)
            \cap X,
  \qquad (n \in \N) , \\
  \StandardSigmaNS
  &:=
\SNNreal_{\infty,\mathscr{L}(n),n}^{\varrho,d,k}(\Omega)
             \cap X,
  \qquad (n \in \N) ,
\end{align*}
and the 
associated approximation classes that we denote by
\begin{align*}
  \SWASpace[X] &= \GenApproxSpace[\alpha][q]
  \quad \text{where} \quad
  \Sigma = (\Sigma_n)_{n \in \N_0}
  \quad \text{with}\quad
  \AppSet_{n} := \StandardSigmaWS \\
  \SNASpace[X] &= \GenApproxSpace[\alpha][q]
  \quad \text{where} \quad
  \Sigma = (\Sigma_n)_{n \in \N_0}
  \quad \text{with}\quad
  \AppSet_{n} := \StandardSigmaNS.
\end{align*}

Since generalized networks are at least as expressive as strict ones,
these approximation classes embed into the corresponding classes for generalized networks,
as we now formalize.

\begin{prop}\label{prop:EmbeddingStrictIntoGeneralized}
Consider $\varrho$ an activation function, $\mathscr{L}$ a depth growth function,
and $X$ a quasi-Banach space of (equivalence classes of) functions from a subset
$\Omega \subseteq \R^{d}$ to $\R^{k}$.
For any $\alpha>0$ and $q \in (0,\infty]$, we have $\|\cdot\|_{\WASpace[X]}  \leq \|\cdot\|_{\SWASpace[X]}$
and $\|\cdot\|_{\NASpace[X]}  \leq \|\cdot\|_{\SNASpace[X]}$; hence
\[
  \SWASpace[X] \hookrightarrow \WASpace[X]
  \quad \text{and} \quad
  \SNASpace[X] \hookrightarrow \NASpace[X].
  \qedhere
\]
\end{prop}

\begin{proof}
  We give the proof for approximation spaces associated to connection complexity;
  the proof is similar for the case of neuron complexity.
  Obviously $\StandardSigmaWS \subset \StandardSigmaW$
  for all $n \in \N_{0}$, so that the approximation errors satisfy
  \(
    \AppErr \big( f,\StandardSigmaW \big)_X \leq \AppErr \big( f, \StandardSigmaWS \big)_X
  \)
  for all $n \in \N_0$.
  This implies
  \(
    \|\cdot\|_{\WASpace[X]}
    \leq \|\cdot\|_{\SWASpace[X]},
  \)
  whence $\SWASpace[X] \subset \WASpace[X]$.
\end{proof}

Under mild conditions on $\varrho$, the converse holds on bounded domains when approximating in $L_{p}$.
This also holds on unbounded domains for activation functions \emph{that can represent the identity}.

\begin{thm}[Approximation classes of strict {\em vs.} generalized networks]
  \label{th:DNNApproxSpaceWellDefinedStrict}
  Consider $d \in \N$, a measurable set $\Omega \subseteq \R^{d}$
  with nonzero measure, and $\varrho: \R \to \R$ an activation function.
  Assume either that:
  \begin{itemize}[leftmargin=0.6cm]
    \item $\Omega$ is bounded, $\varrho$ is \emph{continuous} and $\varrho$ is differentiable
          at some $x_{0} \in \R$ with $\varrho'(x_{0}) \neq 0$; or that

    \item $\varrho$ can represent the identity $\identity: \R \to \R, x \mapsto x$ with $m$ terms
          for some $m \in \N$.
  \end{itemize}
  Then for any depth growth function $\mathscr{L}$, $k \in \N$, $\alpha  > 0$, $p,q \in (0,\infty]$,
  with $X := \StandardXSpace(\Omega)$ as in Equation~\eqref{eq:StandardXSpace}, we have the identities
  \[
    \SWASpace[X] = \WASpace[X]\quad \text{and}\quad \SNASpace[X] = \NASpace[X]\, ,
  \]
  and there exists $C < \infty$ such that
  \begin{align*}
    & \|\bullet\|_{\WASpace[X]}
      \leq \|\bullet\|_{\SWASpace[X]}
      \leq C \, \|\bullet\|_{\WASpace[X]} \\
    \quad \text{and} \quad
    & \|\bullet\|_{\NASpace[X]}
      \leq \|\bullet\|_{\SNASpace[X]}
      \leq C \, \|\bullet\|_{\NASpace[X]} .
    \qedhere
  \end{align*}
\end{thm}



Before giving the proof, let us clarify the precise choice of (quasi)-norm for the vector-valued spaces
$X := \StandardXSpace(\Omega)$ from Equation~\eqref{eq:StandardXSpace}.
For $f = (f_{1},\ldots,f_{k}): \Omega \to \R^{k}$ and $0 < p < \infty$ it is defined by
\(
  \|f\|_{L_p(\Omega;\R^k)}^{p}
  := \sum_{\ell=1}^{k} \|f_{\ell}\|_{L_{p}(\Omega;\R)}^{p}
  = \int_{\Omega} |f(x)|_{p}^{p} \, dx
\),
where $|u|_{p}^{p} := \sum_{\ell=1}^{k}|u_{\ell}|^{p}$ for each $u \in \R^{k}$.
For $p=\infty$ we use the definition
$\|f\|_{\infty} := \max_{\ell = 1,\dots,k} \|f_{\ell}\|_{L_{\infty}(\Omega;\R)}$.

\begin{proof}
When $\varrho$ can represent the identity with $m$ terms,
we rely on Lemma~\ref{lem:IdentityRectifierForFree} and on the estimate
$\mathscr{L}(n) \leq \mathscr{L}(m^2 n)$
to obtain for any $n \in \N$ that 
\[
  \StandardSigmaW
  =
\NNreal_{n,\mathscr{L}(n),\infty}^{\varrho,d,k}(\Omega)
          \cap X
  \subset
 \SNNreal_{m^{2}n,\mathscr{L}(m^{2}n),\infty}^{\varrho,d,k}
          \cap X
  = \StandardSigmaWS[m^{2}n],
\]
and similarly $\StandardSigmaN[n] \subset \StandardSigmaNS[mn]$, so that
\begin{align*}
  \AppErr \big( f, \StandardSigmaWS[m^{2}n] \big)_X
  & \leq \AppErr \big( f,\StandardSigmaW \big)_X
  \quad \forall n \in \N_{0} \, ,\\
  \AppErr \big( f, \StandardSigmaNS[mn] \big)_X
  & \leq \AppErr \big( f,\StandardSigmaN \big)_X
  \quad \forall n \in \N_{0} \, .
\end{align*}

We now establish similar results for the case where $\Omega$ is bounded,
$\varrho$ is continuous and $\varrho'(x_{0}) \neq 0$ is well defined for some $x_{0} \in \R$.
We rely on Lemma~\ref{lem:GeneralizedVSStrict}.
First, note by continuity of $\varrho$ that any
$f \in \NNreal^{\varrho,d,k} \supset \SNNreal^{\varrho,d,k}$
is a continuous function $f : \R^d \to \R^k$.
Furthermore, since $\Omega$ is bounded, $\overline{\Omega}$ is compact,
so that $f|_{\overline{\Omega}}$ is \emph{uniformly} continuous and bounded.
Clearly, this implies that $f|_{\Omega}$ is uniformly continuous and bounded as well.
Since $X = \StandardXSpace(\Omega)$, this implies
\[
  \StandardSigmaWS 
  =
\SNNreal^{\varrho,d,k}_{n, \mathscr{L}(n),\infty}(\Omega)
     \cap X
  =
\SNNreal^{\varrho,d,k}_{n, \mathscr{L}(n),\infty}(\Omega)
\]
and similarly for $\StandardSigmaNS$.
Since $\Omega \subset \R^d$ is bounded,
locally uniform convergence on $\R^d$ implies convergence in $\StandardXSpace(\Omega)$.
Hence for any $n \in \N_0$, using that $\mathscr{L}(n) \leq \mathscr{L}(4n)$,
Lemma~\ref{lem:GeneralizedVSStrict} yields
\[
  \StandardSigmaW
  \subset
         \overline{
\SNNreal_{4n,\mathscr{L}(4n),\infty}^{\varrho,d,k}(\Omega)
         }^{\StandardXSpace (\Omega)}
  \!\! \subset \overline{\StandardSigmaWS[4n]}^{\StandardXSpace (\Omega)} \, ,
\]
where the closure is taken with respect to the topology induced by
$\| \cdot \|_{\StandardXSpace (\Omega)}$.
Similarly, we have
\[
  \StandardSigmaN
  \subset
       \overline{
\SNNreal_{\infty,\mathscr{L}(2n),2n}^{\varrho,d,k}(\Omega) 
       }^{\StandardXSpace (\Omega)}
  \!\! \subset \overline{\StandardSigmaNS[2n]}^{\StandardXSpace (\Omega)} \, .
\]
Now for an arbitrary subset $\Gamma \subset \StandardXSpace (\Omega)$, observe
by continuity of $\|\cdot\|_{\StandardXSpace(\Omega)}$ that
\[
  \inf_{\theta \in \Gamma} \| f - \theta \|_{\StandardXSpace (\Omega)}
  = \inf_{\theta \in \overline{\Gamma}} \| f - \theta\|_{\StandardXSpace (\Omega)}
  \, ;
\]
that is, if one is only interested in the distance of functions $f$ to the set
$\Gamma$, then switching from $\Gamma$ to its closure $\overline{\Gamma}$
(computed in $\StandardXSpace$) does not change the resulting distance.
Therefore,
\begin{align*}
 \AppErr \big( f, \StandardSigmaWS[4n] \big)_X
 & \leq \AppErr \big( f,\StandardSigmaW \big)_X
  \quad \forall n \in \N_{0} \, ,\\
   \AppErr \big( f, \StandardSigmaNS[2n] \big)_X
  & \leq \AppErr \big( f,\StandardSigmaN \big)_X
  \quad \forall n \in \N_{0} \, .
\end{align*}

In both settings ($\varrho$ can represent the identity, or $\Omega$ is bounded
and $\varrho$ differentiable at $x_0$),
Lemma~\ref{lem:ApproximationSpaceElementaryNesting} shows
$\|\cdot\|_{\SWASpace[X]} \leq C \|\cdot\|_{\WASpace[X]}$
and ${\|\cdot\|_{\SNASpace[X]} \leq C \|\cdot\|_{\NASpace[X]}}$
for some $C \in (0,\infty)$.
The conclusion follows using Proposition~\ref{prop:EmbeddingStrictIntoGeneralized}.
\end{proof}

\subsection{Connectivity \emph{vs.} number of neurons}
\label{sec:weightsvsneurons}

\begin{lem}\label{lem:weightsvsneurons}
Consider $\varrho: \R \to \R$ an activation function, $\mathscr{L}$ a depth growth function,
$d,k \in \N$, $p \in (0,\infty]$ and
a measurable $\Omega \subseteq \R^d$ with nonzero measure.
With $X := \StandardXSpace(\Omega)$, we have for any $\alpha>0$ and $q \in (0,\infty]$
\begin{alignat*}{5}
  && \WASpace[X]
   & \hookrightarrow \NASpace[X]
  && \hookrightarrow \WASpace[X][\varrho][q][\alpha/2],
  \\
  \text{and} \quad
  &&  \SWASpace[X]
   & \hookrightarrow \SNASpace[X]
  && \hookrightarrow \SWASpace[X][\varrho][q][\alpha/2],
\end{alignat*}
and there exists $c > 0$ such that
\begin{alignat*}{5}
  && \|\cdot\|_{\WASpace[X]}
   & \geq \|\cdot\|_{\NASpace[X]}
  && \geq c \, \|\cdot\|_{\WASpace[X][\varrho][q][\alpha/2]}
  \, ,  \\
  \text{and} \quad
  && \|\cdot\|_{\SWASpace[X]}
   & \geq \|\cdot\|_{\SNASpace[X]}
  && \geq c \, \|\cdot\|_{\SWASpace[X][\varrho][q][\alpha/2]} \, .
\end{alignat*}
When $L := \sup_{n} \mathscr{L}(n)=2$ (i.e., for shallow networks)
the exponent $\alpha/2$ can be replaced by $\alpha$;
that is, $\WASpace[X] = \NASpace[X]$ with equivalent norms.
\end{lem}

\begin{rem*}
  We will see in Lemma~\ref{lem:WeightAndNeuronSpacesDistinct} below that
  $\WASpace[X] \neq \NASpace[X]$ if, for instance, $\varrho = \varrho_r$ is a power of the ReLU,
  if $\Omega$ is bounded, and if $L := \sup_{n \in \N} \mathscr{L}(n)$ satisfies $3 \leq L < \infty$.
  In general, however, one cannot expect the spaces to be always distinct.
  For instance, if $\varrho$ is the activation function constructed
  in \cite[Theorem 4]{PinkusLowerBoundsForMLPApproximation}, if $L \geq 3$ and if $\Omega$ is bounded,
  then both $\WASpace[X_p(\Omega)]$ and $\NASpace[X_p(\Omega)]$ coincide with $X_p(\Omega)$.
\end{rem*}

\begin{proof}
We give the proof for generalized networks.
By Lemma~\ref{lem:BoundingLayersAndNeuronsByWeights} and Equation~\eqref{eq:WeightsVsNeurons},
\[
  \NNreal^{\varrho,d,k}_{n,\mathscr{L}(n),\infty}
  \subset \NNreal^{\varrho,d,k}_{n,\mathscr{L}(n),n}
  \subset \NNreal^{\varrho,d,k}_{\infty,\mathscr{L}(n),n}
  \subset \NNreal^{\varrho,d,k}_{n^{2}+(d+k)n+dk,\mathscr{L}(n),n}
  \subset \NNreal^{\varrho,d,k}_{n^{2}+(d+k)n+dk,\mathscr{L}(n),\infty}
\]
for any $n \in \N$.
Hence, the approximation errors satisfy
\begin{equation}
  \AppErr(f,\StandardSigmaW)_{X}
  \geq \AppErr(f,\StandardSigmaN)_{X}
  \geq \AppErr(f,\StandardSigmaW[n^{2}+(d+k)n+dk])_{X}.
  \label{eq:WeightsVsNeuronsAppErr}
\end{equation}
By the first inequality in~\eqref{eq:WeightsVsNeuronsAppErr},
$\|\cdot\|_{\WASpace[X]} \geq \|\cdot\|_{\NASpace[X]}$ and $\WASpace[X] \subset \NASpace[X]$.

When $L=2$, by the remark below Equation~\eqref{eq:WeightsVsNeurons} we get
$\NNreal^{\varrho,d,k}_{\infty,\mathscr{L}(n),n} \subset \NNreal^{\varrho,d,k}_{(d+k)n,\mathscr{L}(n),\infty}$;
hence $\AppErr(f,\StandardSigmaN)_{X} \geq \AppErr(f,\StandardSigmaW[(d+k)n])_{X}$
so that Lemma~\ref{lem:ApproximationSpaceElementaryNesting} shows
$\WASpace[X][\varrho][q][\alpha] \supset \NASpace[X][\varrho][q][\alpha]$,
with a corresponding (quasi)-norm estimate;
hence, these spaces coincide with equivalent (quasi)-norms.

\smallskip{}

For the general case, observe that $n^{2}+(d+k)n +dk \leq (n+\gamma)^{2}$
with $\gamma := \max \{ d,k \}$.
Let us first consider the case $q < \infty$.
In this case, we note that if $(n+\gamma)^2 + 1 \leq m \leq (n+\gamma+1)^2$,
then $n^2 \leq m \leq (2\gamma+2)^2 \, n^2$, and thus $m^{\alpha q - 1} \lesssim n^{2 \alpha q - 2}$,
where the implied constant only depends on $\alpha, q$, and $\gamma$.
This implies
\[
  \sum_{m=(n+\gamma)^2 + 1}^{(n+\gamma+1)^2} m^{\alpha q - 1}
  \leq C \cdot n^{2 \alpha q - 1}
  \qquad \forall \, n \in \N
\]
where $C = C(\alpha,q,\gamma) < \infty$, since the sum has 
$((n+\gamma)+1)^2 - (n+\gamma)^2 = 2 n+2\gamma + 1 \leq 4n (2\gamma+1)$ many summands.
By the second inequality in~\eqref{eq:WeightsVsNeuronsAppErr} we get for any $n \in \N$
\begin{align*}
  \sum_{m=(n+\gamma)^2+1}^{(n+1+\gamma)^2}
    \left[m^{\alpha} \AppErr(f,\StandardSigmaW[m-1])_{X}\right]^{q} \tfrac{1}{m}
  & \leq \left(
           \sum_{m=(n+\gamma)^2+1}^{(n+1+\gamma)^2}
           m^{\alpha q-1}
         \right)
         \cdot \big[ \AppErr(f,\StandardSigmaW[(n+\gamma)^{2}])_{X} \big]^q \\
  & \leq C \cdot n^{2\alpha q-1}
           \cdot \big[ \AppErr(f,\StandardSigmaW[n^{2}+(d+k)n+dk])_{X} \big]^q \\
  & \leq C \cdot n^{2\alpha q-1}
           \cdot \big[ \AppErr(f,\StandardSigmaN[n])_{X} \big]^q .
\end{align*}
It follows that
\begin{align*}
  \sum_{m \geq 1+(\gamma+1)^{2}}
    \left[m^{\alpha} \AppErr(f,\StandardSigmaW[m-1])_{X}\right]^{q} \tfrac{1}{m}
  &= \sum_{n \in \N}
        \sum_{m=(n+\gamma)^2+1}^{(n+1+\gamma)^2}
          \left[m^{\alpha} \AppErr(f,\StandardSigmaW[m-1])_{X}\right]^{q} \tfrac{1}{m} \\
  &\leq C \sum_{n \in \N }
             n^{2\alpha q-1} \cdot \big[ \AppErr(f,\StandardSigmaN[n])_{X} \big]^q
   \leq C \|f\|_{\NASpace[X][\varrho][q][2\alpha]}^{q} .
\end{align*}
To conclude we use that
\(
  \sum_{m=1}^{(\gamma+1)^{2}}
    \left[m^{\alpha} \AppErr(f,\StandardSigmaW[m-1])_{X}\right]^{q} \tfrac{1}{m}
  \leq C' \|f\|_{X}^{q}
  \leq C'  \|f\|_{\NASpace[X][\varrho][q][2\alpha]}^{q}
\)
with $C' = \sum_{m=1}^{(\gamma+1)^{2}} m^{\alpha q-1}$.
%
%

The proof for $q=\infty$ is similar.
The proof for strict networks follows along similar lines.
\end{proof}

The final result in this subsection shows that the inclusions in Lemma~\ref{lem:weightsvsneurons}
are quite sharp.

\begin{lem}\label{lem:WeightAndNeuronSpacesDistinct}
  For $r \in \N$, define $\varrho_r : \R \to \R, x \mapsto (x_+)^r$.

  Let $\Omega \subset \R^d$ be bounded and measurable with nonempty interior.
  Let $L, L' \in \N_{\geq 2}$, let $r_1, r_2 \in \N$,
  let $p_1,p_2,q_1,q_2 \in (0,\infty]$, and $\alpha, \beta > 0$.
  Then the following hold:
  \begin{enumerate}
    \item If
          \(
            \WASpace[X_{p_1}(\Omega)][\varrho_{r_1}][q_1][\alpha][L]
            \subset
            \NASpace[X_{p_2}(\Omega)][\varrho_{r_2}][q_2][\beta][L']
          \),
          then $L' - 1 \geq \tfrac{\beta}{\alpha} \cdot \lfloor L/2 \rfloor$.
          \smallskip{}

    \item If
          \(
            \NASpace[X_{p_2}(\Omega)][\varrho_{r_2}][q_2][\beta][L']
            \subset
            \WASpace[X_{p_1}(\Omega)][\varrho_{r_1}][q_1][\alpha][L]
          \),
          then $\lfloor L/2 \rfloor \geq \frac{\alpha}{\beta} \cdot (L' - 1)$.
  \end{enumerate}
  In particular, if
  \(
    \WASpace[X_{p_1}(\Omega)][\varrho_{r_1}][q_1][\alpha][L]
    = \NASpace[X_{p_2}(\Omega)][\varrho_{r_2}][q_2][\alpha][L]
  \),
  then $L = 2$.
\end{lem}

The proof of this result is given in Appendix~\ref{sec:WeightAndNeuronSpacesDistinct}.

\subsection{Role of the depth growth function}
\label{sec:RoleOfDepthGeneric}

In this subsection, we investigate the relation between approximation classes
associated to different depth growth functions.
First we define a comparison rule between depth growth functions.

\begin{defn}[Comparison between depth growth functions]\label{def:GrowthEquivalence}
The depth growth function $\mathscr{L}$ is \emph{dominated}
by the depth growth function $\mathscr{L}'$
(denoted $\mathscr{L} \preceq \mathscr{L}'$ or $\mathscr{L}' \succeq \mathscr{L}$)
if there are $c,n_{0} \in \N$ such that
\begin{equation}
  \forall \, n \geq n_{0}: \quad
  \mathscr{L}(n) \leq \mathscr{L}'(cn) \, .
  \label{eq:GrowthFunctionEquivalenceCond}
\end{equation}
Observe that $\mathscr{L} \leq \mathscr{L}'$ implies $\mathscr{L} \preceq \mathscr{L}'$.

The two depth growth functions are  \emph{equivalent}
(denoted $\mathscr{L} \sim \mathscr{L}'$)
if $\mathscr{L} \preceq \mathscr{L}'$ and $\mathscr{L} \succeq \mathscr{L}'$,
that is to say if there exist $c,n_{0} \in \N$ such that for each $n \geq n_{0}$,
$\mathscr{L}(n) \leq \mathscr{L}'(c n)$
and $\mathscr{L}'(n) \leq \mathscr{L}(c n)$.
This defines an equivalence relation on the set of depth growth functions.
\end{defn}

\begin{lem}\label{lem:RoleGrowthFunctionRate}
Consider two depth growth functions $\mathscr{L}$, $\mathscr{L}'$.
If $\mathscr{L} \preceq \mathscr{L}'$, then for each $\alpha > 0$ and
$q \in (0,\infty]$, there is a constant
$C = C(\mathscr{L},\mathscr{L}',\alpha,q) \in [1,\infty)$ such that:
\begin{alignat*}{5}
  \WASpace[X]
  & \hookrightarrow \WASpace[X][\varrho][q][\alpha][\mathscr{L}']
  && \quad \text{and} \quad
  \|\cdot\|_{\WASpace[X][\varrho][q][\alpha][\mathscr{L}']}
  \leq
  C \cdot \|\cdot\|_{\WASpace[X]} \\
  \NASpace[X]
  & \hookrightarrow \NASpace[X][\varrho][q][\alpha][\mathscr{L}']
  && \quad \text{and} \quad
  \|\cdot\|_{\NASpace[X][\varrho][q][\alpha][\mathscr{L}']}
  \leq
  C \cdot \|\cdot\|_{\NASpace[X]}
  \end{alignat*}
for each activation function $\varrho : \R \to \R$, each (bounded or unbounded) set $\Omega \subset \R^d$,
and each quasi-Banach space $X$ of (equivalence classes of) functions $f: \Omega \to \R^{k}$.

The same holds with $\SWASpace[X]$ (resp. $\SNASpace[X]$) instead of $\WASpace[X]$ (resp. $\NASpace[X]$).

The constant $C$ depends only on the constants $c,n_{0} \in \N$
involved in~\eqref{eq:GrowthFunctionEquivalenceCond} and on $\alpha, q$.
\end{lem}

\begin{proof}
  Let $c, n_0 \in \N$ as in Equation~\eqref{eq:GrowthFunctionEquivalenceCond}.
  For $n \geq n_0$, we then have $\mathscr{L}(n) \leq \mathscr{L}'(c n)$, and hence
  \begin{alignat*}{3}
    \NNreal^{\varrho,d,k}_{n,\mathscr{L}(n),\infty}
    &\subset \NNreal^{\varrho,d,k}_{n, \mathscr{L}'(c n),\infty}
    && \subset \NNreal^{\varrho,d,k}_{c n, \mathscr{L}'(c n),\infty} \, ,\\
    \NNreal^{\varrho,d,k}_{\infty,\mathscr{L}(n),n}
    &\subset \NNreal^{\varrho,d,k}_{\infty, \mathscr{L}'(c n),n}
    && \subset \NNreal^{\varrho,d,k}_{\infty, \mathscr{L}'(c n),cn} \, ,
  \end{alignat*}
  from which we easily get
  \begin{align*}
    \AppErr \big(f, \mathtt{W}_{c n} (X, \varrho, \mathscr{L}') \big)_X
    &\leq \AppErr \big( f, \mathtt{W}_{n} (X, \varrho, \mathscr{L}) \big)_X
    \qquad \forall \, n \geq n_0 \, ,\\
    \AppErr \big(f, \mathtt{N}_{c n} (X, \varrho, \mathscr{L}') \big)_X
    &\leq \AppErr \big( f, \mathtt{N}_{n} (X, \varrho, \mathscr{L}) \big)_X
    \qquad \forall \, n \geq n_0 \, .
  \end{align*}
  Now, Lemma~\ref{lem:ApproximationSpaceElementaryNesting} and the associated remark
  complete the proof.
  Exactly the same proof works for strict networks; one just has to replace
  $\NNreal$ by $\SNNreal$ everywhere.
\end{proof}



As a direct consequence of Lemma~\ref{lem:RoleGrowthFunctionRate}, we see that equivalent
depth growth functions induce the same approximation spaces.

\begin{thm}\label{thm:RoleGrowthFunction}
If $\mathscr{L}, \mathscr{L}'$ are two depth-growth functions satisfying
$\mathscr{L} \sim \mathscr{L}'$, then for any $\alpha > 0$ and $q \in (0,\infty]$,
there is a constant $C \in [1,\infty)$ such that
\begin{alignat*}{5}
  \WASpace[X]
  &= \WASpace[X][\varrho][q][\alpha][\mathscr{L}']
  &&\quad \text{and} \quad
   \tfrac{1}{C} \, \|\cdot\|_{\WASpace[X]}
  && \leq \|\cdot\|_{\WASpace[X][\varrho][q][\alpha][\mathscr{L}']}
  && \leq C \cdot \|\cdot\|_{\WASpace[X]}\\
  \NASpace[X]
  &= \NASpace[X][\varrho][q][\alpha][\mathscr{L}']
  && \quad \text{and} \quad
    \tfrac{1}{C} \, \|\cdot\|_{\NASpace[X]}
  && \leq \|\cdot\|_{\NASpace[X][\varrho][q][\alpha][\mathscr{L}']}
  && \leq C \cdot \|\cdot\|_{\NASpace[X]}
  \end{alignat*}
for each activation function $\varrho$, each $\Omega \subset \R^d$,
and each quasi-Banach space $X$ of (equivalence classes of) functions $f: \Omega \to \R^{k}$.
The same holds with $\SWASpace[X]$ (resp. $\SNASpace[X]$) instead of $\WASpace[X]$ (resp. $\NASpace[X]$).
The constant $C$ depends only on the constants $c,n_{0} \in \N$ in
Definition~\ref{def:GrowthEquivalence} and on $\alpha, q$.
\end{thm}

Theorem~\ref{thm:RoleGrowthFunction} shows in particular that if
$L := \sup_{n} \mathscr{L}(n) < \infty$, then $\WASpace[X] = \WASpace[X][\varrho][q][\alpha][L]$
with equivalent ``proto-norms'' (and similarly with $\NASpace[X]$ instead of $\WASpace[X]$
or with strict networks instead of generalized ones).
Indeed, it is easy to see that $\mathscr{L} \sim \mathscr{L}'$ if
${\sup_{n} \mathscr{L}(n) = \sup_{n} \mathscr{L}'(n) = L < \infty}$.

%
%

\begin{lem}\label{lem:DepthGrowthLemma}
  Consider $\mathscr{L}$ a depth growth function and $\eps > 0$.
  \begin{enumerate}[leftmargin=0.6cm]
    \item if $\mathscr{L}+\eps \preceq \mathscr{L}$ then $\mathscr{L}+b \sim \mathscr{L}$
          for each $b \geq 0$;

    \item if $e^{\eps}\mathscr{L} \preceq \mathscr{L}$ then $a\mathscr{L}+b \sim \mathscr{L}$
          for each $a \geq 1$, $b \geq 1-a$.
          \qedhere
  \end{enumerate}
\end{lem}

\begin{proof}
For the first claim, we first show by induction on $k \in \N$
that $\mathscr{L}+k\eps \preceq \mathscr{L}$.
For $k=1$ this holds by assumption.
For the induction step, recall that $\mathscr{L}+k\eps \preceq \mathscr{L}$ simply means
that there are $c, n_{0} \in \N$ such that $\mathscr{L}(n)+k\eps \leq \mathscr{L}(cn)$
for all $n \in \N_{\geq n_{0}}$.
Therefore, if $n \geq n_{0}$ then
$\mathscr{L}(n)+(k+1)\eps \leq \mathscr{L}(c n)+\eps \leq \mathscr{L}(c^{2}n)$
since $n' = cn \geq n \geq n_{0}$.
Now, note that if $\mathscr{L} \leq \mathscr{L}'$, then also $\mathscr{L} \preceq \mathscr{L}'$.
Therefore, given $b \geq 0$ we choose $k \in \N$ such that $b \leq k\eps$ and get
$\mathscr{L} \preceq \mathscr{L} + b \preceq \mathscr{L} + k \eps \preceq \mathscr{L}$,
so that all these depth-growth functions are equivalent.

For the second claim, a similar induction yields $e^{k\eps} \mathscr{L} \preceq \mathscr{L}$
for all $k \in \N$.
Now, given $a \geq 1$ and $b \geq 1-a$, we choose $k \in \N$ such that $a+b_+ \leq e^{k\eps}$,
where $b_+ = \max \{0, b\}$.
There are now two cases:
If $b \geq 0$, then clearly $\mathscr{L} \leq a \mathscr{L} \leq a \mathscr{L} + b$.
If otherwise $b < 0$, then $b \mathscr{L} \leq b$, since $\mathscr{L} \geq 1$, and hence
\(
  \mathscr{L}
  = a \mathscr{L} + (1 - a) \mathscr{L}
  \leq a \mathscr{L} + b \mathscr{L}
  \leq a \mathscr{L} + b
\).
Therefore, we see in both cases that
\(
  \mathscr{L}
  \leq a \mathscr{L} + b_+
  \leq (a + b_+) \mathscr{L}
  \leq e^{k \eps} \, \mathscr{L}
  \preceq \mathscr{L}
\).
\end{proof}

The following two examples discuss elementary properties of poly-logarithmic and polynomial
growth functions, respectively.

\begin{example}\label{ex:polyloggrowth}
Assume there are $q \geq 1$, $\alpha, \beta > 0$ such that
$|\mathscr{L}(n) - \alpha \log^{q} n| \leq \beta$ for all $n \in \N$.

Choosing $c \in \N$ such that $\eps := \alpha \log^{q} c - 2\beta>0$, we have
\[
  \mathscr{L}(n)+\eps
  \leq \alpha \log^{q}n + \beta + \eps
  =    \alpha \log^{q} n + \alpha \log^{q} c -\beta
  \leq \alpha (\log c + \log n)^{q} -\beta
   =   \alpha \log^{q} (cn) -\beta \leq \mathscr{L}(cn)
\]
for all $n \in \N$; hence $\mathscr{L}+\eps \preceq \mathscr{L}$.
Here, we used that $x^q + y^q = \|(x,y)\|_{\ell^q}^q \leq \|(x,y)\|_{\ell^1}^q = (x+y)^q$
for $x,y \geq 0$.

By Lemma~\ref{lem:DepthGrowthLemma} we get $\mathscr{L} \sim \mathscr{L}+b$ for arbitrary $b \geq 0$.
Moreover as $\lfloor \alpha \log^{q}n\rfloor \leq \alpha \log^{q} n \leq \mathscr{L}(n)+\beta$ we
have $\max(1,\lfloor\alpha \log^{q}(\cdot)\rfloor) \preceq \mathscr{L}+\beta \sim \mathscr{L}$.
Similarly $\mathscr{L}(n) \leq \lfloor\alpha \log^{q}n\rfloor + \beta+1$ hence
$\mathscr{L} \sim \max(1,\lfloor \alpha \log^{q} (\cdot) \rfloor)$.
\end{example}

\begin{example}\label{ex:polygrowth}
Assume there are $\gamma > 0$ and $C  \geq 1$ such that $1/C \leq \mathscr{L}(n)/n^{\gamma} \leq C$
for all $n \in \N$.

Choosing any integer $c \geq (2C^2)^{1/\gamma}$ we have $2C^{2}c^{-\gamma} \leq 1$, and hence
\[
  2\mathscr{L}(n)
  \leq 2C n^{\gamma}
  \leq 2Cc^{-\gamma} (cn)^{\gamma}
  \leq 2Cc^{-\gamma} C \mathscr{L}(cn)
  =    2C^{2}c^{-\gamma} \mathscr{L}(cn)
  \leq \mathscr{L}(cn)
\]
for all $n \in \N$; hence $2\mathscr{L} \preceq \mathscr{L}$.
By Lemma~\ref{lem:DepthGrowthLemma} we get $\mathscr{L} \sim a\mathscr{L}+b$ for each $a \geq 1, b \geq 1-a$.
Moreover, we have $\lceil n^{\gamma}\rceil \leq n^{\gamma}+1 \leq C\mathscr{L}(n)+1$
for all $n \in \N$ hence $\lceil (\cdot)^{\gamma}\rceil \preceq C\mathscr{L}+1 \sim \mathscr{L}$.
Similarly $\mathscr{L}(n) \leq C n^{\gamma} \leq C \lceil n^{\gamma} \rceil$,
and thus $\mathscr{L} \sim \lceil (\cdot)^{\gamma}\rceil$.
\end{example}

In the next sections we conduct preliminary investigations on the role of the (finite or infinite)
depth $L$ in terms of the associated approximation spaces for $\varrho_{r}$-networks.
A general understanding of the role of depth growth largely remains an open question.
A very surprising result in this direction was recently obtained by Yarotsky \cite{YarotskySurprise}.

\begin{rem}\label{sec:RestrictionCartesian}
  It is not difficult to show that approximation classes defined on nested sets
  $\Omega' \subset \Omega \subset \R^{d}$ satisfy natural restriction properties.
  More precisely, the map
  \[
    \WASpace[\StandardXSpace(\Omega)] \to \WASpace[\StandardXSpace(\Omega')], f \mapsto f|_{\Omega'}
  \]
  is well-defined and bounded (meaning,
  $\|f|_{\Omega'}\|_{\WASpace[\StandardXSpace(\Omega')]} \leq \|f\|_{\WASpace[\StandardXSpace(\Omega)]}$),
  and the same holds for the spaces $N_q^\alpha$ instead of $W_q^\alpha$.

  Furthermore, the approximation classes of vector-valued functions
  $f: \Omega \to \R^{k}$ are cartesian products of real-valued function classes;
  that is,
  \[
    \WASpace[\StandardXSpace(\Omega ; \R^k)] \to \big( \WASpace[\StandardXSpace(\Omega;\R)] \big)^k,
    f \mapsto (f_1,\dots,f_k)
  \]
  is bijective and
  \(
    \|f\|_{\WASpace[\StandardXSpace(\Omega ; \R^k)]}
    \asymp \sum_{\ell=1}^k \|f_\ell\|_{\WASpace[\StandardXSpace(\Omega;\R)]}
  \).
  Again, the same holds for the spaces $N_q^\alpha$ instead of $W_q^\alpha$.
  For the sake of brevity, we omit the easy proofs.
\end{rem}

\subsection{Approximation classes are approximation spaces}

We now verify that the main axioms needed to apply
Proposition~\ref{prop:ApproxSpaceWellDefined} are satisfied.
Properties~\ref{enu:GammaContainsZero}--\ref{enu:GammaAdditive} hold
without any further assumptions:

\begin{lem}\label{lem:ApproximationSpacesWellDefined}
Let $\varrho : \R \to \R$ be arbitrary, and let $\mathscr{L}$ be a depth growth function.
The sets $\AppSet_{n}$ defined in \eqref{eq:DefWASpace}--\eqref{eq:DefNASpace} satisfy
Properties~\ref{enu:GammaContainsZero}--\ref{enu:GammaAdditive} on Page~\pageref{enu:GammaAdditive},
with $c=2+\min\{d,k\}$ for Property~\ref{enu:GammaAdditive}.
\end{lem}

\begin{proof}
We generically write $\StandardSigma$ to indicate either $\StandardSigmaW$ or $\StandardSigmaN$.

{\bf Property~\ref{enu:GammaContainsZero}.}
We have $\StandardSigma[0] = \{0\}$ by definition.
For later use, let us also verify that $0 \in \StandardSigma$ for $n \in \N$.
Indeed, Lemma~\ref{lem:ConstantMaps} shows
$0 \in \NNreal^{\varrho,d,k}_{0,1,0} \subset \NNreal^{\varrho,d,k}_{n,L,m}$
for all $n,m,L \in \N \cup \{\infty\}$, and hence $0 \in \StandardSigma$ for all $n \in \N$.

\medskip{}

{\bf Property~\ref{enu:GammaIncreasing}.}
The inclusions $\NNreal_{W, L,\infty}^{\varrho,d,k} \subset \NNreal_{W+1, L',\infty}^{\varrho,d,k}$
and $\NNreal_{\infty, L,N}^{\varrho,d,k} \subset \NNreal_{\infty, L',N+1}^{\varrho,d,k}$
for $W,N \in \N_{0}$ and $L, L' \in \N \cup \{\infty\}$ with $L \leq L'$
hold by the very definition of these sets.
As $\mathscr{L}$ is non-decreasing (that is, $\mathscr{L}(n+1) \geq \mathscr{L}(n)$),
we thus get $\StandardSigma \subset \StandardSigma[n+1]$ for all $n \in \N$.
As seen in the proof of Property~\ref{enu:GammaContainsZero}, this also holds for $n = 0$.

\medskip{}

{\bf Property~\ref{enu:GammaScaling}.}
By Lemma~\ref{lem:SummationLemma}-(\ref{enu:ScalMult}), if $f \in \NNreal_{W, L,N}^{\varrho,d,k}$,
then $a \cdot f \in \NNreal_{W, L, N}^{\varrho,d,k}$ for any $a \in \R$.
Therefore, $a \cdot \StandardSigma \subset \StandardSigma$ for each $a \in \R$ and $n \in \N$.
The converse is proved similarly for $a \neq 0$; hence
$a \cdot \StandardSigma = \StandardSigma$ for each $a \in \R \setminus \{0\}$ and $n \in \N$.
For $n = 0$, this holds trivially.

\medskip{}


{\bf Property~\ref{enu:GammaAdditive}.}
The claim is trivial for $n=0$.
For $n \in \N$, let $f_{1},f_{2} \in \StandardSigma$ be arbitrary. 

For the case of $\StandardSigma = \StandardSigmaW$,
let $g_{1},g_{2} \in\NNreal_{n,\mathscr{L}(n),\infty}^{\varrho,d,k}$
such that $f_{i} = g_{i}|_{\Omega}$.
Lemma~\ref{lem:BoundingLayersAndNeuronsByWeights} shows that
$g_{i} \in \NNreal_{n,L',\infty}^{\varrho,d,k}$ with $L' := \min\{\mathscr{L}(n), n\}$.
By Lemma~\ref{lem:SummationLemma}-(\ref{enu:LinComb}), setting $c_{0} := \min \{d,k\}$, and
$W' := 2n + c_{0} \cdot (L'-1) \leq (2+c_{0})n$, we have
\(
  g_{1}+g_{2} \in \NNreal_{W',L'}^{\varrho,d,k}
              \subset  \NNreal_{(2+c_{0})n,\mathscr{L}((2+c_{0})n)}^{\varrho,d,k}
\)
where for the last inclusion we used that $L' \leq \mathscr{L}(n)$,
that $\mathscr{L}$ is non-decreasing, and that $n \leq (2+c_{0})n$.

For the case of $\StandardSigma = \StandardSigmaN$,
consider similarly $g_{1},g_{2} \in\NNreal_{\infty,\mathscr{L}(n),n}^{\varrho,d,k}$
such that $f_{i} = g_{i}|_{\Omega}$.
By \eqref{eq:LayersBoundedByNeurons}, $g_{i} \in \NNreal_{\infty,L',n}^{\varrho,d,k}$
with $L' := \min\{\mathscr{L}(n), n+1\}$.
By Lemma~\ref{lem:SummationLemma}-(\ref{enu:LinComb}) again, setting $c_{0} := \min \{d,k\}$, and
$N' := 2n + c_{0} \cdot (L'-1) \leq (2+c_{0})n$, we get
\(
  g_{1}+g_{2} \in \NNreal_{\infty,L',N'}^{\varrho,d,k}
              \subset \NNreal_{\infty,\mathscr{L}((2+c_{0})n),(2+c_{0})n}^{\varrho,d,k}
\).

By Definitions~\eqref{eq:WSigmaDefinition}--\eqref{eq:NSigmaDefinition},
this shows in all cases that $f_{1} + f_{2} \in \StandardSigma[(2+c_{0})n]$.
\end{proof}

%
We now focus on Property~\ref{enu:GammaDense}, in the function space $X = \StandardXSpace(\Omega)$
with $p \in (0,\infty]$ and $\Omega \subset \R^d$ a measurable set with nonzero measure.
First, as proved in Appendix~\ref{app:C0FunctionsCanBeExtended},
these spaces are indeed complete, and each $f \in X_p^k (\Omega)$
can be extended to an element $\widetilde{f} \in X_p^k (\R^d)$.
\begin{defn}[admissible domain]
For brevity, in the rest of the paper we refer to $\Omega \subseteq \R^{d}$
as an 
{\em admissible domain}
if, and only if, it is Borel-measurable with nonzero measure.
\end{defn}

\begin{lem}\label{lem:C0FunctionsCanBeExtended}
  Consider $\Omega \subseteq \R^{d}$ an admissible domain, $k \in \N$, and $C_{0}(\R^{d};\R^{k})$
  the space of continuous functions $f: \R^{d} \to \R^{k}$ that vanish at infinity.

  For $0<p<\infty$, we have $\StandardXSpace(\Omega) = \{f|_{\Omega}: f \in \StandardXSpace(\R^{d})\}$;
  likewise, $\StandardXSpace[k][\infty](\Omega) = \{f|_{\Omega}: f \in C_{0}(\R^{d};\R^{k})\}$.
  The spaces $\StandardXSpace(\Omega)$ are quasi-Banach spaces.
\end{lem}

In light of definitions~\eqref{eq:WSigmaDefinition}--\eqref{eq:NSigmaDefinition}, we have
\[
  \bigcup_{n \in \N_{0}} \StandardSigmaW = \bigcup_{n \in \N_{0}} \StandardSigmaN
  =
\NNreal_{\infty,L,\infty}^{\varrho,d,k}(\Omega) \cap X
  =: \StandardSigma[\infty],
\]
with $L := \sup_{n} \mathscr{L}(n) \in \N \cup \{+\infty\}$.
Properties~\ref{enu:GammaScaling} and~\ref{enu:GammaAdditive} imply that $\StandardSigma[\infty]$
is a linear space. 
We study its density in $X$, dealing first with a few degenerate cases.

\subsubsection{Degenerate cases}

Property~\ref{enu:GammaDense} can fail to hold for certain activation functions:
when $\varrho$ is a polynomial and $\mathscr{L}$ is bounded,
the set $\StandardSigma[\infty]$ only contains polynomials of bounded degree,
hence for nontrivial $\Omega$, $\StandardSigma[\infty]$ is not dense in $X$.
Property~\ref{enu:GammaDense} fails again for networks with a single hidden layer ($L=2$)
and certain domains such as $\Omega = \R^{d}$.
Indeed, the realization of any network in $\NNreal^{\varrho,d,k}_{\infty,2,\infty}$
is a finite linear combination of ridge functions $x \mapsto \varrho(A_{i}x+b_{i})$.
A ridge function is in $L_{p}(\R^{d})$ ($p<\infty$) only if it is zero.
Moreover, one can check that
if a linear combination of ridge functions belongs to $L_{p}(\R^{d})$ ($1 \leq p \leq 2$),
then it vanishes, hence $ \StandardSigma[\infty] = \{0\}$.


%
%
%
%

\subsubsection{Non-degenerate cases}
We now
 show that Property~\ref{enu:GammaDense} holds
under proper assumptions on the activation function $\varrho$,
the depth growth function $\mathscr{L}$, and the domain $\Omega$.
The proof uses the celebrated universal approximation theorem for multilayer feedforward networks
\cite{PinkusUniversalApproximation}.
%
%
In light of the above observations we introduce the following definition:
\begin{defn}\label{defn:NonDegenerateRho}
  An activation function $\varrho : \R \to \R$ is called \emph{non-degenerate} if the following hold:
  \begin{enumerate}[leftmargin=0.6cm]
    \item $\varrho$ is Borel measurable;

    \item $\varrho$ is locally bounded, that is, $\varrho$ is bounded on $[-R,R]$ for each $R > 0$;

    \item there is a closed null-set $A \subset \R$
          such that $\varrho$ is continuous at every $x_0 \in \R \setminus A$;

    \item there does not exist a polynomial $p : \R \to \R$ such that
          $\varrho(x) = p(x)$ for almost all $x \in \R$.
          \qedhere
  \end{enumerate}
\end{defn}

\begin{rem*}
  A \emph{continuous} activation function is non-degenerate if and only if it is not a polynomial.
\end{rem*}

These are precisely the assumptions imposed on the activation function
in \cite{PinkusUniversalApproximation}, where the following version of the universal approximation
theorem is shown:

\begin{thm}[{\cite[Theorem 1]{PinkusUniversalApproximation}}]\label{thm:PinkusUniversalApproximation}
  Let $\varrho : \R \to \R$ be a non-degenerate activation function,
  $K \subset \R^d$ be compact, $\eps > 0$, and $f : \R^d \to \R$ be continuous.
  Then there is $N \in \N$ and suitable $b_j, c_j \in \R$, $w_j \in \R^d$, $1 \leq j \leq N$ such that
  $g : \R^d \to \R, x \mapsto \sum_{j=1}^N c_j \, \varrho(\langle w_j , x \rangle + b_j)$
  satisfies $\| f - g\|_{L_\infty(K)} \leq \eps$.
\end{thm}

We prove in Appendix~\ref{app:ApproximationSpacesWellDefinedDensity} that Property~\ref{enu:GammaDense}
holds under appropriate assumptions:

\begin{thm}[Density]\label{th:ApproximationSpacesWellDefinedDensity}
  Consider $\varrho: \R \to \R$ a Borel measurable, locally bounded activation function,
  $\mathscr{L}$ a depth growth function, and $p \in (0,\infty]$.
  Set $L := \sup_{n \in \N} \mathscr{L}(n) \in \N \cup \{+\infty\}$.
  \begin{enumerate}[leftmargin=0.6cm]
    \item \label{enu:AppSetInLp}
          Let $\Omega \subset \R^{d}$ be a \emph{bounded} admissible domain, and assume that $L  \geq 2$.
          \begin{enumerate}
            \item \label{enu:AppSetInLp1}
                  For $p  \in (0,\infty)$ we have
                  \(
                    \NNreal_{\infty,\infty,\infty}^{\varrho,d,k}(\Omega)
                    \subset \StandardXSpace(\Omega)
                  \);

            \item \label{enu:AppSetInLp2}
                  For $p = \infty$ the same holds if $\varrho$ is continuous;

            \item \label{enu:AppSetInDense1}
                  For $p \in (0,\infty)$, if $\varrho$ is non-degenerate then
                  $\Sigma_\infty (\StandardXSpace(\Omega), \varrho, \mathscr{L})$
                  is dense in $\StandardXSpace(\Omega)$;

            \item \label{enu:AppSetInDense2}
                  For $p=\infty$, the same holds if $\varrho$ is non-degenerate and continuous.
          \end{enumerate}

        %

    \item \label{enu:AppSetDenseIndicator}
          Assume that the $L_p$-closure of
          $\NNreal_{\infty, L,\infty}^{\varrho,d,1} \cap \StandardXSpace[] (\R^d)$ contains
          a function $g : \R^d \to \R$ such that: 
          \begin{enumerate}
            \item There is a non-increasing function $\mu : [0,\infty) \to [0,\infty)$
                  satisfying $\int_{\R^d} \mu(|x|) \, d x < \infty$
                  and furthermore $|g(x)| \leq \mu(|x|)$ for all $x \in \R^d$.

            \item $\int_{\R^d} g(x) \, d x \neq 0$; note that this integral is well-defined,
                  since $\int_{\R^d} |g(x)| \, d x \leq \int_{\R^d} \mu(|x|) \, d x < \infty$.
          \end{enumerate}
          Then $\Sigma_\infty (\StandardXSpace(\Omega), \varrho, \mathscr{L})$ is dense in
          $\StandardXSpace[k](\Omega)$ for every admissible domain $\Omega \subseteq \R^{d}$
          and every $k \in \N$.\qedhere

  \end{enumerate}
\end{thm}


\begin{rem*}
  Claim~(\ref{enu:AppSetDenseIndicator}) applies to any admissible domain, bounded or not.
  Furthermore, it should be noted that the first assumption (the existence of $\mu$)
  is always satisfied if $g$ is bounded and has compact support.
\end{rem*}


\begin{cor}\label{cor:DensityBounded}
  Property~\ref{enu:GammaDense} holds for any \emph{bounded} admissible domain $\Omega \subset \R^d$
  and $p \in (0,\infty]$ as soon as $\sup_{n} \mathscr{L}(n) \geq 2$ and $\varrho$ is continuous
  and not a polynomial.
\end{cor}

\begin{cor}\label{cor:DensityBoundedOrNot}
  Property~\ref{enu:GammaDense} holds for any \emph{(even unbounded)} admissible domain
  $\Omega \subseteq \R^d$ and $p \in (0,\infty]$ as soon as $L := \sup_{n} \mathscr{L}(n) \geq 2$
  and as long as $\varrho$ is continuous and such that $\NNreal_{\infty, L,\infty}^{\varrho,d,1}$
  contains a \emph{compactly supported, bounded, non-negative} function $g \neq 0$.
\end{cor}


In Section~\ref{sec:AppSpacesReLU}, we show that the assumptions of Corollary~\ref{cor:DensityBoundedOrNot}
indeed hold when $\varrho$ is the ReLU or one of its powers, provided $L \geq 3$
(or $L \geq 2$ in input dimension $d=1$).
This is a consequence of the following lemma, whose proof we defer to
Appendix~\ref{app:ApproximationOfIndicatorCube}.

\begin{lem}\label{lem:ApproximationOfIndicatorCube}
  Consider $\varrho:\R \to \R$ and $W,N,L \in \N$.
  Assume there is $\sigma \in \NNreal^{\varrho,1,1}_{W,L,N}$ such that
  \begin{equation}
    \sigma(x)
    = \begin{cases}
        0, & \text{if } x \leq 0 \\
        1, & \text{if } x \geq 1
      \end{cases}
    \qquad \text{and} \qquad
    0 \leq \sigma(x) \leq 1 \quad \forall \, x \in \R \, .
    \label{eq:ExactSquashing}
  \end{equation}
  Then the following hold:
  \begin{enumerate}[leftmargin=0.6cm]
    \item For $d \in \N$ and $0<\varepsilon<\tfrac{1}{2}$ there is
          $h \in \NNreal^{\varrho,d,1}_{2dW(N+1), 2L - 1,(2d+1)N}$
          with $0 \leq h \leq 1$,  $\supp(h) \subset [0,1]^{d}$, and
          \begin{equation}\label{eq:PointwiseIndicatorApproximation}
            |h(x) - \Indicator_{[0,1]^d} (x)|
            \leq \Indicator_{[0,1]^d \setminus [\varepsilon, 1-\varepsilon]^d} (x)
            \quad \forall \, \, x \in \R^d \, .
          \end{equation}
          For input dimension $d=1$, this holds for some  $h \in \NNreal_{2W,L,2N}^{\varrho,1,1}$.

    \item There is $L' \leq 2L-1$ (resp. $L' \leq L$ for input dimension $d=1$)
          such that for each hyper-rectangle $[a,b] := \prod_{i=1}^d [a_i, b_i]$
          with $d \in \N$ and $-\infty < a_{i} < b_{i} < \infty$,
          each $p \in (0,\infty)$, and each $\varepsilon > 0$, there is a compactly supported,
          nonnegative function $0 \leq g \leq 1$ such that $\supp(g) \subset [a,b]$,
          \[
            \| g - \Indicator_{[a,b]} \|_{L_p (\R^d)} < \varepsilon,
          \]
          and $g = \Realization(\Phi)$ for some
          $\Phi \in \NNsymbol^{\varrho,d,1}_{2dW(N+1), L',(2d+1)N}$
          with $L(\Phi) = L'$.
          For input dimension $d=1$, this holds for some
          $\Phi \in \NNsymbol_{2W,L',2N}^{\varrho,1,1}$ with $L(\Phi) = L'$.
          \qedhere
\end{enumerate}
\end{lem}

With the elements established so far, we immediately get the following theorem.
\begin{thm}\label{th:DNNApproxSpaceWellDefined}
  Consider $\varrho: \R \to \R$ an activation function, $\mathscr{L}$ a depth growth function,
  $d \in \N$, $p \in (0,\infty]$ and $\Omega \subseteq \R^d$ an admissible domain.
  Set $L := \sup_{n \in \N} \mathscr{L}(n) \in \N \cup \{+\infty\}$.
  Assume that at least one of the following properties holds:
  \begin{enumerate}[leftmargin=0.6cm]
    \item $\varrho$ is continuous and not a polynomial, $L \geq 2$, \emph{and $\Omega$ is bounded};

    \item $\NNreal^{\varrho,d,1}_{\infty,L,\infty} \cap \StandardXSpace[] (\R^d)$
          contains some compactly supported, bounded, non-negative $g \neq 0$.
  \end{enumerate}
  Then for every $k \in \N$, $\alpha  > 0$, $q \in (0,\infty]$,
  and with $X = \StandardXSpace (\Omega)$ as in Equation~\eqref{eq:StandardXSpace}, we have:
  \begin{itemize}[leftmargin=0.6cm]
    \item Properties~\ref{enu:GammaContainsZero}--\ref{enu:GammaDense} are satisfied
          for $\Sigma_{n} = \StandardSigmaW$ (resp.~for $\Sigma_{n} = \StandardSigmaN$);

    \item $\big( \WASpace[X], \|\cdot\|_{\WASpace[X]} \big)$
          and $\big( \NASpace[X], \|\cdot\|_{\NASpace[X]} \big)$ are (quasi)-Banach spaces.
          \qedhere
  \end{itemize}
\end{thm}

In particular, if $\varrho$ is continuous and satisfies the assumptions
of Lemma~\ref{lem:ApproximationOfIndicatorCube} for some $L \in \N$
and if $\sup_{n \in \N} \mathscr{L}(n) \geq 2 L - 1$
(or $\sup_{n \in \N} \mathscr{L}(n) \geq L$ in case of $d = 1$),
then the conclusions of Theorem~\ref{th:DNNApproxSpaceWellDefined} hold
on \emph{any} admissible domain.




\subsection{Discussion and perspectives}


One could envision defining approximation classes where the sets $\AppSet_{n}$
incorporate additional constraints besides $L \leq \mathscr{L}(n)$.
For the theory to hold, one must however ensure either that:
a) the additional constraints are weak enough to ensure the approximation errors
   (and therefore the approximation spaces) are unchanged---cf.~the discussion of strict
   {\em vs} generalized networks; or, more interestingly, that
b) the constraint gets sufficiently relaxed when $n$ grows,
   to ensure compatibility with the additivity property.

As an example, constraints of potential interest include a lower (resp. upper)
bound on the minimum width $\min_{1 \leq \ell \leq L-1} N_{\ell}$
(resp.~maximum width $\max_{1 \leq \ell \leq L-1} N_{\ell}$),
since they impact the memory needed to compute ``in place'' the output of the network.

While network families with a fixed lower bound on their minimum width do satisfy the additivity
Property~\ref{enu:GammaAdditive}, this is no longer the case of families with a \emph{fixed} upper bound
on their \emph{maximum} width.
Consider now a complexity-dependent upper bound $f(n)$ for the maximum width.
Since ``adding'' two networks of a given width yields one with width at most doubled,
the additivity property will be preserved provided that $2f(n) \leq f(cn)$ for some $c \in \N$
and all $n \in \N$.
This can, e.g., be achieved with $f(n) := \lfloor\alpha n\rfloor$,
with the side effect that for $n<1/\alpha$ the set $\AppSet_{n}$
only contains affine functions.

\section{Approximation spaces of the ReLU and its powers}
\label{sec:AppSpacesReLU}

The choice of activation function
has a decisive influence on the approximation spaces $\WASpace[X]$ and $\NASpace[X]$.
As evidence of this, consider the following result.

\begin{thm}[{\cite[Theorem 4]{PinkusLowerBoundsForMLPApproximation}}]
  \label{th:PinkusLowerBoundsForMLPApproximation}

  There exists an \emph{analytic} squashing function%
  \footnote{A function $\sigma: \R \to \R$ is a {\em squashing function}
            if it is nondecreasing with $\lim_{x \to -\infty}\sigma(x) \to 0$
            and $\lim_{x \to \infty} \sigma (x) \to 1$;
            see \cite[Definition 2.3]{Hornik1989universalApprox}.}
  $\varrho: \R \to \R$ such that: for any $d \in \N$, any continuous function from $\Omega = [0,1]^{d}$
  to $\R$ can be approximated arbitrarily well in the uniform norm by a strict
  $\varrho$-network with $L=3$ layers and $W \leq 21 d^2 + 15d + 3$ connections.
\end{thm}

Consider the pathological activation function $\varrho$ from
Theorem~\ref{th:PinkusLowerBoundsForMLPApproximation} and a depth growth function $\mathscr{L}$
satisfying $L := \sup_{n} \mathscr{L}(n) \geq 2$.
Since $\varrho$ is continuous and not a polynomial,
we can apply Theorem~\ref{th:DNNApproxSpaceWellDefined};
hence $\WASpace[X]$ and $\NASpace[X]$ are well defined quasi-Banach spaces
for each bounded admissible domain $\Omega$, $p \in (0,\infty]$ and $X = \StandardXSpace(\Omega)$.
Yet, if $L \geq 3$ there is $n_{0}$ so that $\mathscr{L}(n) \geq 3$ for $n \geq n_{0}$,
and  the set $\StandardSigma$ is dense in $X$ for any $p \in (0,\infty]$
provided that $n \geq \max \{ n_{0}, 21 d^2 + 15 d + 3 \}$; hence $\AppErr (f,\StandardSigma)_X = 0$
for any $f \in X$ and any such $n$, showing that $\WASpace[X] = \NASpace[X] = X$
with equivalent (quasi)-norms.

The approximation spaces generated by pathological activation functions such as in
Theorem~\ref{th:PinkusLowerBoundsForMLPApproximation} are so degenerate that they are uninteresting
both from a practical perspective
(computing a near best approximation with such an activation function is hopeless)
and from a theoretical perspective (the whole scale of approximation spaces collapses
to $\StandardXSpace (\Omega)$).

Much more interesting is the study of approximation spaces generated by commonly used
activation functions such as the ReLU $\varrho_{1}$ or its powers $\varrho_{r}$, $r \in \N$.
For any admissible domain, generalized and strict $\varrho_{r}$-networks
indeed yield well-defined approximations spaces that coincide.

\begin{thm}[Approximation spaces of generalized and strict $\varrho_{r}$-networks]
  \label{th:ReLUDNNApproxSpaceWellDefined}

  Let $r \in \N$ and define $\varrho_r : \R \to \R, x \mapsto (x_+)^r$, where $x_+ := \max \{0, x\}$.
  Consider $X := \StandardXSpace(\Omega)$ with $p \in (0,\infty]$, $d,k \in \N$
  and $\Omega \subseteq \R^{d}$ an arbitrary admissible domain.
  Let $\mathscr{L}$ be any depth growth function.
  \begin{enumerate}[leftmargin=0.6cm]
    \item  For each $\alpha  > 0,q \in (0,\infty]$, $r \in \N$ we have
           \[
             \SWASpace[X][\varrho_r] = \WASpace[X][\varrho_r]
             \quad \text{and} \quad
             \SNASpace[X][\varrho_r] = \NASpace[X][\varrho_r]
           \]
           and there is $C < \infty$ such that
           \begin{alignat*}{3}
             \|\cdot\|_{\WASpace[X][\varrho_r]}
             &\leq \|\cdot\|_{\SWASpace[X][\varrho_r]}
             &&\leq C \|\cdot\|_{\WASpace[X][\varrho_r]}\, ,\\
             \|\cdot\|_{\NASpace[X][\varrho_r]}
             &\leq \|\cdot\|_{\SNASpace[X][\varrho_r]}
             &&\leq C \|\cdot\|_{\NASpace[X][\varrho_r]} \, .
           \end{alignat*}

    \item If the depth growth function $\mathscr{L}$ satisfies
          \[
            \sup_{n \in \N} \mathscr{L}(n)
            \geq \begin{cases}
                   2, & \text{if}\ \Omega\ \text{is bounded \emph{or}}\ d=1 \\
                   3, & \text{otherwise}
                 \end{cases}
          \]
          then, for each $\alpha  > 0$, $q \in (0,\infty]$, $r \in \N$ and $\varrho := \varrho_{r}$,
          the following hold:
          \begin{itemize}[leftmargin=0.6cm]
            \item Properties~\ref{enu:GammaContainsZero}--\ref{enu:GammaDense} are satisfied
                  for $\Sigma_{n} = \StandardSigmaW$ (resp.~for $\Sigma_{n} = \StandardSigmaN$);

            \item $\big( \WASpace[X], \|\cdot\|_{\WASpace[X]} \big)$
                  and $\big( \NASpace[X], \|\cdot\|_{\NASpace[X]} \big)$ are (quasi)-Banach spaces.
                  \qedhere
          \end{itemize}
  \end{enumerate}
%
%
%
\end{thm}

\begin{rem}
For a bounded domain or when $d=1$, the second claim holds for any depth growth function
allowing at least one hidden layer.
In the other cases, the restriction to at least two hidden layers is unavoidable
(except for some exotic unbounded domains with vanishing mass at infinity)
as the only realization of a $\varrho_{r}$-network of depth two that belongs to $\StandardXSpace[](\R^{d})$
is the zero network.
\end{rem}

\begin{proof}[Proof of Theorem~\ref{th:ReLUDNNApproxSpaceWellDefined}]
By Lemma~\ref{lem:ReLUPowerRepresentsIdentity}, $\varrho_{r}$ can represent the identity
using $2r + 2$ terms.
By Theorem~\ref{th:DNNApproxSpaceWellDefinedStrict}, this establishes the first claim.
The second claim follows from Theorem~\ref{th:DNNApproxSpaceWellDefined},
once we show that we can apply the latter.
For bounded $\Omega$, this is clear, since $\varrho_r$ is continuous and not a polynomial,
and hence non-degenerate. For general $\Omega$, we relate $\varrho_{r}$ to B-splines
to establish the following lemma (which we prove below).

\begin{lem}\label{lem:ExactSquashingReLUPower}
  For any $r \in \N$ there is $\sigma_{r} \in \SNNreal^{\varrho_{r},1,1}_{2(r+1),2,r+1}$
  satisfying~\eqref{eq:ExactSquashing}.
\end{lem}

\noindent Combined with Lemma~\ref{lem:ApproximationOfIndicatorCube}, we obtain
the existence of a compactly supported, continuous, non-negative function $g \neq 0$ such that
\(
  g \in \NNreal^{\varrho_{r},d,1}_{\infty, 3,\infty} \cap \StandardXSpace[](\R^d)
\)
(respectively
\(
  g \in \NNreal^{\varrho_{r},d,1}_{\infty,2,\infty} \cap \StandardXSpace[](\R)
\)
for input dimension $d=1$).
Hence, Theorem~\ref{th:DNNApproxSpaceWellDefined} is applicable.
%
\end{proof}

\begin{defn}[B-splines]\label{defn:Bsplines}
  For any function $f: \R \to \R$, define $\Delta f: x \mapsto f(x)-f(x-1)$.
  Let $\varrho_{0} := \Indicator_{[0,\infty)}$ denote the Heaviside function,
  and $\beta_{+}^{(0)} := \Indicator_{[0,1)} = \Delta \varrho_{0}$ the B-spline of degree $0$.
  The B-spline of degree $n$ is obtained by convolving $\beta_{+}^{(0)}$ with itself $n+1$ times:
  \[
    \beta_{+}^{(n)}
    := \underbrace{\beta^{(0)}_{+} \star \ldots \star \beta^{(0)}_{+}}_{n+1\ \text{factors}}.
  \]
  For $n \geq 0$, $\beta_{+}^{(n)}$ is non-negative and is zero except for $x \in [0,n+1]$.
  We have $\beta_{+}^{(n)} \in C^{n-1}_{c}(\R)$ for $n \geq 1$.
  Indeed, this follows since $\varrho_n \in C^{n-1}(\R)$, and since it is known
  (see \cite[Equation (10)]{Unser:1999cr}, noting that \cite{Unser:1999cr} uses \emph{centered} B-splines)
  that the B-spline of degree $n$ can be decomposed as
  \begin{equation} \label{eq:DecompBSpline}
    \beta_{+}^{(n)}
    = \frac{\Delta^{n+1} \varrho_{n}}{n!}
    = \frac{1}{n!} \sum_{k=0}^{n+1} \binom{n+1}{k} (-1)^{k} \varrho_{n}(x-k).\qedhere
  \end{equation}
\end{defn}

\begin{proof}[Proof of Lemma~\ref{lem:ExactSquashingReLUPower}]
For $n \geq 0$, $\beta_{+}^{(n)}$ is non-negative and is zero except for $x \in [0,n+1]$.
Its primitive
\[
  g_{n}(x)
  := \int_{0}^{x} \beta_{+}^{(n)} (t) dt
\]
is thus non-decreasing, with $g_{n}(x) = 0$ for $x \leq 0$
and $g_{n}(x) = g_{n}(n+1)$ for $x \geq n+1$.
Since $\beta_{+}^{(n)} \in C^{n-1}_{c}(\R)$ for $n \geq 1$, we have $g_{n} \in C^{n}(\R)$ for $n \geq 1$.
Furthermore, $g_0 \in C^0 (\R)$ since $\beta^{(0)}_+$ is bounded.

For $r \geq 1$, the above facts imply that the function
$\sigma_{r}(x) := g_{r-1}(rx) / g_{r-1}(r)$ belongs to $C^{r-1}(\R)$ and satisfies \eqref{eq:ExactSquashing}.
To conclude, we now prove that $\sigma_{r} \in \SNNreal^{\varrho_{r},1,1}_{2(r+1),2,r+1}$.
For $0 \leq k \leq n+1$ we have
\begin{align*}
  \int_{0}^{x} \varrho_{n}(t-k)dt
  &=  \begin{cases}
        0,                          & \text{if } x \leq k \\
        \int_{k}^{x} (t-k)^{n} dt
        = \int_{0}^{x-k}t^{n}dt
        = \tfrac{(x-k)^{n+1}}{n+1}, & \text{otherwise}
      \end{cases}
  = \frac{\varrho_{n+1}(x-k)}{n+1} \, .
\end{align*}
By~\eqref{eq:DecompBSpline} it follows that
\[
  g_{n}(x)
  = \frac{1}{(n+1)!} \sum_{k=0}^{n+1} \binom{n+1}{k} (-1)^{k} \varrho_{n+1}(x-k),
\]
and hence
\[
  \sigma_{r}(x)
  = \frac{g_{r-1}(rx)}{g_{r-1}(r)}
  = \frac{1}{r!\ g_{r-1}(r)} \sum_{k=0}^{r} \binom{r}{k} (-1)^{k} \varrho_{r}(rx-k) \, .
\]
Setting $\alpha_{1} := \varrho_{r} \otimes \ldots \otimes \varrho_{r}: \R^{r+1} \to \R^{r+1}$
as well as $T_1: \R \to \R^{r+1}, x \mapsto (rx-k)_{k=0}^{r}$ and
\[
  T_{2}: \R^{r+1} \to \R,
         y = (y_{k})_{k=0}^{r} \mapsto \tfrac{1}{r!\ g_{r-1}(r)} \sum_{k=0}^{r}
                                                                   \binom{r}{k} (-1)^{k} y_{k}
\]
and $\Phi := \big( (T_{1},\alpha_{1}),(T_{2},\identity_{\R}) \big)$, it is then easy to check that
$\sigma_{r} = \Realization(\Phi)$.
Obviously $L(\Phi)=2$, $N(\Phi)=r+1$, and $\|T_{i}\|_{\ell^{0}} = r+1$ for $i=1,2$,
hence as $\Phi$ is strict we have $\Phi \in \SNNsymbol^{\varrho_{r},1,1}_{2(r+1),2,r+1}$.
\end{proof}

%
%
%
%


\subsection{Piecewise polynomial activation functions \emph{vs.}~\texorpdfstring{$\varrho_{r}$}{ϱᵣ}}

In this subsection, we show that approximation spaces of $\varrho_{r}$-networks contain
the approximation spaces of continuous piecewise polynomial activation functions,
and match those of (free-knot) spline activation functions.

\begin{defn}\label{defn:PiecewisePolynomial}
  Consider an interval $I \subseteq \R$.
  A function $f: I \to \R$ is \emph{piecewise polynomial} 
  if there are \emph{finitely} many intervals $I_i \subset I$ such that $I = \bigcup_i I_i$
  and $f|_{I_i}$ is a polynomial.
  It is \emph{of degree at most $r \in \N$} when each $f|_{I_{i}}$ is of degree at most $r$,
  and \emph{with at most $n \in \N$ pieces} (or \emph{with at most $n-1 \in \N_{0}$ breakpoints})
  when there are at most $n$ such intervals.
  The set of piecewise polynomials of degree at most $r$ with at most $n$ pieces
  is denoted $\PPoly^{r}_{n}(I)$, and we set
  $\PPoly^{r}(I) := \cup_{n \in \N} \PPoly^{r}_{n}(I)$.

  A function $f \in \Spline^{r}_{n}(I) := \PPoly^{r}_{n}(I) \cap C^{r-1}(I)$
  is called a \emph{free-knot spline} of degree at most $r$ with at most $n$ pieces
  (or at most $n-1$ breakpoints).
  We set $\Spline^{r}(I) := \cup_{n \in \N} \Spline^{r}_n(I)$.
\end{defn}


\begin{thm}\label{thm:ReLUPowersApproxSpaces}
Consider a depth growth function $\mathscr{L}$, an admissible domain
$\Omega \subset \R^d$, and let $X = \StandardXSpace(\Omega)$ with $d,k \in \N$, $p \in (0,\infty]$.
Let $r \in \N$, set $\varrho_r : \R \to \R, x \mapsto (x_+)^r$,
and let $\alpha>0$, $q \in (0,\infty]$.
\begin{enumerate}[leftmargin=0.6cm]
\item \label{it:Nested0}
      If $\varrho:\R \to \R$ is continuous and piecewise polynomial of degree at most $r$ then,
      \begin{equation}\label{eq:UnboundedDomainGeneralNesting}
        \WASpace[X]
        \hookrightarrow
        \begin{cases}
          \WASpace[X][\varrho_r][q][\alpha][\max(\mathscr{L}+1,2)] , & \text{if } d = 1 \\
          \WASpace[X][\varrho_r][q][\alpha][\max(\mathscr{L}+1,3)] , & \text{if } d \geq 2.
        \end{cases}
      \end{equation}
      Moreover
         if $\Omega$ is bounded, or if $r=1$, or if $\mathscr{L}+1 \preceq \mathscr{L}$, then
              we further have
              \begin{equation}\label{eq:Nesting1}
                \WASpace[X] \hookrightarrow \WASpace[X][\varrho_r].
              \end{equation}

        \item \label{it:Nested01} If $\varrho \in \Spline^{r}(\R)$ is not a polynomial and $\Omega$ is bounded, then we have (with equivalent norms)
              \begin{equation}\label{eq:Nesting1Eq}
                \WASpace[X] = \WASpace[X][\varrho_r]. 
              \end{equation}

\item \label{it:Nested1}
      For any $s \in \N$ we have
      \begin{equation}\label{eq:Nesting0}
        \WASpace[X][\varrho_{r^{s}}]
        \hookrightarrow \WASpace[X][\varrho_{r}][q][\alpha][1+s(\mathscr{L}-1)]. 
      \end{equation}
\end{enumerate}
The same results hold with $\NASpace[X][\cdot][q][\alpha][\cdot]$
instead of $\WASpace[X][\cdot][q][\alpha][\cdot]$.
\end{thm}

Examples~\ref{ex:polyloggrowth} and \ref{ex:polygrowth} provide
important examples of depth growth functions $\mathscr{L}$
with $\mathscr{L}+1 \preceq \mathscr{L}$, so that~\eqref{eq:Nesting1} 
holds on any domain.

\begin{rem}[Nestedness]\label{rem:ReLUrNested}
  For $1 \leq r' \leq r$, the function $\varrho := \varrho_{r'}$ is indeed
  a continuous piecewise polynomial with two pieces of degree at most $r$.
  Theorem~\ref{thm:ReLUPowersApproxSpaces} thus implies that if $\Omega$ is bounded
  or $\mathscr{L} + 1 \preceq \mathscr{L}$, then
  $\WASpace[X][\varrho_{r'}] \hookrightarrow \WASpace[X][\varrho_r]$
  and $\NASpace[X][\varrho_{r'}] \hookrightarrow \NASpace[X][\varrho_r]$.
  We will see in Corollary~\ref{cor:SaturationProp} below that
  if $2 \mathscr{L} \preceq \mathscr{L}$, then these embeddings
  are indeed equalities if $2 \leq r' \leq r$.
\end{rem}

The main idea behind the proof of Theorem~\ref{thm:ReLUPowersApproxSpaces} given below
is to combine Lemma~\ref{lem:RecursiveNNsets} and its consequences with the following results
proved in Appendices~\ref{app:PiecewisePolynomialCont}--\ref{app:SplineVsReLUPow}.

\begin{lem}\label{lem:PiecewisePolynomialCont}
  Consider $\varrho:\R \to \R$ a continuous piecewise polynomial function
  with at most $n \in \N$ pieces of degree at most $r \in \N$.
  With%
  \footnote{Note that $4 = 1 \!\mod 3$ and hence $4^n - 1 = 0 \!\mod 3$, so that $w \in \N$.}
  $w := 2 \cdot (4^{r}-1) / 3$ and $m := 2^{r} - 1$ we have
  \[
    \varrho \in \overline{\NNreal^{\varrho_{r},1,1}_{4(r+1)+(n-1)w,2,2(r+1)+(n-1)m}},
  \]
  where the closure is with respect to the topology of locally uniform convergence.
  For $r=1$ (that is, when $\varrho$ is continuous and piecewise affine with at most $n \in \N$ pieces
  and $\varrho_r = \varrho_1$), we even have
  ${\varrho \in \SNNreal^{\varrho_{r},1,1}_{2(n+1),2,n+1}}$.
\end{lem}

\begin{lem}\label{lem:SplineVsReLUPow}
  Consider $r \in \N$ and $\varrho \in \Spline^{r} (\R)$.
  If $\varrho$ is not a polynomial then $\varrho_r \in \overline{\NNreal^{\varrho,1,1}_{5^{r}r!,2,3^{r}r!}}$,
  where the closure is with respect to locally uniform convergence.
\end{lem}

For bounded $\Omega$, locally uniform convergence on $\R^{d}$ implies
convergence in $X = \StandardXSpace(\Omega)$ for all $p \in (0,\infty]$.
To similarly ``upgrade'' locally uniform convergence to convergence in $X$ on unbounded domains,
we use the following localization lemma which is proved in Appendix~\ref{app:NetworkThatLocalizesAnyFunction}.

\begin{lem}
\label{lem:NetworkThatLocalizesAnyFunction}
  Consider $d,k \in \N$, $r \in \N_{\geq 2}$.
  There is $c = c(d,k,r) \in \N$ such that%
  \footnote{Notice the restriction to $W,N \geq 1$;
            in fact, the result of Lemma~\ref{lem:NetworkThatLocalizesAnyFunction}
            as stated cannot hold for $W=0$ or $N=0$.}
  for any $W,L,N \in \N$, $g \in \NNreal^{\varrho_{r},d,k}_{W,L,N}$, $R \geq 1,\delta > 0$,
  there is $g_{R,\delta} \in \NNreal^{\varrho_{r},d,k}_{cW,\max \{ L+1, 3 \},cN}$,
  such that
  \begin{equation}
    |g_{R,\delta} ( x) - (\Indicator_{[-R,R]^{d}} \cdot g)(x)|
    \leq  2 \cdot |g(x)| \cdot \Indicator_{[-R-\delta,R+\delta]^{d}  \setminus [-R,R]^{d}} (x)
    \qquad \forall \, x \in \R^d \, .
    \label{eq:SpecialCutoffMainEstimateEndToEnd}
  \end{equation}
  For $d=1$ the same holds with  $\max \{ L+1,2 \}$ layers instead of $\max \{ L+1,3 \}$.
\end{lem}

The following proposition describes how one can ``upgrade'' the locally uniform convergence
to convergence in $\StandardXSpace[] (\Omega)$,
at the cost of slightly increasing the depth of the approximating networks.

\begin{prop}\label{prop:HighReLUPowersApproximateLowPowersUnbounded}
  Consider $\Omega \subset \R^d$ an admissible domain and $X = \StandardXSpace(\Omega)$
  with $d,k \in \N$, $p \in (0,\infty]$.
  Assume $\varrho \in \overline{\NNreal^{\varrho_{r},1,1}_{\infty,2,m}}$ where the closure is
  with respect to locally uniform convergence and $r \in \N_{\geq 2}$, $m \in \N$.
  For any $W,N \in \N_{0} \cup \{\infty\}$, $L \in \N \cup \{\infty\}$ we have,
  with closure in $X$, 
  \[
    \NNreal_{W,L,N}^{\varrho,d, k}(\Omega) \cap X
    \subset \overline
            {
              \NNreal^{\varrho_r, d, k}_{cWm^{2}, \max \{ L+1,3 \},cNm}(\Omega)
             \cap X
            }^{X} \, ,
  \]
  where $c = c(d,k, r) \in \N$
  is as in Lemma~\ref{lem:NetworkThatLocalizesAnyFunction}.
  If $d=1$ the same holds with $\max \{ L+1,2 \}$ layers instead of $\max \{ L+1,3 \}$.
  If $\Omega$ is bounded, or if $\varrho \in \NNreal^{\varrho_{r},1,1}_{\infty,2,m}$ with $r=1$,
  then the same holds with $c=1$ and $L$ layers instead of $\max \{ L+1, 3 \}$
  (resp.~instead of $\max \{ L+1, 2 \}$ when $d=1$).
\end{prop}

The proof is in  Appendix~\ref{sub:ReLUPowerNestingUnbounded}.
We are now equipped to prove Theorem~\ref{thm:ReLUPowersApproxSpaces}.

\begin{proof}[Proof of Theorem~\ref{thm:ReLUPowersApproxSpaces}]
  We give the proof for $\WASpace[X][\cdot]$;
  minor adaptations yield the results for $\NASpace[X][\cdot]$.

  For Claim~(\ref{it:Nested0}), first note that Lemma~\ref{lem:PiecewisePolynomialCont}
  shows that there is some $m \in \N$ satisfying
  $\varrho \in \overline{\NNreal_{\infty,2,m}^{\varrho_r,1,1}}$,
  where the closure is with respect to locally uniform convergence.
  Define $\ell := 3$ if $d \geq 2$ (resp. $\ell := 2$ if $d=1$)
  and $\widetilde{\mathscr{L}}:= \max \{ \mathscr{L}+1,\ell \}$
  (resp.~$\widetilde{\mathscr{L}}:= \mathscr{L}$ when $\Omega$ is bounded or $r=1$)
  and consider $c \in \N$ as in Proposition~\ref{prop:HighReLUPowersApproximateLowPowersUnbounded}.
  Thus, since $\widetilde{\mathscr{L}}$ is non-decreasing, by
  Proposition~\ref{prop:HighReLUPowersApproximateLowPowersUnbounded} and
  Lemma~\ref{lem:BoundingLayersAndNeuronsByWeights} we have for all $n \in \N$
  \begin{align*}
    \WeightClassSymbol_n(X,\varrho,\mathscr{L})
    =           \NNreal^{\varrho,d,k}_{n,\mathscr{L}(n),\infty}(\Omega) \cap X
    & \subset \overline{
                \NNreal^{\varrho_r,d,k}_{cnm^{2},\widetilde{\mathscr{L}}(n),\infty}(\Omega) \cap X
              }^{X} \\
    & \subset \overline{
                \NNreal^{\varrho_r,d,k}_{cnm^{2},\widetilde{\mathscr{L}}(cnm^{2}),\infty}(\Omega) \cap X
              }^{X}
    =         \overline{\WeightClassSymbol_{cm^{2}n}(X,\varrho_r, \widetilde{\mathscr{L}})}^{X}.
  \end{align*}
Hence, for any $f \in X$ and $n \in \N$
\[
  \AppErr(f,\WeightClassSymbol_n(X,\varrho,\mathscr{L}))_{X}
  \geq \AppErr \big( f,\WeightClassSymbol_{cm^{2}n}(X,\varrho_r,\widetilde{\mathscr{L}})_{X} \big) .
\]
Thus, Lemma~\ref{lem:ApproximationSpaceElementaryNesting}
yields~\eqref{eq:UnboundedDomainGeneralNesting}. 
When $\Omega$ is bounded or $r=1$, as $\widetilde{\mathscr{L}} = \mathscr{L}$,
this yields~\eqref{eq:Nesting1}. 
When $\mathscr{L}+1 \preceq \mathscr{L}$, as
$\widetilde{\mathscr{L}} \leq \max \{ \mathscr{L}+1,\ell \} \leq \mathscr{L}+\ell+1$,
we have $\widetilde{\mathscr{L}} \preceq \mathscr{L}+\ell+1 \preceq \mathscr{L}$
by Lemma~\ref{lem:DepthGrowthLemma}, yielding again~\eqref{eq:Nesting1} 
by Lemma~\ref{lem:RoleGrowthFunctionRate}.

For Claim~(\ref{it:Nested01}),
if $\Omega$ is bounded and $\varrho \in \Spline^{r}(\R)$ is not a polynomial,
combining Lemma~\ref{lem:SplineVsReLUPow} with Lemma~\ref{lem:Recursivity}, we similarly get
the converse to
~\eqref{eq:Nesting1}. 
This establishes~\eqref{eq:Nesting1Eq}.

We now prove Claim~(\ref{it:Nested1}).
Since $\varrho_{r^s} = \varrho_r \circ \cdots \circ \varrho_r$ (where $\varrho$ appears $s$ times),
Lemma~\ref{lem:NestednessBasic} shows that
\(
  \NNreal^{\varrho_{r^{s}},d,k}_{W,L,N} \subset \NNreal^{\varrho_r,d,k}_{W+(s-1)N,1+s(L-1),sN}
\)
for all $W,L,N$.
Combining this with Lemma~\ref{lem:BoundingLayersAndNeuronsByWeights}, we obtain
\[
  \NNreal^{\varrho_{r^{s}},d,k}_{n,\mathscr{L}(n),\infty}
  \subset \NNreal^{\varrho_{r^{s}},d,k}_{n,\mathscr{L}(n),n}
  \subset \NNreal^{\varrho_r,d,k}_{sn,1+s(\mathscr{L}(n)-1),sn}
  \subset \NNreal^{\varrho_r,d,k}_{sn,1+s(\mathscr{L}(sn)-1),\infty}
  \quad \forall \, n \in \N .
\]
Therefore, we get for any $f \in X$ and $n \in \N$
\[
  \AppErr(f,\WeightClassSymbol_n(X,\varrho_{r^{s}},\mathscr{L}))_{X}
  \geq \AppErr \big( f,\WeightClassSymbol_{sn}(X,\varrho_{r},1+s(\mathscr{L}-1))_{X} \big). 
\]
Hence, we can finally apply Lemma~\ref{lem:ApproximationSpaceElementaryNesting} to obtain
\eqref{eq:Nesting0}.
\end{proof}

\begin{rem}
Inspecting the proofs, we see that if $\varrho \in \Spline^{r}$ has exactly one breakpoint
then $\varrho \in \NNreal^{\varrho_{r},1,1}_{w,2,m}$ and $\varrho_{r} \in \NNreal^{\varrho,1,1}_{w,2,m}$
for some $w,m \in \N$.
This is stronger than $\varrho \in \overline{\NNreal^{\varrho_{r},1,1}_{w,2,m}}$
(resp.~than $\varrho_{r} \in \overline{\NNreal^{\varrho,1,1}_{w,2,m}}$)
and implies~\eqref{eq:Nesting1Eq} 
with equivalent norms \emph{even on unbounded domains}.
Examples include the leaky ReLU \cite{Maas:tn},
the parametric ReLU \cite{He:2015:DDR:2919332.2919814},
and the absolute value which is used in scattering transforms \cite{Mallat:2016jr}.

Another spline of degree one is soft-thresholding, $\sigma(x) := x(1-\lambda/|x|)_{+}$,
which appears in Iterative Shrinkage Thresholding Algorithms (ISTA) for $\ell^{1}$ sparse recovery
in the context of linear inverse problems \cite[Chap.~3]{Foucart:2012wp} and has been used
in the Learned ISTA (LISTA) method \cite{Gregor2010}.
As  $\sigma \in \Spline^{1}$, using soft-thresholding as an activation function
on bounded $\Omega$ is exactly as expressive as using the ReLU.
\end{rem}

\subsection{Saturation property of approximation spaces with polynomial depth growth}
For certain depth growth functions, the approximation spaces of $\varrho_{r}$-networks
are independent of the choice of $r \geq 2$.

\begin{cor}\label{cor:SaturationProp}
  With the notations of Theorem~\ref{thm:ReLUPowersApproxSpaces},
  if $2 \mathscr{L} \preceq \mathscr{L}$ then for every $r \in \N_{\geq 2}$ we have
  \begin{align}
    \WASpace[X][\varrho_{1}]
    &\hookrightarrow \WASpace[X][\varrho_{2}]
    =                \WASpace[X][\varrho_{r}] \, ,\label{eq:SaturationPropWSpace} \\
    \NASpace[X][\varrho_{1}]
    & \hookrightarrow \NASpace[X][\varrho_{2}]
    =                 \NASpace[X][\varrho_{r}] \, ,\label{eq:SaturationPropNSpace}
  \end{align}
  where the equality is with equivalent quasi-norms.
\end{cor}

\begin{example}
  By Example~\ref{ex:polygrowth}, for polynomially growing depth we do have $2\mathscr{L} \preceq\mathscr{L}$.
  This includes the case $\mathscr{L}(n) = n+1$,
  which gives the same approximation spaces as $\mathscr{L} \equiv \infty$;
  see Remark~\ref{rem:InfiniteDepthOrN}.
\end{example}

In words, approximation spaces of $\varrho_{r}$-networks with appropriate depth growth
have a saturation property: increasing the degree $r$ beyond $r=2$ does not pay off
in terms of the considered function spaces.
Note, however, that the constants in the norm equivalence may still play a qualitative role in practice.

\begin{proof}
We prove~\eqref{eq:SaturationPropWSpace}, the proof of~\eqref{eq:SaturationPropNSpace} is similar.
By Lemma~\ref{lem:DepthGrowthLemma}, since  $2\mathscr{L} \preceq \mathscr{L}$
we have $a \mathscr{L} + b \sim \mathscr{L}$ for all $a \geq 1$, $b \geq 1-a$.
In particular, $\mathscr{L} + 1 \preceq \mathscr{L}$ hence \eqref{eq:Nesting1} holds
with $\varrho = \varrho_{r'}$, $r' \in \N$, $1 \leq r' \leq r$.
Combined with~\eqref{eq:Nesting0} and Lemma~\ref{lem:RoleGrowthFunctionRate},
since $r \leq 2^r$ for $r \in \N$ we see
\[
  \WASpace[X][\varrho_{r}][q][\alpha][\mathscr{L}]
  \hookrightarrow
  \WASpace[X][\varrho_{2^r}][q][\alpha][\mathscr{L}]
  \hookrightarrow
  \WASpace[X][\varrho_2][q][\alpha][1 + r (\mathscr{L} - 1)]
  \hookrightarrow
   \WASpace[X][\varrho_2][q][\alpha][\mathscr{L}]
  \hookrightarrow
   \WASpace[X][\varrho_{r}][q][\alpha][\mathscr{L}]
\]
for all $r \in \N_{\geq 2}$.
In the middle we used that
$1 + r (\mathscr{L} - 1) \preceq 1 + r \mathscr{L} \preceq (1 + r) \mathscr{L} \preceq \mathscr{L}$.
\end{proof}

\subsection{Piecewise polynomial activation functions yield non-trivial approximation spaces}
\label{sub:ApproxSpacesNonTrivial}

In light of the pathological example of Theorem~\ref{th:PinkusLowerBoundsForMLPApproximation},
it is important to check that the approximation spaces $\WASpace$ and $\NASpace$
with $\varrho = \varrho_{r}$, $r \in \N$, are \emph{non-trivial}:
they are proper subspaces of $\StandardXSpace(\Omega)$.
This is what we prove for any continuous and piecewise polynomial activation function $\varrho$.

\begin{thm}\label{thm:ApproximationSpacesNonTrivial}
  Let $\varrho : \R \to \R$ be continuous and piecewise polynomial (with finitely many pieces),
  let $\Omega \subset \R^d$ be measurable with nonempty interior, and let $s > 0$.
  Let $p,q \in (0,\infty]$, $k \in \N$, $\alpha \in (0,\infty)$, and $X = \StandardXSpace(\Omega)$.
  Finally, let $\mathscr{L}$ be a depth-growth function
  satisfying $\sup_{n\in\N} \mathscr{L}(n) \geq 2$.
  Then $\WASpace[X] \subsetneq X$ and $\NASpace[X] \subsetneq X$.
\end{thm}

The proof is given at the end of Appendix~\ref{sec:WeightAndNeuronSpacesDistinct}.


\subsection{ReLU-networks of bounded depth have limited expressiveness}

In this subsection, we show that approximation spaces of ReLU-networks \emph{of bounded depth}
and high approximation rate $\alpha$ are non-trivial in a very explicit sense:
they fail to contain any nonzero function in $C_{c}^{3}(\R^{d})$.
This quite general obstruction to the expressiveness of shallow ReLU-networks, and to the embedding
of ``classical'' function spaces into the approximation spaces of shallow ReLU-networks,
is obtained by translating \cite[Theorem 4.5]{PetersenVoigtlaenderReLU} into the language
of approximation spaces.


\begin{thm}\label{thm:UniversalLowerBound}
  Let $\Omega \subseteq \R^d$ be an \emph{open} admissible domain, $p,q \in (0,\infty]$,
  $X = \StandardXSpace[](\Omega)$, $L \in \N$, and $\alpha > 0$.
  \begin{itemize}[leftmargin=0.6cm]
    \item If  $C_c^3 (\Omega) \cap \WASpace[X][\varrho_{1}][q][\alpha][L] \neq \{0\}$
          then $\lfloor L/2 \rfloor \geq \alpha / 2$;

    \item If  $C_c^3 (\Omega) \cap \NASpace[X][\varrho_{1}][q][\alpha][L] \neq \{0\}$
          then $L - 1 \geq \alpha/2$.
          \qedhere
  \end{itemize}
\end{thm}

Before we give a proof we immediately highlight a consequence.

\begin{cor}\label{cor:EmbeddingObstruction}
  Let $Y$ be a function space such that $C_c^3 (\Omega) \cap Y \neq \{0\}$ where
  $\Omega \subseteq \R^d$ is an \emph{open} admissible domain.
  For $p \in (0,\infty]$, $X = \StandardXSpace[](\Omega)$, $L \in \N$, $\alpha > 0$
  and $q \in (0,\infty]$ we have
  \begin{itemize}[leftmargin=0.6cm]
    \item If $Y \subset  \WASpace[X][\varrho_{1}][q][\alpha][L]$ then $\lfloor L/2 \rfloor \geq \alpha/2$;
    \item If $Y \subset  \NASpace[X][\varrho_{1}][q][\alpha][L]$ then $L-1 \geq \alpha/2$.
          \qedhere
  \end{itemize}
\end{cor}

\begin{rem*}
  All  ``classical'' function spaces (Sobolev, Besov, or modulation spaces, \ldots)
  include $C_c^\infty (\Omega)$, hence this shows that none of these
  spaces embed into $\WASpace[X][\varrho_{1}][q][\alpha][L]$
  (resp.~into $\NASpace[X][\varrho_{1}][q][\alpha][L]$) for $\alpha > 2L$. 
  In other words,  {\em to achieve embeddings into approximation spaces of ReLU-networks
  with a good approximation rate, one needs depth!}
\end{rem*}

\begin{proof}[Proof of Theorem~\ref{thm:UniversalLowerBound}]
The claimed estimates are trivially satisfied in case of $L = 1$;
hence, we will assume $L \geq 2$ in what follows.

Let $f \in C_c^3 (\Omega)$ be not identically zero.
We derive necessary criteria on $L$ which have to be satisfied if
$f \in \WASpace[X][\varrho_1][q][\alpha][L]$ or $f \in \NASpace[X][\varrho_1][q][\alpha][L]$.
By Equation~\eqref{eq:ApproximationSpaceEmbedding}, we have
$\WASpace[X][\varrho_1][q][\alpha][L] \subset \WASpace[X][\varrho_1][\infty][\alpha][L]$
and the same for $\NASpace[X][\varrho_1][q][\alpha][L]$;
thus, it suffices to consider the case $q = \infty$.

Extending $f$ by zero outside $\Omega$, we can assume $f \in C_c^3(\R^d)$ with $\supp f \subset \Omega$.
We claim that there is $x_0 \in \supp(f) \subset \Omega$ with $\hess_f (x_0) \neq 0$,
where $\hess_f$ denotes the Hessian of $f$.
If this was false, we would have $\hess_f \equiv 0$ on all of
$\R^d$, and hence $\nabla f \equiv v$ for some $v \in \R^d$.
This would imply $f (x) = \langle v, x \rangle + b$ for all $x \in \R^d$, with $b = f(0)$.
However since $f \equiv 0$ on the nonempty open set $\R^d \setminus \supp(f)$,
this would entail $v = 0$, and then $f \equiv 0$, contradicting our choice of $f$.

Now, choose $r > 0$ such that $\Omega_0 := B_r (x_0) \subset \Omega$.
Then $f|_{\Omega_0}$ is \emph{not} an affine-linear function, so that
\cite[Proposition C.5]{PetersenVoigtlaenderReLU} yields a constant $C_1 = C_1(f,p) > 0$ satisfying
\begin{equation}
  \| f - g \|_{L^p (\Omega_0)}
  \geq C_1 \cdot P^{-2}
  \quad \text{for each $P$-piecewise slice affine function $g : \R^d \to \R$}.
  \label{eq:PiecewiseAffineLowerBound}
\end{equation}
Here, a function $g : \R^d \to \R$ is called \emph{$P$-piecewise slice affine} if for arbitrary
$x_0, v \in \R^d$ the function $g_{x_0, v} : \R \to \R, t \mapsto g(x_0 + t v)$ is piecewise
affine-linear with at most $P$ pieces; that is, $g_{x_0, v} \in \PPoly_{P}^1 (\R)$.

Now, Lemma~\ref{lem:MinGrowthRateNbPiecesSufficient} (which will be proved independently)
shows that there is a constant $K = K(L) \in \N$ such that
\[
  \NNreal_{W,L,\infty}^{\varrho_1, 1, 1} \subset \PPoly_{K \cdot W^{\lfloor L/2 \rfloor}}^1 (\R)
  \qquad \text{and} \qquad
  \NNreal_{\infty,L,N}^{\varrho_1, 1, 1} \subset \PPoly_{K \cdot N^{L-1}}^1 (\R)
\]
for all $N \in \N$.
Furthermore, if $g \in \NNreal_{W,L,N}^{\varrho_1,d,1}$, then Lemma~\ref{lem:NetworkCalculus}
shows $g_{x_0, v} \in \NNreal_{W,L,N}^{\varrho_1,1,1}$;
here, we used that the affine map $T : \R \to \R^d , t \mapsto x_0 + t v$ satisfies
$\| T \|_{\ell^{0,\infty}_\ast} \leq 1$.
In combination, we see that each $g \in \NNreal_{W,L,\infty}^{\varrho_1,d,1}$ is
$P$-piecewise slice affine with $P = K \cdot W^{\lfloor L/2 \rfloor}$,
and each $g \in \NNreal_{\infty,L,N}^{\varrho_1,d,1}$ is
$P$-piecewise slice affine with $P = K \cdot N^{L-1}$.

Now, if $f \in \WASpace[X][\varrho_1][\infty][\alpha][L]$, then there is a constant
$C_2 = C_2(f,\alpha,p) > 0$ such that for each $n \in \N$ there is
$g_n \in \NNreal_{n,L,\infty}^{\varrho_1,d,1}$ satisfying
$\| f - g_n \|_{L^p(\Omega_0)} \leq \| f - g_n \|_{X} \leq C_2 \cdot n^{-\alpha}$.
Furthermore, since $g_n$ is $P$-piecewise slice affine with $P = K \cdot n^{\lfloor L/2 \rfloor}$,
Equation~\eqref{eq:PiecewiseAffineLowerBound} shows that
\(
  K^{-2} C_1 \cdot n^{-2 \lfloor L/2 \rfloor} \leq \| f - g_n \|_{L^p(\Omega_0)} \leq C_2 \cdot n^{-\alpha}
\).
Since this holds for all $n \in \N$, we get $\alpha - 2 \lfloor L/2 \rfloor \leq 0$, as claimed.

The proof in case of $f \in \NASpace[X][\varrho_1][q][\alpha][L]$ is almost identical,
and hence omitted.
\end{proof}

Our next result shows that for networks of \emph{fixed} depth, neural networks using the
activation function $\varrho_r$ with $r \geq 2$ are strictly more expressive than ReLU networks---at
least in the regime of very high approximation rates.

\begin{cor}\label{cor:ReLUPowerStrongerForFixedDepth}
  Consider $\Omega \subseteq \R^d$ an open admissible domain, $p \in (0,\infty]$, $X = \StandardXSpace[](\Omega)$, $L \in \N$.
  In case of $d = 1$, assume that $r \geq 4$ and $L \geq 2$, or that $r \in \{2,3\}$ and $L \geq 3$.
  In case of $d > 1$, assume instead that $r \geq 4$ and $L \geq 3$,
  or that $r \in \{2,3\}$ and $L \geq 5$.
  Then the following hold:
  \[
    \alpha > 2 \lfloor L/2 \rfloor
    \Longrightarrow \WASpace[X][\varrho_{r}][q][\alpha][L]
                    \not\hookrightarrow \WASpace[X][\varrho_{1}][q][\alpha][L]
    \quad \text{and} \quad
    \alpha > 2 L
    \Longrightarrow \NASpace[X][\varrho_{r}][q][\alpha][L]
                    \not\hookrightarrow \NASpace[X][\varrho_{1}][q][\alpha][L].
    \qedhere
  \]
\end{cor}

\begin{proof}
  We use Lemma~\ref{lem:EmbeddingObstructionR1R2} below to get
  $\WASpace[X][\varrho_{r}][q][\alpha][L] \cap C^{3}_{c}(\Omega) \neq \{0\}$ and
  $\NASpace[X][\varrho_{r}][q][\alpha][L] \cap C^{3}_{c}(\Omega) \neq \{0\}$,
  and we conclude using Corollary~\ref{cor:EmbeddingObstruction}.
\end{proof}


\begin{lem}\label{lem:EmbeddingObstructionR1R2}
  Consider $d,r,L \in \N$, $\Omega \subset \R^d$ an \emph{open} admissible domain,
  $p \in (0,\infty]$,  $X = \StandardXSpace[](\Omega)$.
  In case of $d = 1$, assume that $r \geq 4$ and $L \geq 2$, or that $r \in \{2,3\}$ and $L \geq 3$.
  In case of $d > 1$, assume instead that $r \geq 4$ and $L \geq 3$,
  or that $r \in \{2,3\}$ and $L \geq 5$.

  Then for each $\alpha > 0$ and $q \in (0,\infty]$, we have
  \(
    \NASpace[X][\varrho_r][q][\alpha][L] \cap C_c^3 (\Omega)
    \neq \{0\}
    \neq \WASpace[X][\varrho_{r}][q][\alpha][L] \cap C^{3}_{c}(\Omega).
  \)
  \qedhere
\end{lem}

\begin{proof}
Since $\Omega$ is an admissible domain, it is non empty.
Being open, $\Omega$ thus contains a hyper-rectangle
$[a,b] := \prod_{i=1}^d [a_i,b_i] \subset \Omega$, where $a_i < b_i$.

For $r' \geq 2$, let $\sigma_{r'} \in \SNNreal^{\varrho_{r'},1,1}_{2(r'+1),2,r'+1}$
be the function constructed in Lemma~\ref{lem:ExactSquashingReLUPower}.
As $\sigma_{r'}$ satisfies~\eqref{eq:ExactSquashing}, the function $g$ built from $\sigma_{r'}$
in Lemma~\ref{lem:ApproximationOfIndicatorCube}-(2) for small enough $\varepsilon$ is nonzero
and satisfies ${\supp(g) \subset [a,b] \subset \Omega}$ and $g \in \NNreal_{\infty,3,\infty}^{\varrho_{r'},d,1}$
(resp.~$g \in \NNreal_{\infty,2,\infty}^{\varrho_{r'},d,1}$ when $d=1$).
Note that if $r' \geq 4$ then $\varrho_{r'} \in C^{3}(\R)$,
hence $g \in C^{3}_{c}(\R^{d}) \setminus \{0\}$.

When $r \geq 4$, set $r':=r$ so that $g \in \NNreal^{\varrho_{r},d,1}_{\infty,3,\infty}$
($g \in \NNreal^{\varrho_{r},d,1}_{\infty,2,\infty}$ when $d=1$).
When $r \in \{2,3\}$ set $r' := r^{2} \geq 4$.
As $\varrho_{r'} = \varrho_{r} \circ \varrho_{r}$,
Lemma~\ref{lem:NestednessBasic} with $s=2$ yields $g \in \NNreal^{\varrho_{r},d,1}_{\infty,5,\infty}$
($g \in \NNreal^{\varrho_{r},d,1}_{\infty,3,\infty}$ for $d=1$).

It is not hard to see that our assumptions regarding $L$ imply in each case for $n$ large enough
that
\(
  g|_{\Omega}
  \in \WeightClassSymbol_{n}(X,\varrho_{r},L)
      \cap \NeuronClassSymbol_{n}(X,\varrho_{r},L)
\),
and hence
\(
  0 \neq g|_{\Omega}
  \in \WASpace[X][\varrho_{r}][q][\alpha][L]
      \cap C_{c}^{3}(\Omega)
      \cap \NASpace[X][\varrho_{r}][q][\alpha][L]
\).
\end{proof}

\section{Direct and inverse estimates with Besov spaces}
\label{sec:directinverse}

In this section we characterize certain embeddings
\begin{itemize}
\item of Besov spaces into $\WASpace[X][\varrho_r]$ and $\NASpace[X][\varrho_r]$;
      these are called {\em direct estimates};

\item of $\WASpace[X][\varrho_r]$ and $\NASpace[X][\varrho_r]$ into Besov spaces;
      these are called {\em inverse estimates}.
\end{itemize}

Since the approximation classes for output dimension $k > 1$ are $k$-fold cartesian products
of the classes for $k=1$ (cf.~Remark~\ref{sec:RestrictionCartesian}),
we focus on scalar output dimension $k=1$.
We will use so-called \emph{Jackson inequalities} and \emph{Bernstein inequalities},
as well as the notion of real interpolation spaces.
These concepts are recalled in Section~\ref{sec:interpolation},
while Besov spaces and some of their properties are briefly recalled in Section~\ref{sec:remindBesov}
before we proceed to our main results.

\subsection{Reminders on interpolation theory}
\label{sec:interpolation}

Given two quasi-normed vector spaces $(Y_J, \|\cdot\|_{Y_J})$ and $(Y_B, \|\cdot\|_{Y_B})$
with $Y_J \hookrightarrow X$ and $Y_B \hookrightarrow X$
for a given quasi-normed linear space $(X, \|\cdot\|_X)$,
we say that \emph{$Y_J$ fulfills a Jackson inequality with exponent $\gamma > 0$}
with respect to the family $\AppSet = (\Sigma_n)_{n \in \N_0}$,
if there is a constant $C_J > 0$ such that
\begin{equation*}
  \label{eq:JacksonInequality}
  \AppErr(f,\AppSet_{n})_{X} \leq C_J \cdot n^{-\gamma} \cdot \| f \|_{Y_J}
  \qquad \forall \, f \in Y_J \text{ and } n \in \N.
  \tag{J}
\end{equation*}
We say that \emph{$Y_B$ fulfills a Bernstein inequality with exponent $\gamma > 0$}
with respect to $\AppSet = (\Sigma_n)_{n \in \N_0}$, if there is a constant $C_B > 0$ such that
\begin{equation*}
    \label{eq:BernsteinInequality}
    \| \varphi \|_{Y_B} \leq C_B \cdot n^\gamma \cdot \| \varphi \|_X
    \qquad \forall \, n \in \N \text{ and } \varphi \in \AppSet_n .
    \tag{B}
\end{equation*}
As shown in the proof of \cite[Chapter 7, Theorem 9.1]{ConstructiveApproximation},
we have the following:

\begin{prop}\label{prop:BernsteinJacksonConsequence}
  Denote by $(X,Y)_{\theta, q}$ the real interpolation space obtained from $X,Y$,
  as defined e.g.\@ in \cite[Chapter 6, Section 7]{ConstructiveApproximation}.
  Then the following hold:
  \begin{itemize}[leftmargin=0.6cm]
    \item If $Y_J \hookrightarrow X$ fulfills the Jackson inequality with exponent $\gamma > 0$,
          then
          \[
            ( X, Y_J )_{\alpha / \gamma, q} \hookrightarrow \GenApproxSpace
            \qquad \forall \,\, 0 < \alpha < \gamma
                           \text{ and } 0 < q \leq \infty.
          \]

    \item If $Y_B \hookrightarrow X$ fulfills the Bernstein inequality with exponent $\gamma>0$,
          then
          \[
            \GenApproxSpace \hookrightarrow (X, Y_B)_{\alpha / \gamma , q}
            \qquad \forall \,\, 0 < \alpha < \gamma
                           \text{ and } 0 < q \leq \infty.
            \qedhere
          \]
  \end{itemize}
\end{prop}

In particular, if the single space $Y = Y_J = Y_B$ satisfies both inequalities
with the same exponent $\gamma$, then $\GenApproxSpace = (X, Y)_{\alpha / \gamma , q}$
for all $0 < \alpha < \gamma$ and $0 < q \leq \infty$.

By \cite[Chapter 7, Theorem 9.3]{ConstructiveApproximation},
if $\AppSet$ satisfies Properties~\ref{enu:GammaContainsZero}--\ref{enu:GammaDense} then
for $0 < \tau \leq \infty$, $0 < \alpha < \infty$ the space $Y :=  A^{\alpha}_{\tau}(X,\Sigma)$
satisfies matching Jackson and Bernstein inequalities with exponent $\gamma:=\alpha$.
The Bernstein inequality reads
\begin{equation}
  \exists \, C = C (\alpha,\tau,X) > 0
    \quad \forall \, n \in \N \text{ and } \varphi \in \AppSet_n :
    \quad
    \| \varphi \|_{A^{\alpha}_{\tau}(X,\Sigma)} \leq C \cdot n^\alpha \cdot \| \varphi \|_X .
    \label{eq:TrivialBernsteinInequality}
\end{equation}

We will also use the following well-known property of (real) interpolation spaces
(see \cite[Chapter 6, Theorem 7.1]{ConstructiveApproximation}):
For quasi-Banach spaces $X_1, X_2$ and $Y_1, Y_2$, assume that
$T : X_1 + X_2 \to Y_1 + Y_2$ is linear and such that $T|_{X_i} : X_i \to Y_i$ is bounded
for $i \in \{1,2\}$.
Then $T|_{(X_1, X_2)_{\theta,q}} : (X_1, X_2)_{\theta,q} \to (Y_1, Y_2)_{\theta,q}$
is well-defined and bounded for all $\theta \in (0,1)$ and $q \in (0,\infty]$.


\subsection{Reminders on Besov spaces}\label{sec:remindBesov}

We refer to \cite[Section 2]{DeVSharp93} for the definition
of the Besov spaces $B^{s}_{\sigma,\tau} (\Omega) := B^s_{\tau}(X_\sigma(\Omega;\R))$
with $\sigma, \tau \in (0,\infty]$, $s \in (0,\infty)$ and with a Lipschitz domain%
\footnote{Here, the term ``domain'' is to be understood as an open connected set.}
$\Omega \subset \R^{d}$
(see \cite[Definition 4.9]{AdamsSobolevSpaces} for the precise definition of these domains).

As shown in \cite[Theorem 7.1]{DeVore:1988fe},
we have for all $p, s \in (0,\infty)$ the embedding
\[
  B^{s}_{\sigma,p}( (0,1)^d ) \hookrightarrow L_p( (0,1)^d ;\R),
  \quad \text{provided} \quad \sigma = (s/d + 1/p)^{-1}.
\]
Combined with the embedding $B^{s}_{p,q}(\Omega) \hookrightarrow B^s_{p,q'}(\Omega)$
for $q \leq q'$ (see \cite[Displayed equation on Page 92]{DeV98})
and because of $\sigma = (s/d + 1/p)^{-1} \leq p$, we see that
\begin{equation}
  B^{s}_{\sigma,\sigma} ( (0,1)^d )
  \hookrightarrow B^{s}_{\sigma,p}( (0,1)^d )
  \hookrightarrow L_p( (0,1)^d ;\R),
  \quad \text{provided} \quad \sigma = (s/d + 1/p)^{-1}.
  \label{eq:BesovEmbedsLp}
\end{equation}

For the special case $\Omega = (0,1) \subset \R$ and each fixed $p \in (0,\infty)$,
the sub-family of Besov spaces $B^{s}_{\sigma,\sigma}( (0,1) )$
with $\sigma = (s/d+1/p)^{-1}$ satisfies
\begin{equation}
  \big( L_p( (0,1);\R), B^s_{\sigma, \sigma}( (0,1) ) \big)_{\theta,q}
  = B^{\theta s}_{q,q}( (0,1) ),
  \quad \text{for}\ 0 < \theta < 1, \ \text{where}\ q= (\theta s + 1/p)^{-1}.
  \label{eq:LpBesovInterpolation}
\end{equation}
This is shown in \cite[Chapter 12, Corollary 8.5]{ConstructiveApproximation}.

Finally, from the definition of Besov spaces given in
\cite[Equation (2.2)]{DeVSharp93} it is clear that
\begin{equation}
  B_{p,q}^\alpha (\Omega) \hookrightarrow B_{p,q}^\beta(\Omega)
  \quad \text{if} \quad p,q \in (0,\infty] \text{ and } 0 < \beta < \alpha .
  \label{eq:BesovEmbedding}
\end{equation}


\subsection{Direct estimates}
\label{sec:direct}

In this subsection, we investigate embeddings of Besov spaces into the approximation spaces
$\WASpace[X][\varrho_{r}]$ 
where $\Omega \subseteq \R^{d}$ is an admissible domain
and $X := \StandardXSpace[](\Omega)$ with $p \in (0,\infty]$.
For technical reasons, we further assume $\Omega$ to be a bounded Lipschitz domain,
such as $\Omega=(0,1)^d$, see \cite[Definition 4.9]{AdamsSobolevSpaces}. 
The main idea is to exploit known direct estimates for Besov spaces
on such domains which give error bounds for the $n$-term approximations
with B-spline based wavelet systems, see \cite{DeVore:1988fe}.

For $t \in \N_0$ and $d \in \N$, the tensor product B-spline is
\(
  \beta_d^{(t)} (x_1,\dots,x_d)
  := \beta_{+}^{(t)}(x_1) \, \beta_{+}^{(t)}(x_2) \, \cdots\beta_{+}^{(t)}(x_d),
\)
where $\beta_+^{(t)}$ is as introduced in Definition~\ref{defn:Bsplines}.
Notice that $\beta_d^{(0)} = \Indicator_{[0,1)^{d}}$.

By Lemma~\ref{lem:ExactSquashingReLUPower} there is
$\sigma_{r} \in \NNreal^{\varrho_{r},1,1}_{2(r+1),2,r+1}$ satisfying~\eqref{eq:ExactSquashing};
hence by Lemma~\ref{lem:ApproximationOfIndicatorCube} there is $L \leq 3$
such that for $\varepsilon > 0$, we can approximate $\beta_d^{(0)}$
with $g_\eps = \Realization(\Phi_\eps)$ with precision
$\|\beta_d^{(0)} - g_\eps\|_{L_p(\R^d)} < \varepsilon$, where $L(\Phi_\eps) = L$ and
\(
  \Phi_\eps
  \in \NNreal^{\varrho_{r},d,1}_{w,3,m}
\),
for suitable $w = w(d,r), m = m(d,r) \in \N$.
Furthermore, if $d = 1$, then Lemma~\ref{lem:ApproximationOfIndicatorCube} shows that the same holds
for some $\Phi_\eps \in \NNreal^{\varrho_r,d,1}_{w,2,m}$.

For approximating $\beta^{(t)}_d$ (with $t \in \N$) instead of $\beta_d^{(0)}$,
we can actually do better.
In fact, we prove in Appendix~\ref{app:MultiBspline} that one can implement $\beta_d^{(t)}$
as a $\varrho_{t}$-network, provided that $t \geq \min \{ d,2 \}$.

\begin{lem}\label{lem:MultiBspline}
  Let $d,t \in \N$ with $t \geq \min \{ d, 2 \}$.
  Then the tensor product B-spline
  \begin{equation}
    \beta_d^{(t)} : \R^d \to \R, \quad
    \beta_d^{(t)} (x) := \beta_{+}^{(t)}(x_1) \beta_{+}^{(t)}(x_2) \cdots \beta_{+}^{(t)} (x_d)
    \label{eq:TensorBSpline}
  \end{equation}
  satisfies $\beta_d^{(t)} \in \NNreal^{\varrho_t,d,1}_{w,L,m}$
  with $L = 2 + 2 \lceil \log_{2}d\rceil$ and
  \[
    \begin{cases}
      w = 28 d(t+1) \text{ and } m = 13d(t+1), & \text{if } d > 1, \\
      w =  2  (t+2) \text{ and } m =     t+2 , & \text{if } d = 1.
    \end{cases}
    \qedhere
  \]
\end{lem}

In the following, we will consider $n$-term approximations with respect to the continuous
wavelet-type system generated by $\beta_d^{(t)}$.
Precisely, for $a > 0$ and $b \in \R^d$, define $\beta^{(t)}_{a,b} := \beta_d^{(t)} (a \cdot + b)$.
The continuous wavelet-type system generated by $\beta_d^{(t)}$ is then
$\CalD_d^{t} := \{ \beta^{(t)}_{a,b} \colon a \in (0,\infty), b \in \R^d \}$.
For any $t \in \N_{0}$, we define $\AppSet_{0}(\mathcal{D}_d^t) := \{0\}$,
and the reservoir of all $n$-term expansions from $\mathcal{D}_d^{t}$, $n \in \N$, is given by
\[
  \AppSet_n(\mathcal{D}_d^{t})
  := \bigg\{
       g = \sum_{i=1}^n c_i g_i
       \colon
       c_i \in \R, g_i \in \mathcal{D}_d^t
     \bigg\}.
\]
The following lemma relates $\AppSet_n(\mathcal{D}_d^{t})$ to $\NNreal^{\varrho_{r},d,1}_{cn,L,cn}$
for a suitably chosen constant $c=c(d,r,t) \in \N$.

\begin{lem}\label{lem:WaveletNTermsSubsetNeurons}
  Consider $d \in \N$, $t \in \N_{0}$, $p \in (0,\infty]$, $X = \StandardXSpace[](\R^{d})$.
  \begin{enumerate}[leftmargin=0.6cm]
    \item If $t=0$ and $p<\infty$ then, with $L := \min \{ d+1,3 \}$ and $c = c(d,r) \in \N$,
          we have 
          \begin{equation}\label{eq:SigmaN0}
            \AppSet_{n}(\mathcal{D}_d^0)
            \subset \overline{\NNreal^{\varrho_{r},d,1}_{cn,L,cn} \cap X}^{X}
            \qquad \forall \, n,r \in \N.
          \end{equation}

    \item If $t \geq \min \{ d,2 \}$ then, with $L := 2 + 2 \lceil \log_{2} d \rceil$
          we have for any $p \in (0,\infty]$ that
          \begin{equation}\label{eq:Sigma_m}
            \AppSet_n(\mathcal{D}_d^t) \subset \NNreal^{\varrho_t,d,1}_{cn,L,cn} \cap X
            \qquad \forall \, n \in \N ,
          \end{equation}
          where $c = c(d,t) \in \N$.
          \qedhere
  \end{enumerate}
\end{lem}

\begin{proof}
  \textbf{Part (1):} For $t = 0$, $r \in \N$, $0 < p < \infty$, we have already noticed
  before Lemma~\ref{lem:MultiBspline} that there exist $w = w(d,r), m = m(d,r) \in \N$ such that
  $\beta_d^{(0)} \in \overline{\NNreal^{\varrho_r,d,1}_{w,L,m} \cap X}^{X}$,
  where $L = \min \{ d+1,3 \}$.
  Since $\beta_{a,b}^{(0)} = \beta_d^{(0)} \circ P_{a,b}$ for the affine map
  $P_{a,b} : \R^d \to \R^d, x \mapsto a x + b$ and since $\|P_{a,b}\|_{\ell^{0,\infty}_*} = 1$,
  Lemma~\ref{lem:SummationLemma}-(\ref{enu:ScalMult})
  and Lemma \ref{lem:NetworkCalculus}-(\ref{enu:PrePostAffine}) yield
  $\beta_{a,b}^{(0)} \in \overline{\NNreal^{\varrho_r,d,1}_{c,L,c} \cap X}^{X}$
  with $c := \max \{ w,m \}$.
  Thus, the claim follows from Parts (1) and (3) of Lemma~\ref{lem:SummationLemma}.

  \textbf{Part (2):}
  For $t \geq \min \{ d,2 \}$, Lemma~\ref{lem:MultiBspline} shows that
  $\beta_{a,b}^{(t)} \in \NNreal^{\varrho_t,d,1}_{c,L,c} \cap X$
  with $L = 2 + 2 \lceil \log_2 d \rceil$ and $c := \max \{ w,m \}$
  where $w = w(d,t)$ and $m = m(d,t)$ are as in Lemma~\ref{lem:MultiBspline}.
  As before, we conclude using Parts (1) and (3) of Lemma~\ref{lem:SummationLemma}.
\end{proof}



\begin{cor}\label{cor:JacksonAbstractReLUPower}
  Consider $d \in \N$, $\Omega \subset \R^{d}$ an admissible domain,
  $p \in (0, \infty]$, $X = \StandardXSpace[](\Omega)$, $\mathscr{L}$ a depth growth function,
  $L := \sup_{n} \mathscr{L}(n) \in \N \cup \{\infty\}$.
  For $t \in \N_{0}$ define
  $\AppSet(\mathcal{D}_d^{t}) := ( \AppSet_{n}(\mathcal{D}_d^t))_{n \in \N_{0}}$.
  \begin{enumerate}
    \item \label{enu:Jackson3}
          If $L \geq \min \{ d+1, 3 \}$ and $p < \infty$, then for any $r \geq 1$
          \[
            A^{\alpha}_{q}(X,\Sigma(\mathcal{D}_d^0)) \hookrightarrow \WASpace[X][\varrho_{r}]
            \qquad \text{for each} \quad \alpha \in (0,\infty),\ q \in (0,\infty].
          \]

    \item \label{enu:JacksonLogD}
          If $L \geq 2 + 2 \lceil \log_{2} d\rceil$ then for any $r \geq \min \{ d,2 \}$, we have
          \[
            A^{\alpha}_{q}(X,\Sigma(\mathcal{D}_d^r)) \hookrightarrow \WASpace[X][\varrho_{r}]
            \qquad \text{for each}\quad \alpha \in (0,\infty),\ q \in (0,\infty].
            \qedhere
          \]
  \end{enumerate}
\end{cor}

\begin{proof}
  For the proof of Part (1) let $L_0 := \min \{d+1, 3\}$, while $L_0 := 2 + 2 \lceil \log_2 d \rceil$
  for the proof of Part (2).
  Since $L \geq L_0$, there is $n_{0} \in \N$ such that $\mathscr{L}(n) \geq L_0$ for all $n \geq n_{0}$.

  We first start with the proof of Part (2).
  By Lemma~\ref{lem:WaveletNTermsSubsetNeurons}-(2), with $t = r \geq \min \{ d,2 \}$,
  Equation~\eqref{eq:Sigma_m} holds for some $c \in \N$.
  For $n \geq n_{0}/c$ we have $2 + 2 \lceil \log_2 d \rceil = L_0 \leq \mathscr{L}(cn)$, whence
  \[
    \Sigma_n(\mathcal{D}_d^t)
    \subset \overline{\WeightClassSymbol_{cn}(X,\varrho_{r},\mathscr{L})}^{X}.
  \]
  Therefore, we see that 
  \begin{equation}\label{eq:InstanceOpt}
    E(f,\Sigma_n(\mathcal{D}_d^t))_{X}
    \geq E(f,\WeightClassSymbol_{cn}(X,\varrho_{r},\mathscr{L}))_{X}
    \quad \forall \, f \in X \text{ and } n \geq \tfrac{n_0}{c}.
  \end{equation}

  For the proof of Part (1), the same reasoning with~\eqref{eq:SigmaN0}
  instead of \eqref{eq:Sigma_m} yields \eqref{eq:InstanceOpt} with $t=0$ and any $r \in \N$.
  For both parts, we conclude using Lemma~\ref{lem:ApproximationSpaceElementaryNesting}
  and the associated remark.
\end{proof}

\begin{thm}\label{thm:better_direct}
  Let $\Omega \subset \R^{d}$ be a bounded Lipschitz domain of positive measure.
  For $p \in (0,\infty]$, define $X_p (\Omega)$ as in Equation~\eqref{eq:StandardXSpace}.
  Let $\mathscr{L}$ be a depth growth function.

  \begin{enumerate}
    \item Suppose that $d = 1$ and $L := \sup_{n \in \N} \mathscr{L}(n) \geq 2$.
          Then the following holds for each $r \in \N$:
          \begin{equation}
            B^s_{p,q}(\Omega) \hookrightarrow \WASpace[X_{p}(\Omega)][\varrho_r][q][s]
            \quad \forall \, p,q \in (0,\infty] \text{ and } 0 < s < r + \min \{ 1, p^{-1} \} .
            \label{eq:BesovDirectDimension1}
          \end{equation}

    \item Suppose that $d > 1$ and $L := \sup_{n \in \N} \mathscr{L}(n) \geq 3$, and let $r \in \N$.
          Define $r_0 := r$ if $r \geq 2$ and $L \geq 2 + 2 \lceil \log_2 d \rceil$,
          and $r_0 := 0$ otherwise.
          Then
          \begin{equation}
            B^{s d}_{p,q} (\Omega) \hookrightarrow \WASpace[X_{p}(\Omega)][\varrho_r][q][s]
            \quad \forall \, p,q \in (0,\infty] \text{ and } 0 < s < \frac{r_0 + \min \{1, p^{-1} \}}{d} .
            \label{eq:BesovDirectGeneralDimension}
            \qedhere
          \end{equation}
  \end{enumerate}
\end{thm}

%

\begin{rem}
  If $\Omega$ is open then each Besov space $B_{p,q}^{s d}(\Omega)$ contains $C_{c}^{3}(\Omega)$.
  Hence, by Corollary~\ref{cor:EmbeddingObstruction},
  the embeddings~\eqref{eq:BesovDirectDimension1} or \eqref{eq:BesovDirectGeneralDimension}
  with $r = 1$ imply that $\lfloor L/2 \rfloor \geq s / 2$.
  This is indeed the case, since these embeddings for $r = 1$ are only established when
  $L \geq 2$ and $0 < d s < 1 + \min \{ p^{-1},1 \} \leq 2$,
  which implies $s / 2 < 1/d \leq 1 \leq \lfloor L/2 \rfloor$.
\end{rem}

\begin{proof}[Proof of Theorem~\ref{thm:better_direct}]
  See Appendix~\ref{sub:BetterDirectProof}.
\end{proof}

\subsection{Limits on possible inverse estimates} 
\label{sub:LimitsOnInverseEstimates}

For networks of finite depth $\mathscr{L} \equiv L < \infty$,
there are limits on possible embeddings of $\WASpace[X][\varrho_1]$
(resp.~of $\NASpace[X][\varrho_1]$) into Besov spaces.

\begin{thm}\label{thm:LimitInverseBesov}
Consider $\Omega = (0,1)^{d}$, $p \in (0,\infty]$, $X = \StandardXSpace[](\Omega)$, $\mathscr{L}$
a depth growth function such that $L := \sup_{n} \mathscr{L}(n) \in \N_{\geq 2} \cup \{\infty\}$
and $r \in \N$.
For $\sigma, \tau, q \in (0,\infty]$ and $\alpha, s \in (0,\infty)$, the following claims hold%
\footnote{with the convention $\lfloor \infty/2\rfloor = \infty-1 = \infty$}:
\begin{enumerate}
  \item \label{enu:LimitInverseBesov1}
        If $\WASpace[X][\varrho_{r}] \hookrightarrow B^{s}_{\sigma,\tau}(\Omega)$ then
        \(
          \alpha \geq \lfloor L/2\rfloor \cdot \min \{ s, 2 \}
        \).

  \item \label{enu:LimitInverseBesov2}
        If $\NASpace[X][\varrho_{r}] \hookrightarrow B^{s}_{\sigma,\tau}(\Omega)$
        then
        \(
          \alpha \geq (L-1) \cdot \min \{ s, 2 \}
        \).
        \qedhere
\end{enumerate}
\end{thm}

A direct consequence is that for networks of unbounded depth ($L=\infty$),
none of the spaces $\WASpace[X][\varrho_{r}]$, $\NASpace[X][\varrho_r]$ embed into
any Besov space of strictly positive smoothness $s>0$.



\begin{rem}
  For $L=2$, as $\lfloor L/2\rfloor = L-1$ the two inequalities resulting from
  Theorem~\ref{thm:LimitInverseBesov} match.
  This is natural as for $L=2$ we know from Lemma~\ref{lem:weightsvsneurons} that
  $\WASpace[X][\varrho_{r}] = \NASpace[X][\varrho_{r}]$.
  For $L \geq 3$ the inequalities no longer match.
  Each inequality is in fact stronger than what would be achieved by simply combining
  the other one with Lemma~\ref{lem:weightsvsneurons}.
  Note also that in contrast to the direct estimate~\eqref{eq:BesovDirectGeneralDimension}
  of Theorem~\ref{thm:better_direct} where the Besov spaces are of smoothness $s d$,
  here the dimension $d$ does not appear.
  \qedhere


\end{rem}

The proof of Theorem~\ref{thm:LimitInverseBesov} employs a particular family
of oscillating functions that have a long history \cite{Hastad:1tbON80T}
in the analysis of neural networks and of the benefits of depth \cite{telgarsky2016benefits}.

\begin{figure}
  \begin{center}
    \includegraphics[width=0.5\textwidth]{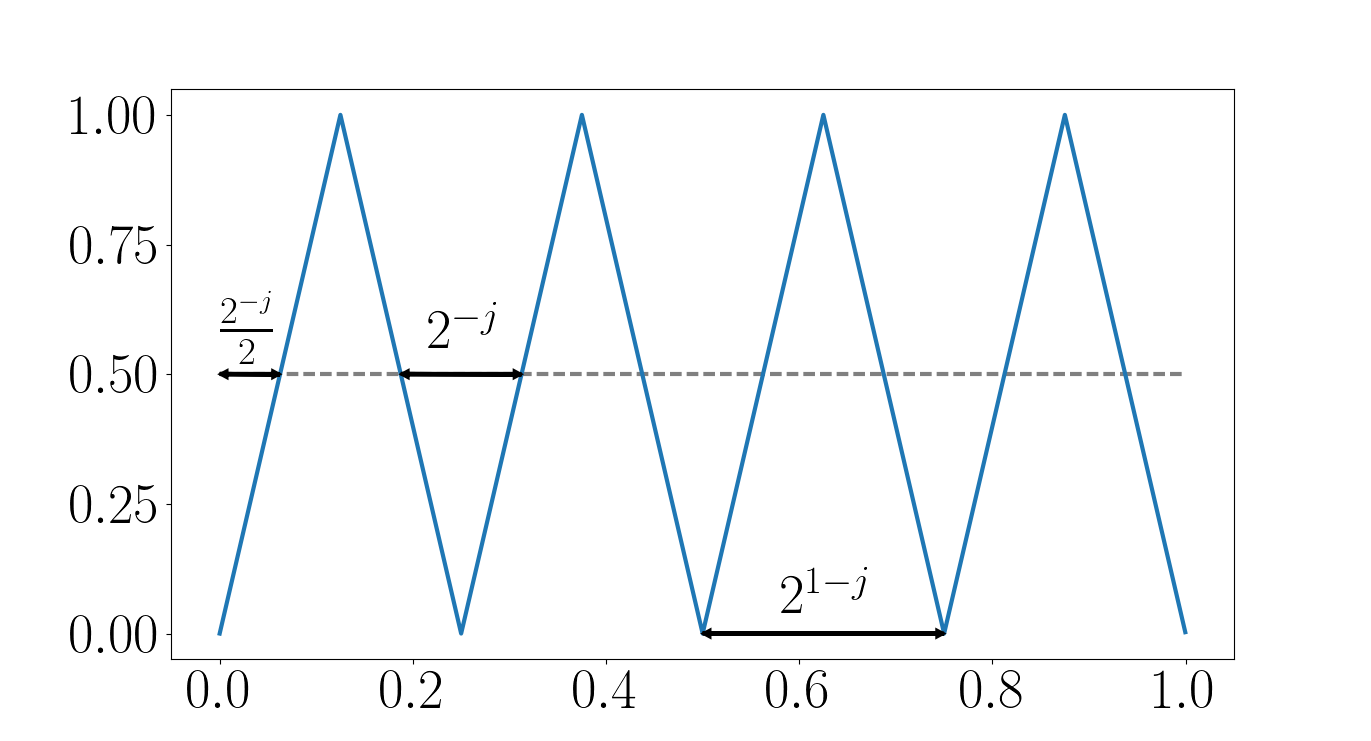}
  \end{center}
  \caption{\label{fig:SawToothPlot}
  A plot of the function $\Delta_j$ (for $j = 3$).}
\end{figure}

\begin{defn}[Sawtooth functions]\label{def:SawtoothFunction}
  Consider $\beta_{+}^{(1)}$ the B-spline of degree one,
  and $\Delta_{1} := \beta_{+}^{(1)}(2\cdot)$ the ``hat'' function supported on $[0,1]$.
  For $j \geq 1$ the univariate ``sawtooth'' function of order $j$,
  \begin{equation}\label{eq:DefSawtooth}
    \Delta_{j} = \sum_{k=0}^{2^{j-1}-1} \Delta_{1}(2^{j-1}\cdot-k),
  \end{equation}
  has support in $[0,1]$ and is made of $2^{j-1}$ triangular ``teeth''
  (see Figure~\ref{fig:SawToothPlot}).
  The \emph{multivariate} sawtooth function $\Delta_{j,d}$ is defined as
  $\Delta_{j,d}(x) := \Delta_{j}(x_{1})$ for $x = (x_{1},\ldots,x_{d}) \in \R^{d}$, $j \in \N$.
\end{defn}

An important property of $\Delta_j$ is that it is a realization of a $\varrho_1$-network
of specific complexity.
The proof of this lemma is in Appendix~\ref{sub:SawtoothImplementation}.

\begin{lem}\label{lem:SawtoothImplementation}
  Let $L \in \N_{\geq 2}$ and define $C_L := 4 \, L + 2^{L-1}$.
  Then
  \[
    \Delta_j \in \NNreal^{\varrho_1, 1, 1}_{\infty, L, C_L \cdot 2^{j / (L-1)}}
    \qquad \text{and} \qquad
    \Delta_j \in \NNreal^{\varrho_1, 1, 1}_{C_L \cdot 2^{j / \lfloor L/2 \rfloor}, L, \infty}
    \qquad \forall \, j \in \N.
    \qedhere
  \]

\end{lem}

\begin{cor}\label{cor:SawtoothMultidimImplementation}
  For $L \in \N_{\geq 2}$, let $C_L$ as in Lemma~\ref{lem:SawtoothImplementation}.
  Then
  \[
    \Delta_{j,d} \in \NNreal^{\varrho_1, d, 1}_{\infty, L, C_L \cdot 2^{j / (L-1)}}
    \qquad \text{and} \qquad
    \Delta_{j,d} \in \NNreal^{\varrho_1, d, 1}_{C_L \cdot 2^{j / \lfloor L/2 \rfloor}, L, \infty}
    \qquad \forall \, j \in \N.
    \qedhere
  \]
\end{cor}

\begin{proof}
  We have $\Delta_{j,d} = \Delta_j \circ T$ for the affine map
  $T : \R^d \to \R, (x_1,\dots,x_d) \mapsto x_1$, which satisfies $\|T\|_{\ell^{0,\infty}_\ast} = 1$.
  Now, the claim is a direct consequence of Lemmas~\ref{lem:SawtoothImplementation}
  and \ref{lem:NetworkCalculus}-(1).
\end{proof}

We further prove in Appendix~\ref{sub:BesovNormSawToothProof} that that the Besov norm of $\Delta_{j,d}$ grows exponentially with $j$:

\begin{lem}\label{lem:BesovNormSawTooth}
  Let $d \in \N$ and $\Omega = (0,1)^d$.
  Let $p,q \in (0,\infty]$ and $s \in (0, \infty)$ be arbitrary.
  Let $s' \in (0,2)$ with $s' \leq s$.
  There is a constant $c = c(d,p,q,s,s') > 0$ such that
  \[
    \|\Delta_{j,d}\|_{B^s_{p,q} (\Omega)} \geq c \cdot 2^{s' j}
    \qquad \forall \, j \in \N.
    \qedhere
  \]
\end{lem}

Given this lower bound on the Besov space norm of the sawtooth function $\Delta_{j,d}$,
we can now prove the limitations regarding possible inverse estimates that we announced above.

\begin{proof}[Proof of Theorem~\ref{thm:LimitInverseBesov}]
  We start with the proof for $\WASpace[X][\varrho_r]$.
  Let us fix $\ell \in \N$ with $\ell \leq \lfloor L/2 \rfloor$, and note that $2\ell \leq L$,
  so that there is some $j_0 = j_0 (\ell, \mathscr{L}) \in \N$ such that $\mathscr{L}(2^j) \geq 2\ell$
  for all ${j \geq j_0}$.
  Now, Corollary~\ref{cor:SawtoothMultidimImplementation} (applied with $2\ell$ instead of $L$)
  shows that
  \(
    \Delta_{\ell j, d}
    \in \NNreal^{\varrho_1, d, 1}_{C_{2\ell} 2^{(\ell j)/\ell}, 2\ell, \infty}
    \subset \NNreal^{\varrho_1, d, 1}_{C_{2\ell} 2^{j}, \mathscr{L}(C_{2\ell} 2^{j}), \infty}
  \)
  for all $j \geq j_0$ and a suitable constant $C_{2\ell} \in \N$.
  Therefore, the Bernstein inequality \eqref{eq:TrivialBernsteinInequality} yields a constant
  $C = C(d,\alpha, q, p) > 0$ such that
  \[
    \| \Delta_{\ell j, d} |_{\Omega} \|_{\WASpace[X][\varrho_1]}
    \leq C \cdot \big( C_{2 \ell} \, 2^j \big)^\alpha \cdot \| \Delta_{\ell j, d} |_{\Omega} \|_{X}
    \leq C_{2\ell}^\alpha C \cdot 2^{\alpha j}
    \quad \forall \, j \geq j_0 .
  \]
  Let $s_0 := \min \{ 2, s \}$, let $0 < s' < s_0$ be arbitrary,
  and note as a consequence of Equation~\eqref{eq:Nesting1} that
  \[
    \WASpace[X][\varrho_1]
    \hookrightarrow \WASpace[X][\varrho_r]
    \hookrightarrow B^s_{\sigma,\tau}(\Omega).
  \]
  Here we used that $\Omega$ is bounded, so that Equation~\eqref{eq:Nesting1} is applicable.
  Overall, as a consequence of this embedding and of Lemma~\ref{lem:BesovNormSawTooth},
  we obtain $c = c(d,s',s,\sigma,\tau) > 0$ and
  ${C' = C'(\sigma,\tau,s,p,q,\alpha,\mathscr{L},\Omega) > 0}$ satisfying
  \[
    c \cdot 2^{s' \ell j}
    \leq \| \Delta_{\ell j, d} |_{\Omega} \|_{B^{s}_{\sigma,\tau}(\Omega)}
    \leq C' \cdot \| \Delta_{\ell j, d} |_{\Omega} \|_{\WASpace[X][\varrho_1]}
    \leq C' C_{2\ell}^\alpha C \cdot 2^{\alpha j}
  \]
  for all $j \geq j_0$.
  This implies $s' \cdot \ell \leq \alpha$.
  Since this holds for all $s' \in (0,s_0)$ and all $\ell \in \N$
  with $\ell \leq \lfloor L/2 \rfloor$, we get $\lfloor L/2 \rfloor \cdot s_0 \leq \alpha$,
  as claimed.

  \medskip{}

  Now, we prove the claim for $\NASpace[X][\varrho_r]$.
  In this case, fix $\ell \in \N$ with $\ell+1 \leq L$, and note that there is some
  $j_0 \in \N$ satisfying $\mathscr{L}(2^j) \geq \ell + 1$ for all $j \geq j_0$.
  Now, Corollary~\ref{cor:SawtoothMultidimImplementation} (applied with $\ell+1$ instead of $L$)
  yields a constant $C_{\ell+1} \in \N$ such that
  \(
    \Delta_{\ell j, d}
    \in \NNreal^{\varrho_1, d, 1}_{\infty, \ell+1, C_{\ell+1} 2^{(\ell j)/((\ell+1) - 1)}}
    \subset \NNreal^{\varrho_1, d, 1}_{\infty, \mathscr{L}(C_{\ell+1} 2^{j}), C_{\ell+1} 2^{j}}
  \)
  for all $j \geq j_0$.
  As above, the Bernstein inequality~\eqref{eq:TrivialBernsteinInequality} therefore shows
  \(
    \| \Delta_{\ell j, d} |_{\Omega} \|_{\NASpace[X][\varrho_1]}
    \leq C_{\ell+1}^\alpha C \cdot 2^{\alpha j}
  \)
  for all $j \geq j_0$ and some constant $C = C(d,\alpha,q,p) < \infty$.
  Reasoning as above, we get that
  \[
    c \cdot 2^{s' \ell j}
    \leq \| \Delta_{\ell j, d} |_{\Omega} \|_{B^{s}_{\sigma,\tau}(\Omega)}
    \leq C' \cdot \| \Delta_{\ell j, d} |_{\Omega} \|_{\NASpace[X][\varrho_1]}
    \leq C' C_{\ell+1}^\alpha C \cdot 2^{\alpha j}
  \]
  for all $j \geq j_0$ and $0 < s' < s_0 = \min \{2,s\}$.
  Therefore, $s' \cdot \ell \leq \alpha$.
  Since this holds for all $s' \in (0,s_0)$ and all $\ell \in \N$ with $\ell + 1 \leq L$,
  we get $\alpha \geq s_0 \cdot (L-1)$, as claimed.
\end{proof}

\subsection{Univariate inverse estimates (\texorpdfstring{$d=1$}{d = 1})}

The ``no-go theorem'' (Theorem~\ref{thm:LimitInverseBesov})
holds for $\Omega= (0,1)^{d}$ in any dimension $d \geq 1$, for any $0 < p \leq \infty$.
In this subsection, we show in dimension $d = 1$ that Theorem~\ref{thm:LimitInverseBesov}
is quite sharp.
Precisely, we prove the following:

\begin{thm}\label{thm:ApproxSpaceIntoBesovForBoundedDepth}
Let $X = L_p(\Omega)$ with $\Omega=(0,1)$ and $p \in (0,\infty)$,
let $r \in \N$, and let $\mathscr{L}$ be a depth growth function.
Assume that $L := \sup_{n} \mathscr{L}(n) < \infty$.
Setting $\nu := \lfloor L/2\rfloor$, the following statements hold:
\begin{enumerate}[leftmargin=0.6cm]
  \item 
        For $s \in (0,\infty)$, $\alpha \in (0, \nu s)$ and $q \in (0,\infty]$,
        we have
        \begin{equation*}
          \WASpace[X][\varrho_{r}]
          \hookrightarrow \big(
                            L_p (\Omega;\R), B_{\sigma, \sigma}^s (\Omega)
                          \big)_{\tfrac{\alpha}{s\nu}, q}
          \quad \text{where} \quad \sigma := (s+1/p)^{-1}.
        \end{equation*}

  \item 
        For $\alpha \in (0,\infty)$, we have
        \begin{equation*}
          \WASpace[X][\varrho_{r}] \hookrightarrow B^{\alpha/\nu}_{q,q} (\Omega)
          \quad \text{where} \quad
          q := \left(\alpha/\nu+1/p\right)^{-1}.
        \end{equation*}
\end{enumerate}
The same holds for $\NASpace[X][\varrho_{r}]$ instead of $\WASpace[X][\varrho_r]$ if we set $\nu := L-1$.
\end{thm}

The proof involves a Bernstein inequality for piecewise polynomials
by Petrushev \cite{PetrushevSplineAndRational}, and new bounds on the number of pieces
of piecewise polynomials implemented by $\varrho_{r}$-networks.
Petrushev considers the (nonlinear) set $\tilde{\mathtt{S}}(k,n)$
of all piecewise polynomials on $(0,1)$ of degree at most $r=k-1$ ($k \in \N$)
with at most $n-1$ breakpoints in $[0,1]$.
In the language of Definition~\ref{defn:PiecewisePolynomial},
$\tilde{\mathtt{S}}(k,n) = \PPoly_n^r ( (0,1) )$ is the set of piecewise polynomials
of degree at most $r=k-1 \in \N_{0}$ with at most $n$ pieces on $(0,1)$.


%

By \cite[Chapter 12, Theorem 8.2]{ConstructiveApproximation}
(see \cite[Theorem 2.2]{PetrushevSplineAndRational} for the original proof)
the following Bernstein-type inequality holds for each family
$\Sigma := (\tilde{\mathtt{S}}(k,n))_{n \in \N}$, $k \in \N$:

\begin{thm}[{\cite[Theorem 2.2]{PetrushevSplineAndRational}}]
  \label{thm:FreeKnotSplineBernsteinInequality}

    Let $\Omega = (0,1)$, $p \in (0,\infty)$, $r \in \N_{0}$, and $s \in (0,r+1)$ be arbitrary,
    and set $\sigma := (s + 1/p)^{-1}$.
    Then there is a constant $C < \infty$ such that we have
    \begin{equation*}
      \| f \|_{B_{\sigma, \sigma}^{s}(\Omega)} \leq C \cdot n^s \cdot \|f\|_{L_p(\Omega)}
      \qquad \forall \, n \in \N \text{ and } f \in \PPoly_{n}^{r}(\Omega). 
      \qedhere
    \end{equation*}
\end{thm}

\begin{rem}
  Theorem~\ref{thm:FreeKnotSplineBernsteinInequality} even holds for \emph{discontinuous}
  piecewise polynomial functions, see \cite[Theorem 2.2]{PetrushevSplineAndRational}.
  Hence, the Besov spaces in Theorem~\ref{thm:ApproxSpaceIntoBesovForBoundedDepth}
  also contain discontinuous functions.
  This is natural, as $\varrho_{r}$-networks with bounded number of connections or neurons
  approximate indicator functions arbitrarily well (though with weight values going to infinity,
  see the proof of Lemma~\ref{lem:ApproximationOfIndicatorCube}).
\end{rem}
When $f$ is a realization of a $\varrho_{r}$-network of depth $L$,
it is piecewise polynomial \cite{telgarsky2016benefits}.
As there are $L-1$ hidden layers, the polynomial pieces are of degree $r^{L-1}$ at most,
hence $f|_{(0,1)} \in \PPoly^{r^{L-1}}_{n} ( (0,1) )$ for large enough $n$.
This motivates the following definition.

\begin{defn}[Number of pieces]\label{def:NbPieces}
  Define $n_{r}(W,L,N)$ to be the optimal bound on the number of polynomial pieces
  for a $\varrho_{r}$-network with $W \in \N_{0}$ connections,
  depth $L \in \N$ and $N \in \N_{0}$ neurons; that is,
  \[
    n_{r}(W,L,N)
    := \min \left\{
              n \in \N
              \quad \colon \quad
              \forall \, g \in \NNreal^{\varrho_{r},1,1}_{W,L,N}
                \,:\, g|_{(0,1)} \in \PPoly_{n}^{r^{L-1}}( (0,1) )
            \right\}.
  \]
  Furthermore, let $n_{r}(W,L,\infty) := \sup_{N \in \N_{0}} n_{r}(W,L,N)$
  and $n_{r}(\infty,L,N) := \sup_{W \in \N_{0}} n_{r}(W,L,N)$.
\end{defn}

\begin{rem}\label{rem:PolynomialPiecesIndependentOfInterval}
  The definition of $n_r (W,L,N)$ is independent of the non-degenerate interval $I \subset \R$
  used for its definition.
  To see this, write $n_r^{(I)} (W,L,N)$ for the analogue of $n_r (W,L,N)$,
  but with $(0,1)$ replaced by a general non-degenerate interval $I \subset \R$.
  First, note that $n_r^{(I)} (W,L,N) \leq n_r^{(J)} (W,L,N)$ if $I \subset J$.

  Next, note for $g \in \NNreal^{\varrho_r,1,1}_{W,L,N}$ and $a \in (0,\infty)$, $b \in \R$
  that $g_{a,b} := g (a \cdot + b) \in \NNreal^{\varrho_r,1,1}_{W,L,N}$ as well
  (see Lemma~\ref{lem:NetworkCalculus}) and that $g|_I \in \PPoly_n^{r^{L-1}} (I)$
  if and only $g_{a,b} |_{a^{-1} (I - b)} \in \PPoly_n^{r^{L-1}} (a^{-1} (I-b))$.
  Therefore, $n_r^{(I)}(W,L,N) = n_r^{(a I + b)}(W,L,N)$ for all $a \in (0,\infty)$ and $b \in \R$.

  Now, if $J \subset \R$ is any non-degenerate interval,
  and if $I \subset \R$ is a \emph{bounded} interval,
  then $a I + b \subset J$ for suitable $a > 0$, $b \in \R$.
  Hence, $n_r^{(I)} = n_r^{(a I + b)} \leq n_r^{(J)}$.
  In particular, this shows $n_r^{(I)} = n_r^{(J)}$ for all \emph{bounded}
  non-degenerate intervals $I,J \subset \R$.

  Finally, if $g \in \NNreal^{\varrho_r,1,1}_{W,L,N}$ is arbitrary,
  then $g \in \PPoly_n^{r^{L-1}}(\R)$ for \emph{some} $n \in \N$.
  Thus, there are $a,b \in \R$, $a < b$ such that $g|_{(-\infty,a+1)}$ and $g|_{(b-1, \infty)}$
  are polynomials of degree at most $r^{L-1}$.
  Let $k := n_r^{((a,b))} (W,L,N) = n_r^{((0,1))} (W,L,N)$,
  so that $g|_{(a,b)} \in \PPoly_k^{r^{L-1}} ( (a,b) )$.
  Clearly, $g \in \PPoly_{k}^{r^{L-1}} (\R)$.
  Hence, $n_r^{(\R)}(W,L,N) \leq k = n_r^{((0,1))}(W,L,N)$.
\end{rem}


We now have the ingredients to establish the first main lemma behind the proof
of Theorem~\ref{thm:ApproxSpaceIntoBesovForBoundedDepth}.

\begin{lem}\label{lem:ApproxSpaceIntoBesovForBoundedDepth}
  Let $X = L_{p}(\Omega)$ with $\Omega = (0,1)$ and $p \in (0,\infty)$.
  Let $r \in \N$ and $\nu \in (0,\infty)$, and let $\mathscr{L}$ be a depth growth function
  such that $L := \sup_{n} \mathscr{L}(n) < \infty$.
  Assume that
  \begin{equation}\label{eq:GrowthRateNbPieces}
    \sup_{W \in \N} W^{-\nu}\ n_{r}(W,L,\infty) < \infty.
  \end{equation}
  \begin{enumerate}[leftmargin=0.6cm]
    \item For $s \in (0,r+1)$, $\alpha \in (0, \nu \cdot s)$, and $q \in (0,\infty]$, we have
          \begin{equation}\label{eq:ApproxSpaceIntoBesovForBoundedDepth1}
            \WASpace[X][\varrho_{r}]
            \hookrightarrow (L_p (\Omega), B_{\sigma, \sigma}^s (\Omega))_{ \frac{\alpha}{s\nu}, q}
            \quad \text{where} \quad
            \sigma := (s+1/p)^{-1}
            \, .
          \end{equation}

    \item For $\alpha \in (0, \nu(r+1))$, we have
          \begin{equation}\label{eq:ApproxSpaceIntoBesovForBoundedDepth2}
            \WASpace[X][\varrho_{r}]
            \hookrightarrow B^{\alpha/\nu}_{q,q} (\Omega)
            \quad \text{where} \quad
            q := \left( \alpha/\nu + 1/p\right)^{-1} .
          \end{equation}
  \end{enumerate}
  The same results hold with $\NASpace[X][\varrho_{r}]$ instead of $\WASpace[X][\varrho_{r}]$
  if we assume instead that
  \begin{equation}
    \sup_{N \in \N} N^{-\nu}\ n_{r}(\infty,L,N) < \infty.
    \label{eq:GrowthRateNbPiecesbis}
    \qedhere
  \end{equation}
\end{lem}

\begin{proof}[Proof of Lemma~\ref{lem:ApproxSpaceIntoBesovForBoundedDepth}]
As
\(
  \NNreal^{\varrho_{r},1,1}_{n,\mathscr{L}(n),\infty}
  \subset \NNreal^{\varrho_{r},1,1}_{n,L,\infty}
\)
for each $n \in \N$, Theorem~\ref{thm:FreeKnotSplineBernsteinInequality}
and Equation~\eqref{eq:GrowthRateNbPieces} yield a constant $C < \infty$ such that
\begin{equation}
  \| f \|_{B_{\sigma, \sigma}^s(\Omega)}
  \leq C \cdot n^{\nu s} \cdot \| f \|_{L_p(\Omega)},
  \quad \text{for all }\ n \in \N \text{ and } f \in \WeightClassSymbol_{n}(X,\varrho_r,\mathscr{L}),
  \label{eq:BernsteinForBoundedDepth}
\end{equation}
where $\sigma := (s+1/p)^{-1} = (s/d+1/p)^{-1}$ (recall $d=1$).
By \eqref{eq:BesovEmbedsLp} we further get that
\(
  Y_{B} := B_{\sigma, \sigma}^s(\Omega) \hookrightarrow L_p(\Omega)
\),
whence \eqref{eq:BernsteinForBoundedDepth} is a valid Bernstein inequality
for $Y_{B}$ with exponent $\gamma := s \cdot \nu > \alpha$.
Proposition~\ref{prop:BernsteinJacksonConsequence} with $\theta := \alpha/\gamma=\alpha/(s\nu)$
and $0 < q \leq \infty$ yields~\eqref{eq:ApproxSpaceIntoBesovForBoundedDepth1}.

When $0 < \alpha < \nu(r+1)$, there is $s \in (0, r+1)$ such that $0 < \alpha < \nu \cdot s$;
hence, \eqref{eq:ApproxSpaceIntoBesovForBoundedDepth1} holds for any $0 <q \leq \infty$.
By \eqref{eq:LpBesovInterpolation}, we see for $\theta := \frac{\alpha}{s \nu} \in (0,1)$
and $q := (\theta s + 1/p)^{-1} = (\alpha/\nu + 1/p)^{-1}$ that the right hand side
of \eqref{eq:ApproxSpaceIntoBesovForBoundedDepth1} is simply
$B_{q, q}^{\theta s} (\Omega) = B_{q,q}^{\alpha/\nu}(\Omega)$.

The proof for $\NASpace[X][\varrho_{r}]$ follows the same steps.
\end{proof}

Theorem~\ref{thm:ApproxSpaceIntoBesovForBoundedDepth} is a corollary of
Lemma~\ref{lem:ApproxSpaceIntoBesovForBoundedDepth} once we establish~\eqref{eq:GrowthRateNbPieces}
(resp.~\eqref{eq:GrowthRateNbPiecesbis}).
The smaller $\nu$ the better, as it yields a larger value for $\alpha/\nu$, hence a smoother
(smaller) Besov space in~\eqref{eq:ApproxSpaceIntoBesovForBoundedDepth2}.

\begin{lem}\label{lem:MinGrowthRateNbPiecesSufficient}
  Consider $L\in \N_{\geq 2}$, $r \in \N$.
  \begin{itemize}[leftmargin=0.6cm]
    \item Property~\eqref{eq:GrowthRateNbPieces} holds if and only if
          $\nu \geq \lfloor L/2\rfloor$;

    \item Property~\eqref{eq:GrowthRateNbPiecesbis} holds if and only if $\nu \geq L-1$.
          \qedhere
  \end{itemize}
\end{lem}

\begin{proof}
If~\eqref{eq:GrowthRateNbPieces} holds with some exponent $\nu$,
then Lemma~\ref{lem:ApproxSpaceIntoBesovForBoundedDepth}-(2) with $\mathscr{L} \equiv L$,
arbitrary $p \in (0,\infty)$, $\alpha:=\nu$ and $q:=(\alpha/\nu+1/p)^{-1}$
yields $\WASpace[X][\varrho_{r}] \hookrightarrow B^{1}_{q,q}(\Omega)$ with $\Omega := (0,1)$.
If we set $s := 1$, then $\min \{s,2\} = s = 1$.
Hence,
Theorem~\ref{thm:LimitInverseBesov} implies $\nu = \alpha \geq \lfloor L/2\rfloor$.
The same argument shows that if~\eqref{eq:GrowthRateNbPiecesbis} holds with some exponent $\nu$,
then $\nu \geq L-1$.
For the converse results it is clearly sufficient to establish~\eqref{eq:GrowthRateNbPieces}
with $\nu = \lfloor L/2\rfloor$ and~\eqref{eq:GrowthRateNbPiecesbis} with $\nu = L-1$.
The proofs are in Appendix~\ref{sub:NumberOfPiecesSimplified}.
\end{proof}

\begin{proof}[Proof of Theorem~\ref{thm:ApproxSpaceIntoBesovForBoundedDepth}]
  We only prove Part (1) for the spaces $W_q^{\alpha}$.
  The proof for the $N_q^{\alpha}$ spaces and that of Part (2) are similar.

  Let $s \in (0,\infty)$ be arbitrary, and choose $r' \in \N$ such that
  $r \leq r'$ and $s \in (0,r'+1)$.
Combining Lemmas~\ref{lem:ApproxSpaceIntoBesovForBoundedDepth} and
  \ref{lem:MinGrowthRateNbPiecesSufficient}, we get
  \(
    \WASpace[X][\varrho_{r'}]
    \hookrightarrow \big( L_p(\Omega), B^s_{\sigma,\sigma}(\Omega) \big)_{\frac{\alpha}{s \nu},q}
  \).
  Since $\Omega$ is bounded, Theorem~\ref{thm:ReLUPowersApproxSpaces}
  shows that $\WASpace[X][\varrho_{r}] \hookrightarrow \WASpace[X][\varrho_{r'}]$.
  By combining the two embeddings, we get the claim.
\end{proof}

\bibliographystyle{plain}
\bibliography{main.bib}

\begin{thebibliography}{10}

\bibitem{AdamsSobolevSpaces}
R.~A. Adams and J.~J.~F. Fournier.
\newblock {\em Sobolev spaces}, volume 140 of {\em Pure and Applied Mathematics
  (Amsterdam)}.
\newblock Elsevier/Academic Press, Amsterdam, second edition, 2003.

\bibitem{JO17}
J.~Adler and O.~\"{O}ktem.
\newblock Solving ill-posed inverse problems using iterative deep neural
  networks.
\newblock {\em Inverse Problems}, 33:124007, 2017.

\bibitem{AliprantisBorderHitchhiker}
C.~D. Aliprantis and K.~C. Border.
\newblock {\em Infinite dimensional analysis: A hitchhiker's guide}.
\newblock Springer, Berlin, third edition, 2006.

\bibitem{MR2029742}
J.M. Almira and U.~Luther.
\newblock Generalized approximation spaces and applications.
\newblock {\em Math. Nachr.}, 263/264:3--35, 2004.

\bibitem{Barron1993}
A.~R. Barron.
\newblock Universal approximation bounds for superpositions of a sigmoidal
  function.
\newblock {\em IEEE Trans. Inf. Theory}, 39(3):930--945, 1993.

\bibitem{Barron1994}
A.~R. Barron.
\newblock {Approximation and estimation bounds for artificial neural networks}.
\newblock {\em Mach. Learn.}, 14(1):115--133, 1994.

\bibitem{Bartlett:2017ux}
Peter~L Bartlett, Nick Harvey, Christopher Liaw, and Abbas Mehrabian.
\newblock {Nearly-tight VC-dimension and Pseudodimension Bounds for Piecewise
  Linear Neural Networks.}
\newblock {\em J. Mach. Learn. Res.}, 2019.

\bibitem{BBGJJ18}
C.~Beck, S.~Becker, P.~Grohs, N.~Jaafari, and A.~Jentzen.
\newblock Solving stochastic differential equations and kolmogorov equations by
  means of deep learning.
\newblock {\em arXiv preprint arXiv:1806.00421}, 2018.

\bibitem{bolcskei2017optimal}
H.~B{\"o}lcskei, P.~Grohs, G.~Kutyniok, and P.~Petersen.
\newblock Optimal approximation with sparsely connected deep neural networks.
\newblock {\em SIAM J. Math. Data Sci.}, 1:8--45, 2019.

\bibitem{Bubb18}
T.A. Bubba, G.~Kutyniok, M.~Lassas, M.~M\"arz, W.~Samek, S.~Siltanen, and
  V.~Srinivasan.
\newblock Learning the invisible: A hybrid deep learning-shearlet framework for
  limited angle computed tomography.
\newblock {\em Inverse Probl.}, 35(6), 2019.

\bibitem{LaugesenAffineSystemsSpanLebesgueSpaces}
H.-Q. Bui and R.~S. Laugesen.
\newblock Affine systems that span {L}ebesgue spaces.
\newblock {\em J. Fourier Anal. Appl.}, 11(5):533--556, 2005.

\bibitem{CandesDiss}
E.~J. Cand\`{e}s.
\newblock {R}idgelets: {T}heory and {A}pplications, 1998.
\newblock {Ph.D.} thesis, {S}tanford {U}niversity.

\bibitem{ChuXM1994networksforlocApprox}
C.~K. Chui, Xin Li, and H.~N. Mhaskar.
\newblock Neural networks for localized approximation.
\newblock {\em Math. Comp.}, 63(208):607--623, 1994.

\bibitem{cohen2016expressive}
N.~Cohen, O.~Sharir, and A.~Shashua.
\newblock On the expressive power of deep learning: A tensor analysis.
\newblock In {\em Conference on Learning Theory}, pages 698--728, 2016.

\bibitem{cohen2016convolutional}
N.~Cohen and A.~Shashua.
\newblock Convolutional rectifier networks as generalized tensor
  decompositions.
\newblock In {\em International Conference on Machine Learning}, pages
  955--963, 2016.

\bibitem{Cybenko1989}
G.~Cybenko.
\newblock {Approximation by superpositions of a sigmoidal function}.
\newblock {\em Mathematics of Control, Signals and Systems}, 2(4):303--314,
  1989.

\bibitem{DeV98}
R.~A. DeVore.
\newblock {Nonlinear approximation}.
\newblock In {\em Acta numerica}, pages 51--150. Cambridge Univ. Press,
  Cambridge, 1998.

\bibitem{DeVore1997approxfeedforward}
R.~A. DeVore, K.I. Oskolkov, and P.P. Petrushev.
\newblock Approximation by feed-forward neural networks.
\newblock {\em Ann. Numer. Math.}, 4:261--287, 1996.

\bibitem{DeVore:1988fe}
R.~A. DeVore and V.~A. Popov.
\newblock {Interpolation of Besov spaces}.
\newblock {\em Trans. Amer. Math. Soc.}, 305(1):397--414, January 1988.

\bibitem{DeVSharp93}
R.~A. DeVore and R.~C. Sharpley.
\newblock Besov spaces on domains in {${\bf R}^d$}.
\newblock {\em Trans. Amer. Math. Soc.}, 335(2):843--864, 1993.

\bibitem{ConstructiveApproximation}
R.A. DeVore and G.G. Lorentz.
\newblock {\em Constructive approximation}, volume 303 of {\em Grundlehren der
  Mathematischen Wissenschaften [Fundamental Principles of Mathematical
  Sciences]}.
\newblock Springer-Verlag, Berlin, 1993.

\bibitem{Elad17}
M.~Elad.
\newblock Deep, deep trouble. deep learning’s impact on image processing,
  mathematics, and humanity.
\newblock {\em SIAM News}, 2017.

\bibitem{Eldan2016PowerofDepth}
R.~Eldan and O.~Shamir.
\newblock The power of depth for feedforward neural networks.
\newblock In {\em Proceedings of the 29th Conference on Learning Theory, {COLT}
  2016, New York, USA, June 23-26, 2016}, pages 907--940, 2016.

\bibitem{ellacott1994aspects}
SW~Ellacott.
\newblock Aspects of the numerical analysis of neural networks.
\newblock {\em Acta Numer.}, 3:145--202, 1994.

\bibitem{Elstrodt}
J.~Elstrodt.
\newblock {\em Ma\ss - und {I}ntegrationstheorie}.
\newblock Springer Spektrum. Springer Spektrum, Berlin, Heidelberg, eigth
  edition, 2018.

\bibitem{FollandRA}
G.B. Folland.
\newblock {\em {R}eal {A}nalysis: {M}odern {T}echniques and {T}heir
  {A}pplications}.
\newblock Pure and applied mathematics. Wiley, second edition, 1999.

\bibitem{FollandAHA}
Gerald~B. Folland.
\newblock {\em A course in abstract harmonic analysis}.
\newblock Studies in Advanced Mathematics. CRC Press, Boca Raton, FL, 1995.

\bibitem{Foucart:2012wp}
Simon Foucart and Holger Rauhut.
\newblock {\em {A Mathematical Introduction to Compressive Sensing}}.
\newblock Springer, May 2012.

\bibitem{Funahashi1989183}
Ken-Ichi Funahashi.
\newblock {On the approximate realization of continuous mappings by neural
  networks}.
\newblock {\em Neural Networks}, 2(3):183--192, 1989.

\bibitem{Gregor2010}
Karol Gregor and Yann LeCun.
\newblock {Learning Fast Approximations of Sparse Coding}.
\newblock In {\em Proceedings of the 27th Annual International Conference on
  Machine Learning}, pages 399--406, 2010.

\bibitem{Hastad:1tbON80T}
J~T H{\aa}stad.
\newblock {Computational Limitations for Small-Depth Circuits. ACM Doctoral
  Dissertation Award (1986)}, 1987.

\bibitem{ResidualNetworks}
K.~He, X.~Zhang, S.~Ren, and J.~Sun.
\newblock Deep residual learning for image recognition.
\newblock In {\em Proceedings of the IEEE conference on computer vision and
  pattern recognition}, pages 770--778, 2016.

\bibitem{He:2015:DDR:2919332.2919814}
Kaiming He, Xiangyu Zhang, Shaoqing Ren, and Jian Sun.
\newblock Delving deep into rectifiers: Surpassing human-level performance on
  imagenet classification.
\newblock In {\em Proceedings of the 2015 IEEE International Conference on
  Computer Vision (ICCV)}, ICCV '15, pages 1026--1034, Washington, DC, USA,
  2015. IEEE Computer Society.

\bibitem{HoffmanKunzeLinearAlgebra}
K.~Hoffman and R.~Kunze.
\newblock {\em Linear algebra}.
\newblock Second edition. Prentice-Hall, Inc., Englewood Cliffs, N.J., 1971.

\bibitem{Hornik1991251}
K.~Hornik.
\newblock Approximation capabilities of multilayer feedforward networks.
\newblock {\em Neural Networks}, 4(2):251 -- 257, 1991.

\bibitem{Hornik1989universalApprox}
K.~Hornik, M.~Stinchcombe, and H.~White.
\newblock Multilayer feedforward networks are universal approximators.
\newblock {\em Neural Netw.}, 2(5):359--366, 1989.

\bibitem{JohnenSchererEquivalenceKFunctionalModulusContinuity}
H.~Johnen and K.~Scherer.
\newblock On the equivalence of the {$K$}-functional and moduli of continuity
  and some applications.
\newblock In {\em Constructive theory of functions of several variables
  ({P}roc. {C}onf., {M}ath. {R}es. {I}nst., {O}berwolfach, 1976)}, pages
  119--140. Lecture Notes in Math., Vol. 571. Springer, Berlin, 1977.

\bibitem{AlexNet}
A.~Krizhevsky, I.~Sutskever, and G.~E. Hinton.
\newblock Imagenet classification with deep convolutional neural networks.
\newblock In {\em Proceedings of the 25th International Conference on Neural
  Information Processing Systems - Volume 1}, NIPS'12, pages 1097--1105, USA,
  2012. Curran Associates Inc.

\bibitem{LaugesenAffineSynthesisQuasiBanach}
R.~S. Laugesen.
\newblock Affine synthesis onto {$L^p$} when {$0<p\leq 1$}.
\newblock {\em J. Fourier Anal. Appl.}, 14(2):235--266, 2008.

\bibitem{LaxCalculusWithApplications}
P.~D. Lax and M.~S. Terrell.
\newblock {\em Calculus with applications}.
\newblock Undergraduate Texts in Mathematics. Springer, New York, second
  edition, 2014.

\bibitem{lemagoarou:hal-01167948}
Luc Le~Magoarou and Remi Gribonval.
\newblock {Flexible Multi-layer Sparse Approximations of Matrices and
  Applications}.
\newblock {\em IEEE Journal of Selected Topics in Signal Processing}, June
  2016.

\bibitem{LeCunDeepLearningNaturePaper}
Y.~LeCun, Y.~Bengio, and G.~Hinton.
\newblock Deep learning.
\newblock {\em Nature}, 521(7553):436--444, 2015.

\bibitem{PinkusUniversalApproximation}
M.~Leshno, V.~Ya. Lin, A.~Pinkus, and S.~Schocken.
\newblock Multilayer feedforward networks with a nonpolynomial activation
  function can approximate any function.
\newblock {\em Neural Netw.}, 6(6):861--867, 1993.

\bibitem{Maas:tn}
A.~L. Maas, A.~Y. Hannun, and A.~Y. Ng.
\newblock Rectifier nonlinearities improve neural network acoustic models.
\newblock In {\em Proc. ICML}, volume~30, page~3, 2013.

\bibitem{PinkusLowerBoundsForMLPApproximation}
V.~Maiorov and A.~Pinkus.
\newblock Lower bounds for approximation by {MLP} neural networks.
\newblock {\em Neurocomputing}, 25(1):81--91, 1999.

\bibitem{Mallat:2016jr}
St{\'e}phane Mallat.
\newblock {Understanding deep convolutional networks}.
\newblock {\em Phil. Trans. R. Soc. A}, 374(2065):20150203--16, March 2016.

\bibitem{MPWN18}
A.~Mardt, L.~Pasquali, H.~Wu, and F.~No\'{e}.
\newblock Vampnets: Deep learning of molecular kinetics.
\newblock {\em Nature communications}, 9:5, 2018.

\bibitem{McCullochPitts}
W.~S. McCulloch and W.~Pitts.
\newblock A logical calculus of the ideas immanent in nervous activity.
\newblock {\em Bull. Math. Biophys.}, 5(4):115--133, 1943.

\bibitem{Mhaskar93}
H.~N. Mhaskar.
\newblock Approximation properties of a multilayered feedforward artificial
  neural network.
\newblock {\em Adv. Comput. Math.}, 1(1):61--80, 1993.

\bibitem{Mhaskar1996NNapprox}
H.~N. Mhaskar.
\newblock Neural networks for optimal approximation of smooth and analytic
  functions.
\newblock {\em Neural Comput.}, 8(1):164--177, 1996.

\bibitem{Mhaskar2016DeepVSShallow}
H.~N. Mhaskar and T.~Poggio.
\newblock {Deep vs. shallow networks: An approximation theory perspective}.
\newblock {\em Analysis and Applications}, 14(06):829--848, 2016.

\bibitem{Mhaskar1995151}
H.N. Mhaskar and C.A. Micchelli.
\newblock Degree of approximation by neural and translation networks with a
  single hidden layer.
\newblock {\em Adv. Appl. Math.}, 16(2):151--183, 1995.

\bibitem{NguyenThien1999687}
T.~Nguyen-Thien and T.~Tran-Cong.
\newblock Approximation of functions and their derivatives: A neural network
  implementation with applications.
\newblock {\em Appl. Math. Model.}, 23(9):687--704, 1999.

\bibitem{SkipConnections}
Emin Orhan and Xaq Pitkow.
\newblock Skip connections eliminate singularities.
\newblock In {\em 6th International Conference on Learning Representations,
  {ICLR} 2018, Vancouver, BC, Canada, April 30 - May 3, 2018, Conference Track
  Proceedings}. OpenReview.net, 2018.

\bibitem{PetersenVoigtlaenderReLU}
P.~Petersen and F.~Voigtlaender.
\newblock Optimal approximation of piecewise smooth functions using deep {ReLU}
  neural networks.
\newblock {\em Neural Netw.}, 108:296--330, 2018.

\bibitem{PetrushevSplineAndRational}
P.P. Petrushev.
\newblock Direct and converse theorems for spline and rational approximation
  and {B}esov spaces.
\newblock In {\em Function spaces and applications ({L}und, 1986)}, volume 1302
  of {\em Lecture Notes in Math.}, pages 363--377. Springer, Berlin, 1988.

\bibitem{pinkus1999approximation}
A.~Pinkus.
\newblock Approximation theory of the {MLP} model in neural networks.
\newblock {\em Acta Numer.}, 8:143--195, 1999.

\bibitem{Ronneberger:2015gk}
Olaf Ronneberger, Philipp Fischer, and Thomas Brox.
\newblock {U-Net: Convolutional Networks for Biomedical Image Segmentation}.
\newblock Springer International Publishing, Cham, 2015.

\bibitem{RudinFA}
W.~Rudin.
\newblock {\em Functional analysis}.
\newblock International Series in Pure and Applied Mathematics. McGraw-Hill,
  Inc., New York, second edition, 1991.

\bibitem{SchmidtHieber:2017vn}
J.~Schmidt-Hieber.
\newblock {Nonparametric regression using deep neural networks with ReLU
  activation function.}
\newblock {\em arXiv preprint arXiv:1708.06633}, math.ST, 2017.
\newblock To appear as a discussion article in Annals of Statistics.

\bibitem{TACMT17}
K.~T. Schütt, F.~Arbabzadah, S.~Chmiela, K.~R. Müller, and A.~Tkatchenko.
\newblock Quantum-chemical insights from deep tensor neural networks.
\newblock {\em Nature communications}, 8:13890, 2017.

\bibitem{ShaCC2015provableAppDNN}
U.~Shaham, A.~Cloninger, and R.~R. Coifman.
\newblock Provable approximation properties for deep neural networks.
\newblock {\em Appl. Comput. Harmon. Anal.}, 44(3):537--557, May 2018.

\bibitem{TopologyWithApplications}
A.~N. Somashekhar and J.~F. Peters.
\newblock {\em Topology with Applications}.
\newblock World Scientific, 2013.

\bibitem{telgarsky2016benefits}
Matus Telgarsky.
\newblock {Benefits of depth in neural networks.}
\newblock {\em Journal of Machine Learning Research}, 49(June):1517--1539, June
  2016.
\newblock 29th Conference on Learning Theory, COLT 2016 - New York, United
  States.

\bibitem{Unser:1999cr}
M.~A. Unser.
\newblock {Splines: a perfect fit for signal and image processing}.
\newblock {\em IEEE Signal Processing Magazine}, 16(6):22--38, 1999.

\bibitem{VoigtlaenderPhDThesis}
F.~Voigtlaender.
\newblock {\em Embedding Theorems for Decomposition Spaces with Applications to
  Wavelet Coorbit Spaces}.
\newblock PhD thesis, RWTH Aachen University, 2015.
\newblock \url{http://publications.rwth-aachen.de/record/564979}.

\bibitem{WiderOrDeeper}
Zifeng Wu, Chunhua Shen, and Anton van~den Hengel.
\newblock {Wider or Deeper: Revisiting the ResNet Model for Visual
  Recognition}.
\newblock {\em Pattern Recognition}, 90:119--133, June 2019.

\bibitem{yarotsky2017error}
D.~Yarotsky.
\newblock Error bounds for approximations with deep {R}e{L}{U} networks.
\newblock {\em Neural Netw.}, 2017.

\bibitem{YarotskySurprise}
Dmitry Yarotsky.
\newblock {Optimal approximation of continuous functions by very deep ReLU
  networks.}
\newblock {\em Journal of Machine Learning Research}, pages 639--649, 2018.
\newblock COLT 2018.

\end{thebibliography}
\vspace{2cm}
\appendix


\section{Proofs for Section~\ref{sec:NetworkCalculus}}
\label{sec:AppendixTechnical}


For a matrix $A \in \R^{n \times d}$, we write $A^T \in \R^{d \times n}$
for the transpose of $A$.
For $i \in \FirstN{n}$ we write $A_{i,-} \in \R^{1 \times d}$ for the $i$-th
row of $A$, while $A_{{(i)}} \in \R^{(n-1) \times d}$ denotes the matrix
obtained by deleting the $i$-th row of $A$.
We use the same notation $b_{(i)}$ for vectors $b\in\R^{n}\cong\R^{n \times 1}$.
Finally, for $j \in \FirstN{d}$, $A_{[j]} \in \R^{n \times (d-1)}$ denotes
the matrix obtained by removing the $j$-th column of $A$.


\subsection{Proof of Lemma~\ref{lem:BiasWeightsDontMatter}}
\label{app:PfBiasWeightsDontMatter}
Write $N_0 (\Phi) := \InDim(\Phi) + \OutDim(\Phi) + N(\Phi)$ for the total number of
neurons of the network $\Phi$, \emph{including the ``non-hidden'' neurons}.

The proof is by contradiction.
Assume that there is a network $\Phi$ for which the claim fails.
Among all such networks, consider one with minimal value of $N_0(\Phi)$,
i.e., such the claim holds for all networks $\Psi$ with $N_0(\Psi) < N_0(\Phi)$.
Let us write $\Phi = \big( (T_1, \alpha_1), \dots, (T_L, \alpha_L) \big)$
with $T_\ell \, x = A^{(\ell)} x + b^{(\ell)}$,
for certain $A^{(\ell)} \in \R^{N_\ell \times N_{\ell-1}}$
and $b^{(\ell)} \in \R^{N_\ell}$.

Let us first consider the case that
\begin{equation}
    \forall \, \,  \ell \in \FirstN{L} \,\,
        \forall \,\, i \in \FirstN{N_\ell} \, : \,
            A^{(\ell)}_{i, -} \neq 0.
  \label{eq:BiasWeightsDontMatterFirstCase}
\end{equation}
By 
\eqref{eq:BiasWeightsDontMatterFirstCase}, we get
$\|A^{(\ell)}\|_{\ell^0} \geq N_\ell \geq \|b^{(\ell)}\|_{\ell^0}$, so that
\[
    W_0(\Phi)
    = \sum_{\ell = 1}^L (\|A^{(\ell)}\|_{\ell^0} + \|b^{(\ell)}\|_{\ell^0})
    \leq 2 \cdot \sum_{\ell = 1}^L \|A^{(\ell)}\|_{\ell^0}
    = 2 W(\Phi) \leq \OutDim(\Phi) + 2 W(\Phi).
\]
Hence, with $\widetilde{\Phi} = \Phi$, $\Phi$ satisfies the claim of the lemma,
in contradiction to our assumption. 

Thus, there is some $\ell_0 \in\FirstN{L}$ and some $i \in\FirstN{N_{\ell_0}}$
satisfying $A^{(\ell_0)}_{i, -} = 0$.
In other words, there is a neuron that is not connected to the previous layers.
Intuitively, one can ``remove it'' without changing $\Realization(\Phi)$.
This is what we now show formally.

Let us write $\alpha_\ell = \bigotimes_{j=1}^{N_\ell} \varrho_j^{(\ell)}$
for certain $\varrho_j^{(\ell)} \in \{\identity_{\R}, \varrho\}$, and set
$\theta_\ell := \alpha_\ell \circ T_\ell$, so that
$\Realization (\Phi) = \theta_L \circ \cdots \circ \theta_1$.
By our choice of $\ell_0$ and $i$, note
\begin{equation}
    \big( \theta_{\ell_0} (x) \big)_i
    = \varrho_i^{(\ell_0)} \left( (A^{(\ell_0)} x + b^{(\ell_0)})_i \right)
    = \varrho_i^{(\ell_0)}
        \left( \langle A^{(\ell_0)}_{i, -}, x \rangle + b_i^{(\ell_0)} \right)
    = \varrho_i^{(\ell_0)} (b_i^{(\ell_0)}) =: c \in \R,
    \label{eq:VanishingWeightRowConsequence}
\end{equation}
for arbitrary $x \in \R^{N_{\ell_0 - 1}}$.
After these initial observations, we now distinguish four cases:

\medskip{}

\textbf{Case 1 (Neuron on the output layer of size $\OutDim(\Phi) = 1$):}
We have $\ell_0 = L$ and $N_L = 1$,
so that necessarily $i = 1$.
In view of Equation~\eqref{eq:VanishingWeightRowConsequence}, we then have
$\Realization (\Phi) \equiv c$.
Thus, if we choose the affine-linear map $S_1 : \R^{N_0}\to\R^1, x\mapsto c$,
and set $\gamma_1 := \identity_\R$, then the strict $\varrho$-network
$\widetilde{\Phi} := \big( (S_1, \gamma_1) \big)$ satisfies
$\Realization(\, \widetilde{\Phi} \,) \equiv c \equiv \Realization(\Phi)$,
and $L(\, \widetilde{\Phi} \,) = 1 \leq L(\Phi)$,
as well as $W_0(\, \widetilde{\Phi} \,) =1=\OutDim(\Phi) \leq \OutDim(\Phi) +2 W(\Phi)$
and $N(\, \widetilde{\Phi} \,) = 0 \leq N(\Phi)$.
Thus, $\Phi$ satisfies the claim of the lemma, contradicting our assumption.

\medskip{}

\textbf{Case 2 (Neuron on the output layer of size $\OutDim(\Phi)>1$):}
We have $\ell_0 = L$ and $N_L > 1$. Define
\[
  B^{(\ell)} := A^{(\ell)},
  \quad c^{(\ell)} := b^{(\ell)},
  \quad \text{and } \quad \beta_\ell := \alpha_\ell
  \quad \text{for} \quad \ell \in \FirstN{L - 1}.
\]
We then set $B^{(L)} := A^{(L)}_{(i)} \in \R^{(N_L - 1) \times N_{L-1}}$
and $c^{(L)} := b^{(L)}_{(i)} \in \R^{N_{L} - 1}$, as well as
$\beta_L := \identity_{\R^{N_L - 1}}$.

Setting $S_\ell \, x := B^{(\ell)} x+c^{(\ell)}$ for $x \in \R^{N_{\ell - 1}}$,
the network $\Phi_0 := \big( (S_1, \beta_1), \dots, (S_L, \beta_L) \big)$ then satisfies
$\Realization (\Phi_0) (x) = \big( \Realization (\Phi) (x) \big)_{(i)}$ for all $x \in \R^{N_0}$,
and $N_0 (\Phi_0) = N_0 (\Phi) - 1 < N_0 (\Phi)$.
Furthermore, if $\Phi$ is strict, then so is $\Phi_0$.

By the ``minimality'' assumption on $\Phi$, there is thus a network $\widetilde{\Phi}_0$
(which is strict if $\Phi$ is strict)
with $\Realization (\, \widetilde{\Phi} \,_0) = \Realization (\Phi_0)$
and such that $L' := L(\, \widetilde{\Phi} \,_0) \leq L(\Phi_0) = L(\Phi)$,
as well as $N (\, \widetilde{\Phi} \,_0) \leq N (\Phi_0) = N(\Phi)$, and
\[
    W (\, \widetilde{\Phi} \,_0)
    \leq W_0 (\, \widetilde{\Phi} \,_0)
    \leq \OutDim(\Phi_0) + 2 \cdot W(\Phi_0)
    \leq \OutDim(\Phi) - 1 + 2 \cdot W(\Phi).
\]

Let us write $\widetilde{\Phi}_0 = \big( (U_1, \gamma_1), \dots, (U_{L'}, \gamma_{L'}) \big)$,
with affine-linear maps $U_\ell : \R^{M_{\ell - 1}} \to \R^{M_\ell}$, so that
$U_\ell \, x = C^{(\ell)} x + d^{(\ell)}$ for $\ell \in \FirstN{L'}$ and
$x \in \R^{M_{\ell - 1}}$.
Note that $M_{L'} = N_L - 1$, and define
\[
    \widetilde{C}^{(L')}
    := \left(
          \begin{matrix}
              C^{(L')}_{1, -} \\
              \vdots \\
              C^{(L')}_{i-1, -} \\
              0 \\
              C^{(L')}_{i, -} \\
              \vdots \\
              C^{(L')}_{M_{L'}, -}
          \end{matrix}
       \right) \in \R^{N_L \times M_{L'-1}}
   \quad \text{and} \quad
    \widetilde{d}^{(L')}
    := \left(
          \begin{matrix}
              d^{(L')}_1 \\
              \vdots \\
              d^{(L')}_{i-1} \\
              c \\
              d^{(L')}_{i} \\
              \vdots \\
              d^{(L')}_{M_{L'}}
          \end{matrix}
       \right) \in \R^{N_L} ,
\]
as well as $\widetilde{\gamma}_{L'} := \identity_{\R^{N_L}}$, and
$\widetilde{U}_{L'} : \R^{M_{L' - 1}} \to \R^{N_L},
x \mapsto \widetilde{C}^{(L')} x + \widetilde{d}^{(L')}$, and finally
\[
   \widetilde{\Phi}
   := \big(
         (U_1, \gamma_1),
         \dots,
         (U_{L' - 1}, \gamma_{L' - 1}),
         (\widetilde{U}_{L'}, \widetilde{\gamma}_{L'})
      \big).
\]

By virtue of Equation~\eqref{eq:VanishingWeightRowConsequence}, we then
have $\Realization (\, \widetilde{\Phi} \,) = \Realization (\Phi)$,
and if $\Phi$ is strict, then so is $\Phi_0$ and thus also $\widetilde{\Phi}_0$ and $\widetilde{\Phi}$.
Furthermore, we have $L (\, \widetilde{\Phi} \,) = L' \leq L(\Phi)$, and
$N(\, \widetilde{\Phi} \,) = N(\widetilde{\Phi}_0) \leq N(\Phi)$, as well as
$W (\, \widetilde{\Phi} \,) \leq W_0 (\, \widetilde{\Phi} \,)
\leq 1 + W_0 (\, \widetilde{\Phi} \,_0) \leq \OutDim(\Phi) + 2 W(\Phi)$.
Thus, $\Phi$ satisfies the claim of the lemma, contradicting our assumption.

\medskip{}

\textbf{Case 3 (Hidden neuron on layer $\ell_0$ with $N_{\ell_0} = 1$):}
We have $1 \leq \ell_0 < L$ and $N_{\ell_0} = 1$.
In this case, Equation~\eqref{eq:VanishingWeightRowConsequence} implies
$\theta_{\ell_0} \equiv c$, whence
$\Realization(\Phi) = \theta_L \circ \cdots \circ \theta_1 \equiv \widetilde{c}$
for some $\widetilde{c} \in \R^{N_L}$.

Thus, if we choose the affine map $S_1 : \R^{N_0} \to \R^{N_L}, x \mapsto \widetilde{c}$,
then the strict $\varrho$-network $\widetilde{\Phi} = \big( (S_1, \gamma_1) \big)$
satisfies $\Realization(\widetilde{\Phi}) \equiv \widetilde{c} \equiv \Realization(\Phi)$
and $L(\widetilde{\Phi}) = 1 \leq L(\Phi)$, as well as
$W_0 (\widetilde{\Phi}) \leq d_{\mathrm{out}} (\Phi) \leq d_{\mathrm{out}} (\Phi) + 2 \, W(\Phi)$
and $N(\widetilde{\Phi}) = 0 \leq N(\Phi)$.
Thus, $\Phi$ satisfies the claim of the lemma, in contradiction to our choice of $\Phi$.

\medskip{}

\textbf{Case 4 (Hidden neuron on layer $\ell_0$ with $N_{\ell_0} > 1$):}
In this case, we have $1 \leq \ell_0 < L$ and $N_{\ell_0} > 1$.
Define $S_\ell := T_\ell$ and $\beta_\ell := \alpha_\ell$ for
$\ell \in \FirstN{L} \setminus \{\ell_0, \ell_0 + 1\}$, and let us choose
${S_{\ell_0} : \R^{N_{\ell_0 - 1}} \to \R^{N_{\ell_0} - 1}, x \mapsto B^{(\ell_0)} x + c^{(\ell_0)}}$,
where
\[
    B^{(\ell_0)} := A^{(\ell_0)}_{(i)},
    \quad
    c^{(\ell_0)} := b^{(\ell_0)}_{(i)},
    \quad \text{and} \quad
    \beta_{\ell_0}
    := \varrho_1^{(\ell_0)}
       \otimes \cdots
       \otimes \varrho_{i - 1}^{(\ell_0)}
       \otimes \varrho_{i+1}^{(\ell_0)}
       \otimes \cdots
       \otimes \varrho_{N_{\ell_0}}^{(\ell_0)} .
\]
Finally, for $x \in \R^{N_{\ell_0} - 1}$, let
\(
    \iota_c (x)
    :=  \left( x_1,\dots, x_{i-1}, c, x_i,\dots, x_{N_{\ell_0} - 1}  \right)^T
    \in \R^{N_{\ell_0}} \, ,
\)
and set $\beta_{\ell_{0} + 1} := \alpha_{\ell_0 + 1}$, as well as
\[
    S_{\ell_0 + 1} : \R^{N_{\ell_0} - 1} \to \R^{N_{\ell_0 + 1}},
                     x \mapsto A_{[i]}^{(\ell_0 + 1)} \, x
                               + c \cdot A^{(\ell_0 + 1)} e_{i}
                               + b^{(\ell_0 + 1)}
                       =       A^{(\ell_0 + 1)} (\iota_c (x) )
                               + b^{(\ell_0 + 1)},
\]
where $e_i$ is the $i$-th element of the standard basis of $\R^{N_{\ell_0}}$.

Setting $\vartheta_{\ell} := \beta_\ell \circ S_\ell$
and recalling that $\theta_\ell = \alpha_\ell \circ T_\ell$ for $\ell \in\FirstN{L}$,
we then have $\vartheta_{\ell_0} (x) = (\theta_{\ell_0} (x) )_{(i)}$ for all
$x \in \R^{N_{\ell_0 - 1}}$.
By virtue of Equation~\eqref{eq:VanishingWeightRowConsequence}, this implies
$\theta_{\ell_0} (x) = \iota_c ( \vartheta_{\ell_0} (x) )$, so that
\[
    S_{\ell_0 + 1} (\vartheta_{\ell_0} (x) )
    = A^{(\ell_0 + 1)} \big(\iota_c ( \vartheta_{\ell_0 } (x) )\big)
      + b^{(\ell_0 + 1)}
    = A^{(\ell_0 + 1)} (\theta_{\ell_0} (x) ) + b^{(\ell_0 + 1)}
    = T_{\ell_0 + 1} (\theta_{\ell_0} (x) ) .
\]
Recalling that $\beta_{\ell_0 + 1} = \alpha_{\ell_0 + 1}$, we thus see
$\vartheta_{\ell_0 + 1} \circ \vartheta_{\ell_0}
= \theta_{\ell_0 + 1} \circ \theta_{\ell_0}$, which then easily shows
$\Realization (\Phi_0) = \Realization (\Phi)$ for
$\Phi_0 := \big( (S_1, \beta_1),\dots, (S_L, \beta_L) \big)$.
Note that if $\Phi$ is strict, then so is $\Phi_0$.
Furthermore, we have $N_{0}(\Phi_0) = N_{0}(\Phi) - 1 < N_{0}(\Phi)$ so that by ``minimality''
of $\Phi$, there is a network $\widetilde{\Phi}_0$ (which is strict if $\Phi$ is strict)
satisfying $\Realization(\, \widetilde{\Phi}_0\,)=\Realization(\Phi_0)=\Realization(\Phi)$
and furthermore $L(\, \widetilde{\Phi}_0 \,) \leq L(\Phi_0) = L(\Phi)$,
as well as $N(\, \widetilde{\Phi}_0\, ) \leq N(\Phi_0) \leq N(\Phi)$, and finally
\(
    W (\, \widetilde{\Phi}_0 \,)
    \leq W_0 (\, \widetilde{\Phi}_0 \,)
    \leq \OutDim (\Phi_0) + 2 W(\Phi_0)
    \leq \OutDim (\Phi) + 2 W(\Phi).
\)
Thus, the claim holds for $\Phi$, contradicting our assumption.
\hfill$\square$

\subsection{Proof of Lemma~\ref{lem:BoundingLayersAndNeuronsByWeights}}


\label{app:PfBoundingLayersAndNeuronsByWeights}

We begin by showing $\NNreal_{W, L,W}^{\varrho,d,k} \subset \NNreal_{W,W,W}^{\varrho,d,k}$.
Let $f \in \NNreal_{W, L,W}^{\varrho,d,k}$.
By definition there is $\Phi \in \NNsymbol_{W, L,W}^{\varrho,d,k}$ such that
$f = \Realization(\Phi)$.
Note that $W(\Phi) \leq W$, and let us distinguish two cases:
If $L(\Phi) \leq W(\Phi)$ then $L(\Phi) \leq W$, whence in fact
$\Phi \in \NNsymbol_{W, W, W}^{\varrho,d,k}$ and
$f \in \NNreal_{W, W, W}^{\varrho,d,k}$ as claimed.
Otherwise, $W(\Phi) < L(\Phi)$ and by
Corollary~\ref{cor:BoundingConnectionsWithLayers2}
we have $f = \Realization(\Phi) \equiv c$ for some $c \in \R^{k}$.
Therefore, Lemma~\ref{lem:ConstantMaps} shows that
$f \in \NNreal_{0,1,0}^{\varrho,d,k} \subset \NNreal_{W,W,W}^{\varrho,d,k}$,
where the inclusion holds by definition of these sets.

The inclusion $\NNreal_{W,L,W}^{\varrho,d,k} \subset \NNreal_{W,L,\infty}^{\varrho,d,k}$
is trivial.
Similarly, if  $L \geq W$ then trivially
$\NNreal_{W, W, W}^{\varrho,d,k} \subset \NNreal_{W,L,W}^{\varrho,d,k}$.

Thus, it remains to show
$\NNreal_{W,L,\infty}^{\varrho,d,k} \subset \NNreal_{W,L,W}^{\varrho,d,k}$.
%
To prove this, we will show that for each network
\(
  \Phi
  = \big( (T_{1}, \alpha_{1}), \dots, (T_{K}, \alpha_{K}) \big)
  \in \NNsymbol_{W,L,\infty}^{\varrho,d,k}
\)
(so that necessarily $K \leq L$)
with $N (\Phi) > W$, one can find a neural network
$\Phi' \in \NNsymbol_{W,L, \infty}^{\varrho,d,k}$ with
$\Realization (\Phi') = \Realization(\Phi)$, and such that $N(\Phi') < N(\Phi)$.
If $\Phi$ is strict, then we show that $\Phi'$ can also be chosen to be strict.
The desired inclusion can then be obtained by repeating this ``compression''
step until one reaches the point where $N(\Phi') \leq W$.

For each $\ell \in \FirstN{K}$, let $b^{(\ell)} \in \R^{N_\ell}$ and
$A^{(\ell)} \in \R^{N_{\ell} \times N_{\ell-1}}$ be such that
$T_\ell = A^{(\ell)} \bullet + b^{(\ell)}$.
Since $\Phi\in \NNsymbol_{W,L,\infty}^{\varrho,d,k}$, we have $W(\Phi)\leq W$.
In combination with $N(\Phi) > W$, this implies
\[
    \sum_{\ell = 1}^{K-1} N_\ell
    = N(\Phi)
    > W
    \geq W(\Phi)
    = \sum_{\ell = 1}^K \| A^{(\ell)} \|_{\ell_0}
    \geq \sum_{\ell = 1}^{K-1}
           \sum_{i = 1}^{N_\ell}
           \| A^{(\ell)}_{i, -} \|_{\ell^0} \, .
\]
Therefore, $K > 1$, and there must be some $\ell_0 \in \FirstN{K-1}$ and
$i \in \FirstN{N_{\ell_0}}$ with $A^{(\ell_0)}_{i, -} = 0$.
We now distinguish two cases:

\medskip{}

\textbf{Case 1} (Single neuron on layer $\ell_{0}$):
We have $N_{\ell_0} = 1$.
In this case, $A^{(\ell_0)} = 0$ and hence $T_{\ell_0} \equiv b^{(\ell_0)}$.
Therefore, $\Realization (\Phi)$ is constant; say $\Realization (\Phi) \equiv c \in \R^k$.
Choose $S_{1} : \R^d \to \R^k, x \mapsto c$, and $\beta_{1} := \identity_{\R^k}$.
Then $\Realization (\Phi) \equiv c \equiv \Realization (\Phi')$ for the
strict $\varrho$-network
\(
  \Phi '
  := \big( (S_{1},\beta_{1}) \big)
  \in \NNsymbol_{0,1,0}^{\varrho,d,k}
  \subset \NNsymbol_{W,L,\infty}^{\varrho,d,k}
\),
which indeed satisfies $N(\Phi') = 0 \leq W < N(\Phi)$.

\medskip{}

\textbf{Case 2} (Multiple neurons on layer $\ell_{0}$):
We have $N_{\ell_0} > 1$.
Recall that $\ell_0 \in \FirstN{K-1}$, so that $\ell_0 + 1 \in \FirstN{K}$.
Now define $S_{\ell} := T_{\ell}$ and 
$\beta_{\ell} := \alpha_{\ell}$ for
$\ell \in \FirstN{K} \setminus \{\ell_0, \ell_0 + 1\}$.
Further, define
\[
    S_{\ell_0}
    : \R^{N_{\ell_0 - 1}} \to \R^{N_{\ell_0}-1},
    \quad \! \text{with} \! \quad
    (S_{\ell_0} \, x)_{j} := \!
    \begin{cases}
        (T_{\ell_0} \, x)_j ,    & \text{if } j < i, \\
        (T_{\ell_0} \, x)_{j+1}, & \text{if } j \geq i
    \end{cases}
    \quad \! \text{for } j \in \FirstN{N_{\ell_0}-1} .
\]
Using the notation $A_{(i)}, b_{(i)}$ from the beginning of Appendix~\ref{sec:AppendixTechnical},
this means
$S_{\ell_0} \, x=A^{(\ell_0)}_{(i)} x + b^{(\ell_0)}_{(i)} = (T_{\ell_0} \, x)_{(i)}$.

Finally, writing
$\alpha_{\ell} = \varrho_1^{(\ell)} \otimes \cdots \otimes \varrho_{N_\ell}^{(\ell)}$
for $\ell \in \FirstN{K}$, define $\beta_{\ell_0 + 1} := \alpha_{\ell_0 + 1}$, as well as
\[
    \beta_{\ell_0}
    := \varrho_1^{(\ell_0)}
       \otimes \cdots
       \otimes \varrho_{i-1}^{(\ell_0)}
       \otimes \varrho_{i+1}^{(\ell_0)}
       \otimes \cdots
       \otimes \varrho_{N_{\ell_0}}^{(\ell_0)}
       \quad : \quad
       \R^{N_{\ell_0}-1} \to \R^{N_{\ell_0}-1}  \, ,
\]
and
\begin{align*}
    S_{\ell_0 + 1}
    : \R^{N_{\ell_0}-1} \to \R^{N_{\ell_0 + 1}},
    y & \mapsto
      T_{\ell_0 + 1}
      \left(
          y_1, \dots, y_{i-1},
          \varrho_i^{(\ell_0)}(b_i^{(\ell_0)}),
          y_i, \dots, y_{N_{\ell_0}-1}
      \right) \\
      & = A_{[i]}^{(\ell_0 + 1)} y
        + b^{(\ell_0 + 1)}
        + \varrho_i^{(\ell_0)} (b_i^{(\ell_0)}) \cdot A^{(\ell_0 + 1)} \, e_i,
\end{align*}
where $e_i \in \R^{N_{\ell_0}}$ denotes the $i$-th element of the standard
basis, and where $A_{[i]}$ is the matrix obtained from a given matrix $A$
by removing its $i$-th column.


Now, for arbitrary $x \in \R^{N_{\ell_0 - 1}}$, let
$y := S_{\ell_0} \, x \in \R^{N_{\ell_0} - 1}$ and
$z := T_{\ell_0} \, x \in \R^{N_{\ell_0}}$.
Because of $A^{(\ell_0)}_{i, -} = 0$, we then have $z_i = b_i^{(\ell_0)}$.
Further, by definition of $S_{\ell_0}$,
we have $y_{j}=(T_{\ell_0} \, x)_j = z_j$ for $j<i$,
and $y_j =(T_{\ell_0} \, x)_{j+1}=z_{j+1}$ for $j\geq i$.
All in all, this shows
\begin{align*}
  S_{\ell_0 + 1} \big( \beta_{\ell_0} (S_{\ell_0} x) \big)
  & = S_{\ell_0 + 1} (\beta_{\ell_0} (y) ) \\
  & = T_{\ell_0 + 1}
    \left(
        \varrho_1^{(\ell_0)} (y_1),
        \dots,
        \varrho_{i-1}^{(\ell_0)} (y_{i-1}),
        \varrho_i^{(\ell_0)} (b_i^{(\ell_0)}),
        \varrho_{i+1}^{(\ell_0)} (y_i),
        \dots,
        \varrho_{N_{\ell_0}}^{(\ell_0)} (y_{N_{\ell_0} - 1})
    \right) \\
  & = T_{\ell_0 + 1}
    \left(
        \varrho_1^{(\ell_0)} (z_1),
        \dots,
        \varrho_{i-1}^{(\ell_0)} (z_{i-1}),
        \varrho_i^{(\ell_0)} (z_i),
        \varrho_{i+1}^{(\ell_0)} (z_{i+1}),
        \dots,
        \varrho_{N_{\ell_0}}^{(\ell_0)} (z_{N_{\ell_0}})
    \right) \\
  & = T_{\ell_0 + 1} \big( \alpha_{\ell_0} (z) \big)
    = T_{\ell_0 + 1} \big( \alpha_{\ell_0} (T_{\ell_0} x) \big) \, .
\end{align*}
Recall that this holds for all $x \in \R^{N_{\ell_0 - 1}}$. From this,
it is not hard to see $\Realization (\Phi) = \Realization (\Phi')$ for the network
\(
  \Phi '
  := \big( (S_{1}, \beta_{1}), \dots, (S_{K}, \beta_{K}) \big)
  \in \NNsymbol_{\infty, K,\infty}^{\varrho,d,k}
  \subset \NNsymbol_{\infty, L,\infty}^{\varrho,d,k}
\).
Note that $\Phi'$ is a strict network if $\Phi$ is strict.
Finally, directly from the definition of $\Phi'$, we see
$W(\Phi') \leq W(\Phi) \leq W$, so that
$\Phi' \in \NNsymbol_{W, L,\infty}^{\varrho,d,k}$.
Also, $N(\Phi') = N(\Phi) - 1 < N(\Phi)$, as desired.
\hfill$\square$

\ifarxiv
\subsection{Proof of Lemma~\ref{lem:DeepeningLemma}}
\label{app:PfDeepeningLemma}

Write $\Phi = \big( (T_1,\alpha_1),\dots,(T_{L},\alpha_{L}) \big)$ with $L = L(\Phi)$.
If $L_0 = 0$, we can simply choose $\Psi = \Phi$.
Thus, let us assume $L_0 > 0$, and distinguish two cases:

\textbf{Case 1}: If $k \leq d$, so that $c = k$, set
\[
    \Psi
    := \Big(
         (T_1, \alpha_1),
         \dots,
         (T_{L}, \alpha_L),
         \underbrace{(\identity_{\R^k}, \identity_{\R^k}),
         \dots,
         (\identity_{\R^k}, \identity_{\R^k})}_{L_0 \text{ terms}}
       \Big) ,
\]
and note that the affine map $T := \identity_{\R^{k}}$ satisfies $\|T\|_{\ell^{0}} = k=c$,
and hence $W(\Psi) = W(\Phi) + c \, L_0$.
Furthermore, $\Realization (\Psi) = \Realization (\Phi)$,
$L(\Psi) = L(\Phi) + L_0$, and $N(\Psi) = N(\Phi) + c L_0$.
Here we used crucially that the definition of \emph{generalized} neural
networks allows us to use the identity as the activation function for
some neurons.

\textbf{Case 2}: If $d < k$, so that $c = d$, the proof proceeds as in
the previous case, but with
\begin{flalign*}
    &&
    \Psi
    := \Big(
         \underbrace{(\identity_{\R^d}, \identity_{\R^d}),
         \dots,
         (\identity_{\R^d}, \identity_{\R^d})}_{L_0 \text{ terms}} \, ,
         (T_1, \alpha_1),
         \dots,
         (T_{L}, \alpha_{L})
       \Big) .
    && \square
\end{flalign*}
\else
\fi


\ifarxiv
\subsection{Proof of Lemma~\ref{lem:SummationLemma}}
\label{app:PfSummationLemma}


For the proof of the first part,
denoting $\Phi = \big( (T_{1}, \alpha_{1}), \dots, (T_L, \alpha_{L}) \big)$, we set
\(
  \Psi
  := \big( (T_{1}, \alpha_{1}), \dots, (c \cdot T_L, \alpha_{L}) \big)
\).
By Definition~\ref{defn:NeuralNetworks} we have $\alpha_{L} = \identity_{\R^{k}}$,
hence one easily sees $\Realization(\Psi) = c \cdot \Realization(\Phi)$.
If $\Phi$ is strict, then so is $\Psi$.
By construction $\Phi$ and $\Psi$ have the same number of layers and neurons, and $W(\Psi) \leq W(\Phi)$ with equality if $c \neq 0$.

\medskip{}

For the second and third part,
we proceed by induction, using two auxiliary claims.

\begin{lem}\label{lem:AuxConcatenationTwo}
  Let $\Psi_1 \in \NNsymbol^{\varrho,d,k_{1}}$ and $\Psi_2 \in \NNsymbol^{\varrho,d,k_{2}}$.
  There is a network $\Psi \in \NNsymbol^{\varrho,d,k_{1}+k_{2}}$
  with ${L(\Psi) = \max \{ L(\Psi_1), L(\Psi_2) \}}$ such that $\Realization(\Psi) = g$, where
  \(
    g : \R^{d} \to \R^{k_{1}+k_{2}},
        x \mapsto \big( \Realization(\Psi_{1})(x),\Realization(\Psi_{2})(x) \big)
  \).
  Furthermore, setting $c := \min \big\{ d,\max \{ k_{1},k_{2} \} \big\}$,
  $\Psi$ can be chosen to satisfy
  \begin{align*}
    W(\Psi) & \leq W(\Psi_{1}) + W(\Psi_{2}) + c \cdot |L(\Psi_2) - L(\Psi_1)| \\
    N(\Psi) & \leq N(\Psi_1) + N(\Psi_2) + c \cdot |L(\Psi_2) - L(\Psi_1)| \, .\qedhere
  \end{align*}
\end{lem}

\begin{lem}\label{lem:AuxSummationTwo}
  Let $\Psi_{1}, \Psi_{2} \in \NNsymbol^{\varrho,d,k}$.
  There is  $\Psi \in \NNsymbol^{\varrho,d,k}$ with $L(\Psi) = \max \{ L(\Psi_1), L(\Psi_2) \}$
  such that $\Realization(\Psi) = \Realization(\Psi_{1}) + \Realization(\Psi_{2})$ and,    with $c = \min\{d,k\}$,
  \begin{align*}
    W(\Psi) & \leq W(\Psi_{1}) + W(\Psi_{2}) + c \cdot |L(\Psi_{2})-L(\Psi_{1})| \\
    N(\Psi) & \leq N(\Psi_1) + N(\Psi_2) + c \cdot |L(\Psi_2) - L(\Psi_1)| \, . \qedhere
  \end{align*}
\end{lem}
\begin{proof}[Proof of Lemmas~\ref{lem:AuxConcatenationTwo} and \ref{lem:AuxSummationTwo}]
  Set $L := \max \{ L(\Psi_1), L(\Psi_2) \}$ and $L_i := L(\Psi_i)$ for $i \in \{1,2\}$.
  By Lemma~\ref{lem:DeepeningLemma} applied to $\Psi_i$ and $L_0 = L - L_i \in \N_0$,
  we get for each $i \in \{1,2\}$ a network ${\Psi_i ' \in \NNsymbol^{\varrho,d,k_{i}}}$ with
  $\Realization(\Psi_i ') = \Realization(\Psi_i)$ and such that $L(\Psi_i ') = L$,
  as well as $W(\Psi_i ') \leq W(\Psi_i) + c (L - L_i)$
  and furthermore $N(\Psi_i') \leq N(\Psi_i) + c (L - L_i)$.
  By choice of $L$, we have $(L - L_1) + (L - L_2) = |L_1 - L_2|$,
  whence $W(\Psi_1 ') + W(\Psi_2 ') \leq W(\Psi_1) + W(\Psi_2) + c \, |L_1 - L_2|$,
  and $N(\Psi_1 ') + N(\Psi_2 ') \leq N(\Psi_1) + N(\Psi_2) + c \, |L_1 - L_2|$.

  \medskip{}

  First we deal with the pathological case $L = 1$.
  In this case, each $\Psi'_i$ is of the form $\Psi'_i = \big( ( T_i, \identity_{\R^k}) \big)$,
  with $T_i : \R^d \to \R^k$ an affine-linear map.
  For proving Lemma~\ref{lem:AuxConcatenationTwo},
  we set $\Psi := \big( (T,\identity_{\R^{k_1+k_2}}) \big)$
  with the affine-linear map $T:\R^{d} \to \R^{k_1+k_2},\ x \mapsto \big( T_1(x),T_2(x) \big)$,
  so that $\Realization(\Psi) = g$.
  For proving Lemma~\ref{lem:AuxSummationTwo}, we set $\Psi := \big( (T, \identity_{\R^k}) \big)$
  with $T = T_1+T_2$, so that
  \(
    \Realization(\Psi)
    = T_1 + T_2
    = \Realization(\Psi'_1) + \Realization(\Psi'_2)
    = \Realization(\Psi_1) + \Realization(\Psi_2)
  \).
  Finally, we see for both cases that $N(\Psi) = 0 = N(\Psi'_1) + N(\Psi'_2)$ and
  \[
    W(\Psi)
    =    \left\| T \right\|_{\ell^0}
    \leq \left\| T_1 \right\|_{\ell^0} + \left\| T_2 \right\|_{\ell^0}
    = W(\Psi'_1) + W(\Psi'_2) \, .
  \]
  This establishes the result for the case $L=1$.

  \medskip{}

  For $L > 1$, write $\Psi_1 ' = \big( (T_1, \alpha_1), \dots, (T_L, \alpha_L) \big)$ and
  $\Psi_2 ' = \big( (S_1, \beta_1), \dots, (S_L, \beta_L) \big)$ with affine-linear maps
  $T_\ell : \R^{N_{\ell - 1}} \to \R^{N_\ell}$ and
  $S_\ell : \R^{M_{\ell-1}} \to \R^{M_\ell}$ for $\ell \in \FirstN{L}$.
  Let us define $\theta_\ell := \alpha_\ell \otimes \beta_\ell$ for $\ell \in \FirstN{L}$---except
  for $\ell = L$ when proving Lemma~\ref{lem:AuxSummationTwo},
  in which case we set $\theta_L := \identity_{\R^k}$.
  Next, set
  \[
    R_1 : \R^d \to \R^{N_1 + M_1}, x \mapsto (T_1 x , S_1 x)
    \quad \text{and} \quad
    R_\ell : \R^{N_{\ell-1} + M_{\ell-1}} \to \R^{N_\ell + M_\ell},
             (x,y) \mapsto (T_\ell \, x , S_\ell \, y)
  \]
  for $2 \leq \ell \leq L$---except if $\ell = L$ when proving Lemma~\ref{lem:AuxSummationTwo}.
  In this latter case, we instead define $R_L$ as
  \(
    {R_L : \R^{N_{L-1} + M_{L-1}} \to \R^{k}, (x,y) \mapsto T_L \, x + S_L \, y}
  \).
  Finally set $\Psi := \big( (R_1, \theta_1), \dots, (R_L, \theta_L) \big)$.

  When proving Lemma~\ref{lem:AuxConcatenationTwo},
  it is straightforward to verify that $\Psi$ satisfies
  \[
      \Realization (\Psi) (x)
    = \big( \Realization (\Psi_1') (x), \Realization(\Psi_2')(x) \big)
    = \big( \Realization (\Psi_1) (x), \Realization(\Psi_2)(x) \big)
    = g(x)
    \qquad \forall \, x \in \R^d \, .
  \]
  Similarly, when proving Lemma~\ref{lem:AuxSummationTwo}, one can easily check that
  \(
    \Realization(\Psi)
    = \Realization(\Psi'_1) + \Realization(\Psi'_2)
    = \Realization(\Psi_1) + \Realization(\Psi_2)
  \).

  Further, for arbitrary $\ell \in \FirstN{L}$, we have
  $\|R_\ell\|_{\ell^0} \leq \|T_\ell\|_{\ell^0}+ \|S_\ell\|_{\ell^0}$
  so that
  \begin{equation*}
    W(\Psi)
    = \sum_{\ell=1}^{L}
        \|R_\ell\|_{\ell^0}
    \leq \sum_{\ell=1}^{L}
           (\|T_\ell\|_{\ell^0}+ \|S_\ell\|_{\ell^0})
    =   W(\Psi_1 ')+W(\Psi'_2) \, .
  \end{equation*}
  Finally, $N(\Psi) = \sum_{\ell=1}^{L-1} (N_{\ell}+M_{\ell}) = N(\Psi'_1)+N(\Psi'_2)$.
  Given the estimates for $W(\Psi_1') + W(\Psi_2')$ and $N(\Psi_1') + N(\Psi_2')$
  stated at the beginning of the proof, this yields the claim.
\end{proof}

Let us now return to the proof of Parts 2 and 3 of Lemma~\ref{lem:SummationLemma}.
Set $f_{i} := \Realization(\Phi_{i})$ and $L_i := L(\Phi_i)$.
We first show that we can without loss of generality assume $L_1 \leq \dots \leq L_n$.
To see this, note that there is a permutation $\sigma \in S_n$ such that if we set
$\Gamma_j := \Phi_{\sigma(j)}$, then $L(\Gamma_1) \leq \dots \leq L(\Gamma_n)$.
Furthermore, $\sum_{j=1}^n \Realization(\Gamma_j) = \sum_{j=1}^n \Realization(\Phi_j)$.
Finally, there is a permutation matrix $P \in \GL(\R^d)$ such that
\[
  P \circ \big( \Realization(\Gamma_1), \dots, \Realization(\Gamma_n) \big)
  = \big( \Realization(\Phi_1), \dots, \Realization(\Phi_n) \big)
  = (f_1, \dots, f_n)
  = g \, .
\]
Since the permutation matrix $P$ has exactly one non-zero entry per row and column,
we have $\|P\|_{\ell^{0,\infty}} = 1$ in the notation of Equation~\eqref{eq:DefL0MixedNorms}.
Therefore, the first part of Lemma~\ref{lem:NetworkCalculus} (which will be proven independently)
shows that $g \in \NNreal^{\varrho,d,K}_{W,L,N}$, provided that
$\big( \Realization(\Gamma_1), \dots, \Realization(\Gamma_n) \big) \in \NNreal^{\varrho,d,K}_{W,L,N}$.
These considerations show that we can assume $L(\Phi_1) \leq \dots \leq L(\Phi_n)$
without loss of generality.

We now prove the following claim by induction on $j \in \FirstN{n}$:
There is $\Theta_{j} \in \NNsymbol^{\varrho,d,K_{j}}$ satisfying
$W(\Theta_{j}) \leq \sum_{i=1}^{j} W(\Phi_{i}) + c \, (L_j - L_1)$,
and
$N(\Theta_{j}) = \sum_{i=1}^{j} N(\Phi_{i}) + c \, (L_j - L_1)$,
as well as ${L(\Theta_{j}) = L_j}$,
and such that $\Realization(\Theta_j) = g_{j} := \sum_{i=1}^{j} f_{i}$
and $K_{j} := k$ for the summation,
respectively such that ${\Realization(\Theta_{j}) = g_{j} := (f_1, \dots, f_j)}$
and $K_{j} := \sum_{i=1}^{j} k_{i}$ for the cartesian product.
Here, $c$ is as in the corresponding claim of Lemma~\ref{lem:SummationLemma}.

Specializing to $j=n$ then yields the conclusion of Lemma~\ref{lem:SummationLemma}.

\medskip{}

We now proceed to the induction.
The claim trivially holds for $j=1$---just take $\Theta_1 = \Phi_1$.
Assuming that the claim holds for some $j \in \FirstN{n-1}$, we define $\Psi_{1} := \Theta_{j}$
and $\Psi_{2} := \Phi_{j+1}$.
Note that $L(\Psi_1) = L(\Theta_j) = L_j \leq L_{j+1} = L(\Psi_2)$.
For the summation, by Lemma~\ref{lem:AuxSummationTwo} there is a network
$\Psi \in \NNsymbol^{\varrho,d,k}$ with $L(\Psi) = L_{j+1}$ and
\(
  \Realization(\Psi)
  =  \Realization(\Psi_{1}) + \Realization(\Psi_{2})
  = \Realization(\Theta_{j}) +  \Realization(\Phi_{j+1})
  = g_{j}+ f_{j+1}
  = g_{j+1}
\),
and such that
\[
  W(\Psi)
  \leq W(\Psi_{1}) + W(\Psi_{2}) + c' \cdot |L(\Psi_2) - L(\Psi_1)|
  \leq W(\Theta_{j}) + W(\Phi_{j+1}) + c' \cdot (L_{j+1} - L_j)
\]
and likewise
$N(\Psi) \leq N(\Theta_j) + N(\Phi_{j+1}) + c' \cdot (L_{j+1} - L_j)$, where $c' = \min\{d,k\} = c$.
For the cartesian product, Lemma~\ref{lem:AuxConcatenationTwo} yields a network
$\Psi \in \NNsymbol^{\varrho,d,K_{j}+k_{j+1}} = \NNsymbol^{\varrho,d,K_{j+1}}$ satisfying
\[
  \Realization(\Psi)
  = \big( \Realization(\Psi_{1}),\Realization(\Psi_{2}) \big)
  = \big( \Realization(\Theta_{j}), \Realization(\Phi_{j+1}) \big)
  = g_{j+1}
\]
and such that, setting
$c' := \min \big\{ d, \max \{ K_{j}, k_{j+1} \} \big\} \leq \min\{d,K-1\} = c$,
we have
\[
  W(\Psi)
  \leq W(\Psi_{1})+W(\Psi_{2}) + c' \cdot | L(\Psi_{2}) - L(\Psi_{1}) |
  = W(\Theta_{j}) + W(\Phi_{j+1}) + c' \cdot (L_{j+1} - L_j)
\]
and $N(\Psi) \leq N(\Theta_j) + N(\Phi_{j+1}) + c' \cdot (L_{j+1} - L_j)$.

With $\Theta_{j+1} := \Psi$ we get $\Realization(\Theta_{j+1}) = g_{j+1}$,
$L(\Theta_{j+1}) = L_{j+1}$ and, by the induction hypothesis,
\[
  W(\Theta_{j+1})
    \leq \sum_{i=1}^{j} W(\Phi_{i})
         + c \, (L_j - L_1)
         + W(\Phi_{j+1})
         + c \, (L_{j+1} - L_j)
    =    \sum_{i=1}^{j+1} W(\Phi_{i}) + c \, (L_{j+1} - L_1) \, .
\]
Similarly,
\(
  N(\Theta_{j+1})
  \leq \sum_{i=1}^{j+1}
         N(\Phi_{i})
       + c \cdot (L_{j+1} - L_1)
\).
This completes the induction and the proof.
\hfill$\square$

\else
\fi

\subsection{Proof of Lemma~\ref{lem:NetworkCalculus}}
\label{app:PfNetworkCalculus}

\ifarxiv
We prove each part of the lemma individually.
\fi
\ifarxiv
\medskip{}

\noindent
\textbf{Part~\eqref{enu:Compos}:}
Let
\(
  \Phi_1
  = \big( (T_1, \alpha_1), \dots, (T_{L_1}, \alpha_{L_1}) \big)
  \in \NNsymbol^{\varrho, d, d_1}
\)
and
\({
  \Phi_2
  = \big( (S_1, \beta_1), \dots, (S_{L_2}, \beta_{L_2}) \big)
  \in \NNsymbol^{\varrho, d_1, d_2}
}\)
Define
\[
  \Psi
  :=  \big(
        (T_1, \alpha_1),
        \dots,
        (T_{L_1}, \alpha_{L_1}),
        (S_1, \beta_1),
        \dots,
        (S_{L_2}, \beta_{L_2})
      \big) \, .
\]
We emphasize that $\Psi$ is indeed a \emph{generalized} $\varrho$-network,
since all $T_\ell$ and all $S_\ell$ are affine-linear (with ``fitting'' dimensions),
and since all $\alpha_\ell$ and all $\beta_\ell$ are $\otimes$-products of $\varrho$
and $\identity_{\R}$, with $\beta_{L_2} = \identity_{\R^{d_2}}$.
Furthermore, we clearly have $L(\Psi) = L_1 + L_2 = L(\Phi_1) + L(\Phi_2)$, and
\[
  W(\Psi)
  =    \sum_{\ell=1}^{L_1}
         \|T_\ell\|_{\ell^0}
       + \sum_{\ell'=1}^{L_2}
           \|S_{\ell'}\|_{\ell^0}
  =    W(\Phi_1) + W(\Phi_2).
\]
Clearly, $N(\Psi) = N(\Phi_1) + d_1 + N(\Phi_2)$.
Finally, the property $\Realization(\Psi) = \Realization(\Phi_2) \circ \Realization(\Phi_1)$ is a
direct consequence of the definition of the realization of neural networks.

\medskip{}

\noindent
\textbf{Part~\eqref{enu:PrePostAffine}:}
\else
\fi
Let $\Phi = \big( (T_{1},\alpha_{1}), \dots, (T_L,\alpha_{L}) \big) \in \NNsymbol^{\varrho,d,k}$.
We give the proof for $Q \circ \Realization(\Phi)$,
since the proof for $\Realization(\Phi) \circ P$ is similar but simpler;
the general statement in the lemma then follows from the identity
\(
  Q \circ \Realization(\Phi) \circ P
  = (Q \circ \Realization(\Phi)) \circ P
  = \Realization(\Psi_1) \circ P
\).

We first treat the special case $\|Q\|_{\ell^{0,\infty}}=0$
which implies $\|Q\|_{\ell^{0}} = 0$, and hence
$Q \circ \Realization(\Phi) \equiv c$ for some $c \in \R^{k_1}$.
Choose $N_0,\dots,N_L$ such that $T_\ell : \R^{N_{\ell - 1}} \to \R^{N_\ell}$
for $\ell \in \FirstN{L}$, and define $S_\ell : \R^{N_{\ell - 1}} \to \R^{N_\ell}, x \mapsto 0$
for $\ell \in \FirstN{L-1}$ and $S_L : \R^{N_{L-1}} \to \R^{k_1}, x \mapsto c$.
It is then not hard to see that the network $\Psi := \big( (S_1,\alpha_1),\dots,(S_L,\alpha_L) \big)$
satisfies $L(\Psi) = L(\Phi)$ and $N(\Psi) = N(\Phi)$, as well as
$W(\Psi) = 0$ and $\Realization(\Psi) \equiv c = Q \circ \Realization(\Phi)$.

We now consider the case $\|Q\|_{\ell^{0,\infty}} \geq 1$.
Define $U_{\ell} := T_{\ell}$ for $\ell \in \FirstN{L-1}$
and ${U_{L} := Q \circ T_{L}}$.
By Definition~\ref{defn:NeuralNetworks} we have $\alpha_{L} = \identity_{\R^{k}}$, whence
\({
  \Psi
  := \big( (U_{1},\alpha_{1}), \dots, (U_{L-1}, \alpha_{L-1}), (U_L, \identity_{\R^{k_{1}}}) \big)
     \in \NNsymbol_{\infty,L,N(\Phi)}^{\varrho,d,k_1}
}\)
satisfies $\Realization(\Psi) = Q \circ  \Realization(\Phi)$.
To control $W(\Psi)$, we use the following lemma.
The proof is slightly deferred.

\begin{lem}\label{lem:MatrixSparsityForComposition}
Let $p,q,r \in \N$ be arbitrary.
\begin{enumerate}
  \item For arbitrary affine-linear maps $T : \R^p \to \R^q$ and $S : \R^q \to \R^r$, we have
        \[
          \| S \circ T \|_{\ell^0}
          \leq \| S \|_{\ell^{0,\infty}} \cdot \| T \|_{\ell^0}
          \quad \text{and} \quad
          \| S \circ T \|_{\ell^0}
          \leq \| S \|_{\ell^0} \cdot \| T \|_{\ell^{0,\infty}_{\ast}} \, .
        \]

  \item For affine-linear maps $T_1, \dots, T_n$, we have
        $\|T_1 \otimes \cdots \otimes T_n\|_{\ell^0} \leq \sum_{i=1}^n \|T_i\|_{\ell^0}$,
        as well as
        \[
          \quad \quad
          \| T_1 \otimes \cdots \otimes T_n \|_{\ell^{0,\infty}}
          \leq \max_{i \in \FirstN{n}} \| T_i \|_{\ell^{0,\infty}}
          \quad \text{and} \quad
          \| T_1 \otimes \cdots \otimes T_n \|_{\ell^{0,\infty}_{\ast}}
          \leq \max_{i \in \FirstN{n}} \| T_i \|_{\ell^{0,\infty}_{\ast}} \, .
          \qedhere
        \]
\end{enumerate}
\end{lem}

Let us continue with the proof from above.
By definition,
\(
  \|U_{\ell}\|_{\ell^{0}}
  = \|T_{\ell}\|_{\ell^{0}}
  \leq \| Q \|_{\ell^{0,\infty}} \cdot \|T_{\ell}\|_{\ell^{0}}
\)
for $\ell \in \FirstN{L-1}$.
By Lemma~\ref{lem:MatrixSparsityForComposition} we also have
$\|U_{L}\|_{\ell^{0}} \leq \| Q \|_{\ell^{0,\infty}} \cdot \|T_{L}\|_{\ell^{0}}$, and hence
\[
    W(\Psi)
    = \sum_{\ell=1}^L
        \|U_\ell\|_{\ell^0}
    \leq \| Q \|_{\ell^{0,\infty}} \, \sum_{\ell=1}^L \|T_\ell\|_{\ell^0}
    =    \| Q \|_{\ell^{0,\infty}} \cdot W(\Phi).
\]
Finally, if $\Phi$ is strict, then $\Psi$ is strict as well;
thus, the claim also holds with $\SNNreal$ instead of $\NNreal$.
\ifarxiv
\medskip{}

\noindent
\textbf{Part~(\ref{enu:ComposLessDepth}):}
Let
\(
  \Phi_1
  = \big( (T_1, \alpha_1), \dots, (T_L, \alpha_L) \big)
  \in \NNsymbol^{\varrho, d, d_1}
\)
and
\(
  \Phi_2
  = \big( (S_1, \beta_1), \dots, (S_K, \beta_K) \big)
  \in \NNsymbol^{\varrho, d_1, d_2}
\).

We distinguish two cases:
First, if $L = 1$, then $\Realization(\Phi_1) = T_1$.
Since $T_1 : \R^d \to \R^{d_1}$, this implies $\|T_1\|_{\ell^{0,\infty}_\ast} \leq d$.
Thus, Part~(\ref{enu:PrePostAffine}) shows that
\[
  \Realization(\Phi_2) \circ \Realization(\Phi_1)
  = \Realization(\Phi_2) \circ T_1
  \in \NNreal_{d \cdot W(\Phi_2), K, N(\Phi_2)}^{\varrho,d,d_2}
  \subset \NNreal_{W(\Phi_1) + N \cdot W(\Phi_2), L + K - 1, N(\Phi_1) + N(\Phi_2)}^{\varrho,d,d_2},
\]
where $N := \max \{ N(\Phi_1), d\}$.

Let us now assume that $L > 1$.
In this case, define
\[
  \Psi := \big(
            (T_1, \alpha_1),
            \dots,
            (T_{L-1},\alpha_{L-1}),
            (S_1 \circ T_L, \beta_1),
            (S_2, \beta_2)
            \dots,
            (S_K,\beta_K)
          \big) \, .
\]
It is not hard to see that $N(\Psi) \leq N(\Phi_1) + N(\Phi_2)$ and---because of
$\alpha_L = \identity_{\R^{d_1}}$---that
\[
  \Realization(\Psi)
  = (\beta_K \circ S_K)
    \circ \cdots
    \circ (\beta_1 \circ S_1)
    \circ (\alpha_L \circ T_L)
    \circ \cdots
    \circ (\alpha_1 \circ T_1)
  = \Realization(\Phi_2) \circ \Realization(\Phi_1).
\]
Note $T_\ell : \R^{M_{\ell - 1}} \to \R^{M_\ell}$ for certain $M_0,\dots,M_L \in \N$.
Since $L > 1$, we have $M_{L - 1} \leq N(\Phi_1) \leq N$.
Furthermore, since $T_L : \R^{M_{L-1}} \to \R^{M_L}$,
we get $\|T_L\|_{\ell^{0,\infty}_\ast} \leq M_{L-1} \leq N$ directly from the definition.
Thus, Lemma~\ref{lem:MatrixSparsityForComposition} shows
\(
  \|S_1 \circ T_L\|_{\ell^0}
  \leq \|S_1\|_{\ell^0} \cdot \|T_L\|_{\ell^{0,\infty}_\ast}
  \leq N \cdot \|S_1\|_{\ell^0}
\).
Therefore, and since $N \geq 1$, we see that
\[
  W(\Psi)
  = \sum_{\ell=1}^{L-1}
      \|T_\ell\|_{\ell^0}
    + \|S_1 \circ T_L\|_{\ell^0}
    + \sum_{\ell=2}^{K}
        \|S_\ell\|_{\ell^0}
  \leq W(\Phi_1)
       + N \cdot \|S_1\|_{\ell^0}
       + N \cdot \sum_{\ell=2}^{K}
                   \|S_\ell\|_{\ell^0}
  =    W(\Phi_1) + N \cdot W(\Phi_2).
\]
Finally, note that if $\Phi_1,\Phi_2$ are strict networks, then so is $\Psi$.
\else
\fi
\hfill$\square$



\begin{proof}[Proof of Lemma~\ref{lem:MatrixSparsityForComposition}]
\label{app:PfTechnicalA}

The stated estimates follow directly from the definitions by direct computations
and are thus left to the reader.
For instance, the main observation for proving that
$\| B A \|_{\ell^0} \leq \| B \|_{\ell^{0,\infty}} \cdot \| A \|_{\ell^0}$
is that
\[
  \| A x \|_{\ell^0}
  = \left\| \smash{\sum_{i=1}^p}\vphantom{\sum} x_i \cdot A e_i \right\|_{\ell^0}
  \leq \sum_{i \, : \, x_i \neq 0} \| A e_i \|_{\ell^0}
  \leq \| x \|_{\ell^0} \cdot \| A \|_{\ell^{0,\infty}}
  \quad \text{for} \quad A \in \R^{q \times p} \text{ and } x \in \R^p.
  \qedhere
\]
\end{proof}

\subsection{Proof of Lemma~\ref{lem:RecursiveNNsets}}
\label{app:RecursiveNNsets}

We start with an auxiliary lemma.

\begin{lem}\label{lem:TensorActivation}
Consider two activation functions $\varrho,\sigma$ such that $\sigma = \Realization(\Psi_{\sigma})$
for some
\(
  \Psi_{\sigma} \in \NNsymbol^{\varrho,1,1}_{w,\ell,m}
\)
with $L(\Psi_{\sigma}) = \ell \in \N$, $w \in \N_{0}$, $m \in \N$.
Furthermore, assume that $\sigma \not\equiv \mathrm{const}$.

Then, for any $d \in \N$ and $\alpha_{i} \in \{\identity_{\R},\sigma\}$,
$1 \leq i \leq d$ we have $\alpha_{1} \otimes \cdots \otimes \alpha_{d} = \Realization(\Phi)$
for some network
\[
  \Phi
  = \big( (U_1, \gamma_1), \ldots, (U_{\ell}, \gamma_{\ell}) \big)
  \in \NNsymbol^{\varrho,d,d}_{dw,\ell,dm}
\]
satisfying $\|U_1\|_{\ell^{0,\infty}} \leq m$,
$\|U_1\|_{\ell^{0,\infty}_{\ast}} \leq 1$, $\|U_{\ell}\|_{\ell^{0,\infty}} \leq 1$,
and $\|U_{\ell}\|_{\ell^{0,\infty}_{\ast}} \leq m$.

If $\Psi_\sigma$ is a strict network and $\alpha_i = \sigma$ for all $i$,
then $\Phi$ can be chosen to be a strict network.
\end{lem}

\begin{proof}[Proof of Lemma~\ref{lem:TensorActivation}]
First we show that any $\alpha \in \{\identity_{\R},\sigma\}$
satisfies $\alpha = \Realization(\Psi_{\alpha})$ for some network
\[
  \Psi_{\alpha}
  = \big(
      (U_{1}^{\alpha}, \gamma_{1}^{\alpha}),
      \ldots,
      (U_{\ell}^{\alpha}, \gamma_{\ell}^{\alpha})
    \big)
  \in \NNsymbol^{\varrho,1,1}_{w,\ell,m}
\]
with $\|U_1^{\alpha}\|_{\ell^{0,\infty}} \leq m$,
$\|U_1^{\alpha}\|_{\ell^{0,\infty}_{\ast}} \leq 1$,
$\|U_{\ell}^{\alpha}\|_{\ell^{0,\infty}} \leq 1$
and $\|U_{\ell}^{\alpha}\|_{\ell^{0,\infty}_{\ast}} \leq m$.

For $\alpha = \sigma$ we have $\alpha = \Realization(\Psi_{\sigma})$
where $\Psi_\sigma$ is of the form
\(
  \Psi_{\sigma}
  = \big( (T_{1}, \beta_{1}), \ldots, (T_{\ell}, \beta_{\ell}) \big)
  \in \NNsymbol^{\varrho,1,1}_{w,\ell,m}
\).
For $\alpha = \identity_{\R}$, observe that $\alpha = \Realization(\Psi_{\identity_{\R}})$ with
\[
  \Psi_{\identity_{\R}}
  := \big( (T_1', \identity_{\R}), \ldots, (T_\ell ', \identity_{\R}) \big)
  := \big( (\identity_{\R}, \identity_{\R}), \ldots, (\identity_{\R}, \identity_{\R}) \big),
\]
where it is easy to see that $N(\Psi_{\identity_{\R}}) = \ell - 1 \leq m$
and $W(\Psi_{\identity_{\R}}) = \ell \leq w$.
Indeed, Equation~\eqref{eq:LayersBoundedByNeurons} shows that
$\ell = L(\Psi_\sigma) \leq 1 + N(\Psi_\sigma) \leq 1 + m$.
On the other hand, since $\sigma \not\equiv \mathrm{const}$,
Corollary~\ref{cor:BoundingConnectionsWithLayers2} shows that
$\ell = L(\Psi_\sigma) \leq W(\Psi_\sigma) \leq w$.

Denoting by $N_{i}$ the number of neurons in the $i$-th layer of $\Psi_{\sigma}$
(where layer $0$ is the input layer, and layer $\ell$ the output layer),
we get because of $\Psi_{\sigma} \in \NNsymbol^{\varrho,1,1}_{w,\ell,m}$ that $N_{i} \leq m$
for $1 \leq i \leq L-1$.
Furthermore, since $T_{1}: \R \to \R^{N_{1}}$,
we have $\|T_{1}\|_{\ell^{0,\infty}} \leq N_{1} \leq m$ and $\|T_{1}\|_{\ell^{0,\infty}_{\ast}} \leq 1$.
Similarly, as $T_{\ell}: \R^{N_{\ell-1}} \to \R$ we have $\|T_{\ell}\|_{\ell^{0,\infty}} \leq 1$
and $\|T_{\ell}\|_{\ell^{0,\infty}_{\ast}} \leq m$.
The same bounds trivially hold for $T'_{1}$ and $T'_{\ell}$.

We now prove the claim of the lemma by induction on $d$.
The result is trivial for $d=1$ using $\Phi = \Psi_{\alpha_{1}}$.
Assuming it is true for $d \in \N$, we prove it for $d+1$.

Define $\alpha = \alpha_{1} \otimes \cdots \otimes \alpha_{d}$ and
$\overline{\alpha} = \alpha_{1} \otimes \cdots \otimes \alpha_{d+1} = \alpha \otimes \alpha_{d+1}$.
By induction, there are networks
\(
  \Psi_{1}
  = \big( (V_{1},\lambda_{1}),\ldots,(V_{\ell},\lambda_{\ell}) \big)
  \in \NNsymbol^{\varrho,d,d}_{dw,\ell,dm}
\)
and
\(
  \Psi_{2}
  = \big( (W_{1},\mu_{1}),\ldots,(W_{\ell},\mu_{\ell}) \big)
  \in \NNsymbol^{\varrho,1,1}_{w,\ell,m}
\)
such that $\Realization(\Psi_1)=\alpha$ and $\Realization(\Psi_2)=\alpha_{d+1}$
and such that $\|V_1\|_{\ell^{0,\infty}} \leq m$, $\|V_1\|_{\ell^{0,\infty}_\ast} \leq 1$,
$\|V_\ell\|_{\ell^{0,\infty}} \leq 1$, and $\|V_\ell\|_{\ell^{0,\infty}_\ast} \leq m$,
and likewise for $W_1$ instead of $V_1$ and $W_\ell$ instead of $V_\ell$.

Define $U_{i} := V_{i} \otimes W_{i}$ and
$\gamma_{i}:= \lambda_{i} \otimes \mu_{i}$ for $1 \leq i \leq \ell$,
and $\Phi := \big( (U_{1},\gamma_{1}),\ldots,(U_{\ell},\gamma_{\ell}) \big)$.
One can check that $\Realization(\Phi) = \overline{\alpha}$.
Moreover, Lemma~\ref{lem:MatrixSparsityForComposition} shows that
$\|U_{i}\|_{\ell^0} = \|V_{i}\|_{\ell^0}+\|W_{i}\|_{\ell^0}$ for $1 \leq i \leq \ell$,
whence $W(\Phi) = W(\Psi_1) + W(\Psi_2) \leq dw+d = (d+1)w$
and similarly $N(\Phi) = N(\Psi_{1}) + N(\Psi_{2}) \leq (d+1)m$.
Finally, Lemma~\ref{lem:MatrixSparsityForComposition} shows that
\begin{align*}
  & \|U_{1}\|_{\ell^{0,\infty}}
    \leq \max
           \big\{
             \|V_{1}\|_{\ell^{0,\infty}},
             \|W_{1}\|_{\ell^{0,\infty}}
           \big\}
    \leq m,
  \quad
  && \|U_{1}\|_{\ell^{0,\infty}_{\ast}}
  \leq \max
         \big\{
           \|V_{1}\|_{\ell^{0,\infty}_{\ast}},
           \|W_{1}\|_{\ell^{0,\infty}_{\ast}}
         \big\}
  \leq 1, \\
  & \|U_{\ell}\|_{\ell^{0,\infty}}
    \leq \max
           \big\{
             \|V_{\ell}\|_{\ell^{0,\infty}},
             \|W_{\ell}\|_{\ell^{0,\infty}}
           \big\}
    \leq 1,
  \quad
  && \|U_{\ell}\|_{\ell^{0,\infty}_{\ast}}
  \leq \max
         \big\{
           \|V_{\ell}\|_{\ell^{0,\infty}_{\ast}},
           \|W_{\ell}\|_{\ell^{0,\infty}_{\ast}}
         \big\}
  \leq m \, .
\end{align*}
Clearly, if $\Psi_\sigma$ is strict, and if $\alpha_i = \sigma$ for all $i$,
then the same induction shows that $\Phi$ can be chosen to be a strict network.
\end{proof}

\begin{proof}[Proof of Lemma~\ref{lem:RecursiveNNsets}]
For the first statement with $\ell=2$ consider $f = \Realization(\Psi)$ for some
\[
  \Psi
  = \big( (S_{1}, \alpha_{1}), \ldots, (S_{K-1}, \alpha_{K-1}), (S_K, \identity_{\R^{k}}) \big)
  \in \NNsymbol^{\sigma,d,k}_{W,L,N}.
\]
In case of $K = 1$, we trivially have $\Psi \in \NNsymbol_{W,L,N}^{\varrho,d,k}$,
so that we can assume $K \geq 2$ in the following.

Denoting by $N_{i}$ the number of neurons at the $i$-th layer of $\Psi$,
Lemma~\ref{lem:TensorActivation} yields for each ${i \in \FirstN{K-1}}$
a network
\(
  \Phi_{i}
  =   \big( (U_{1}^{i}, \gamma_{i}), (U_{2}^{i}, \identity_{\R^{N_{i}}}) \big)
  \in \NNsymbol^{\varrho,N_{i},N_{i}}_{N_{i}w,2,N_{i}m}
\)
satisfying $\alpha_i = \Realization(\Phi_{i})$ and
$\gamma_{i}: \R^{N(\Phi_{i})} \to \R^{N(\Phi_{i})}$ with $N(\Phi_{i}) \leq N_{i}m$
and finally $\|U_{1}^{i}\|_{\ell^{0,\infty}} \leq m$
and $\|U_{2}^{i}\|_{\ell^{0,\infty}_{\ast}} \leq m$.
With $T_{1} := U_{1}^{1} \circ S_{1}$, $T_{K} := S_{K} \circ U_{2}^{K-1}$,
$T_{i} := U_{1}^{i} \circ S_{i} \circ U_{2}^{i-1}$ for $2 \leq i \leq K-1$ and
\[
  \Phi
  := \big( (T_1, \gamma_1), \ldots, (T_{K-1},\gamma_{K-1}), (T_K, \identity_{\R^k}) \big) \, ,
\]
one can check that $f = \Realization(\Phi)$.

By Lemma~\ref{lem:MatrixSparsityForComposition},
\(
  \|T_{i}\|_{\ell^{0}}
  \leq \|U_{1}^{i}\|_{\ell^{0,\infty}}
       \|S_{i}\|_{\ell^{0}}
       \|U_{2}^{i-1}\|_{\ell^{0,\infty}_{\ast}}
  \leq m^{2}\|S_{i}\|_{\ell^{0}}
\)
for $2 \leq i \leq K-1$, and the same overall bound also holds for $i \in \{1,K\}$.
As a result we get $L(\Phi) = K \leq L$ as well as
\begin{align*}
  \frac{W(\Phi)}{m^2}
  =    \sum_{i=1}^{K} \frac{\|T_{i}\|_{\ell^{0}}}{m^2}
  \leq \sum_{i=1}^{K} \|S_{i}\|_{\ell^{0}}
  =    W(\Psi)
  \leq W
  \quad \text{and} \quad
  \frac{N(\Phi)}{m}
  =    \sum_{i=1}^{K-1} \frac{N(\Phi_i)}{m}
  \leq \sum_{i=1}^{K-1} N_{i}
  =    N(\Psi)
  \leq N .
\end{align*}

\medskip{}

For the second statement, we prove by induction on $L \in \N$
that $\NNreal_{W,L,N}^{\sigma,d,k} \subset \NNreal^{\varrho,d,k}_{mW + Nw , 1 + (L-1)\ell, N(1+m)}$.

For $L = 1$, it is easy to see $\NNreal_{W,1,N}^{\sigma,d,k} = \NNreal^{\varrho,d,k}_{W,1,N}$,
simply because on the last (and for $L=1$ only) layer, the activation function is always given
by $\identity_{\R^k}$.
Thus, the claim follows from the trivial inclusion
$\NNreal_{W,1,N}^{\varrho,d,k} \subset \NNreal^{\varrho,d,k}_{mW + Nw , 1, N(1+m)}$,
since $m \geq 1$.


Now, assuming the claim holds true for $L$, we prove it for $L+1$.
Consider $f \in \NNreal^{\sigma,d,k}_{W,L+1,N}$.
In case of ${f \in \NNreal^{\sigma,d,k}_{W,L,N}}$,
we get
\(
  f \in \NNreal^{\varrho,d,k}_{mW + Nw , 1 + (L-1)\ell, N(1+m)}
    \subset \NNreal^{\varrho,d,k}_{mW + Nw , 1 + ( (L+1)-1)\ell, N(1+m)}
\)
by the induction hypothesis.
In the remaining case where $f \notin \NNreal^{\sigma,d,k}_{W,L,N}$, there is a network
$\Psi \in \NNsymbol^{\sigma,d,k}_{W,L+1,N}$ of the form
\(
  \Psi
  =   \big( (S_{1},\alpha_{1}),\ldots,(S_L,\alpha_L),(S_{L+1},\identity_{\R^{k}}) \big)
\)
such that $f = \Realization(\Psi)$.
Observe that $S_{L+1}: \R^{\overline{k}} \to \R^{k}$ with $\overline{k} := N_L$
the number of neurons of the last hidden layer.
Defining
\(
  \Psi_{1}
  := \big(
       (S_{1}, \alpha_1),
       \ldots,
       (S_{L-1}, \alpha_{L-1}),
       (S_L, \identity_{\R^{\overline{k}}})
     \big),
\)
we have $\Psi_{1} \in \NNsymbol^{\sigma,d,\overline{k}}_{\overline{W},L,\overline{N}}$
where $\overline{W} := W(\Psi_{1})$ and $\overline{N} := N(\Psi_{1})$ satisfy
\[
  \overline{W} + \|S_{L+1}\|_{\ell^{0}}
  \leq W(\Psi)
  \leq W
  \quad \text{and} \quad
  \overline{N} + \overline{k}
  \leq N(\Psi)
  \leq N .
\]
Define $g := \Realization(\Psi_1)$, so that $f = S_{L+1} \circ \alpha_L \circ g$.
We now exhibit a $\varrho$-network $\Phi$ (instead of the $\sigma$-network $\Psi$)
of controlled complexity such that $f = \Realization(\Phi)$.
As $g := \Realization(\Psi_{1}) \in \NNreal^{\sigma,d,\overline{k}}_{\overline{W},L,\overline{N}}$,
the induction hypothesis shows that $g = \Realization(\Phi_{1})$ for some network
\[
  \Phi_{1}
  = \big(
      (T_{1}, \beta_{1}),
      \ldots,
      (T_{K-1}, \beta_{K-1})
      (T_{K}, \identity_{\R^{\overline{k}}})
    \big)
  \in \NNsymbol^{\varrho,d,\overline{k}}_{m \overline{W} + \overline{N} w, 1 + (L-1)\ell, \overline{N}(1+m)} \, .
\]
Moreover, Lemma~\ref{lem:TensorActivation} shows that $\alpha_L = \Realization(\Phi_{2})$
for a network
\[
  \Phi_{2}
  =   \big(
        (U_{1}, \gamma_{1}),
        \ldots,
        (U_{\ell-1}, \gamma_{\ell-1}),
        (U_{\ell}, \identity_{\R^{\overline{k}}})
      \big)
  \in \NNsymbol^{\varrho,\overline{k},\overline{k}}_{\overline{k}w,\ell,\overline{k}m}
\]
with $\|U_\ell\|_{\ell^{0,\infty}_\ast} \leq m$.
By construction, we have $f = S_{L+1} \circ \alpha_{L} \circ g = \Realization(\Phi)$ for the network
\[
  \Phi
  := \big(
       (T_1, \beta_1),
       \dots,
       (T_{K-1}, \beta_{K-1}),
       (T_K, \identity_{\R^{\overline{k}}}),
       (U_1, \gamma_1),
       \dots,
       (U_{\ell-1}, \gamma_{\ell-1}),
       (S_{L+1} \circ U_\ell, \identity_{\R^k})
     \big).
\]
To conclude, we observe that
\(
  L(\Phi) = K + \ell \leq 1 + (L-1)\ell + \ell = 1 + \big( (L+1) - 1 \big) \ell
\),
as well as
\begin{align*}
  W(\Phi)
  & = W(\Phi_1)
      + \big( W(\Phi_2) - \|U_\ell\|_{\ell^0} \big)
      + \|S_{L+1} \circ U_\ell\|_{\ell^0} \\
  ({\scriptstyle{\text{Lemma}~\ref{lem:MatrixSparsityForComposition}}})
  & \leq m \overline{W}
         + \overline{N} w
         + W(\Phi_2)
         + \|S_{L+1}\|_{\ell^0} \cdot \|U_\ell\|_{\ell^{0,\infty}_\ast} \\
  & \leq m \overline{W} + \overline{N} w + \overline{k} w + m \cdot \|S_{L+1}\|_{\ell^0}
    \leq m W + N w.
\end{align*}
Finally, we also have
\(
  N(\Phi)
  = N(\Phi_1) + \overline{k} + N(\Phi_2)
  \leq \overline{N} (1 + m) + \overline{k} + \overline{k} \cdot m
  =    (\overline{N} + \overline{k}) (1 + m)
  \leq N (1+m)
\).
\end{proof}

\subsection{Proof of Lemma~\ref{lem:NestednessBasic}}
\label{app:NestednessBasic}

Let
\(
  \Psi
  = \big( (S_{1},\alpha_{1}),\ldots, (S_{K-1}, \alpha_{K-1}), (S_K,\identity_{\R^{k}}) \big)
  \in \NNsymbol^{\sigma,d,k}_{W,L,N}
\)
be arbitrary and ${g = \Realization(\Psi)}$.
We prove that there is some $\Phi \in \NNsymbol^{\varrho,d,k}_{W+(s-1)N,1+s(L-1),sN}$ such that
$g = \Realization(\Phi)$.
This is easy to see if $s=1$ or $K=1$; hence we now assume $K \geq 2$ and $s \geq 2$.
Denoting by $N_{\ell}$ the number of neurons at the $\ell$-th layer of $\Psi$,
for $1 \leq \ell \leq K-1$ we have
$\alpha_{\ell} = \alpha_{\ell}^{(1)} \otimes \ldots \otimes \alpha_{\ell}^{(N_{\ell})}$
where $\alpha_{\ell}^{(i)} \in \{\identity_{\R},\sigma\}$.
For $1 \leq \ell \leq L-1$, $1 \leq j \leq K_{\ell}$, $1 \leq i \leq s$, define
\[
  \beta_{s(\ell-1)+i}^{(j)}
  := \begin{cases}
       \varrho,        & \text{if } \alpha_{\ell}^{(j)} = \sigma, \\
       \identity_{\R}, & \text{otherwise}
     \end{cases}
\]
and let $\beta_{s(\ell-1)+i}
         := \beta_{s(\ell-1)+i}^{(1)} \otimes \ldots \otimes \beta_{s(\ell-1)+i}^{(N_{\ell})}$.
Define also $T_{s(\ell-1)+1} := S_{\ell}: \R^{N_{\ell-1}} \to \R^{N_{\ell}}$
and $T_{s(\ell-1)+i} := \identity_{\R^{N_{\ell}}}$ for $2 \leq i \leq s$.
It is painless to check that
\begin{align*}
  \alpha_{\ell} \circ S_{\ell}
  &=  \beta_{s(\ell-1)+s}
      \circ T_{s(\ell-1)+s}
      \circ \cdots
      \circ \beta_{s(\ell-1)+2}
      \circ T_{s(\ell-1)+2}
      \circ \beta_{s(\ell-1)+1}
      \circ T_{s(\ell-1)+1} \\
  &=  \beta_{s\ell}
      \circ T_{s\ell}
      \circ \cdots
      \circ \beta_{s(\ell-1)+1}
      \circ T_{s(\ell-1)+1} \, ,
\end{align*}
and hence
\[
  g
  = S_{K} \circ \alpha_{K-1} \circ S_{K-1} \circ \cdots \circ \alpha_{1} \circ S_{1}
  = S_{K} \circ \beta_{s(K-1)} \circ T_{s(K-1)} \circ \cdots \circ \beta_{1} \circ T_{1} \, .
\]
That is to say, $g = \Realization(\Phi)$ with
\[
  \Phi
  := \big( (T_{1},\beta_{1}),\ldots,(T_{s(K-1)},\beta_{s(K-1)}),(S_{K},\identity_{\R^{k}}) \big)
  \in \NNsymbol^{\varrho,d,k}_{W',1+s(K-1),sN}
  \subset \NNsymbol^{\varrho,d,k}_{W', 1 + s(L-1), sN} ,
\]
where we compute
\begin{align*}
  W'
  &:=  \| S_{K} \|_{\ell^{0}}
       + \sum_{j =1}^{s(K-1)}
           \| T_{j} \|_{\ell^{0}}
  = \|S_{K}\|_{\ell^{0}}
    + \sum_{\ell=1}^{K-1}
        \sum_{i=1}^{s}
          \|
            T_{s(\ell-1)+i}
          \|_{\ell^{0}} \\
  &=  \| S_{K} \|_{\ell^{0}}
      + \sum_{\ell=1}^{K-1}
          \Big(
            \| T_{s(\ell-1)+1} \|_{\ell^{0}}
            + \sum_{i=2}^{s}
                \| T_{s(\ell-1)+i} \|_{\ell^{0}}
          \Big) \\
  &=  \| S_{K} \|_{\ell^{0}}
      + \sum_{\ell=1}^{K-1}
          \left(
            \|S_\ell\|_{\ell^{0}}
            + (s-1) N_{\ell}
          \right)
   =  \sum_{\ell=1}^{K}
        \|S_\ell\|_{\ell^{0}}
      + (s-1) \sum_{\ell=1}^{K-1}
                N_{\ell} \\
  & =  W(\Psi) + (s-1) N(\Psi) \leq W+(s-1)N \, .
\end{align*}
We conclude as claimed that $\Phi \in \NNsymbol^{\varrho,d,k}_{W+(s-1)N,1+s(L-1),sN}$.
Finally, if $\Psi$ is strict, then so is $\Phi$.
\hfill$\square$

\subsection{Proof of Lemma~\ref{lem:Recursivity}}
\label{app:Recursivity}

For $f \in \NNreal^{\sigma,d,k}_{W,L,N}$ there is
\(
  \Phi
  = \big(
      (S_{1},\alpha_{1}),\ldots,(S_{L'},\alpha_{L'})
    \big)
  \in \NNsymbol^{\sigma,d,k}_{W,L',N}
\)
with $L(\Phi) = L' \leq L$ and such that $f = \Realization(\Phi)$.
Replace each occurrence of the activation function $\sigma$ by $\sigma_{h}$
in the nonlinearities $\alpha_{j}$ to define a $\sigma_{h}$-network
$\Phi_{h}
 := \big( (S_{1},\alpha_{1}^{(h)}),\ldots,(S_{L'},\alpha_{L'}^{(h)}) \big)
 \in  \NNsymbol^{\sigma_{h},d,k}_{W,L',N}$ and its realization
$f_{h} := \Realization(\Phi_{h})\in \NNreal^{\sigma_{h},d,k}_{W,L',N}$.
Since $\sigma$ is continuous and $\sigma_{h} \to \sigma$
locally uniformly on $\R$ as $h \to 0$, we get by Lemma~\ref{lem:LocallyUniformComposition}
(which is proved independently below) that $f_{h} \to f$ locally uniformly on $\R^{d}$.
To conclude for $\ell=2$ observe that $\sigma_{h} = \Realization(\Psi_{h})$
with $\Psi_{h} \in \NNsymbol^{\varrho,1,1}_{w,\ell,m}$ and $L(\Psi_{h}) = \ell$,
whence Lemma~\ref{lem:RecursiveNNsets} yields
\[
  f_{h} \in \NNreal^{\sigma_{h},d,k}_{W,L',N}
  \subset \NNreal^{\varrho,d,k}_{Wm^{2},L',Nm}
  \subset \NNreal^{\varrho,d,k}_{Wm^{2},L,Nm} \, .
\]
For arbitrary $\ell$ we similarly conclude that
\begin{flalign*}
  && f_{h} \in \NNreal^{\sigma_{h},d,k}_{W,L',N}
  \subset \NNreal^{\varrho,d,k}_{W+Nw,1+(L'-1)(\ell+1),N(2+m)}
  \subset \NNreal^{\varrho,d,k}_{W+Nw,1+(L-1)(\ell+1),N(2+m)} \, .
  && \square
\end{flalign*}

\subsection{Proof of Lemmas~\ref{lem:GeneralizedVSStrict} and \ref{lem:IdentityRectifierForFree}}
\label{app:PfGeneralizedVSStrict}

In this section, we provide a unified proof for Lemmas~\ref{lem:GeneralizedVSStrict}
and \ref{lem:IdentityRectifierForFree}.
To be able to handle both claims simultaneously, the following concept will be important.

\begin{defn}\label{def:NetworkCompatibleTopology}
  For each $d,k \in \N$, let us fix a subset $\mathcal{G}_{d,k} \subset \{ f : \R^d \to \R^k \}$
  and a topology $\CalT_{d,k}$ on the space of all functions $f : \R^d \to \R^k$.
  Let $\CalG := (\CalG_{d,k})_{d,k \in \N}$ and $\CalT := (\CalT_{d,k})_{d,k \in \N}$.
  The tuple $(\CalG,\CalT)$ is called a \emph{network compatible topology family}
  if it satisfies the following:
  \begin{enumerate}
    \item We have $\{ T : \R^d \to \R^k \,\mid\, T \text{ affine-linear} \} \subset \CalG_{d,k}$
          for all $d,k \in \N$.

    \item If $p \in \N$ and for each $i \in \FirstN{p}$, we are given a sequence
          $(f_i^{(n)})_{n \in \N_0}$ of functions $f_i^{(n)} : \R \to \R$
          satisfying $f_i^{(0)} \in \CalG_{1,1}$ and
          $f_i^{(n)} \xrightarrow[n\to\infty]{\CalT_{1,1}} f_i^{(0)}$, then
          \(
            f_1^{(n)} \otimes \cdots \otimes f_p^{(n)}
            \xrightarrow[n\to\infty]{\CalT_{p,p}}
            f_1^{(0)} \otimes \cdots \otimes f_p^{(0)}
          \)
          and $f_1^{(0)} \otimes \cdots \otimes f_p^{(0)} \in \CalG_{p,p}$.

    \item If $f_n : \R^d \to \R^k$ and $g_n : \R^k \to \R^\ell$ for all $n \in \N_0$
          and if $f_0 \in \CalG_{d,k}$ and $g_0 \in \CalG_{k,\ell}$ as well as
          $f_n \xrightarrow[n\to\infty]{\CalT_{d,k}} f_0$
          and $g_n \xrightarrow[n\to\infty]{\CalT_{k,\ell}} g_0$,
          then $g_0 \circ f_0 \in \CalG_{d,\ell}$ and
          $g_n \circ f_n \xrightarrow[n\to\infty]{\CalT_{d,\ell}} g_0 \circ f_0$.
          \qedhere
  \end{enumerate}
\end{defn}

\begin{rem*}
  Roughly speaking, the above definition introduces certain topologies $\CalT_{d,k}$
  and certain sets of ``good functions'' $\CalG_{d,k}$ such that---for limit functions
  that are ``good''---convergence in the topology is compatible
  with taking $\otimes$-products and with composition.

  By induction, it is easy to see that if $p \in \N$ and if for each $i \in \FirstN{p}$
  we are given a sequence $(f_i^{(n)})_{n \in \N}$ with $f_i^{(n)} : \R^{d_{i-1}} \to \R^{d_i}$
  and $f_i^{(0)} \in \CalG_{d_{i-1}, d_i}$ as well as
  $f_i^{(n)} \xrightarrow[n\to\infty]{\CalT_{d_{i-1},d_i}} f_i^{(0)}$, then also
  $f_p^{(0)} \circ \cdots \circ f_1^{(0)} \in \CalG_{d_0, d_p}$, as well as
  \(
    f_p^{(n)} \circ \cdots \circ f_1^{(0)}
    \xrightarrow[n\to\infty]{\CalT_{d_0,d_p}} f_p^{(0)} \circ \cdots \circ f_1^{(0)}
  \).
  Indeed, the base case of the induction is contained in Definition~\ref{def:NetworkCompatibleTopology}.
  Now, assuming that the claim holds for $p \in \N$, we prove it for $p+1$.
  To this end, let $F_1^{(n)} := f_p^{(n)} \circ \cdots \circ f_1^{(n)}$
  and $F_2^{(n)} := f_{p+1}^{(n)}$.
  By induction, we know $F_1^{(0)} \in \CalG_{d_0, d_p}$ and
  $F_1^{(n)} \xrightarrow[n\to\infty]{\CalT_{d_0,d_p}} F_1^{(0)}$.
  Since also $F_2^{(0)} = f_{p+1}^{(0)} \in \CalG_{d_p, d_{p+1}}$,
  Definition~\ref{def:NetworkCompatibleTopology} implies
  $F_2^{(0)} \circ F_1^{(0)} \in \CalG_{d_0, d_{p+1}}$ and
  $F_2^{(n)} \circ F_1^{(n)} \xrightarrow[n\to\infty]{\CalT_{d_0, d_{p+1}}} F_2^{(0)} \circ F_1^{(0)}$,
  which is precisely the claim for $p+1$ instead of $p$.
\end{rem*}

We now have the following important result:

\begin{prop}\label{prop:NetworkCompatibleTopologyClosure}
  Let $\varrho : \R \to \R$, and let $(\CalG, \CalT)$ be a network compatible topology family
  satisfying the following
  \begin{itemize}
    \item $\varrho \in \CalG_{1,1}$;

    \item There is some $n \in \N$ such that for each $m \in \N$ there are affine-linear maps
          $E_m : \R \to \R^n$ and $D_m : \R^n \to \R$ such that
          $F_m := D_m \circ (\varrho \otimes \cdots \otimes \varrho) \circ E_m : \R \to \R$
          satisfies $F_m \xrightarrow[m\to\infty]{\CalT_{1,1}} \identity_{\R}$.
  \end{itemize}
  Then we have for arbitrary $d,k \in \N$, $W,N \in \N_0 \cup \{\infty\}$
  and $L \in \N \cup \{\infty\}$ the inclusion
  \[
    \NNreal^{\varrho,d,k}_{W,L,N}
    \subset \overline{\SNNreal^{\varrho,d,k}_{n^2 W, L, n N}} ,
  \]
  where the closure is a sequential closure which is taken with respect to the topology $\CalT_{d,k}$.
\end{prop}


\begin{rem*}
  Before we give the proof of Proposition~\ref{prop:NetworkCompatibleTopologyClosure},
  we explain a convention that will be used in the proof.
  Precisely, in the definition of $W(\Phi)$, we always assume that the affine-linear maps
  $T_\ell$ are of the form $T_\ell : \R^{N_{\ell - 1}} \to \R^{N_\ell}$.
  Clearly, the expressivity of networks will not change if instead of the
  spaces $\R^{N_1},\dots, \R^{N_{L - 1}}$, one uses finite-dimensional
  vector spaces $V_1, \dots, V_{L-1}$ with $\dim V_i = N_i$.
  The only nontrivial question is the interpretation of $\|T_\ell\|_{\ell^0}$
  for an affine-linear map $T_\ell : V_{\ell - 1} \to V_\ell$, since for the
  case of $\R^{N_{\ell}}$, we chose the standard basis for obtaining the
  matrix representation of $T_\ell$, while for general vector spaces $V_\ell$,
  there is no such canonical choice of basis.
  Yet, in the proof below, we will consider the case
  $V_\ell = \R^{n_1} \times \cdots \times \R^{n_{m}}$.
  In this case, there is a canonical way of identifying $V_\ell$ with
  $\R^{N_\ell}$ for $N_\ell = \sum_{j=1}^m n_j$, and there is also
  a canonical choice of ``standard basis'' in the space $V_\ell$.
  We will use this convention in the proof below to simplify the notation.
\end{rem*}

\begin{proof}[Proof of Proposition~\ref{prop:NetworkCompatibleTopologyClosure}]
Let $\Phi \in \NNsymbol_{W,L,N}^{\varrho,d,k}$.
We will construct a sequence $(\Phi_m)_{m \in \N} \subset \SNNsymbol_{n^2 W, L, nN}^{\varrho,d,k}$
satisfying $\Realization(\Phi_m) \xrightarrow[m\to\infty]{\CalT_{d,k}} \Realization(\Phi)$.
To this end, note that $\Phi  = \big( (T_1, \alpha_1), \dots, (T_K, \alpha_K) \big)$
for some $K \leq L$ and that there are $N_0, \dots, N_K \in \N$
(with $N_0 = d$ and $N_K = k$) such that $T_\ell : \R^{N_{\ell - 1}} \to \R^{N_\ell}$
is affine-linear for each $\ell \in \FirstN{K}$.

Let us first consider the special case $K = 1$.
By definition of a neural network, we have $\alpha_K = \identity_{\R^k}$,
so that $\Phi$ is already a \emph{strict} $\varrho$-network.
Therefore, we can choose
\(
  \Phi_m
  :=      \Phi \in \SNNsymbol_{W,L,N}^{\varrho,d,k}
  \subset \SNNsymbol_{n^2 W, L, nN}^{\varrho,d,k}
\)
for all $m \in \N$.

From now on we assume $K \geq 2$.
For brevity, set $\varrho_1 := \varrho$ and $\varrho_2 := \identity_{\R}$,
as well as $D(1) := 1$ and $D(2) := n$, and furthermore
\begin{alignat*}{3}
  & E_1^{(m)} \! := \! \identity_{\R} : \R\to \R^{D(1)}
  \quad \text{and} \quad
  & E_2^{(m)} \! := \! E_m : \R \to \R^{D(2)} \, , \\
  \text{as well as} \quad
  & D_1^{(m)} \! := \! \identity_{\R} : \R^{D(1)} \to \R
  \quad \text{and} \quad
  & D_2^{(m)} \! := \! D_m : \R^{D(2)} \to \R \, .
\end{alignat*}
By definition of a generalized $\varrho$-network, for each
$\ell \in \FirstN{K}$ there are
$\iota_1^{(\ell)}, \dots, \iota_{N_\ell}^{(\ell)} \in \{1,2\}$ with
$\alpha_\ell= \varrho_{\iota_1^{(\ell)}} \otimes \cdots \otimes
\varrho_{\iota_{N_\ell}^{(\ell)}}$, and with $\iota_j^{(K)} = 2$ for all
$j \in \FirstN{N_K}$. 
Now, define $V_0 := \R^d =\R^{N_{0}}$, $V_K := \R^k = \R^{N_K}$, and
\[
  V_\ell
  :=    \R^{D(\iota_1^{(\ell)})} \times \cdots \times \R^{D(\iota_{N_\ell}^{(\ell)})}
  \cong \R^{\sum_{i=1}^{N_{\ell}} D(\iota_{i}^{(\ell)})}
  \quad \text{for} \quad 1 \leq \ell \leq K-1.
\]
Since we eventually want to obtain strict networks $\Phi_m$, furthermore set
\[
  \beta^{(1)} := \varrho : \R^{D(1)} \to \R^{D(1)}
  \qquad \text{and} \qquad
  \beta^{(2)} := \varrho \otimes \cdots \otimes \varrho : \R^{D(2)} \to \R^{D(2)} .
\]
Using these maps, finally define 
$\beta_K := \identity_{\R^k}$, as well as
\[
  \beta_\ell
  := \beta^{(\iota_1^{(\ell)})}
     \otimes \cdots \otimes
     \beta^{(\iota_{N_\ell}^{(\ell)})} : V_{\ell} \to V_{\ell}
  \quad \text{for} \quad 1 \leq \ell \leq K-1 \, .
\]
Finally, for $\ell \in \FirstN{K}$ and $m \in \N$, define affine-linear maps
\begin{align*}
   P_\ell^{(m)}
   := E_{\iota_1^{(\ell)}}^{(m)}
    \otimes \cdots \otimes
    E_{\iota_{N_\ell}^{(\ell)}}^{(m)}
  : \R^{N_{\ell}} \to V_\ell 
  \qquad \text{and} \qquad
  Q_\ell^{(m)}
   := D_{\iota_1^{(\ell)}}^{(m)}
      \otimes \cdots \otimes
      D_{\iota_{N_\ell}^{(\ell)}}^{(m)}
   : V_\ell \to \R^{N_\ell} \, .
\end{align*}

The crucial observation is that by assumption regarding the maps $D_m, E_m$, we have
\begin{equation}
  \begin{split}
    & D_2^{(m)} \circ \beta^{(2)} \circ E_2^{(m)}
      = F_m
      \xrightarrow[m \to \infty]{\CalT_{1,1}} \identity_{\R}
      = \varrho_2 , \\
    \text{and} \quad
    & D_1^{(m)} \circ \beta^{(1)} \circ E_1^{(m)}
      = \identity_{\R} \circ \varrho \circ \identity_{\R}
      = \varrho
      = \varrho_1 \, .
  \end{split}
  \label{eq:StrictGeneralizedBoundedDomainMainIdentity}
\end{equation}

Finally, for the construction of the strict networks $\Phi_m$, we define for $m \in \N$
\begin{alignat*}{6}
  && S_1^{(m)}    & := P_1^{(m)} \circ T_1
       & & :  \R^d = \R^{N_{0}} = V_0 && \to V_1, \\
  && S_K^{(m)}    & := T_K \circ  Q_{K-1}^{(m)}
       & & : V_{K-1} && \to \R^{N_K} = \R^k = V_K , \\
  \text{and} \quad
  && S_\ell^{(m)} & := P_\ell^{(m)} \circ T_\ell \circ Q_{\ell-1}^{(m)}
         & &  :  V_{\ell-1} && \to V_\ell
         \qquad \qquad \qquad \qquad \qquad
         \text{for } 2 \leq \ell \leq K - 1 \, ,
\end{alignat*}
and then set $\Phi_m := \big( (S_1^{(m)}, \beta_1), \dots, (S_K^{(m)}, \beta_K) \big)$.
Because of $D(\iota_{i^{(\ell)}}) \in \{1,n\}$, we obtain
\[
  N(\Phi_m)
  = \sum_{\ell = 1}^{K-1} \dim V_\ell
  = \sum_{\ell=1}^{K-1} \sum_{i=1}^{N_\ell} D(\iota_i^{(\ell)})
  \leq \sum_{\ell=1}^{K-1} n N_\ell
  =    n N(\Phi) \leq n N \, .
\]
Furthermore, by the second part of Lemma~\ref{lem:MatrixSparsityForComposition}
and in view of the product structure of $P_\ell^{(m)}$, we have
\[
  \| P_\ell^{(m)} \|_{\ell^{0,\infty}}
  \leq \max \big\{
                \|E^{(m)}_1\|_{\ell^{0,\infty}},
                \|E^{(m)}_2\|_{\ell^{0,\infty}}
            \big\}
  \leq \max \{ D(1), D(2) \}
  \leq n ,
\]
for arbitrary $\ell \in \FirstN{K}$, simply because $E^{(m)}_j : \R \to \R^{D(j)}$ for $j \in \{1,2\}$.
Likewise,
\[
  \| Q_{\ell}^{(m)} \|_{\ell^{0,\infty}_\ast}
  \leq \max \big\{
                  \| D^{(m)}_1 \|_{\ell^{0,\infty}_{\ast}},
                  \| D^{(m)}_2 \|_{\ell^{0,\infty}_{\ast}}
            \big\}
  \leq \max \{ D(1), D(2) \}
  \leq n ,
\]
because $D_j^{(m)} : \R^{D(j)} \to \R$ for $j \in \{1,2\}$.
By the first part of Lemma~\ref{lem:MatrixSparsityForComposition},
we thus see for $2 \leq \ell \leq K - 1$ that
\[
  \| S_\ell^{(m)} \|_{\ell^0}
  \leq \| P_\ell^{(m)}\|_{\ell^{0,\infty}}
       \cdot \| T_\ell \|_{\ell^0}
       \cdot \| Q_{\ell - 1}^{(m)} \|_{\ell^{0,\infty}_\ast}
  \leq n^2 \cdot \| T_\ell \|_{\ell^0} \, .
\]
Similar arguments yield
$\|S_1^{(m)}\|_{\ell^0} \leq n \cdot \|T_1\|_{\ell^0} \leq n^2 \cdot \|T_1\|_{\ell^0}$
and $\|S_K^{(m)}\|_{\ell^0} \leq n \cdot \|T_K\|_{\ell^0} \leq n^2 \cdot \|T_K\|_{\ell^0}$.
All in all, this implies $W(\Phi_m) \leq n^2 \cdot W(\Phi) \leq n^2 W$, as desired.

Now, since $\varrho_1 = \varrho \in \CalG_{1,1}$ by the assumptions of the current proposition,
since $\varrho_2 = \identity_{\R} \in \CalG_{1,1}$ as an affine-linear map,
and since $(\CalG,\CalT)$ is a network compatible topology family,
we see for all $1 \leq \ell \leq K - 1$ that
\(
  \alpha_\ell
  = \varrho_{\iota_1^{(\ell)}}
    \otimes \cdots \otimes
    \varrho_{\iota_{N_\ell}^{(\ell)}} \in \CalG_{N_\ell,N_\ell}
\)
and furthermore that
\begin{equation}
  \begin{split}
    Q_\ell^{(m)} \circ \beta_\ell \circ P_\ell^{(m)}
    & = \left(
          D_{\iota_1^{(\ell)}}^{(m)}
          \circ \beta^{(\iota_1^{(\ell)})}
          \circ E_{\iota_1^{(\ell)}}^{(m)}
      \right)
      \otimes \cdots \otimes
      \left(
          D_{\iota_{N_\ell}^{(\ell)}}^{(m)}
          \circ \beta^{(\iota_{N_\ell}^{(\ell)})}
          \circ E_{\iota_{N_\ell}^{(\ell)}}^{(m)}
      \right) \\
    ({\scriptstyle{\text{Eq.}~\eqref{eq:StrictGeneralizedBoundedDomainMainIdentity}
                   \text{ and compatibility of } (\CalG,\CalT) \text{ with $\otimes$}}})
    & \xrightarrow[m\to\infty]{\CalT_{N_\ell, N_\ell}}
          \varrho_{\iota_1^{(\ell)}}
          \otimes \cdots \otimes
          \varrho_{\iota_{N_\ell}^{(\ell)}}
    =   \alpha_\ell \, .
  \end{split}
  \label{eq:ApproximateIdentityRepresentation}
\end{equation}

Finally, since $\beta_K = \identity_{\R^k} = \alpha_K \in \CalG_{k, k}$,
and since $(\CalG,\CalT)$ is a network compatible topology family
and thus compatible with compositions (as long as the ``factors'' of the limit are ``good'',
which is satisfied here, since $\alpha_\ell \in \CalG_{N_\ell, N_\ell}$ as we just
saw and since $T_\ell \in \CalG_{N_{\ell - 1}, N_\ell}$ as an affine-linear map), we see that
\begin{align*}
   \Realization (\Phi_m)
  & = \beta_K \circ S_K^{(m)} \circ \cdots \circ \beta_1 \circ S_1^{(m)} \\
  & = \alpha_K \circ T_K
      \circ (Q_{K-1}^{(m)} \circ \beta_{K-1} \circ P_{K-1}^{(m)})
      \circ T_{K-1}
      \circ \cdots
      \circ (Q_1^{(m)} \circ \beta_1 \circ P_1^{(m)})
      \circ T_1 \\
  ({\scriptstyle{\text{Eq.}~\eqref{eq:ApproximateIdentityRepresentation}}})
  & \xrightarrow[m\to\infty]{\CalT_{d,k}}
        \alpha_K \circ T_K
        \circ \alpha_{K-1} \circ T_{K-1}
        \circ \cdots
        \circ \alpha_1 \circ T_1
    =   \Realization (\Phi) \, ,
\end{align*}
and hence $\Realization(\Phi) \in \overline{\SNNreal^{\varrho,d,k}_{n^2 W, L, n N}}$.
\end{proof}

Now, we use Proposition~\ref{prop:NetworkCompatibleTopologyClosure} to prove
Lemma~\ref{lem:IdentityRectifierForFree}.

\begin{proof}[Proof of Lemma~\ref{lem:IdentityRectifierForFree}]
  For $d,k \in \N$, let $\CalG_{d,k} := \{ f : \R^d \to \R^k \}$, and let
  $\CalT_{d,k} = 2^{\CalG_{d,k}}$ be the discrete topology on the set $\{ f : \R^d \to \R^k \}$.
  This means that every set is open, so that the only convergent sequences are those
  that are eventually constant.
  It is easy to see that $(\CalG,\CalT)$ is a network compatible topology family and
  $\varrho \in \CalG_{1,1}$.

  Finally, by assumption of Lemma~\ref{lem:IdentityRectifierForFree}, there are
  $a_i, b_i, c_i \in \R$ for $i \in \FirstN{n}$ and some $c \in \R$ such that
  $x = c + \sum_{i=1}^n a_i \, \varrho(b_i \, x + c_i)$ for all $x \in \R$.
  If we define $E_m : \R \to \R^n, x \mapsto (b_1 \, x + c_1, \dots, b_n \, x + c_n)$
  and $D_m : \R^n \to \R, y \mapsto c + \sum_{i=1}^n a_i \, y_i$, then $E_m, D_m$ are affine-linear,
  and $\identity_{\R} = D_m \circ (\varrho \otimes \cdots \otimes \varrho) \circ E_m$ for all
  $m \in \N$.
  Thus, all assumptions of Proposition~\ref{prop:NetworkCompatibleTopologyClosure} are satisfied,
  so that this proposition implies
  \(
    \NNreal^{\varrho,d,k}_{W,L,N}
    \subset \overline{\SNNreal_{n^2 W, L, nN}^{\varrho,d,k}}
    =       \SNNreal_{n^2 W, L, nN}^{\varrho,d,k}
  \)
  for all $d,k \in \N$, $W,N \in \N_0 \cup \{\infty\}$ and $L \in \N \cup \{\infty\}$.
  Here, we used that the (sequential) closure of a set $M$ with respect to the discrete topology
  is simply the set $M$ itself.
\end{proof}

Finally, we will use Proposition~\ref{prop:NetworkCompatibleTopologyClosure}
to provide a proof of Lemma~\ref{lem:GeneralizedVSStrict}.
To this end, the following lemma is essential.

\begin{lem}\label{lem:LocallyUniformComposition}
  Let $(f_n)_{n \in \N_0}$ and $(g_n)_{n \in \N_0}$ be sequences of functions
  $f_n : \R^d \to \R^k$ and $g_n : \R^k \to \R^\ell$.
  Assume that $f_0, g_0$ are continuous and that $f_n \xrightarrow[n\to\infty]{} f_0$
  and $g_n \xrightarrow[n\to\infty]{} g_0$ with locally uniform convergence.
  Then $g_0 \circ f_0$ is continuous, and
  $g_n \circ f_n \xrightarrow[n\to\infty]{} g_0 \circ f_0$ with locally uniform convergence.
\end{lem}

\begin{proof}
  Locally uniform convergence on $\R^d$ is equivalent to uniform convergence on bounded sets.
  Furthermore, the continuous function $f_0$ is bounded on each bounded set $K \subset \R^d$;
  by uniform convergence, this implies that
  $K' := \{ f(x) \colon x \in K \} \cup \{ f_n (x) \colon n \in \N \text{ and } x \in K \} \subset \R^k$
  is bounded as well.
  Hence, the continuous function $g_0$ is \emph{uniformly} continuous on $K'$.
  From these observations, the claim follows easily; the details are left to the reader.
\end{proof}

Given this auxiliary result, we can now prove Lemma~\ref{lem:GeneralizedVSStrict}.

\begin{proof}[Proof of Lemma~\ref{lem:GeneralizedVSStrict}]
  For $d,k \in \N$, define $\CalG_{d,k} := \{ f : \R^d \to \R^k \,\mid\, f \text{ continuous} \}$,
  and let $\CalT_{d,k}$ denote the topology of locally uniform convergence on
  $\{ f : \R^d \to \R^k \}$.
  We claim that $(\CalG,\CalT)$ is a network compatible topology family.
  Indeed, the first condition in Definition~\ref{def:NetworkCompatibleTopology} is trivial,
  and the third condition holds thanks to Lemma~\ref{lem:LocallyUniformComposition}.
  Finally, it is not hard to see that if $f_i^{(n)} : \R \to \R$ satisfy
  $f_i^{(n)} \to f_i^{(0)}$ locally uniformly for all $i \in \FirstN{p}$, then
  \(
    f_1^{(n)} \otimes \cdots \otimes f_p^{(n)}
    \xrightarrow[n\to\infty]{} f_1^{(0)} \otimes \cdots \otimes f_p^{(0)}
  \)
  locally uniformly.
  This proves the second condition in Definition~\ref{def:NetworkCompatibleTopology}.

  We want to apply Proposition~\ref{prop:NetworkCompatibleTopologyClosure} with $n = 2$.
  We have $\varrho \in \CalG_{1,1}$, since $\varrho$ is continuous by the assumptions
  of Lemma~\ref{lem:GeneralizedVSStrict}.
  Thus, it remains to construct sequences $(E_m)_{m \in \N}, (D_m)_{m \in \N}$
  of affine-linear maps $E_m : \R \to \R^2$ and $D_m : \R^2 \to \R$ such that
  $D_m \circ (\varrho \otimes \varrho) \circ E_m \to \identity_{\R}$ with
  locally uniform convergence.
  Once these are constructed, Proposition~\ref{prop:NetworkCompatibleTopologyClosure} shows that
  $\NNreal^{\varrho,d,k}_{W,L,N} \subset \overline{\SNNreal^{\varrho,d,k}_{4W, L, 2N}}$,
  where the closure is with respect to locally uniform convergence.
  This is precisely what is claimed in Lemma~\ref{lem:GeneralizedVSStrict}.

  To construct $E_m, D_m$, let us set $a := \varrho' (x_0) \neq 0$.
  By definition of the derivative, for arbitrary $m \in \N$ and
  $\varepsilon_m := |a|/m$, there is some $\delta_m > 0$ satisfying
  \begin{equation}
    \left|
        \big( \varrho (x_0 + h) - \varrho (x_0) \big) / h - a
    \right|
    \leq \varepsilon_m = |a| / m 
    \qquad \forall \, \, h \in \R \text{ with } 0 < |h| \leq \delta_m \, .
    \label{eq:DifferentiationDefinition}
  \end{equation}
  Now, define affine-linear maps
  \[
      E_m : \R\to\R^2, x \mapsto \left(
                                    x_0 + m^{-1/2} \cdot \delta_m \cdot x  \,
                                    , \, x_0
                                   \right)^T 
      \quad \text{and} \quad
      D_m : \R^2\to\R,
            (y_1, y_2) \mapsto \sqrt{m} \cdot (y_1 - y_2) / (a \cdot \delta_m) \, ,
  \]
  and set $F_m := D_m \circ (\varrho \otimes \varrho) \circ E_m$.

  Finally, let $x \in \R$ be arbitrary with $0 < |x| \leq \sqrt{m}$, and set
  $h := \delta_m \cdot x / \sqrt{m}$, so that $0 < |h| \leq \delta_m$.
  By multiplying Equation~\eqref{eq:DifferentiationDefinition} with $|h|/|a|$, we then get
  \begin{align*}
      & \left|
            a^{-1} \cdot \big( \varrho (x_0 + h) - \varrho(x_0) \big) - h
        \right|
        \leq \frac{|h|}{m} \\
      ({\scriptstyle{\text{multiply by } \sqrt{m} / \delta_m}})
      \Longrightarrow
      & \left|
          \frac{\sqrt{m}}{a \cdot \delta_m}
          \left(
              \varrho \left( x_0 + \frac{\delta_m \cdot x}{\sqrt{m}} \right)
              - \varrho (x_0)
          \right)
          - x
        \right|
        \leq \frac{|h|}{\delta_m \cdot \sqrt{m}}
        =    \frac{|x|}{m}
        \leq \frac{1}{\sqrt{m}} \, ,
  \end{align*}
  where the last step used that $|x| \leq \sqrt{m}$.
  This estimate is trivially valid for $x = 0$.
  Put differently, we have thus shown $|F_m (x) - x|\leq 1/\sqrt{m}$
  for all $x \in \R$ with $|x| \leq \sqrt{m}$.
  That is, $F_m \xrightarrow[m\to\infty]{} \identity_{\R}$ with locally uniform convergence.
\end{proof}

\subsection{Proof of Lemma~\ref{lem:ReLUPowerRepresentsIdentity}}
\label{app:ReLUPowerRepresentsIdentity}

We will need the following lemma that will also be used elsewhere.
\begin{lem}\label{lem:MonomialShiftSpan}
  For $f : \R \to \R$ and $a \in \R$,
  let $T_a f : \R \to \R, x \mapsto T_a f (x) = f(x-a)$.
  Furthermore, for $n \in \N_0$, let $X^n : \R \to \R, x \mapsto x^n$ and
  $V_n := \mathrm{span} \{T_a X^n \, \colon \, a \in \R\}$, with the convention $X^0 \equiv 1$.

  We have $V_n = \R_{\deg \leq n}[x]$, that is, $V_n$ is the space of all polynomials of degree at most $n$.
\end{lem}

\begin{proof}
  Clearly, $V_n \subset \R_{\deg \leq n} [x] =: V$, where $\dim V = n+1$.
  Therefore, it suffices to show that $V_n$ contains $n+1$ linearly independent elements.
  In fact, we show that whenever $a_1,\dots,a_{n+1} \in \R$ are pairwise distinct,
  then the family $(T_{a_i} X^n)_{i=1,\dots,n+1} \subset V_n$ is linearly independent.

  To see this, suppose that $\theta_1,\dots,\theta_{n+1} \in \R$ are such that
  $0 \equiv \sum_{i=1}^{n+1} \theta_i \, T_{a_i} X^n$.
  A direct computation using the binomial theorem shows that this implies
  \(
    0 \equiv \sum_{\ell=0}^n
             \big[
               \binom{n}{\ell} (-1)^\ell X^{n-\ell}
               \sum_{i=1}^{n+1}
                 \theta_i a_i^\ell
             \big]
  \).
  By comparing the coefficients of $X^t$, this leads to
  $0 = \big( \sum_{i=1}^{n+1} a_i^\ell \, \theta_i \big)_{\ell=0,\dots,n} = A^T \theta$,
  where $\theta = (\theta_1,\dots,\theta_{n+1}) \in \R^n$, and where the \emph{Vandermonde matrix}
  $A := (a_i^j)_{i=1,\dots,n+1, j=0,\dots,n} \in \R^{(n+1) \times (n+1)}$ is invertible;
  see \cite[Equation (4-15)]{HoffmanKunzeLinearAlgebra}.
  Hence, $\theta = 0$, showing that $(T_{a_i} X^n)_{i=1,\dots,n+1}$ is a linearly independent family.
\end{proof}


\begin{proof}[Proof of Lemma~\ref{lem:ReLUPowerRepresentsIdentity}]
  First, note
  \begin{equation}
    \varrho_r (x) + (-1)^r \, \varrho_r (-x)
    = \begin{cases}
        \varrho_r (x) = (x_+)^r = x^r ,                                  & \text{if } x \geq 0 \, \\
        (-1)^r \varrho_r (-x) = (-1)^r [(-x)_+]^r = (-1)^r (-x)^r = x^r, & \text{if } x < 0 \, .
      \end{cases}
    \label{eq:ReLUPowerCanRepresentPower}
  \end{equation}
  Next, Lemma~\ref{lem:MonomialShiftSpan} shows that
  $V_r = \R_{\deg \leq r}[x]$ has dimension $r+1$.
  Thus, given any polynomial ${f \in \R_{\deg \leq r}[x]}$, there are $a_1, \dots, a_{r+1} \in \R$
  and $b_1, \dots, b_{r+1} \in \R$ such that for all $x \in \R$
  \[
    f(x)
     = \sum_{\ell = 1}^{r+1}
          a_\ell \cdot (T_{b_\ell} X^r)(x)
     \stackrel{ \eqref{eq:ReLUPowerCanRepresentPower}}{=}
     \sum_{\ell = 1}^{r+1}
          a_\ell \cdot [\varrho_r (x - b_\ell) + (-1)^r \varrho_r \big( - (x - b_\ell) \big)].
     \qedhere
  \]
\end{proof}

\subsection{Proof of Lemma~\ref{lem:MultNetwork}}
\label{app:MultNetwork}

For Part (1), define $w_{j} := 6n(2^{j}-1)$ and $m_{j} := (2n+1)(2^j-1)-1$.
We will prove below by induction on $j \in \N$ that
$M_{2^{j}} \in \NNreal^{\varrho,2^{j},1}_{w_j,2j,m_j}$.
Let us see first that this implies the result.
For arbitrary $d \in \N_{\geq 2}$ and $j = \lceil \log_{2} d \rceil$ it is not hard to see that
\[
  P: \R^{d} \to \R^{2^{j}}, x \mapsto (x,1_{2^{j}-d}) = (x,0_{2^{j}-d}) + (0_{d},1_{2^{j}-d})
\]
is affine-linear with $\|P\|_{\ell^{0,\infty}_\ast}=1$
(cf.~Equation~\eqref{eq:DefL0MixedNorms}) and that $M_{d} = M_{2^{j}} \circ P$.
Using Lemma~\ref{lem:NetworkCalculus}-(\ref{enu:PrePostAffine}) we get
$M_{d} \in \NNreal^{\varrho,2^{j},1}_{w_j,2j,m_j}$ as claimed.

We now proceed to the induction.
As a preliminary, note that by assumption there are $a \in \R$,
$\alpha_1, \dots, \alpha_n \in \R$ and $\beta_1, \dots, \beta_n \in \R$ such that
for all $x \in \R$
\[
  x^2  = a +  \sum_{\ell = 1}^{n} \beta_\ell \, \varrho (x - \alpha_\ell).
\]
Put differently, the affine-linear maps
$T_1 : \R \to \R^{n}, x \mapsto (x-\alpha_{\ell})_{\ell=1}^{n}$ and
${T_2 : \R^{n} \to \R, y \mapsto a + \sum_{\ell=1}^{n} \beta_{\ell} \, y_{\ell}}$
satisfy $x^2 = T_2 \circ (\varrho \otimes \cdots \otimes \varrho) \circ T_1 (x)$ for all $x \in \R$,
where the $\otimes$-product has $n$ factors.
Since ${x \cdot y = \tfrac{1}{4} \big( (x+y)^2 - (x-y)^2 \big)}$ for all $x,y \in \R$,
if we define the maps ${T_0 : \R^2 \to \R^2, (x,y) \mapsto (x + y, x-y)}$ and
$T_3 : \R^2 \to \R, (u,v) \mapsto \frac{1}{4} (u - v)$, then for all $x,y \in \R$
\vspace{-0.3cm}
\[
  x \cdot y
  = \tfrac{1}{4} \cdot \big( (x+y)^2 - (x-y)^2 \big)
  = \big(
      S_2
      \circ \overbracket{(\varrho \otimes \cdots \otimes \varrho)}^{2n \text{ factors}}
      \circ S_1
    \big) (x,y).
\]
where 
${S_1 := (T_1 \otimes T_1) \circ T_0 : \R^2 \to \R^{2n}}$ and
$S_2 := T_3 \circ (T_2 \otimes T_2) : \R^{2n} \to \R$.
As $\|S_1\|_{\ell^{0}} \leq 4n$ and $\|S_2\|_{\ell^{0}} \leq 2n$
we obtain $M_{2} = \Realization(\Phi_{1})$ where
\(
  \Phi_{1}
  = \big(
      (S_{1}, \varrho \otimes \cdots \otimes \varrho),(S_{2},\identity)
    \big)
  \in \NNsymbol_{6n, 2, 2n}^{\varrho, 2, 1}.
\)
This establishes our induction hypothesis for $j=1$:
$M_{2} \in \SNNreal_{6n, 2, 2n}^{\varrho, 2, 1} \subset \NNreal_{w_j, 2^j, m_j}^{\varrho, 2, 1}$
for $j = 1$.

We proceed to the actual induction step.
Define the affine maps $U_1, U_2 : \R^{2^{j+1}} \to \R^{2^{j}}$ by
\[
  U_1(x) := (x_{1}, \ldots, x_{2^{j}}) =: \overline{x}
  \quad \text{and} \quad
  U_2(x) := (x_{2^{j}+1}, \ldots, x_{2^{j+1}}) =: x'
  \quad \text{for} \quad x \in \R^{2^{j+1}}.
\]
With these definitions, observe that
\(
  M_{2^{j+1}}(x)
  = M_{2^{j}}(\overline{x}) M_{2^{j}}(x')
  = M_{2} \big( M_{2^{j}}(U_{1}(x)),M_{2^{j}}(U_{2}(x)) \big)
\).

By the induction hypothesis there is a network
\(
  \Phi_{j}
  = \big( (V_{1},\alpha_{1}), \ldots, (V_{L},\identity) \big)
  \in \NNsymbol_{w_{j}, 2j, m_{j}}^{\varrho, 2^{j}, 1}
\)
with $L(\Phi_{j}) = L \leq 2j$ such that $M_{2^{j}} = \Realization(\Phi_{j})$.
Since $\|U_{i}\|_{\ell^{0,\infty}_\ast}=1$, the second part
of Lemma~\ref{lem:MatrixSparsityForComposition} shows
$\|V_1 \circ U_i\|_{\ell^0} \leq \|V_{1}\|_{\ell^{0}}$,
whence $M_{2^{j}} \circ U_i = \Realization(\Psi_{i})$, where
\(
  \Psi_{i}
  = \big(
      (V_{1} \circ U_{i}, \alpha_{1}), (V_{2}, \alpha_{2}), \ldots, (V_{L},\identity)
    \big)
\)
satisfies $W(\Psi_{i}) \leq W(\Phi_{j})$, $N(\Psi_{i}) \leq N(\Phi_{j})$, $L(\Psi_{i}) = L$,
and $\Psi_{i} \in \NNsymbol^{\varrho,2^{j},1}_{w_{j},2j,m_{j}}$.
\ifarxiv Thus, Lemma~\ref{lem:AuxConcatenationTwo} shows that
\fi
\(
  f
  :=  (M_{2^{j}} \circ U_1, M_{2^{j}} \circ U_2)
  \in \NNreal_{2w_{j},2j,2m_{j}}^{\varrho,2^{j+1},2}
\).
Since $M_{2} \in \NNreal^{\varrho,2,1}_{6n,2,2n}$,
Lemma~\ref{lem:NetworkCalculus}-(\ref{enu:Compos}) shows that
$M_{2^{j+1}} = M_{2} \circ f \in \NNreal^{\varrho,2^{j+1},1}_{2w_{j}+6n,2j+2,2m_{j}+2n+2}$.

To conclude the proof of Part~(1), note that
$2w_{j}+6n = 12 n(2^{j}-1) + 6n  = 6 n(2^{j+1}-1) = w_{j+1}$
and $2m_{j}+2n+2 = 2(2n+1)(2^{j}-1)+2n = (2n+1) (2^{j+1}-2)+2n+1-1 = m_{j+1}$.

\medskip{}

To prove Part (2), we recall from Part (1) that
$M_{2} : \R^2 \to \R, (x,y) \mapsto x \cdot y$ satisfies $M_{2} = \Realization(\Psi)$
with $\Psi \in \SNNsymbol_{6n, 2, 2n}^{\varrho, 2, 1}$ and $L(\Psi) = 2$.
Next, let $P^{(i)} : \R \times \R^k \to \R \times \R, (x, y) \mapsto (x, y_i)$
for each $i \in \FirstN{k}$, and note that $P^{(i)}$ is linear with
$\|P^{(i)}\|_{\ell^{0,\infty}} = 1 = \|P^{(i)}\|_{\ell^{0,\infty}_\ast}$.
Lemma~\ref{lem:NetworkCalculus}-(\ref{enu:PrePostAffine}) shows that
$M_{2} \circ P^{(i)} = \Realization(\Psi_{i})$ where
$\Psi_{i} \in \SNNsymbol^{\varrho,1+k,1}_{6n,2,2n}$ and $L(\Psi_{i}) = L(\Psi)=2$.
To conclude, observe ${(M_{2} \circ P^{(i)}) (x,y) = x \cdot y_i = [m(x,y)]_i}$
for 
${m : \R \times \R^k \to \R^k, (x,y) \mapsto x \cdot y}$.
Therefore, Lemma~\ref{lem:SummationLemma}-(\ref{enu:Cartesian}) shows that
$m = (M_{2} \circ P^{(1)}, \dots, M_{2} \circ P^{(k)}) \in \NNreal^{\varrho, 1+k, k}_{6kn, 2,2kn}$,
as desired.
\hfill$\square$



\section{Proofs for Section~\ref{sec:ApproximationSpaces}}
\label{app:ApproximationSpaces}

\subsection{Proof of Lemma~\ref{lem:ApproximationSpaceElementaryNesting}}
\label{sub:ApproximationSpaceElementaryNestingProof}

Let $f \in \GenApproxSpaceSet[\Sigma']$.
For the sake of brevity, set
$\eps_n := \AppErr(f, \Sigma_n)_X$ and $\delta_n := \AppErr(f, \Sigma_n')_{X}$
for $n \in \N_0$.
First, observe that $\eps_n \leq \|f\|_X = \delta_0$ for all $n \in \N_0$.
Furthermore, we have by assumption that $\eps_{cm} \leq \delta_m$ for all $m \in \N$.
Now, setting $m_n := \lfloor \frac{n - 1}{c} \rfloor \in \N$ for $n \in \N_{\geq c + 1}$, note that
$n - 1 \geq c \, m_n$, and hence $\eps_{n-1} \leq \eps_{c \, m_n} \leq C \cdot \delta_{m_n}$.
Therefore, we see
\[
  \eps_{n-1} \leq \delta_0 \text{ if } 1 \leq n \leq c
  \qquad \text{and} \qquad
  \eps_{n-1} \leq C \cdot \delta_{m_n} \text{ if } n \geq c+1.
\]
Next, note for $n \in \N_{\geq c + 1}$ that $m_n \geq 1$ and $m_n \geq \frac{n - 1}{c} - 1$, whence
$n \leq c \, m_n + c + 1 \leq (2 c + 1) m_n$.
Therefore, $n^\alpha \leq (2c + 1)^\alpha m_n^\alpha$.
Likewise, since $m_n \leq n$, we have $n^{-1} \leq m_n^{-1}$ for all $n \in \N_{\geq c + 1}$.

There are now two cases.
First, if $q < \infty$, and if we set $K := K(\alpha,q,c) := \sum_{n=1}^c n^{\alpha q - 1}$, then
\begin{align*}
  \|f\|_{\GenApproxSpaceSet[\Sigma]}^q
    = \sum_{n = 1}^\infty
      [n^\alpha \, \eps_{n - 1}]^q \, \frac{1}{n}
  & \leq \delta_0^q \cdot \sum_{n=1}^c n^{\alpha q - 1}
         + C^q \sum_{n=c+1}^\infty
                 (n^\alpha \, \delta_{m_n})^q \frac{1}{n} \\
  & \leq K \, \delta_0^q
        + C^q \, (2c+1)^{\alpha q} \sum_{n=c+1}^\infty
                                     (m_n^\alpha \, \delta_{m_n})^q \frac{1}{m_n} \, .
\end{align*}
Further, for $n \in \N_{\geq c + 1}$ satisfying $m_n = m$ for some $m \in \N$,
we have $m \leq \frac{n-1}{c} < m+1$,
which easily implies $|\{ n \in \N_{\geq c + 1} \colon m_n = m\}|
                      \leq |\{ n \in \N \colon c m + 1 \leq n < c m + c + 1 \}| = c$.
Thus,
\begin{align*}
  \sum_{n=c+1}^\infty
    (m_n^\alpha \, \delta_{m_n})^q \frac{1}{m_n}
  & = \sum_{m = 1}^\infty
        (m^\alpha \, \delta_m)^q
        \cdot \frac{1}{m}
        \cdot |\{ n \in \N_{\geq c+1} \colon m_n = m \}| \\
  & \leq c \sum_{m=1}^\infty (m^\alpha \, \delta_m)^q \frac{1}{m}
    \leq c \sum_{m=1}^\infty (m^\alpha \, \delta_{m-1})^q \frac{1}{m}
    =    c \, \|f\|_{\GenApproxSpaceSet[\Sigma']}^q \, .
\end{align*}
Overall, we thus see for $q < \infty$ that
\[
  \|f\|_{\GenApproxSpaceSet[\Sigma]}^q
  \leq (K + C^q (2c+1)^{\alpha q} c) \cdot \|f\|_{\GenApproxSpaceSet[\Sigma']}^q
  <    \infty \, ,
\]
where the constant $K + C^q (2c+1)^{\alpha q} c$ only depends on $\alpha,q,c,C$.

The adaptations for the (easier) case $q = \infty$ are left to the reader.
\hfill$\square$

\subsection{Proof of Lemma~\ref{lem:C0FunctionsCanBeExtended}}
\label{app:C0FunctionsCanBeExtended}

For $p \in (0,\infty)$, the claim is clear, since it is well-known that $L_{p}(\Omega;\R^{k})$
is complete, and since one can extend each $g \in \StandardXSpace(\Omega) = L_p (\Omega;\R^k)$
by zero to a function $f \in L^p(\Omega;\R^k)$ satisfying $g = f|_{\Omega}$.

Now, we consider the case $p = \infty$.
We first prove completeness of $\StandardXSpace[k][\infty](\Omega)$.
Let $(f_n)_{n \in \N} \subset \StandardXSpace[k][\infty](\Omega)$ be a Cauchy sequence.
It is well-known that there is a continuous function $f : \Omega \to \R^k$ such that
$f_n \to f$ uniformly.
In fact (see for instance \cite[Theorem 12.8]{TopologyWithApplications}), $f$ is uniformly continuous.
It remains to show that $f$ vanishes at infinity.
Let $\eps > 0$ be arbitrary, and choose $n \in \N$ such that $\|f - f_n\|_{\sup} \leq \frac{\eps}{2}$.
Since $f_n$ vanishes at $\infty$, there is $R > 0$ such that $|f_n(x)| \leq \frac{\eps}{2}$
for $x \in \Omega$ with $|x| \geq R$.
Therefore, $|f(x)| \leq \eps$ for such $x$, proving that $f \in \StandardXSpace[k][\infty](\Omega)$,
while $\|f - f_n\|_{\StandardXSpace[k][\infty](\Omega)} \to 0$ follows from the uniform convergence
$f_n \to f$.

\medskip{}

Finally, we prove that
\(
  \StandardXSpace[k][\infty](\Omega) = \{f|_{\Omega}: f \in C_{0}(\R^{d};\R^{k})\}
\).
By considering components it is enough to prove that
$\{f|_{\Omega}: f \in C_{0}(\R^{d})\} = \StandardXSpace[][\infty](\Omega)$.
To see that
$\{f|_{\Omega}: f \in C_{0}(\R^{d})\} \subset \StandardXSpace[][\infty](\Omega)$,
simply note that%
\footnote{For instance, \cite[Proposition 4.35]{FollandRA} shows that each function in
$C_0(\R^d)$ is a uniform limit of continuous, compactly supported functions,
\cite[Proposition (2.6)]{FollandAHA} shows that such functions are uniformly continuous,
while \cite[Theorem 12.8]{TopologyWithApplications} shows that the uniform continuity
is preserved by the uniform limit.}
if $f \in C_0 (\R^d)$, then $f$ is not only continuous, but in fact \emph{uniformly} continuous.
Therefore, $f|_{\Omega}$ is also uniformly continuous (and vanishes at infinity),
whence $f|_{\Omega} \in \StandardXSpace[][\infty](\Omega)$.


For proving
$\StandardXSpace[][\infty](\Omega) \subset \{f|_{\Omega}: f \in C_{0}(\R^{d})\}$,
we will use the notion of the \emph{one-point compactification}
$Z_\infty := \{\infty\} \cup Z$ of a locally compact Hausdorff space $Z$
(where we assume that $\infty \notin Z$); see \cite[Proposition~4.36]{FollandRA}.
The topology on $Z_\infty$ is given by
\(
  \mathcal{T}_Z := \{ U \colon U \subset Z \text{ open} \}
                   \cup \{ Z_\infty \setminus K \colon K \subset Z \text{ compact} \}
\).
Then, $(Z_\infty,\mathcal{T}_Z)$ is a compact Hausdorff space and
the topology induced on $Z$ as a subspace of $Z_\infty$ coincides with the original topolog on $Z$;
see \cite[Proposition~4.36]{FollandRA}.
Furthermore, if $A \subset Z$ is \emph{closed}, then a direct verification shows that
the relative topology on $A_\infty$ as a subset of $Z_\infty$ coincides with the topology
$\mathcal{T}_A$.

Now, let $g \in \StandardXSpace[][\infty](\Omega)$.
Since $g$ is uniformly continuous, it follows (see \cite[Lemma 3.11]{AliprantisBorderHitchhiker})
that there is a uniformly continuous function $\widetilde{g} : A \to \R$
satisfying $g = \widetilde{g}|_{\Omega}$, with $A := \overline{\Omega} \subset \R^{d}$
the closure of $\Omega$ in $\R^d$.

Since $g \in C_0(\Omega)$, it is not hard to see that $\widetilde{g} \in C_0(A)$.
Hence, \cite[Proposition~4.36]{FollandRA} shows that the function
$G : A_\infty \to \R$ defined by $G(x) = \widetilde{g}(x)$ for $x \in A$ and $G(\infty) = 0$
is continuous.
Since $A_\infty \subset (\R^d)_\infty$ is compact, the Tietze extension theorem
(see \cite[Theorem 4.34]{FollandRA}) shows that there is a continuous extension
$H : (\R^d)_\infty \to \R$ of $G$.
Again by \cite[Proposition~4.36]{FollandRA}, this implies that $f := H|_{\R^d} \in C_0(\R^d)$.
By construction, we have $g = f|_{\Omega}$.
\hfill$\square$

\subsection{Proof of Theorem~\ref{th:ApproximationSpacesWellDefinedDensity}}
\label{app:ApproximationSpacesWellDefinedDensity}

\subsubsection{Proof of Claims~\ref{enu:AppSetInLp1}-\ref{enu:AppSetInLp2}}

We use the following lemma.

\begin{lem}\label{lem:LocBoundedEtc}
  Let $\mathcal{C}$ be one of the following classes of functions:
  \begin{itemize}
    \item locally bounded functions;

    \item Borel-measurable functions;

    \item continuous functions;

    \item Lipschitz continuous functions;

    \item locally Lipschitz continuous functions.
  \end{itemize}
  If the activation function $\varrho$ belongs to $\mathcal{C}$,
  then any $f \in \NNreal^{\varrho,d,k}$ also belongs to $\mathcal{C}$.
\end{lem}

\begin{proof}
  First, note that each affine-linear map $T : \R^d \to \R^k$
  belongs to \emph{all} of the mentioned classes.
  Furthermore, note that since $\R^d$ is locally compact, a function $f : \R^d \to \R^k$ is
  locally bounded [locally Lipschitz] if and only if $f$ is bounded [Lipschitz continuous]
  on each \emph{bounded} set.
  From this, it easily follows that each class $\mathcal{C}$ is closed under composition.
  Finally, it is not hard to see that if $f_1, \dots, f_n : \R \to \R$
  all belong to the class $\mathcal{C}$, then so does
  $f_1 \otimes \cdots \otimes f_n : \R^n \to \R^n$.

  Combining these facts with the definition of the realization of a neural network,
  we get the claim.
\end{proof}

As $\varrho$ is locally bounded and Borel measurable, by Lemma~\ref{lem:LocBoundedEtc}
each $g \in \NNreal^{\varrho,d,k}$ is locally bounded and measurable.
As $\Omega$ is bounded, we get $g|_{\Omega} \in L_{p}(\Omega;\R^{k})$ for all $p \in (0,\infty]$,
and hence $g \in \StandardXSpace(\Omega)$ if $p < \infty$. This establishes claim~\ref{enu:AppSetInLp1}.
Finally, if $p = \infty$, then by our additional assumption that $\varrho$ is continuous,
$g$ is continuous by Lemma~\ref{lem:LocBoundedEtc}.
On the compact set $\overline{\Omega}$, $g$ is thus uniformly continuous and bounded,
so that $g|_{\Omega}$ is uniformly continuous and bounded as well, that is,
$g|_{\Omega} \in \StandardXSpace[k][\infty] (\Omega)$.
This establishes claim~\ref{enu:AppSetInLp2}.%
\hfill$\square$

\subsubsection{Proof of claims~\ref{enu:AppSetInDense1}-\ref{enu:AppSetInDense2}}


We first consider the case $p < \infty$.
Let $f \in \StandardXSpace (\Omega) = L_p (\Omega; \R^k)$ and $\eps > 0$.
For each $i \in \FirstN{k}$, extend the $i$-th component function $f_i$ by zero to a function
$g_i \in L_p(\R^d)$.
As is well-known (see for instance \cite[Chapter VI, Theorem 2.31]{Elstrodt}),
$C_c^\infty (\R^d)$ is dense in $L_p(\R^d)$, so that we find $h_i \in C_c^\infty (\R^d)$
satisfying $\| g_i - h_i \|_{L_p} < \eps$.
Choose $R > 0$ satisfying $\supp(h_i) \subset [-R,R]^d$ and ${\Omega \subset [-R,R]^d}$.
By the universal approximation theorem (Theorem~\ref{thm:PinkusUniversalApproximation}), we can find
$\gamma_i \in \NNreal^{\varrho,d,1}_{\infty,2,\infty} \subset \NNreal^{\varrho,d,1}_{\infty,L,\infty}$
satisfying $\| h_i - \gamma_i \|_{L_\infty ([-R,R]^d)} \leq \eps / (4R)^{d/p}$.
Note that the inclusion
$\NNreal^{\varrho,d,1}_{\infty,2,\infty} \subset \NNreal^{\varrho,d,1}_{\infty,L,\infty}$
used above is (only) true since we are considering \emph{generalized} neural networks,
and since $L \geq 2$.

Using the elementary estimate $(a + b)^p \leq (2 \max \{a, b \})^p \leq 2^p (a^p + b^p)$, we see
\[
  |\gamma_i (x) - g_i (x)|^p
  \leq \big( |\gamma_i(x) - h_i(x)| + |h_i(x) - g_i(x)| \big)^p
  \leq 2^p \Big( \frac{\eps^p}{(4R)^{d}} + |h_i(x) - g_i(x)|^p  \Big)
  \quad \forall \, x \in [-R,R]^d \, ,
\]
which easily implies
$\|\gamma_i - g_i\|_{L_p ([-R,R]^d)}^p
 \leq 2^p (\eps^p + \|h_i - g_i\|_{L_p([-R,R]^d)}^p)
 \leq 2^{1 + p} \eps^p$.

Lemma~\ref{lem:SummationLemma} shows that
$\gamma := (\gamma_1, \dots, \gamma_k) \in \NNreal^{\varrho,d,k}_{\infty,L,\infty}$,
whence $\gamma|_{\Omega} \in \Sigma_\infty (\StandardXSpace (\Omega), \varrho, \mathscr{L})$
by claims~\ref{enu:AppSetInLp1}-\ref{enu:AppSetInLp2}
of Theorem~\ref{th:ApproximationSpacesWellDefinedDensity}.
Finally, since $g_i|_{\Omega} = f_i$, we have
\[
  \| f - \gamma|_{\Omega} \|_{L_p (\Omega)}^p
  \leq \sum_{i=1}^k \| g_i - \gamma_i \|_{L_p ([-R,R]^d)}^p
  \leq 2^{1 + p} k \cdot \eps^p \, .
\]
Since $\eps > 0$ was arbitrary, this proves the desired density.

\medskip{}

Now, we consider the case $p = \infty$.
Let $f \in \StandardXSpace[k][\infty] (\Omega)$.
Lemma~\ref{lem:C0FunctionsCanBeExtended} shows that there is a continuous function
$g : \R^d \to \R^k$ such that $f = g|_{\Omega}$.
Since $L \geq 2$, we can apply the universal approximation theorem
(Theorem~\ref{thm:PinkusUniversalApproximation}) to each of the component functions $g_i$
of $g = (g_1,\dots,g_k)$ to obtain functions
$\gamma_i \in \NNreal^{\varrho,d,1}_{\infty,2,\infty} \subset \NNreal^{\varrho,d,1}_{\infty,L,\infty}$
satisfying $\|g_i - \gamma_i \|_{L_\infty ([-R,R]^d)} \leq \eps$,
where we chose $R > 0$ so large that $\Omega \subset [-R,R]^d$.
Lemma~\ref{lem:SummationLemma} shows that
$\gamma := (\gamma_1, \dots, \gamma_k) \in \NNreal^{\varrho,d,k}_{\infty,L,\infty}$,
whence $\gamma|_{\Omega} \in \Sigma_\infty (\StandardXSpace (\Omega), \varrho, \mathscr{L})$
by claims~\ref{enu:AppSetInLp1}-\ref{enu:AppSetInLp2}
of Theorem~\ref{th:ApproximationSpacesWellDefinedDensity}, since $\varrho$ is continuous.
Finally, since $g_i |_{\Omega} = f_i$, we have
\[
  \sup_{x \in \Omega} \| f(x) - \gamma(x) \|_{\ell^\infty}
  \leq \sup_{x \in [-R,R]^d} \,\,
         \max_{i \in \FirstN{k}}
           | g_i(x) - \gamma_i(x)|
  \leq \eps \, .
\]
Since $\eps > 0$ was arbitrary, this proves the desired density.
\hfill$\square$

\subsubsection{Proof of Claim (\ref{enu:AppSetDenseIndicator})}
\label{sub:DensityClaim}

Set $\mathcal{V} := \NNreal^{\varrho,d,1}_{\infty,L} \cap \StandardXSpace[](\R^d)$.
Lemma~\ref{lem:SummationLemma} easily shows that $\mathcal{V}$ is a vector space.
Furthermore, Lemma~\ref{lem:NetworkCalculus} shows that if $f \in \mathcal{V}$, $A \in \GL(\R^d)$,
and $b \in \R^d$, then $f (A \bullet + b) \in \mathcal{V}$ as well.
Clearly, all these properties also hold for $\overline{\mathcal{V}}$
instead of $\mathcal{V}$, where the closure is taken in $X_p(\R^d)$.

It suffices to show that $\mathcal{V}$ is dense in $\StandardXSpace[](\R^d)$.
Indeed, suppose for the moment that this is true.
Let $f \in \StandardXSpace(\Omega)$ be arbitrary.
By applying Lemma~\ref{lem:C0FunctionsCanBeExtended} to each of the component functions $f_i$ of $f$,
we see for each $i \in \FirstN{k}$ that there is a function $F_i \in \StandardXSpace[](\R^d)$
such that $f_i = F_i |_{\Omega}$.
Now, let $\eps > 0$ be arbitrary, and set $p_0 := \min \{1,p\}$.
Since $\mathcal{V}$ is dense in $\StandardXSpace[](\R^d)$,
there is for each $i \in \{1,\dots,k\}$ a function $G_i \in \mathcal{V}$ such that
$\|G_i - F_i\|_{L_p}^{p_0} \leq \eps^{p_0} / k$.
Lemma~\ref{lem:SummationLemma} shows
\(
  g := (G_1|_{\Omega},\dots,G_k|_{\Omega})
  \in \NNreal^{\varrho,d,k}_{\infty,L}(\Omega) \cap \StandardXSpace (\Omega)
  = \Sigma_\infty (\StandardXSpace(\Omega), \varrho, \mathscr{L})
\),
and it is not hard to see that
\(
  \| f - g \|_{\StandardXSpace(\Omega)}^{p_0}
  \leq \sum_{i=1}^k
         \| F_i - G_i\|_{\StandardXSpace(\Omega)}^{p_0}
  \leq \eps^{p_0}
\),
and hence $\| f - g\|_{\StandardXSpace(\Omega)} \leq \eps$.
As $\eps > 0$ and $g \in \StandardXSpace(\Omega)$ were arbitrary,
this proves that $\Sigma_\infty (\StandardXSpace(\Omega), \varrho, \mathscr{L})$
is dense in $\StandardXSpace(\Omega)$, as desired.

\medskip{}

It remains to show that $\mathcal{V} \subset \StandardXSpace[](\R^d)$ is dense.
To prove this, we distinguish three cases:

\medskip{}

\noindent
\textbf{Case 1 ($p \in [1,\infty)$):}
First, the existence of the ``radially decreasing $L_1$-majorant'' $\mu$ for $g$,
\cite[Lemma A.2]{LaugesenAffineSystemsSpanLebesgueSpaces} shows that
$P|g| \in L_\infty (\R^d) \subset L_p^{\mathrm{loc}}(\R^d)$, where
$P|g|$ is a certain \emph{periodization} of $|g|$ whose precise definition is immaterial for us.
Since $g \in L_p (\R^d)$ and $P|g| \in L_p^{\mathrm{loc}}(\R^d)$,
and $\int_{\R^d} g(x) \, dx \neq 0$,
\cite[Corollary 1]{LaugesenAffineSystemsSpanLebesgueSpaces} implies that
$\mathcal{V}_0 := \mathrm{span}\{ g_{j,k} \colon j \in \N, k \in \Z^d \}$
is dense in $L_p(\R^d)$, where $g_{j,k}(x) = 2^{jd/p} \cdot g(2^j x - k)$.
As a consequence of the properties of the space $\mathcal{V}$ that we mentioned above,
and since $g \in \overline{\mathcal{V}}$, we have $\mathcal{V}_0 \subset \overline{\mathcal{V}}$.
Hence, $\mathcal{V} \subset L_p (\R^d)$ is dense,
and we have $L_p(\R^d) = \StandardXSpace[](\R^d)$ since $p < \infty$.

\medskip{}

\noindent
\textbf{Case 2 ($p \in (0,1)$):}
Since $g \in L_1(\R^d) \cap L_p(\R^d)$ with $\int_{\R^d} g(x) \, d x \neq 0$,
\cite[Theorem 4 and Proposition 5(a)]{LaugesenAffineSynthesisQuasiBanach}
show that $\mathcal{V}_0 \subset L_p(\R^d)$ is dense, where the space
$\mathcal{V}_0$ is defined precisely as for $p \in [1,\infty)$.
The rest of the proof is as for $p \in [1,\infty)$.

\medskip{}

\noindent
\textbf{Case 3 ($p = \infty$):}
Note $\StandardXSpace[](\R^d) = C_0(\R^d)$.
Let us assume towards a contradiction that $\mathcal{V}$ is not dense in $C_0(\R^d)$.
By the Hahn-Banach theorem (see for instance \cite[Theorem 5.8]{FollandRA}),
there is a bounded linear functional $\varphi \in (C_0(\R^d))^\ast$
such that $\varphi \not\equiv 0$, but $\varphi \equiv 0$ on $\overline{\mathcal{V}}$.

By the Riesz representation theorem for $C_0$ (see \cite[Theorem 7.17]{FollandRA}),
there is a finite real-valued Borel-measure $\mu$ on $\R^d$
such that $\varphi(f) = \int_{\R^d} f(x) \, d \mu(x)$ for all $f \in C_0(\R^d)$.
Thanks to the Jordan decomposition theorem (see \cite[Theorem 3.4]{FollandRA}),
there are finite positive Borel measures $\mu_+$ and $\mu_-$ such that $\mu = \mu_+ - \mu_-$.

Let $f \in C_0 (\R^d)$ be arbitrary.
For $a > 0$, define $g_a : \R^d \to \R, x \mapsto a^d \, g(a x)$,
and note $T_x g_a \in \overline{\mathcal{V}}$ (and hence $\varphi(T_x g_a) = 0$)
for all $x \in \R^d$, where $T_x g_a (y) = g_a (y-x)$.
By Fubini's theorem and the change of variables $y = -z$, we get
\begin{equation}
  \begin{split}
    \int_{\R^d}
      (f \ast g_a)(x)
    \, d \mu(x)
    & = \int_{\R^d}
          \int_{\R^d}
            f(z) \, g_a (x-z)
          \, d z
        \, d \mu (x) \\
    & = \int_{\R^d}
          f(-y)
          \int_{\R^d}
            g_a (y + x)
          \, d \mu (x)
        \, d y
      = \int_{\R^d}
          f(-y) \, \varphi( T_{-y} \, g_a )
        \, d y
      = 0
  \end{split}
  \label{eq:DensityInBanachLpMainIdentityC0Case}
\end{equation}
for all $a \geq 1$.
Here, Fubini's theorem was applied to each of the integrals
$\int (f \ast g_a)(x) \, d \mu_{\pm} (x)$, which is justified since
\begin{align*}
  \int \int |f(z) \, g_a (x-z)| \, d z \, d \mu_{\pm} (x)
  & \leq \mu_{\pm}(\R^d) \, \|f\|_{L_{\infty}} \, \|T_{z} g_a\|_{L_{1}}
  = \mu_{\pm}(\R^d) \, \|f\|_{L_{\infty}} \, \|g_a\|_{L_1} < \infty \, .
\end{align*}
Now, since $f \in C_0(\R^d)$ is bounded and uniformly continuous,
\cite[Theorem 8.14]{FollandRA} shows $f \ast g_a \to f$ uniformly as $a \to \infty$.
Therefore, \eqref{eq:DensityInBanachLpMainIdentityC0Case} implies
\(
  \varphi(f)
  = \int_{\R^d} f(x) \, d \mu (x)
  = \lim_{a \to \infty} \int_{\R^d} (f \ast g_a) (x) \, d \mu (x)
  = 0
\),
since $\mu$ is a finite measure.
This implies $\varphi \equiv 0$ on $C_0 (\R^d)$, which is the desired contradiction.
\hfill$\square$

\subsection{Proof of Lemma~\ref{lem:ApproximationOfIndicatorCube}}
\label{app:ApproximationOfIndicatorCube}



  \textbf{Part (1):}
  Define
  \[
    t : \R \to \R,
        x \mapsto \sigma \big( x/\eps \big)-\sigma \big( 1+(x-1)/\eps \big) \, .
  \]
  A straightforward calculation using the properties of $\sigma$ shows that
  \begin{equation}
    t(x) =
    \begin{cases}
      0,                         & \text{if } x \in \R \setminus [0,1], \\
      1,                         & \text{if } x \in [\varepsilon, 1-\varepsilon].
    \end{cases}
    \label{eq:tProperty}
  \end{equation}
  We claim that $0 \leq t \leq 1$.
  To see this, first note that if $r \geq 1$, then $\sigma(x - r) \leq \sigma(x)$ for all $x \in \R$.
  Indeed, if $x \leq r$, then $\sigma(x - r) = 0 \leq \sigma(x)$;
  otherwise, if $x > r$, then $x \geq 1$, and hence $\sigma(x - r) \leq 1 = \sigma(x)$.
  Since $r := \frac{1}{\eps} - 1 \geq 1$, we thus see that
  $t(x) = \sigma(\frac{x}{\eps}) - \sigma(\frac{x}{\eps} - r) \geq 0$ for all $x \in \R$.
  Finally, we trivially have $t(x) \leq \sigma(\frac{x}{\eps}) \leq 1$ for all $x \in \R$.

  Now, if we define
  \[
    g_0 : \R^d \to \R,
          x \mapsto  \sigma \left( 1+ \sum_{i=1}^d t (x_i)-d \right),
  \]
  we see $0 \leq g_0 \leq 1$.
  Furthermore, for $x \in [\varepsilon, 1-\varepsilon]^d$, we have
  $t(x_i) = 1$ for all $i \in \FirstN{d}$, whence $g_0(x) = 1$.
  Likewise, if $x \notin [0,1]^d$, then $t (x_i) = 0$ for at least one $i \in \FirstN{d}$.
  Since $0 \leq t (x_i) \leq 1$ for all $i$, this implies $\sum_{i=1}^d t (x_i) - d \leq -1$,
  and thus $g_0(x) = 0$.
  All in all, and because of $0 \leq g_0 \leq 1$,
  these considerations imply that $\supp(g_0) \subset [0,1]^{d}$ and
  \begin{equation}
    |g_0(x) - \Indicator_{[0,1]^d} (x)|
    \leq \Indicator_{[0,1]^d \setminus [\varepsilon, 1-\varepsilon]^d} (x)
    \quad \forall \, \, x \in \R^d \, .
    \label{eq:PointwiseIndicatorApproximationInProof}
  \end{equation}
  Now, for proving the general case of Part (1), let $h := g_0$,
  while $h := t$ in case of $d = 1$.
  As a consequence of Equations~\eqref{eq:PointwiseIndicatorApproximationInProof}
  and \eqref{eq:tProperty} and of $0 \leq t \leq 1$, we then see that
  Condition~\eqref{eq:PointwiseIndicatorApproximation} is satisfied in both cases.
  Thus, all that needs to be shown is that $h = g_0 \in \NNreal^{\varrho,d,1}_{2dW(N+1), 2L-1, (2d+1)N}$
  or that $h = t \in \NNreal^{\varrho,1,1}_{2W,L,2N}$ in case of $d = 1$.
  We will verify both of these properties in the proof of Part (2) of the lemma.

  \medskip{}

  \textbf{Part (2):}
  We first establish the claim for the special case $[a,b]= [0,1]^{d}$.
  With $\lambda$ denoting the $d$-dimensional Lebesgue measure,
  and with $h$ as constructed in Part (1), we deduce from \eqref{eq:PointwiseIndicatorApproximation}
  that
  \[
    \| h - \Indicator_{[0,1]^d} \|_{L_p}^{p}
    \leq \lambda ([0,1]^d \setminus [\varepsilon, 1-\varepsilon]^d)
    =    [1 - (1 - 2\varepsilon)^d] \, .
  \]
  Since the right-hand side vanishes as $\varepsilon \to 0$, this proves the
  claim for the special case $[a,b] = [0,1]^d$, once we show $h = \Realization (\Phi)$
  for $\Phi$ with appropriately many layers, neurons, and nonzero weights.

  By assumption on $\sigma$, there is $L_0 \leq L$ such that $\sigma = \Realization(\Phi_\sigma)$
  for some $\Phi_\sigma \in \NNsymbol^{\varrho,1,1}_{W,L_0,N}$ with $L(\Phi_\sigma) = L_0$.
  For $i \in \FirstN{d}$ set $f_{i, 1} : \R^d \to \R, x \mapsto \sigma(\frac{x_i}{\eps})$
  and $f_{i, 2} : \R^d \to \R, x \mapsto -\sigma(1 + \frac{x_i - 1}{\eps})$.
  By Lemma~\ref{lem:NetworkCalculus}-(\ref{enu:PrePostAffine}) there exist
  $\Psi_{i,1}, \Psi_{i,2} \in \NNsymbol^{\varrho,d,1}_{W,L_0,N}$
  with $L(\Psi_{i,1}) = L(\Psi_{i,2}) = L_0$ for any $i \in \FirstN{d}$
  such that $f_{i,1} = \Realization(\Psi_{i,1})$ and $f_{i,2} = \Realization(\Psi_{i,2})$.
  Lemma~\ref{lem:SummationLemma}-(\ref{enu:LinComb}) then shows that
  \[
    F :
    \R^d \to \R,
    x \mapsto \sum_{i = 1}^d t(x_i)
            = \sum_{i=1}^d f_{i,1}(x) + \sum_{i=1}^d f_{i,2}(x)
  \]
  satisfies $F = \Realization(\Phi_F)$ for some $\Phi_F \in \NNsymbol^{\varrho,d,1}_{2dW,L_0,2dN}$
  with $L(\Phi_F) = L_0$.
  Hence, Lemma~\ref{lem:NetworkCalculus}-(\ref{enu:PrePostAffine}) shows that
  $G : \R^d \to \R, x \mapsto 1 + \sum_{i=1}^d t(x_i) - d$ satisfies
  $G = \Realization(\Phi_G)$ for some $\Phi_G \in \NNsymbol^{\varrho,d,1}_{2dW,L_0,2dN}$
  with $L(\Phi_G) = L_0$.

  In case of $d = 1$, set $L' := L_0$ and recall that $h = t = F$,
  where we saw above that $F = \Realization(\Phi_F)$ and
  $\Phi_F \in \NNsymbol^{\varrho,1,1}_{2W,L_0,2N}$ with $L(\Phi_F) = L_0$.
  For general $d \in \N$ set $L' := 2 L_0 - 1$ and recall that $h = g_0 = \sigma \circ G$.
  Hence, Lemma~\ref{lem:NetworkCalculus}-(\ref{enu:ComposLessDepth}) shows
  $h = \Realization(\Phi_h)$ for some $\Phi_h \in \NNsymbol^{\varrho,d,1}$
  with $L(\Phi_h) = L'$, $N(\Phi_h) \leq (2d+1)N$ and
  $W(\Phi_h) \leq 2d W + \max \{ 2 d N, d \} W \leq 2 d W (N+1)$.

  \medskip{}

  It remains to transfer the result from $[0,1]^d$ to the general case $[a,b]$.
  To this end, define the invertible affine-linear map
  \[
      T_0 : \R^d \to \R^d,
            x \mapsto \left(
                          (b_i - a_i)^{-1} \cdot (x_i - a_i)
                      \right)_{i \in \FirstN{d}} \, .
  \]
  A direct calculation shows
  $\Indicator_{[0,1]^d} \circ T_0 = \Indicator_{T_0^{-1} [0,1]^d} = \Indicator_{[a,b]}$.
  Since $\|T_{0}\|_{\ell^{0,\infty}_{\ast}} =1$, the first part of Lemma~\ref{lem:NetworkCalculus}
  shows that
  $g := h \circ T_0 = \Realization(\Phi)$ for some
  $\Phi \in \NNsymbol^{\varrho,d,1}_{2dW (N+1), 2 L_0 - 1, (2d+1) N}$ with $L(\Phi) = 2 L_0 - 1 = L'$
  (resp.~$g := h \circ T_0 = \Realization(\Phi)$ for some
  $\Phi \in \NNsymbol^{\varrho,1,1}_{2W, L_0, 2N}$ with $L(\Phi) = L_0 = L'$ in case of $d = 1$)
  with $h$ as above.
  Moreover, by an application of the change-of-variables-formula, we get
  \begin{align*}
    \| g - \Indicator_{[a,b]} \|_{L_p}
    & = \| h \circ T_0 - \Indicator_{[0,1]^d} \circ T_0 \|_{L_p} \\
    & = \big| \det \diag \big( (b_i - a_i)^{-1} \big)_{i \in \FirstN{d}} \big|^{-1/p}
        \cdot \| g - \Indicator_{[0,1]^d} \|_{L_p} 
      = \| g - \Indicator_{[0,1]^d} \|_{L_p}
        \cdot \prod_{i=1}^d (b_i - a_i)^{1/p} \, .
  \end{align*}
  As seen above, the first factor can be made arbitrarily small
  by choosing $\eps \in (0, \frac{1}{2})$ suitably.
  Since the second factor is constant, this proves the claim.
\hfill$\square$

\section{Proofs for Section~\ref{sec:AppSpacesReLU}}

\subsection{Proof of Lemma~\ref{lem:PiecewisePolynomialCont}}
\label{app:PiecewisePolynomialCont}

We begin with three auxiliary results.

\begin{lem}\label{lem:DerivativeLocallyUniformConvergence}
  Let $f : \R \to \R$ be \emph{continuously} differentiable.
  Define $f_h : \R \to \R, x \mapsto h^{-1} \cdot (f(x+h) - f(x))$
  for $h \in \R \setminus \{0\}$.
  Then $f_h \to f$ as $h \to 0$ with locally uniform convergence on $\R$.
\end{lem}

\begin{proof}
  This is an easy consequence of the mean-value theorem, using that $f'$ is locally uniformly
  continuous.
  For more details, we refer to \cite[Theorem 4.14]{LaxCalculusWithApplications}.
\end{proof}

Since $\varrho_{r+1}$ is continuously differentiable with $\varrho_{r+1} ' = \varrho_r$,
the preceding lemma immediately implies the following result.

\begin{cor}\label{cor:NestednessMainIdea}
  For $r \in \N$, $h > 0$,
  \(
    \sigma_h : \R \to \R,
                x \mapsto (r+1)^{-1}
                          \cdot h^{-1}
                          \cdot \big(
                                  \varrho_{r+1} (x + h)
                                  - \varrho_{r+1} (x)
                                \big)
  \)
  we have $\sigma_{h} = \Realization(\Psi_{h})$ where $\Psi_{h} \in \SNNsymbol_{4,2,2}^{\varrho_{r+1},1,1}$,
  $L(\Psi_{h}) = 2$, and \(\lim_{h \to 0}\sigma_h =  \varrho_{r}\)
  with locally uniform convergence on $\R$.
\end{cor}

We need one more auxiliary result for the proof of Lemma~\ref{lem:PiecewisePolynomialCont}.

\begin{cor}\label{cor:NestednessMainIdeaStep2}
  For any $d,k,r \in \N$, $j \in \N_{0}$, $W,N \in \N_{0}$, $L \in \N$ we have
  \begin{equation}\label{eq:NestedNNFamiliesReLUPowers}
  \NNreal^{\varrho_{r},d,k}_{W,L,N} \subset \overline{\NNreal^{\varrho_{r+j},d,k}_{4^{j}W,L,2^{j}N}}
  \end{equation}
  where closure is with respect to locally uniform convergence on $\R^d$.
\end{cor}

\begin{proof}
We prove by induction on $\delta$ that the result holds for any $0 \leq j \leq \delta$.
This is trivial for $\delta=0$.
By Corollary~\ref{cor:NestednessMainIdea} we can apply
Lemma~\ref{lem:Recursivity} to $\varrho := \varrho_{r+1}$ and $\sigma := \varrho_{r}$
(which is continuous) with $w = 4$, $\ell=2$, $m=2$.
This yields for any $W,N \in \N_{0}$, $L \in \N$ that
\(
  \NNreal^{\varrho_{r},d,k}_{W,L,N}
  \subset \overline{\NNreal^{\varrho_{r+1},d,k}_{4W,L,2N}},
\)
which shows that our induction hypothesis is valid for $\delta = 1$.
Assume now that the hypothesis holds for some $\delta \in \N$,
and consider $W,N \in \N_{0}$, $r,L \in \N$, $0 \leq j \leq \delta+1$.
If $j \leq \delta$ then the induction hypothesis yields~\eqref{eq:NestedNNFamiliesReLUPowers},
so there only remains to check the case $j = \delta+1$.
By the induction hypothesis, for $r' = r+\delta$, $W' = 4^{\delta}W$, $N' = 2^{\delta}N$, $j=1$ we have
\(
  \NNreal^{\varrho_{r+\delta},d,k}_{4^{\delta}W,L,2^{\delta}N}
  \subset \overline{\NNreal^{\varrho_{r+\delta+1},d,k}_{4^{\delta+1}W,L,2^{\delta+1}N}}.
\)
Finally,
\(
 \NNreal^{\varrho_{r},d,k}_{W,L,N}
 \subset \overline{\NNreal^{\varrho_{r+\delta},d,k}_{4^{\delta}W,L,2^{\delta}N}}
 \subset \overline{\NNreal^{\varrho_{r+\delta+1},d,k}_{4^{\delta+1}W,L,2^{\delta+}N}}
\)
by the induction hypothesis for $j=\delta$.
\end{proof}

\begin{proof}[Proof of Lemma~\ref{lem:PiecewisePolynomialCont}] 
The proof is by induction on $n$.
For $n=1$, $\varrho$ is a polynomial of degree at most $r$.
By Lemma~\ref{lem:ReLUPowerRepresentsIdentity}, $\varrho_{r}$ can represent any such polynomial
with $2(r+1)$ terms, whence $\varrho \in \NNreal^{\varrho_{r},1,1}_{4(r+1),2,2(r+1)}$.
When $r=1$, $\varrho$ is an affine function; hence there are $a,b \in \R$
such that $\varrho(x) = b+ax = b+ a\varrho_{1}(x)-a\varrho_{1}(-x)$ for all $x$, showing that
$\varrho \in \SNNreal^{\varrho_{1},1,1}_{4,2,2} = \SNNreal^{\varrho_1,1,1}_{2(n+1),2,n+1}$.

Assuming the result true for $n \in \N$, we prove it for $n+1$.
Consider $\varrho$ made of $n+1$ polynomial pieces:
$\R$ is the disjoint union of $n+1$ intervals $I_{i}$, $0 \leq i \leq n$ and there are polynomials
$p_{i}$ such that $\varrho(x) =  p_{i}(x)$ on the interval $I_{i}$ for $0 \leq i \leq n$.
Without loss of generality order the intervals by increasing ``position''
and define $\bar{\varrho}(x) = \varrho(x)$ for $x \in \cup_{i=0}^{n-1} I_{i} = \R \setminus I_{n}$,
and $\bar{\varrho}(x) = p_{n-1}(x)$ on $I_{n}$.
It is not hard to see that $\bar{\varrho}$ is continuous and made of $n$ polynomial pieces,
the last one being $p_{n-1}(x)$ on $I_{n-1} \cup I_{n}$.
Observe that $\varrho(x) = \bar{\varrho}(x) + f(x - t_{n})$
where $\{t_{n}\} = \overline{I_{n-1}} \cap \overline{I_{n}}$ is the breakpoint between the intervals
$I_{n-1}$ and $I_{n}$, and
\[
  f(x) := \varrho(x + t_n) - \bar{\varrho}(x + t_n)
  = \begin{cases}
      0                                         & \text{for}\ x < 0    \\
      p_{n}(x + t_n)-p_{n-1}(x + t_n) & \text{for}\ x \geq 0.
    \end{cases}
\]
Note that $q(x) := p_{n}(x + t_n) - p_{n-1}(x + t_n)$ satisfies $q(0) = f(0) = 0$,
since $\varrho$ is continuous.
Because $q$ is a polynomial of degree at most $r$, there are $a_i \in \R$ such that
$q(x) = \sum_{i=1}^{r} a_i \, x^i$.
This shows that $f = \sum_{i=1}^{r} a_i \varrho_i$.
In case of $r = 1$, this shows that $f \in \SNNreal^{\varrho_1,1,1}_{2,2,1}$.
For $r \geq 2$, since $\varrho_i \in \NNreal^{\varrho_i,1,1}_{2,2,1}$,
Corollary~\ref{cor:NestednessMainIdeaStep2} yields
\(
  \varrho_i \in \overline{\NNreal^{\varrho_r,1,1}_{2 \cdot 4^{r-i},2,2^{r-i}}}
\),
where the closure is with respect to the topology of locally uniform convergence.
Observing that $2\sum_{i=1}^{r}4^{r-i} = 2 \cdot (4^{r}-1)/3 = w$
and $\sum_{i=1}^{r}2^{r-i} = 2^{r}-1 = m$, Lemma~\ref{lem:SummationLemma}-(\ref{enu:LinComb})
implies that%
\footnote{This implicitly uses that $\varrho_i$ is \emph{not} affine-linear, so that
\(
  \varrho_i
  \in \overline{\NNreal^{\varrho_r,1,1}_{2 \cdot 4^{r-i},2,2^{r-i}}
      \setminus \NNreal^{\varrho_r,1,1}_{\infty,1,\infty}}
\).}
$f \in \overline{\NNreal^{\varrho_{r},1,1}_{w,2,m}}$.
Since $P: \R \to \R, x \mapsto x+t_n$ is affine with
$\|P\|_{\ell^{0,\infty}} = \|P\|_{\ell^{0,\infty}_{\ast}}=1$, by the induction hypothesis,
Lemma~\ref{lem:NetworkCalculus}-(\ref{enu:PrePostAffine})
and Lemma~\ref{lem:SummationLemma}-(\ref{enu:LinComb}) again, we get
\[
  \varrho(\bullet)
  = \bar{\varrho}(\bullet) + f(\cdot - t_n)
  \in \overline{\NNreal^{\varrho_{r},1,1}_{4(r+1)+(n-1)w+w,2,2(r+1)+(n-1)m+m}}
  =  \overline{\NNreal^{\varrho_{r},1,1}_{4(r+1)+(n+1-1)w,2,2(r+1)+(n+1-1)m}}.
\]
For $r=1$, it is not hard to see
\(
  \varrho
  \in \SNNreal^{\varrho_1,1,1}_{2(n+1)+2,2,n+1+1}
  =   \SNNreal^{\varrho_1,1,1}_{2((n+1)+1),2,(n+1)+1}
\).
\end{proof}

\subsection{Proof of Lemma~\ref{lem:SplineVsReLUPow}}
\label{app:SplineVsReLUPow}


First we show that if $s \in \N$ and if $\varrho \in \Spline^s$ is not a polynomial, then
there are $\alpha,\beta,t_{0} \in \R$, $\varepsilon>0$ and $p$ a polynomial of degree at most $s-1$
such that
\begin{equation}\label{eq:PiecewiseAffineStep1}
  \varrho_{s}(z)
  = \alpha \varrho(t_{0}+z) + \beta \varrho(t_{0}-z) -p(z)
  \quad \forall z \in [-\varepsilon,+\varepsilon].
\end{equation}
Consider any $t_{0} \in \R$.
Since $\varrho \in \Spline^{s}$, there are $\varepsilon > 0$ and two polynomials $p_{-},p_{+}$
of degree at most $s$, with matching $s-1$ first derivatives at $t_{0}$, such that
\[
  \varrho(x)
  = \begin{cases}
      p_{+}(x)& \text{for}\ x \in [t_{0},t_{0}+\varepsilon]\\
      p_{-}(x)& \text{for}\ x \in [t_{0}-\varepsilon, t_{0}].
    \end{cases}
\]
Since $\varrho$ is not a polynomial, there is $t_{0}$ such that the $s$-th derivatives
of $p_{\pm}$ at $t_{0}$ do not match, i.e.
$a_{-} := p^{(s)}_{-}(t_{0})/s! \neq p^{(s)}_{+}(t_{0})/s! =: a_{+}$.
A Taylor expansion yields
\[
  \varrho(t_0+z)
  = \begin{cases}
      q(z)+ a_{+} z^s& \text{for}\ z \in [0,\varepsilon]\\
      q(z)+ a_{-} z^s& \text{for}\ z \in [-\varepsilon,0],
    \end{cases}
\]
where $q(z) := \sum_{n=0}^{s-1}  p_{\pm}^{(n)}(t_{0})z^{n}/n!$ is a polynomial of degree at most $s-1$.
As a result, for $|z| \leq \varepsilon$
\[
  a_{+} \cdot [\varrho(t_{0}+z) - q(z)] - (-1)^{s} \, a_{-} \cdot [\varrho(t_{0}-z) - q(-z)]
  = \begin{cases}
      (a_{+}^{2}-a_{-}^{2}) \cdot z^s & \text{for}\ z \in [0,\varepsilon] \\
      0                               & \text{for}\ z \in [-\varepsilon,0]
    \end{cases}
  = (a_{+}^{2}-a_{-}^{2}) \cdot \varrho_s(z).
\]
Since $a_{+} \neq a_{-}$, setting $\alpha := a_{+}/(a_{+}^{2}-a_{-}^{2})$
and $\beta := (-1)^{s+1} a_{-}/(a_{+}^{2}-a_{-}^{2})$,
as well as $p(x) := \alpha q(z)+\beta q(-z)$ we get as claimed
\(
  \varrho_s(z) = \alpha \varrho(z+t_{0}) + \beta \varrho(-z+t_{0}) -p(z)
\)
for every $|z| \leq \varepsilon$.

\medskip{}

Now consider $r \in \N$. Given $R > 0$ we now set
\[
  f_{R}(x) := (\tfrac{R}{\varepsilon})^{r}
              \left[
                \alpha\varrho(\varepsilon x/R+t_{0})
                +\beta \varrho(-\varepsilon x/R+t_{0})
                - p(\eps x / R)
              \right]
\]
with $\alpha,\beta,t_{0},\varepsilon,p$ from~\eqref{eq:PiecewiseAffineStep1}.
Observe that
\(
  \varrho_r(x) = (R/\varepsilon)^{r} \varrho_r(\varepsilon x/R) = f_{R}(x)
\)
for all $x \in [-R,R]$, so that $f_{R}$ converges locally uniformly to $\varrho_{r}$ on $\R$.

We show by induction on $r \in \N$ that $f_{R} \in \NNreal^{\varrho,1,1}_{w,2,m}$
where $w = w(r),m = m(r) \in \N$ only depend on $r$.
For $r=1$, this trivially holds as the polynomial $p$ in~\eqref{eq:PiecewiseAffineStep1} is a constant;
hence $f_{R} \in \NNreal^{\varrho,1,1}_{4,2,2}$.

Assuming the result true for some $r \in \N$ we now prove it for $r+1$.
Consider $\varrho \in \Spline^{r+1}$ that is not a polynomial.
The polynomial $p$ in~\eqref{eq:PiecewiseAffineStep1} with $s=r+1$ is of degree at most $r$;
hence by Lemma~\ref{lem:ReLUPowerRepresentsIdentity} there are $c,a_{i},b_{i},c_{i} \in \R$
such that
\(
  p(x) = c+ \sum_{i=1}^{r+1} a_{i} \, \varrho_r(b_{i} x + c_{i})
\)
for all $x \in \R$.
Now, observe that since $\varrho \in \Spline^{r+1}$ is not a polynomial,
its derivative satisfies $\varrho' \in \Spline^{r}$ and is not a polynomial either.
The induction hypothesis yields $\varrho_r \in \overline{\NNreal^{\varrho',1,1}_{w,2,m}}$
for $w=w(r),m=m(r) \in \N$.
It is not hard to check that this implies $p \in \overline{\NNreal^{\varrho',1,1}_{2(r+1)w,2,(r+1)m}}$.
Finally, as $\varrho'(x)$ is the locally uniform limit of $(\varrho(x+h)-\varrho(x))/h$ as $h \to 0$
(see Lemma~\ref{lem:DerivativeLocallyUniformConvergence}),
we obtain $p \in \overline{\NNreal^{\varrho,1,1}_{4(r+1)w,2,2(r+1)m}}$
thanks to Lemma~\ref{lem:Recursivity}.
Combined with the definition of $f_{R}$ we obtain
$f_{R} \in \overline{\NNreal^{\varrho,1,1}_{4(r+1)w+4,2,2(r+1)m+2}}$.

Finally we quantify $w, m$: First of all, note that $w(1) = 4 \leq 5$ and $m(1) = 2 \leq 3$;
furthermore, $w(r+1) \leq 4(r+1) w(r)+4 \leq 5(r+1) w(r)$
and $m(r+1) \leq 2(r+1)m+2 \leq 3(r+1)m$.
An induction therefore yields $w(r) \leq 5^r r!$ and $m(r) \leq 3^r r!$.
\hfill$\square$

\subsection{Proof of Lemma~\ref{lem:NetworkThatLocalizesAnyFunction}}
\label{app:NetworkThatLocalizesAnyFunction}

\textbf{Step 1:}
In this step, we
construct $\theta_{R,\delta} \in \NNreal^{\varrho_{r},d,1}_{w,\ell,m}$ satisfying
\begin{equation}
  |\theta_{R,\delta} (x) - \Indicator_{[-R,R]^{d}} (x)|
  \leq 2 \cdot \Indicator_{[-R-\delta,R+\delta]^{d} \setminus [-R,R]^{d}}
  \qquad \forall \, x \in \R^d \, ,
  \label{eq:ApproximateBallCutoffAccuracy}
\end{equation}
with $\ell=3$ (resp. $\ell=2$ if $d=1$) and with $w,m$ only depending on $d$ and $r$.

The affine map
\(
  P: \R^{d}\to \R^{d},
     x = (x_{i})_{i=1}^{d} \mapsto \left(\tfrac{x_{i}}{2(R+\delta)}+\tfrac{1}{2}\right)_{i=1}^{d}
\)
satisfies $\|P\|_{\ell^{0,\infty}} = \|P\|_{\ell^{0,\infty}_\ast} = 1$.
For $x \in \R^d$, we have $x \in [-R-\delta,R+\delta]^{d}$ if and only if $P(x) \in [0,1]^{d}$,
and $x \in [-R,R]^{d}$ if and only if $P(x) \in [\varepsilon,1-\varepsilon]^{d}$,
where $\varepsilon := \tfrac{2\delta}{2(R+\delta)}$;
thus, $\Indicator_{[-R,R]^d} (P^{-1} x) = \Indicator_{[\eps, 1-\eps]^d}(x)$
for all $x \in \R^d$.

Next, by combining Lemmas~\ref{lem:ExactSquashingReLUPower} and \ref{lem:ApproximationOfIndicatorCube}
(see in particular Equation~\eqref{eq:PointwiseIndicatorApproximation}),
we obtain $f \in \NNreal^{\varrho_{r},d,1}_{w,\ell,m}$
(with the above mentioned properties for $w,\ell,m$ and $m \geq d$) such that
\(
    |f(x)-\Indicator_{[0,1]^{d}} (x)|
    \leq  \Indicator_{[0,1]^{d} \setminus [\varepsilon,1-\varepsilon]^{d}}
\)
for all $x \in \R^d$.
Therefore, the function $\theta_{R,\delta} := f \circ P$ satisfies
\begin{align*}
  |\theta_{R,\delta} (x) - \Indicator_{[-R,R]^d}(x)|
  & = |f(P x) - \Indicator_{[-R,R]^d} (P^{-1} P x)| \\
  & \leq |f (P x) - \Indicator_{[0,1]^d}(P x)|
         + |\Indicator_{[0,1]^d} (P x) - \Indicator_{[\eps,1-\eps]^d} (P x)| \\
  & \leq 2 \cdot \Indicator_{[-R-\delta, R+\delta]^d \setminus [-R,R]^d} (x)
\end{align*}
for all $x \in \R^{d}$.
Finally, by Lemma~\ref{lem:NetworkCalculus}-(\ref{enu:PrePostAffine}), we have
$\theta_{R,\delta} \in \NNreal^{\varrho_{r},d,1}_{w,\ell,m}$.

\textbf{Step 2:}
Consider $g \in \NNreal^{\varrho_{r},d,k}_{W,L,N}$
and define $g_{R,\delta} (x) := \theta_{R,\delta} (x) \cdot g(x)$ for all $x \in \R^d$.
The desired estimate \eqref{eq:SpecialCutoffMainEstimateEndToEnd} is an easy consequence
of \eqref{eq:ApproximateBallCutoffAccuracy}.
It only remains to show that one can implement $g_{R,\delta}$
with a $\varrho_{r}$-network of controlled complexity.

Since we assume $W \geq 1$ we can use  Lemma~\ref{lem:BoundingLayersAndNeuronsByWeights};
combining it with Equation~\eqref{eq:LayersBoundedByNeurons}
we get $g \in \NNreal^{\varrho_{r},d,k}_{W,L',N'}$ with $L' = \min \{ L,W,N+1 \}$
and $N' = \min \{ N,W \}$.
Lemma~\ref{lem:SummationLemma}-(\ref{enu:Cartesian}) yields
$(\theta_{R,\delta},g) \in \NNreal^{\varrho_r,d,k+1}_{w', L'', m'}$
with $L'' = \max \{ L',\ell \}$ as well as
$w' = W + w + \min \{ d,k \} \cdot (L''-1)$ and
$m' = N' + m + \min \{ d,k \} \cdot (L''-1)$.
Since $L''-1 = \max \{ L'-1,\ell-1 \} \leq \max \{ W-1,\ell-1 \} \leq W+\ell-2$
and $N' \leq W$, we get
\begin{align*}
  w' &\leq W + w + \min \{ d,k \} \cdot (W+\ell-2) = W \cdot (1+\min \{ d,k \}) +c_{1} \\
  m' &\leq W + m + \min \{ d,k \} \cdot (W+\ell-2) = W \cdot (1+\min \{ d,k \}) +c_{2}.
\end{align*}
where $c_{1},c_{2}$ only depend on $d,k,r$.

As $r \geq 2$, Lemma~\ref{lem:ReLUPowerRepresentsIdentity} shows that $\varrho_{r}$
can represent any polynomial of degree two with ${n=2(r+1)}$ terms.
Thus, Lemma~\ref{lem:MultNetwork} shows that the multiplication map
$m : \R \times \R^k \to \R^k, (x,y) \mapsto x \cdot y$ satisfies
$m \in \NNreal^{\varrho_r, 1+k, k}_{12k(r+1), 2,4k(r+1)}$.
Finally, Lemma~\ref{lem:NetworkCalculus}-(\ref{enu:ComposLessDepth}) proves that
$g_{R,\delta} =  m \circ (\theta_{R,\delta},g) \in \NNreal^{\varrho_r,d,k}_{w'',L''',m''}$,
where ${L''' = L''+1}$ and $m'' = m'+ 4k(r+1) = N' + m + \min \{ d,k \} \cdot (L''-1) + 4k(r+1)$
as well as ${w'' = w' + \max \{ m',d \} \cdot 12 k(r+1)}$.

As $L'' = \max \{ L',\ell \} \leq \max \{ L,\ell \}$
we have $L''' \leq \max \{ L+1,4 \}$
(respectively $L''' \leq \max \{ L+1,3 \}$ if $d=1$).
Furthermore, since $m' \geq m \geq d$ we have $\max \{ m',d \} = m'$.
Because of $W \geq 1$, we thus see that
\[
  w'' =    w'+m' \cdot 12 k(r+1)
      \leq W \cdot (1 + \min \{ d,k \}) \cdot (1 + 12k(r+1)) + c_{3}
      \leq c_{4}W
\]
where $c_{3},c_{4}$ only depend on $d,k,r$.
Finally, $L''-1 = \max \{ L'-1,\ell-1 \} \leq \max \{ N,\ell-1 \} \leq N + \ell - 1$.
Since $N' \leq N$, we get $m'' \leq N \cdot (1+\min \{ d,k \}) + c_{5} \leq c_{6} N$
where again $c_{5},c_{6}$ only depend on $d,k,r$.
To conclude, we set $c := \max \{ c_{4},c_{6} \}$.
\hfill$\square$

\subsection{Proof of Proposition~\ref{prop:HighReLUPowersApproximateLowPowersUnbounded}}
\label{sub:ReLUPowerNestingUnbounded}

When $r=1$ and $\varrho \in \NNreal^{\varrho_{r},1,1}_{\infty,2,m}$ the result follows from
Lemma~\ref{lem:RecursiveNNsets}.

Now, consider $f \in \NNreal_{W,L,N}^{\varrho, d, k}$ such that $f|_{\Omega} \in X$.
Since $\varrho \in \overline{\NNreal^{\varrho_{r},1,1}_{\infty,2,m}}$,
Lemma~\ref{lem:Recursivity} shows that
\begin{equation}\label{eq:Tmp0}
  \NNreal_{W, L,N}^{\varrho, d, k} \subset \overline{\NNreal_{Wm^{2}, L,Nm}^{\varrho_r, d, k}},
  \quad \text{ with closure in the topology of locally uniform convergence on } \R^d.
\end{equation}
For bounded $\Omega$, locally uniform convergence implies convergence in $X = \StandardXSpace(\Omega)$
for all $p \in (0,\infty]$ hence the result.

For unbounded $\Omega$ we need to work a bit harder.
First we deal with the degenerate case where $W=0$ or $N=0$.
If $W=0$ then by Lemma~\ref{lem:ConstantMaps} $f$ is a constant map;
hence $f \in \NNreal^{\varrho_r,d,k}_{0,1,0}$.
If $N=0$ then $f$ is affine-linear with $\|f\|_{\ell^{0}} \leq W$;
hence $f\in \NNreal^{\varrho_r,d,k}_{W,1,0}$.
In both cases the result trivially holds.

From now on we assume that $W,N \geq 1$.
Consider $\eps > 0$.
By the dominated convergence theorem (in case of $p < \infty$) or our special choice
of $\StandardXSpace[k][\infty](\Omega)$ (cf.~Equation~\eqref{eq:StandardXSpace})
(in case of $p = \infty$) we see that there is some $R \geq 1$ such that
\[
  \|f - f \cdot \Indicator_{[-R,R]^{d}}\|_{L_p(\Omega; \R^k)}
  \leq \eps'
  :=   \frac{\eps}{8^{1 / \min\{1,p\}}}.
\]
Denoting by $\lambda(\cdot)$ the Lebesgue measure, \eqref{eq:Tmp0} implies that there is
$g \in \NNreal_{Wm^{2}, L,Nm}^{\varrho_r, d, k}$ such that
\[
  \|f - g\|_{L_\infty ([-R-1,R+1]^{d};\R^{k})}
  \leq \eps' \big/ \, \big[ \lambda([-R-1,R+1]^{d}) \big]^{1/p} .
\]
Consider $c = c(d,k,r)$, $\ell = \min \{ d+1, 3 \}$, $L' = \max \{ L+1,\ell \}$ and the function
$g_{R,1} \in \NNreal^{\varrho_r, d, k}_{cWm^{2},L',cNm}$
from Lemma~\ref{lem:NetworkThatLocalizesAnyFunction}.
By~\eqref{eq:SpecialCutoffMainEstimateEndToEnd} and the fact that $\|\cdot\|_{L_p}^{\min \{1,p\}}$
is subadditive, we see
\begin{align*}
      \|f - g_{R,1}\|_{L_p (\Omega;\R^k)}^{\min \{1,p\}}
    & \leq \| f - f \cdot \Indicator_{[-R,R]^d} \|_{L_p (\Omega; \R^k)}^{\min \{1,p\}}
           + \| (f - g) \Indicator_{[-R,R]^d} \|_{L_p (\Omega; \R^k)}^{\min \{1,p\}} \\
    &      \qquad
           + \|
               g \cdot \Indicator_{[-R,R]^d}(x) - g_{R,1}
             \|_{L_p (\Omega;\R^k)}^{\min \{1,p\}} \\
    & \leq \frac{\eps^{\min\{1,p\}}}{8}
           + \left(
               \|f - g\|_{L_\infty([-R-1,R+1]^d;\R^{k})} \cdot [ \lambda([-R,R]^d) ]^{1/p}
             \right)^{\min \{1,p\}} \\
    &      \qquad
           + \left(
               \|\ 2 \cdot |g| \cdot \Indicator_{[-R-1,R+1]^d \setminus [-R,R]^d}\|_{L_p (\Omega)}
             \right)^{\min \{1,p\}} \, ,
             \\
    & \leq \frac{\eps^{\min\{1,p\}}}{2}
           + \left(
               \|\  2 \cdot |g| \cdot \Indicator_{[-R-1,R+1]^d \setminus [-R,R]^d}\|_{L_p (\Omega)}
             \right)^{\min \{1,p\}} \, .
\end{align*}
To estimate the final term, note that
\begin{align*}
    & \Big(
         \|\ |g| \cdot \Indicator_{[-R-1,R+1]^d \setminus [-R,R]^d} \|_{L_p (\Omega)}
      \Big)^{\min \{1,p\}} \\
    & \leq \Big(
             \|\ |g-f| \cdot \Indicator_{[-R-1,R+1]^d \setminus [-R,R]^d}\|_{L_p (\Omega)}
           \Big)^{\min\{1,p\}}
           + \Big(
               \|\ |f| \cdot \Indicator_{[-R-1,R+1]^d \setminus [-R,R]^d} \|_{L_p (\Omega)}
             \Big)^{\min \{1,p\}} \\
    & \leq \Big(
             \| f - g \|_{L_\infty ([-R-1,R+1]^d;\R^{k})} \cdot [\lambda ([-R-1,R+1]^d)]^{1/p}
           \Big)^{\min \{1,p\}}
           + \Big(
               \| f - f \cdot \Indicator_{[-R,R]^d}\|_{L_p (\Omega; \R^k)}
             \Big)^{\min\{1,p\}} \\
    & \leq \frac{\eps^{\min \{1,p\}}}{8}
           + \frac{\eps^{\min \{1,p\}}}{8} \, .
\end{align*}
Because of $2^{\min\{1,p\}} \leq 2$, this implies
\(
  \Big(
     \|\ 2 \cdot |g| \cdot \Indicator_{[-R-1,R+1]^d \setminus [-R,R]^d} \|_{L_p (\Omega)}
  \Big)^{\min \{1,p\}}
  \leq \tfrac{\eps^{\min \{1,p\}}}{2}
\).
Overall, we thus see that $\| f - g_{R,1} \|_{L_p (\Omega; \R^k)} \leq \eps < \infty$.
Because of $f|_{\Omega} \in X$, this implies in particular that $g_{R,1}|_{\Omega} \in X$.
Since $\eps > 0$ was arbitrary, we get as desired that
\(
  f|_{\Omega} \in \overline{\NNreal^{\varrho_r, d, k}_{cWm^{2}, L', cNm} \cap X}^{X}
\),
where the closure is taken in $X$.
\hfill$\square$


\section{Proofs for Section~\ref{sec:directinverse}}

\subsection{Proof of Lemma~\ref{lem:MultiBspline}}
\label{app:MultiBspline}

In light of~\eqref{eq:DecompBSpline} we have
$\beta_{+}^{(t)} \in \NNreal^{\varrho_t,1,1}_{2(t+2),2,t+2}$.
This yields the result for $d=1$, including when $t=1$.

For $d \geq 2$ and $t \geq \min \{d,2\} = 2$,
define $f_j:\R^d\rightarrow \R$ by $f_j := \beta^{(t)}_+ \circ \pi_j$
with $\pi_j:\R^d\rightarrow \R, x \mapsto x_j$, $j=1,\ldots,d$.
By Lemma~\ref{lem:NetworkCalculus}--(\ref{enu:PrePostAffine}) together with the fact
that $\|\pi_j\|_{\ell^{0,\infty}_*} = 1$ we get $f_j \in \NNreal^{\varrho_t,d,1}_{2(t+2),2,t+2}$.
Form the vector function $f := (f_1,f_2,\ldots,f_d)$.
Using Lemma~\ref{lem:SummationLemma}-(\ref{enu:Cartesian}),
we deduce $f\in \NNreal^{\varrho_t,d,d}_{2d(t+2),2,d(t+2)}$.

As $t \geq 2$, by Lemma~\ref{lem:ReLUPowerRepresentsIdentity}, $\varrho_t$ can represent
any polynomial of degree two with $n = 2(t+1)$ terms.
Hence, for $d \geq 2$, by Lemma~\ref{lem:MultNetwork} the multiplication function
$M_d : \R^d \to \R, (x_1,\dots,x_d) \mapsto x_1 \cdots x_d$ satisfies
$M_{d} \in \NNreal^{\varrho_t,d,1}_{4n(2^{j}-1),2j,(2n+1)(2^{j}-1)-1}$
with $j := \lceil \log_{2} d \rceil$.
By definition, $2^{j-1} < d \leq 2^{j}$, hence $2^{j}-1 \leq 2(d-1)$ and
$6 n (2^{j}-1) \leq  12 n(d-1) = 24(t+1)(d-1)$, as well as
\[
  (2n+1)(2^{j}-1)-1 \leq (4n+2)(d-1)-1 = (8t+10)(d-1)-1 ,
\]
so that $M_{d}  \in \NNreal^{\varrho_t,d,1}_{24(t+1)(d-1),2j,(8t+10)(d-1)-1}$.
As $\beta_d^{(t)} = M_{d} \circ f$, by Lemma~\ref{lem:NetworkCalculus}--(\ref{enu:Compos}) we get
\[
  \beta_d^{(t)} \in \NNreal^{\varrho_t,d,1}_{2d(t+2)+24(t+1)(d-1),2j+2,d(t+2)+(8t+10)(d-1)-1+d}.
\]
To conclude, we observe that
\begin{flalign*}
  && 2 d(t + 2) + 24(t+1)(d-1)
  & \leq d(2t+4+24t+24) = d(26t+28) \leq 28d(t+1) \\
  && d(t+2)+(8t+10)(d-1)-1+d & \leq d(t+2+8t+10+1) = d(9t+13) \leq 13d(t+1).
  && \square
\end{flalign*}

\subsection{Proof of Theorem~\ref{thm:better_direct}}
\label{sub:BetterDirectProof}

We divide the proof into three steps.\smallskip{}

\noindent
\textbf{Step 1 (Recalling results from \cite{DeVore:1988fe}):}
Using the tensor B-splines $\beta_d^{(t)}$ introduced in Equation~\eqref{eq:TensorBSpline},
define $N := N^{(\tau)} := \beta_d^{(\tau-1)}$ for $\tau \in \N$, and note that this coincides
with the definition of $N$ in \cite[Equation (4.1)]{DeVore:1988fe}.
Next, as in \cite[Equations (4.2) and (4.3)]{DeVore:1988fe},
for $k \in \N_0$ and $j \in \Z^d$,
define $N_k^{(\tau)} (x) := N^{(\tau)} (2^k x)$
and $N_{j,k}^{(\tau)} (x) := N^{(\tau)} (2^k x - j)$.
Furthermore, let $\Omega_0 := (- \tfrac{1}{2}, \tfrac{1}{2})^d$ denote the unit cube, and set
\[
  \Lambda^{(\tau)}(k) := \big\{ j \in \Z^d \colon N_{j,k}^{(\tau)}|_{\Omega_0} \not\equiv 0 \big\}
  \quad \text{and} \quad
  \Sigma_k^{(\tau)} := \mathrm{span} \{ N_{j,k}^{(\tau)} \colon j \in \Lambda^{(\tau)}(k) \big\},
\]
and finally $s_k^{(\tau)} (f)_p := \inf_{g \in \Sigma_k^{(\tau)}} \|f - g\|_{L_p}$
for $f \in X_p (\Omega_0)$ and $k \in \N_0$.
Setting $\lambda^{(\tau,p)} := \tau - 1 + \min \{ 1, p^{-1} \}$,
\cite[Theorem 5.1]{DeVore:1988fe} shows
\begin{equation}
  \|f\|_{B_{p,q}^\alpha (\Omega_0)}
  \asymp \|f\|_{L_p} + \big\| \big( s_k^{(\tau)} (f)_p \big)_{k \in \N_0} \big\|_{\ell_q^\alpha}
  \quad \forall \, p,q \in (0,\infty],
                   \alpha \in (0,\lambda^{(\tau,p)}),
                   \text{ and } f \in B_{p,q}^\alpha (\Omega_0).
  \label{eq:BesovSpacesAreBSplineApproximationSpaces}
\end{equation}
Here, $\| (c_k)_{k \in \N_0} \|_{\ell_q^\alpha} = \|(2^{\alpha k} \, c_k)_{k \in \N_0}\|_{\ell^q}$;
see \cite[Equation (5.1)]{DeVore:1988fe}.

\medskip{}

\noindent
\textbf{Step 2 (Proving the embedding
$B_{p,q}^{d \alpha} (\Omega_0) \hookrightarrow A_q^\alpha (X_p(\Omega_0), \Sigma(\CalD_d^{\tau-1}))$):}
Define $\Sigma(\CalD_d^t) := (\Sigma_n (\CalD_d^t))_{n \in \N_0}$.
In this step, we show that
$B_{p,q}^{d \alpha} (\Omega_0) \hookrightarrow A_q^\alpha (X_p (\Omega_0), \Sigma(\CalD_d^{\tau-1}))$
for any $\tau \in \N$ and all $p,q \in (0,\infty]$ and $\alpha > 0$
with $0 < d \alpha < \lambda^{(\tau,p)}$.

To this end, we first show that if we choose $X = X_{p}(\Omega_0)$,
then the family $\Sigma(\CalD_d^{\tau-1})$ satisfies
the properties \ref{enu:GammaContainsZero}--\ref{enu:GammaDense}.
To see this, we first have to show $\Sigma_n(\CalD_d^{\tau-1}) \subset X_p(\Omega_0)$.
For $p < \infty$, this is trivial, since $N^{(\tau)} = \beta_d^{(\tau-1)}$ is bounded and measurable.
For $p = \infty$ this holds as well, since if $\tau \geq 2$, then $N^{(\tau)} = \beta_d^{(\tau-1)}$
is continuous; finally, the case $\tau = 1$ cannot occur for $p = \infty$, since this would imply
\[0 < d \alpha < \lambda^{(\tau,p)} = \tau - 1 + \min \{ 1, p^{-1} \} = 0.\]
Next, Properties~\ref{enu:GammaContainsZero}--\ref{enu:GammaAdditive} are trivially satisfied.
Finally, the density of $\bigcup_{n=0}^\infty \Sigma_n (\CalD_{d}^{\tau-1})$ in $X_p(\Omega_0)$
is well-known for $\tau = 1$, since then $\beta_0^{(\tau-1)} = \Indicator_{[0,1)^d}$ and $p < \infty$.
For $\tau \geq 2$, the density follows with the same arguments that were used for the case
$p = \infty$ in Section~\ref{sub:DensityClaim}.

Next, note that $\supp N^{(\tau)} \subset [0,\tau]^d$ and thus
$\supp N^{(\tau)}_{j,k} \subset 2^{-k} (j + [0,\tau]^d)$.
Therefore, if $j \in \Lambda^{(\tau)}(k)$, then $\emptyset \neq \Omega_0 \cap \supp N_{j,k}$,
so that there is some $x \in \Omega_0 \cap 2^{-k}(j + [0,\tau]^d)$.
This implies ${j \in \Z^d \cap [-2^{k-1} - \tau, 2^{k-1}]^d}$,
and thus $|\Lambda^{(\tau)}(k)| \leq (2^k + \tau + 1)^d$.
Directly by definition of $\Sigma_n(\CalD_d^t)$ and $\Sigma_k^{(\tau)}$, this implies
\begin{equation}
  \Sigma_k^{(\tau)} \subset \Sigma_{(2^k + \tau + 1)^d} (\CalD_d^{\tau-1})
  \qquad \forall \, k \in \N_0 .
  \label{eq:SplineMainInclusion}
\end{equation}

Next, since we are assuming $0 < \alpha d < \lambda^{(\tau,p)}$,
Equation~\eqref{eq:BesovSpacesAreBSplineApproximationSpaces} yields a constant
$C_1 = C_1 (p,q,\alpha,\tau,d) > 0$ such that
\(
  \|f\|_{L_p} + \big\| \big(s_k^{(\tau)} (f)_p \big)_{k \in \N_0} \big\|_{\ell_q^{d \alpha}}
  \leq C_1 \cdot \|f\|_{B_{p,q}^{d \alpha}(\Omega_0)}
\)
for all $f \in B_{p,q}^{d \alpha}(\Omega_0)$.
Therefore, we see for $f \in B_{p,q}^{d \alpha}(\Omega_0)$ and $q < \infty$ that
\begin{align*}
  \|f\|_{A_q^\alpha (X_p (\Omega_0), \Sigma(\CalD_d^{\tau-1}))}^q
  & = \sum_{n=1}^{\infty}
        n^{-1} \cdot [n^\alpha \cdot E(f, \Sigma_{n-1}(\CalD_d^{\tau-1}))_{L_p(\Omega_0)}]^q \\
  & \leq \|f\|_{L_p}^q \sum_{n=1}^{(\tau+2)^d} n^{\alpha q - 1}
         + \sum_{k=0}^{\infty}
             \sum_{n = (2^k + \tau + 1)^d + 1}^{(2^{k+1} + \tau + 1)^d}
               n^{\alpha q - 1}
               [E (f, \Sigma_{n-1}(\CalD_d^{\tau-1}))_{L_p(\Omega_0)}]^q \\
  & \overset{(\ast)}{\leq} C_2 \cdot \|f\|_{L_p}^q
                           + C_4 \sum_{k=0}^\infty
                                   2^{kd} 2^{dk (\alpha q - 1)} [s_k^{(\tau)}(f)_p]^q \\
  & \leq (C_2 + C_4) \cdot \big(
                             \|f\|_{L_p}
                             + \big\|
                                 \big( s_k^{(\tau)} (f)_p \big)_{k \in \N_0}
                               \big\|_{\ell_q^{d\alpha}}
                           \big)^q
    \leq C_1^q \cdot (C_2 + C_4) \cdot \|f\|_{B_{p,q}^{d\alpha}(\Omega_0)}^q .
\end{align*}
At the step marked with $(\ast)$, we used that Equation~\eqref{eq:SplineMainInclusion} yields
\(
  \Sigma_{n-1}(\CalD_d^{\tau-1})
  \supset \Sigma_{(2^k+\tau+1)^d} (\CalD_d^{\tau-1})
  \supset \Sigma_k^{(\tau)}
\)
for all $n \geq 1 + (2^k + \tau + 1)^d$, and furthermore that if
$1 + (2^k + \tau + 1)^d \leq n \leq (2^{k+1} + \tau + 1)^d$, then
$2^{dk} \leq n \leq (\tau+3)^d \cdot 2^{dk}$,
so that $n^{\alpha q - 1} \leq C_3 2^{dk (\alpha q - 1)}$
for some constant $C_3 = C_3(d,\tau,\alpha,q)$, and finally that
\(
  \sum_{n= (2^k + \tau + 1)^d + 1}^{(2^{k+1} + \tau + 1)^d} 1
  \leq (2^{k+1} + \tau + 1)^d
  \leq (\tau+3)^d \cdot 2^{dk}
\).

For $q = \infty$, the proof is similar.
Setting $\ell_k := (2^k + \tau + 1)^d$ for brevity, we see with similar estimates as above that
\begin{align*}
  & \|f\|_{A_\infty^\alpha (X_p(\Omega_0), \Sigma(\CalD_d^{\tau-1}))} \\
  & = \max \Big\{
             \max_{0 \leq n \leq (\tau+2)^d}
               n^\alpha \, E(f, \Sigma_{n-1}(\CalD_d^{\tau-1}))_{L_p(\Omega_0)},
             \quad
             \sup_{k \in \N_0}
               \max_{\ell_k + 1 \leq n \leq \ell_{k+1}}
                 n^\alpha \, E(f, \Sigma_{n-1}(\CalD_d^{\tau-1}))_{L_p (\Omega_0)}
           \Big\} \\
  & \leq \max \Big\{
                (\tau+2)^{\alpha d} \, \|f\|_{L_p(\Omega_0)},
                \quad
                \sup_{k \in \N_0}
                  (\tau+3)^{\alpha d} \, 2^{\alpha d k} s_k^{(\tau)}(f)_p
              \Big\} \\
  & \leq (\tau+3)^{\alpha d} \big(
                            \|f\|_{L_p(\Omega_0)}
                            + \| (s_k^{(\tau)}(f)_p)_{k \in \N_0} \|_{\ell_q^{d \alpha}}
                           \big)
  \leq C_1 \, (\tau+3)^{\alpha d} \, \|f\|_{B_{p,\infty}^{d \alpha}(\Omega_0)}.
\end{align*}
Overall, we have shown
$B_{p,q}^{d \alpha}(\Omega_0) \hookrightarrow A_q^\alpha (X_p (\Omega_0), \Sigma(\CalD_d^{\tau-1}))$
for $\tau \in \N$, $p,q \in (0,\infty]$ and $0 < \alpha d < \lambda^{(\tau,p)}$.

\medskip{}

\noindent
\textbf{Step 3 (Proving the embeddings~\eqref{eq:BesovDirectDimension1} and \eqref{eq:BesovDirectGeneralDimension}):}
In case of $d = 1$, let us set $r_0 := r$,
while $r_0$ is as in the statement of the theorem for $d > 1$.
Since $\Omega$ is bounded and $\Omega_0 = (-\tfrac{1}{2}, \tfrac{1}{2})^d$,
there is some $R > 0$ such that $\Omega \subset R \cdot \Omega_0$.
Let us fix $p,q \in (0,\infty]$ and $s > 0$ such that $d s < r_0 + \min \{1, p^{-1} \}$.

Since $\Omega$ and $R \cdot \Omega_0$ are bounded Lipschitz domains,
there exists a (not necessarily linear) \emph{extension operator}
$\mathcal{E} : B^{d s}_{p,q} (\Omega) \to B^{d s}_{p,q} (R \Omega_0)$
with the properties $(\mathcal{E} f)|_{\Omega} = f$ and
$\|\mathcal{E} f\|_{B^{d s}_{p,q}(R \Omega_0)} \leq C \cdot \|f\|_{B^{d s}_{p,q}(\Omega)}$
for all $f \in B^{d s}_{p,q}(\Omega)$.
Indeed, for $p \in [1,\infty]$ this follows from
\cite[Section 4, Corollary 1]{JohnenSchererEquivalenceKFunctionalModulusContinuity},
since this corollary yields an extension operator
$\mathcal{E} : X_p (\Omega) \to X_p (R \Omega_0)$ with the additional property
that the $j$-th modulus of continuity $\omega_j$ satisfies
$\omega_j (t, \mathcal{E} f)_{R \Omega_0} \leq M_j \cdot \omega_j (t, f)_{\Omega}$
for all $j \in \N$, all $f \in X_p(\Omega)$, and all $t \in [0,1]$.
In view of the definition of the Besov spaces
(see in particular \cite[Chapter 2, Theorem 10.1]{ConstructiveApproximation}),
this easily implies the result.
Finally, in case of $p \in (0,1)$, the existence of the extension operator follows
from \cite[Theorem 6.1]{DeVSharp93}.
In addition to the existence of the extension operator, we will also need that the dilation
operator $D_1 : B^{d s}_{p,q}(R \Omega_0) \to B^{d s}_{p,q} (\Omega_0), f \mapsto f(R \bullet)$
is well-defined and bounded, say $\|D_1\| \leq C_1$; this follows directly from the definition
of the Besov spaces.

We first prove Equation~\eqref{eq:BesovDirectDimension1}, that is, we consider the case $d = 1$.
To this end, define $\tau := r + 1 \in \N$, let $f \in B^{s}_{p,q}(\Omega)$ be arbitrary,
and set $f_1 := D_1 (\mathcal{E} f) \in B^{s}_{p,q}(\Omega_0)$.
By applying Step~2 with $\alpha = s$ (and noting that
$0 < d \alpha = s < r + \min\{1,p^{-1}\} = \lambda^{(\tau,p)}$),
we get $f_1 \in A_q^s (X_p(\Omega_0), \Sigma(\CalD_d^{r}))$, with
$\|f_1\|_{A_q^s (X_p(\Omega_0), \Sigma(\CalD_d^{r}))} \leq C C_1 C_2 \cdot \|f\|_{B^{d s}_{p,q}(\Omega)}$,
where the constant $C_2$ is provided by Step~2.

Next, we note that $L := \sup_{n \in \N} \mathscr{L}(n) \geq 2 = 2 + 2 \lceil \log_2 d \rceil$
and $r \geq 1 = \min \{d,2\}$, so that Corollary~\ref{cor:JacksonAbstractReLUPower}-(2) shows
\(
  f_1
  \in A_q^s (X_p(\Omega_0), \Sigma(\CalD_d^r))
  \hookrightarrow \WASpace[X_p(\Omega_0)][\varrho_r][q][s]
\).
But it is an easy consequence of Lemma~\ref{lem:NetworkCalculus}-(1) that the dilation operator
\(
  D_2 : \WASpace[X_p(\Omega_0)][\varrho_r][q][s] \to \WASpace[X_p(R \Omega_0)][\varrho_r][q][s],
        g \mapsto g (\bullet / R)
\)
is well-defined and bounded.
Hence, we see that $D_2 f_1 \in \WASpace[X_p(R \Omega_0)][\varrho_r][q][s]$
with $\|D_2 f_1\|_{\WASpace[X_p(R \Omega_0)][\varrho_r][q][s]} \lesssim \|f\|_{B^{d s}_{p,q}(\Omega)}$.
Now, note $D_2 f_1(x) = f_1 (x/R) = \mathcal{E} f (x) = f(x)$ for all $x \in \Omega \subset R \Omega_0$,
and hence $f = (D_2 f_1)|_{\Omega}$.
Thus, Remark~\ref{sec:RestrictionCartesian} implies that
$f \in \WASpace[X_p(\Omega)][\varrho_r][q][s]$
with $\|f\|_{\WASpace[X_p(\Omega)][\varrho_r][q][s]} \lesssim \|f\|_{B^{d s}_{p,q}(\Omega)}$,
as claimed.

\medskip{}

Now, we prove Equation~\eqref{eq:BesovDirectGeneralDimension}.
To this end, define $\tau := r_0 + 1 \in \N$, let $f \in B^{s d}_{p,q}(\Omega)$ be arbitrary,
and set $f_1 := D_1 (\mathcal{E} f) \in B^{d s}_{p,q}(\Omega_0)$.
Applying Step~2 with $\alpha = s$
(noting ${0 < d \alpha = d s < r_0 + \min\{1, p^{-1}\} = \lambda^{(\tau,p)}}$),
we get $f_1 \in A_q^s (X_p (\Omega_0), \Sigma(\CalD_d^{r_0}))$, with
\(
  \|f_1\|_{A_q^s (X_p (\Omega_0), \Sigma(\CalD_d^{r_0}))}
  \leq C C_1 C_2 \cdot \|f\|_{B^{d s}_{p,q}(\Omega)}
\),
where the constant $C_2$ is provided by Step~2.

Next, we claim that
$A_q^s (X_p (\Omega_0), \Sigma(\CalD_d^{r_0})) \hookrightarrow \WASpace[X_p(\Omega_0)][\varrho_r][q][s]$.
Indeed, if $r \geq 2$ and $L \geq 2 + 2 \lceil \log_2 d \rceil$, then this follows from
Corollary~\ref{cor:JacksonAbstractReLUPower}--(2). 
Otherwise, we have $r_0 = 0$ and $L \geq 3 \geq \min \{d+1, 3\}$,
so that the claim follows from Corollary~\ref{cor:JacksonAbstractReLUPower}--(1);
here, we note that $p < \infty$, since we would otherwise get the contradiction
$0 < \alpha d < r_0 + \min \{1, p^{-1} \} = 0$.
Therefore, $f_1 \in \WASpace[X_p(\Omega_0)][\varrho_r][q][s]$
with $\|f_1\|_{\WASpace[X_p(\Omega_0)][\varrho_r][q][s]} \lesssim \|f\|_{B^{d s}_{p,q}(\Omega)}$.
The rest of the argument is exactly as in the case $d = 1$.
\hfill$\square$

\subsection{Proof of Lemma~\ref{lem:SawtoothImplementation}}
\label{sub:SawtoothImplementation}

Lemma~\ref{lem:SawtoothImplementation} shows that deeper networks can implement the sawtooth
function $\Delta_j$ using less connections/neurons than more shallow networks.
The reason for this is indicated by the following lemma.

\begin{lem}\label{lem:SawtoothIteration}
  For arbitrary $j \in \N$, we have $\Delta_j \circ \Delta_1 = \Delta_{j+1}$.
\end{lem}

\begin{proof}
  It suffices to verify the identity on $[0,1]$, since if $x \in \R \setminus [0,1]$,
  then $\Delta_1 (x) = 0 = \Delta_{j+1} (x)$, so that
  $\Delta_j(\Delta_1(x)) = \Delta_j (0) = 0 = \Delta_{j+1} (x)$.
  We now distinguish two cases for $x \in [0,1]$.

  \smallskip{}

  \emph{Case 1:} $x \in [0,\tfrac{1}{2}]$.
  This implies $\Delta_1 (x) = 2x$, and hence
  (recall the definition of $\Delta_{j}$ in Equation~\eqref{eq:DefSawtooth})
  \[
    \Delta_j (\Delta_1(x))
    = \sum_{k=0}^{2^{j-1} - 1}
        \Delta_1 \big( 2^{j-1} 2x - k \big)
    = \sum_{k=0}^{2^{j - 1} - 1}
        \Delta_1(2^{(j+1) - 1} x - k)
    = \Delta_{j+1}(x).
  \]
  In the last equality we used that $2^j x - k \leq 2^{j-1} - k \leq 0$
  for $k \geq 2^{j-1}$, so that $\Delta_1 (2^j x - k) = 0$ for those $k$.

  \medskip{}

  \emph{Case 2:} $x \in (\tfrac{1}{2}, 1]$.
  Observe that $\Delta_j (x) = \Delta_j (1-x)$ for all $x \in \R$ and $j \in \N$.
  Since $x' := 1-x \in [0,1/2]$, this identity and Case 1 yield
  \(
    \Delta_j \circ \Delta_1 (x)
    = \Delta_j \circ \Delta_1 (1-x)
    = \Delta_{j+1}(1-x)
    = \Delta_{j+1}(x)
  \).
\end{proof}

Using Lemma~\ref{lem:SawtoothIteration}, we can now provide the proof
of Lemma~\ref{lem:SawtoothImplementation}.

\begin{proof}[Proof of Lemma~\ref{lem:SawtoothImplementation}]
  \emph{Part (1):}
  Write $j = k (L-1) + s$ for suitable $k \in \N_0$ and $0 \leq s \leq L - 2$.
  Note that this implies $k \leq j / (L-1)$.
  Thanks to Lemma~\ref{lem:SawtoothIteration}, we have
  $\Delta_j = \Delta_{k+s} \circ \Delta_k \circ \cdots \circ \Delta_k$, where $\Delta_k$
  occurs $L-2$ times.
  Furthermore, since $\Delta_k : \R \to \R$ is affine with $2 + 2^k$ pieces
  (see Figure~\ref{fig:SawToothPlot}, and note that we consider $\Delta_k$ as a function
  on all of $\R$, not just on $[0,1]$), Lemma~\ref{lem:PiecewisePolynomialCont} shows that
  $\Delta_k \in \NNreal^{\varrho_1,1,1}_{\infty,2,3+2^k}$.
  By the same reasoning, we get $\Delta_{k+s} \in \NNreal_{\infty,2,3+2^{k+s}}^{\varrho_1,1,1}$.
  Now, a repeated application of Lemma~\ref{lem:NetworkCalculus}-(\ref{enu:ComposLessDepth})
  shows that
  \[
    \Delta_j
    = \Delta_{k+s} \circ \Delta_k \circ \cdots \circ \Delta_k
    \in \NNreal_{\infty,L,(L-2)(3 + 2^k) + 3 + 2^{k+s}}^{\varrho_1,1,1} .
  \]
  Finally, $\Delta_j \in \NNreal^{\varrho_1,1,1}_{\infty,L,C_L \cdot 2^{j/(L-1)}}$
  with $C_L := 4 \, L + 2^{L-1}$ since 
  \[
    (L-2)(3+2^k)+3+2^{k+s}
    = 3(L-1) + (L-2+2^{s})2^{k}
    \leq (4L-5 + 2^{L-2}) 2^k \leq (4L+2^{L-1}) 2^{j/(L-1)}
    = C_{L} \cdot 2^{j/(L-1)}.
  \]

  \medskip{}

  \emph{Part (2):}
  Set $\kappa := \lfloor L/2\rfloor$ and write $j = k \kappa + s$ for $k \in \N_0$
  and $0 \leq s \leq \kappa - 1$.
  Note that $k \leq j / \kappa = j / \lfloor L/2 \rfloor$.
  As above, $\Delta_j = \Delta_{k+s} \circ \Delta_k \circ \cdots \circ \Delta_k$, where $\Delta_k$
  occurs $\kappa - 1$ times, and since $\Delta_k : \R \to \R$ is affine with $2 + 2^k$ pieces,
  using Lemma~\ref{lem:PiecewisePolynomialCont} again shows that
  $\Delta_k \in \NNreal^{\varrho_1,1,1}_{6+2^{k+1},2,\infty}$,
  and
  $\Delta_{k+s} \in \NNreal_{6+2^{k+s+1},2,\infty}^{\varrho_1,1,1}$.
  Now, a repeated application of Lemma~\ref{lem:NetworkCalculus}-(\ref{enu:Compos}) shows that
  \[
    \Delta_j
    = \Delta_{k+s} \circ \Delta_k \circ \cdots \circ \Delta_k
    \in \NNreal_{6+2^{k+s+1} + (\kappa - 1)(6 + 2^{k+1}),
                 2 + 2 \cdot (\kappa-1),
                 \infty}^{\varrho_1,1,1} .
  \]
  Finally,
  $\Delta_k \in \NNreal^{\varrho_1,1,1}_{C_{L} 2^{j/\lfloor L/2\rfloor},\lfloor L/2\rfloor,\infty}$,
  as $2+2(\kappa-1) = 2\kappa \leq L$, $s+1 \leq \kappa \leq L/2 \leq L-1$
  (since $L \geq 2$) and
  \[
    6+2^{k+s+1} + (\kappa - 1)(6 + 2^{k+1})
    = 6\kappa + (2^{s+1}+2)2^{k}
    \leq (3L+2^{L-1}+2) 2^{j / \lfloor L/2 \rfloor}
    \leq C_L \cdot 2^{j / \lfloor L/2 \rfloor}.
    \qedhere
  \]
\end{proof}

\subsection{Proof of Lemma~\ref{lem:BesovNormSawTooth}}
\label{sub:BesovNormSawToothProof}

For $h \in \R^d$, we define the translation operator $T_h$ by $(T_h f)(x) = f(x - h)$
for $f : \R^d \to \R$.
With this, the $h$-difference operator of order $k$ is given by $D_h^k = (D_h)^k$,
where $D_h := (T_{-h} - \identity)$.
For later use, we note for $a > 0$ that $D_h [f(a \bullet)](x) = (D_{a h}f)(a x)$,
as can be verified by a direct calculation.
By induction, this implies $D_h^k [f(a \bullet)] = (D_{a h}^k f)(a \bullet)$ for all $k \in \N$.
Furthermore, $T_x D_h^k = D_h^k T_x$ for all $x,h \in \R^d$ and $k \in \N$.

A direct computation shows
\[
  \Delta_1
  = \widetilde{\Delta_1}
    + 2 \widetilde{\Delta_1}(\bullet - \tfrac{1}{4})
    + \widetilde{\Delta_1} (\bullet - \tfrac{1}{2})
  = (T_{1/4} + \identity)^2 \widetilde{\Delta_1}
  \qquad \text{where} \qquad
  \widetilde{\Delta_1} := \frac{1}{2} \Delta_1 (2 \bullet).
\]
Next, note that $(T_{-1/4} - \identity) (T_{1/4} + \identity) = T_{-1/4} - T_{1/4}$ and hence, since $T_{-1/4}$ and $T_{1/4}$ commute,
\begin{equation}
  D_{1/4}^2 \Delta_1
  = (T_{-1/4} - \identity)^2 (T_{1/4} + \identity)^2 \widetilde{\Delta_1}
  = (T_{-1/4} - T_{1/4})^2 \widetilde{\Delta_1}
  = (T_{-1/2} - 2 \, \identity + T_{1/2}) \widetilde{\Delta_1}.
  \label{eq:OneToothSecondDifferenceComputation}
\end{equation}
Moreover by induction on $\ell \in \N_{0}$, we see that
\begin{equation}
    \sum_{k=0}^{\ell}
        T_k (T_{-\frac{1}{2}} - 2 \, \identity + T_{\frac{1}{2}})
       = T_{-\frac{1}{2}}
        + T_{\frac{2\ell + 1}{2}}
        + 2 \sum_{i=0}^{2\ell}
              (-1)^{i-1} \, T_{\frac{i}{2}}.
  \label{eq:SpecialShiftSum}
\end{equation}
Define $h_j := 2^{-(j+1)}$, so that $2^{j-1} h_j = 1/4$.
Since $\Delta_j = \sum_{k=0}^{2^{j-1} - 1} (T_k \Delta_1)(2^{j-1} \bullet)$
(cf.~Equation~\eqref{eq:DefSawtooth}),
Equations~\eqref{eq:OneToothSecondDifferenceComputation}
and \eqref{eq:SpecialShiftSum} and the properties from the beginning of the proof yield
for $x \in \R$ that
\begin{equation}
  \begin{split}
    (D_{h_j}^2 \Delta_j)(x)
    & = \sum_{k=0}^{2^{j-1}-1}
          [D_{2^{j-1} h_j}^2 (T_k \Delta_1)](2^{j-1} x)
      = \bigg[
          \sum_{k=0}^{2^{j-1} - 1}
            T_k (D_{1/4}^2 \Delta_1)
        \bigg]
        (2^{j-1} x) \\
    & = (T_{-\frac{1}{2}} \widetilde{\Delta_1})(2^{j-1} x)
        + (T_{\frac{2^j - 1}{2}} \widetilde{\Delta_1})(2^{j-1} x)
        + 2 \sum_{i=0}^{2^j - 2} (-1)^{i-1} (T_{\frac{i}{2}} \widetilde{\Delta_1}) (2^{j-1} x).
  \end{split}
  \label{eq:SecondDifferenceSawtoothComputed}
\end{equation}
Recall for $g \in X_p (\Omega)$ that the $r$-th modulus of continuity of $g$ is given by
\[
  \omega_r (g)_p (t)
  := \sup_{h \in \R^d, |h| \leq t}
       \| D_h^r g\|_{X_p (\Omega_{r,h})}
  \quad \text{where} \quad
  \Omega_{r,h} := \{ x \in \Omega \colon x + u h \in \Omega \text{ for each } u \in [0,r] \}.
\]
Let $e_1 = (1,0,\dots,0) \in \R^d$.
For $h = h_j \, e_1$, we have $\Omega_{2,h} \supset (0,\frac{1}{2}) \times (0,1)^{d-1}$
since $\Omega = (0,1)^{d}$.
Next, because of ${\supp \, \widetilde{\Delta_1} = [0, \tfrac{1}{2}]}$, the
family $(T_{i/2} \widetilde{\Delta_1})_{i \in \Z}$ has pairwise disjoint supports
(up to null-sets), and
\[
  \supp \big[ (T_{\frac{i}{2}} \widetilde{\Delta_1})(2^{j-1} \bullet) \big]
  = 2^{-j} (i + [0,1])
  \subset \big[ 0, \tfrac{1}{2} \big]
  \qquad \text{for} \quad 0 \leq i \leq 2^{j-1} - 1.
\]
Combining these observations with the fact that
\(
  (T_{\frac{i}{2}} \widetilde{\Delta_1})(2^{j-1} \bullet)
  = \widetilde{\Delta_1}(2^{j-1} \bullet - i/2)
  = \Delta_1(2^j \bullet-i)/2
\),
Equation~\eqref{eq:SecondDifferenceSawtoothComputed} yields for $p < \infty$ that
\begin{align*}
  \| D_{h_j \, e_1}^2 \Delta_{j,d} \|_{L_p (\Omega_{2, h_j e_1})}^p
  & \geq \sum_{i = 0}^{2^{j-1} - 1}
           2^p \| (T_{\frac{i}{2}} \widetilde{\Delta_1}) (2^{j-1} \bullet) \|_{L_p (2^{-j} (i + [0,1]))}^p
    = \sum_{i = 0}^{2^{j-1} - 1}
           \| \Delta_1 (2^{j} \bullet - i) \|_{L_p (2^{-j} (i + [0,1]))}^p \\
  & =    \sum_{i = 0}^{2^{j-1} - 1}
           2^{-j} \| \Delta_1 \|_{L_p ([0,1])}^p
    = \frac{\|\Delta_1\|_{L_p}^p}{2},
\end{align*}
and hence $\|D_{h_j e_1}^2 \Delta_{j,d}\|_{L_p(\Omega_{2,h_je_1})} \geq C_p$,
where $C_p := 2^{-1/p} \, \|\Delta_1\|_{L_p}$ for $p < \infty$. Since $\Omega_{2, h_j e_1} \subset \Omega = (0,1)^d$ has at most measure 1, we have $\|\cdot\|_{L_1(\Omega_{2, h_j e_1})} \leq \|\cdot\|_{L_\infty (\Omega_{2, h_j e_1})}$, hence the same holds for $p=\infty$ with $C_\infty := C_1$. By definition, this implies $\omega_2 (\Delta_{j,d})_p (t) \geq C_p$ for $t \geq |h_je_1| = 2^{-(j+1)}$.

Overall, we get by definition of the Besov quasi-norms in case of $q < \infty$ that
\[
  \|\Delta_{j,d}\|_{B^{s'}_{p,q}(\Omega)}^q
  \geq \int_0^\infty [t^{-{s'}} \omega_2 (\Delta_{j,d})_p (t)]^q \frac{dt}{t}
  \geq C_p^q \cdot \int_{2^{-(j+1)}}^\infty t^{-{s'} q - 1} \, dt
  =    \frac{C_p^q}{{s'} q} \cdot 2^{{s'} q (j+1)},
\]
and hence $\|\Delta_{j,d}\|_{B^{s'}_{p,q}(\Omega)} \geq \frac{C_p}{({s'} q)^{1/q}} \, 2^{{s'} ( j+1)}$
for all $j \in \N$.
In case of $q = \infty$, we see similarly that
\[
  \|\Delta_{j,d}\|_{B^{s'}_{p,q}(\Omega)}
  \geq \sup_{t \in (0,\infty)}
         t^{-{s'}} \, \omega_2 (\Delta_{j,d})_p (t)
  \geq C_p \cdot (2^{-(j+1)})^{-{s'}}
  =    C_p \cdot 2^{{s'} (j+1)}
\]
for all $j \in \N$.
In both cases, we used that ${s'} < 2$ to ensure that we can use
the modulus of continuity of order $2$ to compute the Besov quasi-norm.
Finally, note because of ${s'} \leq s$ that $B^s_{p,q}(\Omega) \hookrightarrow B^{s'}_{p,q}(\Omega)$;
see Equation~\eqref{eq:BesovEmbedding}.
This easily implies the claim.
\hfill$\square$


\subsection{Proof of Lemma~\ref{lem:MinGrowthRateNbPiecesSufficient}}
\label{sub:NumberOfPiecesSimplified}

In this section, we prove Lemma~\ref{lem:MinGrowthRateNbPiecesSufficient},
based on results of Telgarsky \cite{telgarsky2016benefits}.

Telgarsky makes extensive use of two special classes of functions:
First, a function $\sigma : \R \to \R$ is called \textbf{$(t,\beta)$-poly}
(where $t \in \N$ and $\beta \in \N_0$) if there is a partition of $\R$
into $t$ intervals $I_1,\dots,I_t$ such that $\sigma|_{I_j}$ is a polynomial
of degree at most $\beta$ for each $j \in \FirstN{t}$.
In the language of Definition~\ref{defn:PiecewisePolynomial},
these are precisely those functions which belong to $\PPoly_{t}^{\beta}(\R)$.
The second class of functions which is important are the
\textbf{$(t,\alpha,\beta)$-semi-algebraic functions} $f : \R^k \to \R$
(where $t \in \N$ and $\alpha,\beta \in \N_0$).
The definition of this class (see \cite[Definition 2.1]{telgarsky2016benefits}) is somewhat technical.
Luckily, we don't need the definition, all we need to know is the following result:

\begin{lem}\label{lem:MakingSemiAlgebraicFunctions}
  (see \cite[Lemma 2.3-(1)]{telgarsky2016benefits})
  If $\sigma : \R \to \R$ is $(t,\beta)$-poly and $q : \R^d \to \R$ is a (multivariate)
  polynomial of degree at most $\alpha \in \N_0$, then $\sigma \circ q$ is
  $(t,\alpha,\alpha \beta)$-semi-algebraic.
\end{lem}

In most of our proofs, we will mainly be interested in knowing that a function $\sigma : \R \to \R$
is $(t,\alpha)$-poly for certain $t,\alpha$.
The following lemma gives a sufficient condition for this to be the case.

\begin{lem}\label{lem:SemiAlgebraicAndPolyIsPoly}
  (see \cite[Lemma 3.6]{telgarsky2016benefits})
  If $f : \R^k \to \R$ is $(s,\alpha,\beta)$-semi-algebraic and if $g_1,\dots,g_k : \R \to \R$
  are $(t,\gamma)$-poly, then the function $f \circ (g_1,\dots,g_k) : \R \to \R$ is
  $\big( s t (1 + \alpha \gamma) \cdot k , \beta \gamma \big)$-poly.
\end{lem}

For proving Lemma~\ref{lem:MinGrowthRateNbPiecesSufficient}, we begin with the easier case
where we count neurons instead of weights.

\begin{proof}[Proof of the second part of Lemma~\ref{lem:MinGrowthRateNbPiecesSufficient}]
  We want to show that for any depth $L \in \N_{\geq 2}$ and degree $r \in \N$ there is a constant
  $\Lambda_{L,r} \in \N$ such that each function $f \in \NNreal^{\varrho_r, 1, 1}_{\infty,L,N}$
  is $(\Lambda_{L,r}N^{L-1},r^{L-1})$-poly.
  To show this, let $\Phi \in \NNsymbol^{\varrho_r,1,1}_{\infty,L,N}$ with $f = \Realization(\Phi)$,
  say $\Phi = \big( (T_1,\alpha_1), \dots, (T_K, \alpha_K) \big)$, where necessarily $K \leq L$,
  and where each $T_\ell : \R^{N_{\ell - 1}} \to \R^{N_\ell}$ is affine-linear.

  For $\ell \in \FirstN{K}$ and $j \in \FirstN{N_\ell}$, we let $f_j^{(\ell)} : \R \to \R$
  denote the output of neuron $j$ in the $\ell$-th layer.
  Formally, let $f_j^{(1)} : \R \to \R, x \mapsto \big( \alpha_1 (T_1 \, x) \big)_j$, and inductively
  \begin{equation}\label{eq:TmpNeuronOutput}
    f_j^{(\ell+1)}
    : \R \to \R,
    x \mapsto \Big[
                \alpha_{\ell + 1}
                \Big( T_{\ell+1} \big( f_k^{(\ell)}(x) \big)_{k \in \FirstN{N_\ell}} \Big)
              \Big]_j
    \quad \text{ for } \quad 1 \leq \ell \leq L-1 \text{ and } 1 \leq j \leq N_{\ell + 1}.
  \end{equation}
  We prove below by induction on $\ell \in \FirstN{K}$ that there is a constant $C_{\ell,r} \in \N$
  which only depends on $\ell,r$ and such that $f_j^{(\ell)}$ is
  $\big(C_{\ell,r} \prod_{t=0}^{\ell-1} N_t, r^{\gamma(\ell)}\big)$-poly,
  where $\gamma(\ell) := \min\{\ell, L-1\}$.
  Once this is shown, we see that $f = \Realization(\Phi) = f_1^{(K)}$ is
  $\big(C_{K,r} \prod_{t=0}^{K-1} N_t, r^{L-1}\big)$-poly.
  Then, because of $N_0 = 1$, we see that
  \[
    C_{K,r} \prod_{t=0}^{K-1} N_t
    \leq \Lambda_{L,r} \prod_{t=1}^{K-1} N_t
    \leq \Lambda_{L,r} \prod_{t=1}^{K-1} N(\Phi)
    \leq \Lambda_{L,r} \cdot [N(\Phi)]^{K-1}
    \leq \Lambda_{L,r} \cdot N^{L-1} ,
  \]
  where $\Lambda_{L,r} := \max_{1 \leq K \leq L} C_{K,r}$.
  Therefore, $f$ is indeed $(\Lambda_{L,r} \, N^{L-1}, r^{L-1})$-poly.

  \medskip{}

  \noindent
  \emph{Start of induction ($\ell = 1$):}
  Note that $L \geq 2$, so that $\gamma(\ell) = \ell = 1$.
  We have $T_1 x = a x + b$ for certain $a,b \in \R^{N_1}$ and
  $\alpha_1 = \varrho^{(1)} \otimes \cdots \otimes \varrho^{(N_1)}$
  for certain $\varrho^{(j)} \in \{\identity_\R, \varrho_r\}$.
  Thus, $\varrho^{(j)}$ is $(2,r)$-poly, and thus $(2,1,r)$-semi-algebraic
  according to Lemma~\ref{lem:MakingSemiAlgebraicFunctions}.
  Therefore, Lemma~\ref{lem:SemiAlgebraicAndPolyIsPoly} shows
  because of $f_j^{(1)}(x) = \varrho^{(j)} (b_j + a_j x)$ that $f_j^{(1)}$ is
  $(2(1+1), r)$-poly, for any $j \in \FirstN{N_1}$.
  Because of $N_0 = 1$, the claim holds for $C_{1,r} := 4$.

  \medskip{}

  \noindent
  \emph{Induction step ($\ell \to \ell + 1$):}
  Suppose that $\ell \in \FirstN{K-1}$ is such that the claim holds.
  Note that $\ell \leq K-1 \leq L-1$, so that $\gamma(\ell) = \ell$.

  We have $T_{\ell + 1} \, y = A \, y + b$ for certain $A \in \R^{N_{\ell + 1} \times N_\ell}$
  and $b \in \R^{N_{\ell + 1}}$, and
  $\alpha_{\ell + 1} = \varrho^{(1)} \otimes \cdots \otimes \varrho^{(N_{\ell + 1})}$
  for certain $\varrho^{(j)} \in \{\identity_\R, \varrho_r\}$, where $\varrho^{(j)} = \identity_\R$
  for all $j \in \FirstN{N_{\ell + 1}}$ in case of $\ell = K - 1$.
  Hence, $\varrho^{(j)}$ is $(2,r)$-poly, and even $(2,1)$-poly in case of $\ell = K-1$.
  Moreover, each of the polynomials
  \({
    p_{j,\ell}
    : \R^{N_\ell} \to \R,
    y \mapsto (A \, y + b)_j = b_j + \sum_{t=1}^{N_\ell} A_{j,t} \, y_t
  }\)
  is of degree at most $1$, hence by Lemma~\ref{lem:MakingSemiAlgebraicFunctions},
  $\varrho^{(j)} \circ p_{j,\ell}$ is $(2,1,r)$-semi-algebraic,
  and even $(2,1,1)$-semi-algebraic in case of $\ell = K-1$.

  Each function $f_t^{(\ell)}$ is $(C_{\ell,r} \prod_{t=0}^{\ell-1} N_t, r^{\ell})$-poly
  by the induction hypothesis.
  By Lemma~\ref{lem:SemiAlgebraicAndPolyIsPoly}, since
  \[
    f_j^{(\ell + 1)} (x)
    = \varrho^{(j)}
      \Big(
        \big[ A \, \big( f_t^{(\ell)} (x) \big)_{t \in \FirstN{N_\ell}} + b \big]_j
      \Big)
    = (\varrho^{(j)} \circ p_{j,\ell}) \big( f_1^{(\ell)} (x), \dots, f_{N_\ell}^{(\ell)} (x) \big),
  \]
  it follows that $f_j^{(\ell+1)}$ is $(P, r^{\ell+1})$-poly
  [respectively, $(P,r^{\ell})$-poly if $\ell = K-1$], where
  \[
    P \leq 2 C_{\ell,r} (1 + r^{\ell}) \cdot N_\ell \cdot \prod_{t=0}^{\ell-1} N_t
      =:   C_{\ell + 1, r} \cdot \prod_{t=0}^{(\ell+1) - 1} N_t .
  \]
  Finally, note in case of $\ell < K-1$ that $\ell + 1 \leq K - 1 \leq L-1$, and hence
  $\gamma(\ell+1) = \ell+1$, while in case of
  $\ell = K-1$ we have $\ell \leq \min \{\ell+1, L-1\} = \gamma(\ell+1)$ .
  Therefore, each $f_j^{(\ell+1)}$ is
  $(C_{\ell+1,r} \cdot \prod_{t=0}^{(\ell+1) - 1} N_t, r^{\gamma(\ell+1)})$-poly.
  This completes the induction, and thus the proof.
\end{proof}

The proof of the first part of Lemma~\ref{lem:MinGrowthRateNbPiecesSufficient}
uses the same basic arguments as in the preceding proof, but in a more careful way.
In particular, we will also need the following elementary lemma.

\begin{lem}\label{lem:SumOfPolyFunctions}
  Let $k \in \N$, and for each $i \in \FirstN{k}$ let $f_i : \R \to \R$ be $(t_i,\alpha)$-poly
  and continuous.
  Then the function $\sum_{i=1}^k f_i$ is $(t,\alpha)$-poly, where
  $t = 1 - k + \sum_{i=1}^k t_i$.
\end{lem}

\begin{proof}
  For each $i \in \FirstN{k}$, there are ``breakpoints''
  \(
    b_0^{(i)} := - \infty
    < b_1^{(i)}
    < \cdots
    < b_{t_i - 1}^{(i)}
    < \infty =: b_{t_i}^{(i)}
  \)
  such that $f_i |_{\R \cap [b_j^{(i)}, b_{j+1}^{(i)}]}$ is a polynomial of degree at most $\alpha$
  for each $0 \leq j \leq t_{i} - 1$.
  Here, we used the continuity of $f_i$ to ensure that we can use closed intervals.

  Now, let $M := \bigcup_{i=1}^k \{b_1^{(i)}, \ldots, b_{t_i - 1}^{(i)}\}$.
  We have $|M| \leq \sum_{i=1}^k (t_i - 1) = t - 1$, with $t$ as in the statement of the lemma.
  Thus, $M = \{b_1,\dots,b_s\}$ for some $0 \leq s \leq t - 1$, where
  \(
    b_0 := - \infty < b_1 < \cdots < b_s < \infty =: b_{s+1}.
  \)
  It is easy to see that $F := \sum_{i=1}^k f_i$ is such that $F|_{\R \cap [b_j, b_{j+1}]}$
  is a polynomial of degree at most $\alpha$ for each $0 \leq j \leq s$.
  Thus, $F$ is $(s+1,\alpha)$-poly and therefore also $(t,\alpha)$-poly.
\end{proof}

\begin{proof}[Proof of the first part of Lemma~\ref{lem:MinGrowthRateNbPiecesSufficient}]
  Let us first consider an arbitrary network $\Phi \in \NNsymbol^{\varrho_r, 1, 1}_{W,L,\infty}$
  satisfying $L(\Phi) = L$. Let $L_0 := \lfloor L/2 \rfloor \in \N_0$.
  We claim that
  \begin{equation}
    \Realization(\Phi)
    \text{ is }
    \big(\max \{1, \Lambda_{L,r} \, W^{L_0} \}, r^{L-1}\big)\text{-poly}
    \text{ where }
    \Lambda_{L,r} \in \N
    \text{ only depends on } L,r.
    \label{eq:NumberOfLinearPiecesWeightBoundAlmostDone}
  \end{equation}
  In case of $L = 1$, this is trivial, since then $\Realization(\Phi) : \R \to \R$ is affine-linear.
  Thus, we will assume $L \geq 2$ in what follows.
  Note that this entails $L_0 \geq 1$.

  Let $\Phi = \big( (T_1,\alpha_1), \dots, (T_L, \alpha_L) \big)$,
  where $T_\ell : \R^{N_{\ell - 1}} \to \R^{N_\ell}$ is affine-linear.
  We first consider the special case that $\|T_\ell\|_{\ell^0} = 0$ for some $\ell \in \FirstN{L}$.
  In this case, Lemma~\ref{lem:BoundingConnectionsWithLayers} shows that
  $\Realization(\Phi) \equiv c$ for some $c \in \R$.
  This trivially implies that $\Realization(\Phi)$ is
  $(\max \{1, \Lambda_{L,r} \, W^{L_0} \}, r^{L-1})$-poly.
  Thus, we can assume in the following that $\|T_\ell\|_{\ell^0} \neq 0$
  for all $\ell \in \FirstN{L}$.
  As in the proof of the first part of Lemma~\ref{lem:MinGrowthRateNbPiecesSufficient},
  we define $f_j^{(\ell)} : \R \to \R$ to be the function
  computed by neuron $j \in \FirstN{N_\ell}$ in layer $\ell \in \FirstN{L}$,
  cf.~Equation~\eqref{eq:TmpNeuronOutput}.
  \smallskip{}

  {\bf Step 1.} We let $L_1 := \lfloor \tfrac{L - 1}{2} \rfloor \in \N_0$,
  and we show by induction on $t \in \{0,1,\dots, L_1\}$ that
  \begin{equation}
    f_j^{(2 t + 1)}
    \text{ is }
    \Big( C_{t,r} \prod_{\ell=1}^t \|T_{2\ell}\|_{\ell^0}, r^{\gamma(t)} \Big)\text{-poly}
    \quad \forall \, t \in \{0,1,\dots,L_1\} \text{ and } j \in \FirstN{N_{2t + 1}} ,
    \label{eq:NumberOfLinearPiecesWeightBoundInductiveClaim}
  \end{equation}
  where $\gamma (t) := \min\{L-1, 2 t + 1\}$ and where the constant $C_{t,r} \in \N$
  only depends on $t,r$.
  Here, we use the convention that the empty product satisfies
  $\prod_{\ell=1}^0 \|T_{2\ell}\|_{\ell^0} = 1$.

  \medskip{}

  \noindent
  \emph{Induction start ($t=0$):}
  We have $T_1 x = a x + b$ for certain $a,b \in \R^{N_1}$ and
  $\alpha_1 = \varrho^{(1)} \otimes \cdots \otimes \varrho^{(N_1)}$ for certain
  $\varrho^{(j)} \in \{\identity_\R, \varrho_r\}$.
  In any case, $\varrho^{(j)}$ is $(2,r)$-poly, and hence $(2,1,r)$-semi-algebraic
  by Lemma~\ref{lem:MakingSemiAlgebraicFunctions}.
  Now, note
  \(
    f_j^{(2t+1)}(x)
    = f_j^{(1)}(x)
    = \varrho^{(j)} \big( (T_1 x)_j \big)
    = \varrho^{(j)} (a_j x + b_j)
  \),
  so that Lemma~\ref{lem:SemiAlgebraicAndPolyIsPoly} shows that
  $f_j^{(2t+1)}$ is $(2 (1 + 1), r)$-poly.
  Thus, Equation~\eqref{eq:NumberOfLinearPiecesWeightBoundInductiveClaim} holds for $t = 0$
  if we choose $C_{0,r} := 4$.
  Here, we used that $L \geq 2$ and $t=0$, so that $L-1 \geq 2t + 1$
  and hence $\gamma(t) = 2 t + 1 = 1$.

  \medskip{}

  \noindent
  \emph{Induction step $(t \to t+1)$:}
  Let $t \in \N_0$ such that $t+1 \leq \tfrac{L-1}{2}$ and such that
  Equation~\eqref{eq:NumberOfLinearPiecesWeightBoundInductiveClaim} holds for $t$.
  We have $T_{2t + 2} \bullet = A \bullet + b$ for certain $A \in \R^{N_{2t+2} \times N_{2t+1}}$
  and $b \in \R^{N_{2t+2}}$, and furthermore
  $\alpha_{2t+2} = \varrho^{(1)} \otimes \cdots \otimes \varrho^{(N_{2t+2})}$
  for certain $\varrho^{(j)} \in \{ \identity_\R, \varrho_r \}$.
  Recall from Appendix~\ref{sec:AppendixTechnical} that $A_{j,-} \in \R^{1 \times N_{2t+1}}$
  denotes the $j$-th row of $A$.
  For $j \in \FirstN{N_{2t+2}}$, we claim that
  \begin{equation}
    \begin{cases}
      f_j^{(2t+2)} \equiv \varrho^{(j)}(b_j),
      & \text{if } A_{j,-} = 0, \\[0.1cm]
      f_j^{(2t+2)}
      \text{ is }
      (C_{t,r}' \cdot M_j \cdot \prod_{\ell=1}^t \|T_{2\ell}\|_{\ell^0}, r^{2t+2})\text{-poly},
      & \text{if } A_{j,-} \neq 0,
    \end{cases}
    \label{eq:NumberOfLinearPiecesWeightBoundIngredient1}
  \end{equation}
  where $M_j := \|A_{j,-}\|_{\ell^0}$, and where the constant $C_{t,r}' \in \N$ only depends on $t,r$.


  The first case where $A_{j,-}=0$ is trivial.
  For proving the second case where $A_{j,-} \neq 0$, let us define
  $\Omega_j := \{ i \in \FirstN{N_{2t+1}} \colon A_{j,i} \neq 0 \}$, say
  $\Omega_j = \{ i_1, \dots, i_{M_j} \}$ with (necessarily) pairwise distinct $i_1,\dots,i_{M_j}$.
  By introducing the polynomial
  ${p_{j,t} : \R^{M_j} \to \R, y \mapsto b_j + \sum_{m=1}^{M_j} A_{j,i_m} y_m}$,
  we can then write
  \[
    f_j^{(2t+2)}(x)
    = \varrho^{(j)}
      \Big(
        b_j + A_{j,-} \big( f_k^{(2t+1)} (x) \big)_{k \in \FirstN{N_{2t+1}}}
      \Big)
    = (\varrho^{(j)} \circ p_{j,t}) \big( f_{i_1}^{(2t+1)} (x), \dots, f_{i_{M_j}}^{(2t+1)} (x) \big) .
  \]
  Since $\varrho^{(j)}$ is $(2,r)$-poly and $p_{j,t}$ is a polynomial of degree at most $1$,
  Lemma~\ref{lem:MakingSemiAlgebraicFunctions} shows that $\varrho^{(j)} \circ p_{j,t}$
  is $(2,1,r)$-semi-algebraic.
  Furthermore, by the induction hypothesis we know that each function $f_{i_m}^{(2t+1)}$
  is $(C_{t,r} \prod_{\ell=1}^t \|T_{2\ell}\|_{\ell^0}, r^{2t+1})$-poly,
  where we used that $\gamma(t)=2t+1$ since $t+1 \leq (L-1)/2$.
  Therefore---in view of the preceding displayed equation---Lemma~\ref{lem:SemiAlgebraicAndPolyIsPoly}
  shows that the function $f_j^{(2t+2)}$ is indeed
  $(C_{t,r}' \cdot M_j \cdot \prod_{\ell=1}^t \|T_{2\ell}\|_{\ell^0}, r^{2t+2})$-poly,
  where $C_{t,r}' := 2 C_{t,r} \cdot (1 + r^{2t+1})$.

  We now estimate the number of polynomial pieces of the function $f_i^{(2t+3)}$
  for $i \in \FirstN{N_{2t+3}}$.
  To this end, let $B \in \R^{N_{2t+3} \times N_{2t+2}}$ and $c \in \R^{N_{2t+3}}$ such that
  $T_{2t+3} = B \bullet + c$, and choose $\sigma^{(i)} \in \{\identity_\R, \varrho_r\}$
  such that $\alpha_{2t+3} = \sigma^{(1)} \otimes \cdots \otimes \sigma^{(N_{2t+3})}$.
  For $i \in \{1,\dots,N_{2t+3}\}$, let us define
  \[
    G_{i,t}
    : \R \to \R,
    x \mapsto \sum_{j \in \FirstN{N_{2t+2}} \text{ such that } A_{j,-} \neq 0}
                B_{i,j} \, f_j^{(2t+2)} (x) .
  \]
  In view of Equation~\eqref{eq:NumberOfLinearPiecesWeightBoundIngredient1},
  Lemma~\ref{lem:SumOfPolyFunctions} shows that $G_{i,t}$ is $(P,r^{2t+2})$-poly, where
  \begin{align*}
    P
    & \leq 1
           - |\{ j \in \FirstN{N_{2t+2}} \colon A_{j,-} \neq 0 \}|
           + C_{t,r}'
             \cdot \Big( \prod_{\ell=1}^t \|T_{2\ell}\|_{\ell^0} \Big)
             \sum_{j \in \FirstN{N_{2t+2}} \text{ such that } A_{j,-} \neq 0}
               M_j \\
    & \leq C_{t,r}' \cdot \Big( \prod_{\ell=1}^t \|T_{2\ell}\|_{\ell^0} \Big) \cdot \|A\|_{\ell^0}
    =    C_{t,r}' \cdot \prod_{\ell=1}^{t+1} \|T_{2\ell}\|_{\ell^0}
  \end{align*}
  Here, we used that $\| T_{2t+2} \|_{\ell^0} \neq 0$ and hence $A \neq 0$, so that
  $|\{ j \in \FirstN{N_{2t+2}} \colon A_{j,-} \neq 0\}| \geq 1$.

  Next, note because of Equation~\eqref{eq:NumberOfLinearPiecesWeightBoundIngredient1}
  and by definition of $G_{i,t}$ that there is some $\theta_{i,t} \in \R$ satisfying
  \[
    f_i^{(2t+3)}(x)
    = \sigma^{(i)} \Big( c_i + \sum_{j=1}^{N_{2t+2}} B_{i,j} \, f_j^{(2t+2)} (x) \Big)
    = \sigma^{(i)} \big( \theta_{i,t} + G_{i,t}(x) \big)
    \quad \forall \, x \in \R .
  \]
  Now there are two cases: If $2t + 3 > L-1$, then $2t+3 = L$, since $t+1 \leq \tfrac{L-1}{2}$.
  Therefore, $\sigma^{(i)} = \identity_\R$, so that we see that
  $f_i^{(2t+3)} = \theta_{i,t} + G_{i,t}$ is
  $(C_{t,r}' \cdot \prod_{\ell=1}^{t+1} \|T_{2\ell}\|_{\ell^0}, r^{2t+2})$-poly,
  where $2t+2 = L-1 = \gamma(t+1)$.

If $2t + 3 \leq L-1$, then $\gamma(t+1) = 2t + 3$.
  Furthermore, each $\sigma^{(i)}$ is $(2,r)$-poly and hence $(2,1,r)$-semi-algebraic
  by Lemma~\ref{lem:MakingSemiAlgebraicFunctions}.
  In view of the preceding displayed equation, and since $G_{i,t}$ is
  ${(C_{t,r}' \cdot \prod_{\ell=1}^{t+1} \|T_{2\ell}\|_{\ell^0}, r^{2t+2})}$-poly,
  Lemma~\ref{lem:SemiAlgebraicAndPolyIsPoly} shows that $f_i^{(2t+3)}$ is
  $\big( 2 (1 + r^{2t+2}) C_{t,r}' \cdot \prod_{\ell=1}^{t+1} \|T_{2\ell}\|_{\ell^0}, r^{2t+3} \big)$-poly.

In each case, with $C_{t+1,r} := 2 (1 + r^{2t+2}) C_{t,r}'$, we see that
  Equation~\eqref{eq:NumberOfLinearPiecesWeightBoundInductiveClaim} holds for $t+1$ instead of $t$.

  \medskip{}

  {\bf Step 2.}   We now complete the proof of
  Equation~\eqref{eq:NumberOfLinearPiecesWeightBoundAlmostDone},
  by distinguishing whether $L$ is odd or even.

  \smallskip{}

  \emph{If $L$ is odd:} In this case $L_1 = \lfloor \tfrac{L-1}{2} \rfloor = \tfrac{L-1}{2}$,
  so that we can use Equation~\eqref{eq:NumberOfLinearPiecesWeightBoundInductiveClaim}
  for the choice $t = \tfrac{L-1}{2}$ to see that $\Realization(\Phi) = f_1^{(L)} = f_1^{(2t+1)}$
  is $(P, r^{L-1})$-poly, where
  \[
    P \leq C_{t,r} \prod_{\ell=1}^t \|T_{2\ell}\|_{\ell^0}
      \leq C_{t,r} \prod_{\ell=1}^{(L-1)/2} \!\!\! W(\Phi)
      \leq C_{t,r} \cdot [W(\Phi)]^{(L-1)/2}
      \leq C_{t,r} \cdot W^{\lfloor L/2 \rfloor}.
  \]

  \emph{If $L$ is even:}
  In this case, set $t := \tfrac{L}{2} - 1 \in \{ 0,1,\dots,L_1 \}$,
  and note $2t+1 = L-1 =\gamma(t)$.
  Hence, with $A \in \R^{1 \times N_{L-1}}$ and $b \in \R$ such that $T_L = A \bullet + b$, we have
  \[
    \Realization(\Phi)
    = T_L \big( f_k^{(2t+1)} (x) \big)_{k \in \FirstN{N_{L-1}}}
    = b + \sum_{k \in \FirstN{N_{L-1}} \text{ such that } A_{1,k} \neq 0}
            A_{1,k} \, f_k^{(2t+1)}(x) .
  \]
  Therefore, thanks to Equation~\eqref{eq:NumberOfLinearPiecesWeightBoundInductiveClaim},
  Lemma~\ref{lem:SumOfPolyFunctions} shows that $\Realization(\Phi)$ is $(P,r^{2t+1})$-poly, where
  \begin{align*}
    P & \leq 1
             - |\{ k \in \FirstN{N_{L-1}} \colon A_{1,k} \neq 0 \}|
             + C_{t,r} \sum_{\substack{k \in \FirstN{N_{L-1}} \\ \text{such that } A_{1,k} \neq 0}}
                         \prod_{\ell=1}^t \|T_{2\ell}\|_{\ell^0} \\
      & \leq C_{t,r}
             \cdot \|A\|_{\ell^0}
             \cdot \prod_{\ell=1}^{\frac{L}{2} - 1} \|T_{2\ell}\|_{\ell^0}
        =    C_{t,r} \prod_{\ell=1}^{L/2} \|T_{2\ell}\|_{\ell^0} \\
      & \leq C_{t,r} \cdot [W(\Phi)]^{L/2}
        =    C_{t,r} \cdot [W(\Phi)]^{\lfloor L/2 \rfloor}
        \leq C_{t,r} \cdot W^{\lfloor L/2 \rfloor} .
  \end{align*}
  In the second inequality we used
  $|\{ k \in \FirstN{N_{L-1}} \colon A_{1,k} \neq 0\}| = \|A\|_{\ell^{0}} = \|T_L\|_{\ell^0} \geq 1$.
  We have thus established Equation~\eqref{eq:NumberOfLinearPiecesWeightBoundAlmostDone}
  in all cases.

  \medskip{}

{\bf Step 3.}  It remains to prove the actual claim.
  Let $f \in \NNreal^{\varrho_r,1,1}_{W,L,\infty}$ be arbitrary,
  whence $f = \Realization(\Phi)$ for some $\Phi \in \NNreal^{\varrho_r,1,1}_{W,K,\infty}$
  with $L(\Phi) = K$ for some $K \in \N_{\leq L}$.
  In view of Equation~\eqref{eq:NumberOfLinearPiecesWeightBoundAlmostDone}, this implies that
  $f = \Realization(\Phi)$ is $(\max\{1, \Lambda_{K,r} \, W^{\lfloor K/2 \rfloor}\}, r^{K-1})$-poly.
  If we set $\Theta_{L,r} := \max_{1 \leq K \leq L} \Lambda_{K,r}$, then this easily implies that
  $f$ is $(\max\{1, \Theta_{L,r} \, W^{\lfloor L/2 \rfloor}\}, r^{L-1})$-poly, as desired.
\end{proof}


\section{The spaces \texorpdfstring{$\WASpace[\StandardXSpace[]][\varrho_r][q][\alpha][L]$}{W(Xₚᵏ, ϱᵣ, L)}
and \texorpdfstring{$\NASpace[\StandardXSpace[]][\varrho_r][q][\alpha][L]$}{N(Xₚᵏ, ϱᵣ, L)} are distinct}
\label{sec:WeightAndNeuronSpacesDistinct}

In this section, we show that for a \emph{fixed} depth $L \geq 3$ and $\Omega = (0,1)^d$
the approximation spaces defined in terms of the number of weights and in terms of the number
of neurons are distinct; that is, we show
\begin{equation}
  \WASpace[\StandardXSpace[](\Omega)][\varrho_r][q][\alpha][L]
  \neq \NASpace[\StandardXSpace[](\Omega)][\varrho_r][q][\alpha][L].
  \label{eq:WeightSpaceDistinctFromNeuronSpace}
\end{equation}
The proof is based on several results by Telgarsky \cite{telgarsky2016benefits},
which we first collect.
The first essential concept is the notion of the \emph{crossing number} of a function.

\begin{defn}\label{def:CrossingNumber}
  For any piecewise polynomial function $f : \R \to \R$ with finitely many pieces,
  define $\widetilde{f} : \R \to \{0,1\}, x \mapsto \Indicator_{f(x) \geq 1/2}$.
  Thanks to our assumption on $f$, the sets $\widetilde{f}^{-1} (\{0\}) \subset \R$
  and $\widetilde{f}^{-1} (\{1\}) \subset \R$ are finite unions of (possibly degenerate)
  intervals.
  For $i \in \{0,1\}$, denote by $I_f^{(i)} \subset 2^{\R}$ the set of connected components
  of $\widetilde{f}^{-1} (\{i\})$.
  Finally, set $I_f := I_f^{(0)} \cup I_f^{(1)}$ and define the \emph{crossing number}
  $\Cr (f)$ of $f$ as $\Cr (f) := |I_f| \in \N$.
\end{defn}

The following result gives a bound on the crossing number of $f$,
based on bounds on the complexity of $f$.
Here, we again use the notion of $(t,\beta)$--poly functions as introduced at the beginning of
Appendix~\ref{sub:NumberOfPiecesSimplified}.

\begin{lem}\label{lem:CrossingNumberBound}(\cite[Lemma 3.3]{telgarsky2016benefits})
  If $f : \R \to \R$ is $(t,\alpha)$--poly, then $\Cr (f) \leq t (1 + \alpha)$.
\end{lem}

Finally, we will need the following result which tells us that if $\Cr (f) \gg \Cr (g)$,
then the functions $\widetilde{f},\widetilde{g}$ introduced in Definition~\ref{def:CrossingNumber}
differ on a large number of intervals $I \in I_f$.

\begin{lem}\label{lem:TelgarskyFundamental}(\cite[Lemma 3.1]{telgarsky2016benefits})
  Let $f : \R \to \R$ and $g : \R \to \R$ be piecewise polynomial with finitely many pieces.
  Then
  \[
    \frac{1}{\Cr (f)} \cdot
    \Big|
      \big\{
        I \in I_f
        \colon
        \forall x \in I : \widetilde{f}(x) \neq \widetilde{g} (x)
      \big\}
    \Big|
    \geq \frac{1}{2} \Big( 1 - 2 \frac{\Cr(g)}{\Cr (f)} \Big).
    \qedhere
  \]
\end{lem}

The first step to proving Equation~\eqref{eq:WeightSpaceDistinctFromNeuronSpace}
will be the following estimate:


\begin{lem}\label{lem:SawtoothHardToApproximate}
  Let $p \in (0,\infty]$.
  There is a constant $C_p > 0$ such that the error of best approximation
  (cf.~Equation~\eqref{eq:ApproximationErrorDefinition}) of the
  ``sawtooth function'' $\Delta_j$ (cf.~Equation~\eqref{eq:DefSawtooth})
  by piecewise polynomials satisfies
  \[
    \AppErr (\Delta_j , \PPoly_N^\alpha)_{L_p ( (0,1) )}
    \geq C_p
    \qquad \forall \, j,\alpha \in \N,\quad \forall \, 1 \leq N \leq \frac{2^{j} + 1}{4(1+\alpha)}.
    \qedhere
  \]
\end{lem}

For proving this lower bound, we first need to determine the crossing number of $\Delta_j$.

\begin{lem}\label{lem:SawtoothCrossingNumber}
  Let $j \in \N$ and $\Delta_j : \R \to \R$ as in Equation~\eqref{eq:DefSawtooth}.
  We have $\Cr(\Delta_j) = 1 + 2^j$ and
  \[
    \int_{I \cap [0,1]}
      \big| \Delta_j (x) - \frac{1}{2} \big|
    \, d x
    \geq 2^{-j-3}
    \qquad \forall \, I \in I_{\Delta_j}.
    \qedhere
  \]
\end{lem}


\begin{proof}
The formal proof is omitted as it involves tedious but straightforward computations;
graphically, the claimed properties are straightforward consequences of Figure~\ref{fig:SawToothPlot}.
\end{proof}


\begin{proof}[Proof of Lemma~\ref{lem:SawtoothHardToApproximate}]
  Let $j,\alpha \in \N$ and let $N \in \N$ with $N \leq \frac{2^{j} + 1}{4(1+\alpha)}$
  and $f \in \PPoly_N^\alpha$ be arbitrary.
  Lemma~\ref{lem:CrossingNumberBound}
  shows $\Cr(f) \leq N(1 + \alpha) \leq \frac{2^{j} + 1}{4}$,
  so that Lemma~\ref{lem:SawtoothCrossingNumber} implies
  ${\theta := 1 - 2 \frac{\Cr(f)}{\Cr(\Delta_j)} = 1 - 2 \frac{\Cr(f)}{1 + 2^j} \geq \tfrac{1}{2}}$.
  Now, recall the notation of Definition~\ref{def:CrossingNumber}, and set
  \[
    G :=
    \big\{
      I \in I_{\Delta_j}
      \,\big|\,
      \forall \, x \in I : \widetilde{\Delta_j}(x) \neq \widetilde{f}(x)
    \big\}.
  \]
  By Lemma~\ref{lem:TelgarskyFundamental},
  $\frac{1}{\Cr(\Delta_j)} |G| \geq \frac{\theta}{2} \geq \frac{1}{4}$,
  which means $|G| \geq \frac{1 + 2^j}{4} \geq 2^{j-2}$,
  since we have $\Cr(\Delta_j) = 1 + 2^j$.

  For arbitrary $I \in G$, we have $\widetilde{\Delta_j}(x) \neq \widetilde{f}(x)$ for all $x \in I$,
  so that either $f(x) < \tfrac{1}{2} \leq \Delta_j (x)$ or $\Delta_j(x) < \tfrac{1}{2} \leq f(x)$.
  In both cases, we get $|\Delta_j (x) - f(x)| \geq |\Delta_j(x) - \tfrac{1}{2}|$.
  Furthermore, recall that $0 \leq \Delta_j \leq 1$,
  so that $|\Delta_j(x) - \tfrac{1}{2}| \leq \tfrac{1}{2} \leq 1$.
  Because of $\|\Delta_j - f\|_{L_p ( (0,1) )} \geq \|\Delta_j - f\|_{L_1 ( (0,1) )}$ for $p \geq 1$,
  it is sufficient to prove the result for $0 < p \leq 1$.
  For this range of $p$, we see that
  \[
    \big| \Delta_j(x) - \tfrac{1}{2} \big|
    = \big| \Delta_j (x) - \tfrac{1}{2} \big|^{1 - p}
      \cdot \big| \Delta_j (x) - \tfrac{1}{2} \big|^{p}
    \leq \big| \Delta_j (x) - \tfrac{1}{2} \big|^{p} .
  \]
  Overall, we get
  $|\Delta_j(x) - f(x)|^p \geq |\Delta_j(x) - \tfrac{1}{2}|^p \geq |\Delta_j(x) - \tfrac{1}{2}|$
  for all $x \in I$ and $I \in G$.
  Thus, 
  \begin{align*}
    \int_{[0,1]}
      |\Delta_j(x) - f(x)|^p
    \, d x
    & \geq \sum_{I \in G}
             \int_{I \cap [0,1]}
               | \Delta_j(x) - f(x)|^p
             \, dx
     \geq \sum_{I \in G}
             \int_{I \cap [0,1]}
               \Big| \Delta_j(x) - \frac{1}{2} \Big|
             \, d x \\
    ({\scriptstyle{\text{Lemma}~\ref{lem:SawtoothCrossingNumber}}})
    & \geq 
           \sum_{I \in G} 2^{-j-3}
      =    |G| \cdot 2^{-j-3}
      \geq 2^{j-2} \cdot 2^{-j-3}
      =    2^{-5}.
  \end{align*}
  This implies $\|\Delta_j - f\|_{L_p ( (0,1) )} \geq 2^{-5/p} =: C_p$.
\end{proof}


As a consequence of the lower bound in Lemma~\ref{lem:SawtoothHardToApproximate},
we can now prove lower bounds for the neural network approximation space norms of
the multivariate sawtooth function $\Delta_{j,d}$ (cf.~Definition~\ref{def:SawtoothFunction})

\begin{prop}\label{prop:SawtoothNormLowerBound}
Consider $\Omega = [0,1]^{d}$, $r \in \N$, $L \in \N_{\geq 2}$, $\alpha \in (0, \infty)$, $p,q \in (0,\infty]$.
There is a constant ${C = C(d,r,L,\alpha,p,q) > 0}$ such that
\[
  \|\Delta_{j,d}\|_{\WASpace[\StandardXSpace[](\Omega)][\varrho_r][q][\alpha][L]}
  \geq C \cdot 2^{\alpha j / \lfloor L / 2 \rfloor}
  \quad \text{and} \quad
  \|\Delta_{j,d}\|_{\NASpace[\StandardXSpace[](\Omega)][\varrho_r][q][\alpha][L]}
  \geq C \cdot 2^{\alpha j / (L-1)},\qquad \forall j \in \N. \qedhere
\]
\end{prop}

\begin{proof}
  According to Lemma~\ref{lem:MinGrowthRateNbPiecesSufficient}, there is a constant
  $C_1 = C_1 (r,L) \in \N$ such that
  \begin{equation}
    \NNreal^{\varrho_r, 1, 1}_{W,L,\infty}
    \subset \PPoly_{C_1 \cdot W^{\lfloor L/2 \rfloor}}^{\beta}
    \quad \text{and} \quad
    \NNreal^{\varrho_r,1,1}_{\infty,L,N}
    \subset \PPoly_{C_1 \cdot N^{L-1}}^{\beta}
    \quad \text{where} \quad \beta := r^{L-1}.
    \label{eq:NetworksArePiecewisePolynomial}
  \end{equation}
  We first prove the estimate regarding
  $\|\Delta_{j,d}\|_{\WASpace[\StandardXSpace[](\Omega)][\varrho_r][q][\alpha][L]}$.
  To this end, note that there is a constant $C_2 = C_2(L,\beta,C_1) = C_2(L,r) > 0$ such that
  \(
    \big( \tfrac{2^{j+1}}{4 C_1 (1 + \beta)} \big)^{1/\lfloor L/2 \rfloor}
    = C_2 \cdot 2^{(j+1) / \lfloor L/2 \rfloor}
  \).
  Now, let $W \in \N_0$ with $W \leq C_2 \cdot 2^{(j+1) / \lfloor L/2 \rfloor}$
  and $F \in \NNreal^{\varrho_{r},d,1}_{W,L,\infty}$ be arbitrary.
  Define $F_{x'} : \R \to \R, t \mapsto F ((t, x'))$
  for $x' \in [0,1]^{d-1}$.
  According to Lemma~\ref{lem:NetworkCalculus}-(1) and Equation~\eqref{eq:NetworksArePiecewisePolynomial},
  we have
  \(
    F_{x'} \in \NNreal_{W,L,\infty}^{\varrho_r,1,1}
           \subset \PPoly_{C_1 \cdot W^{\lfloor L/2 \rfloor}}^{\beta}.
  \)
  Since
  \(
    C_1 \cdot W^{\lfloor L/2 \rfloor}
    \leq C_1 \cdot \tfrac{2^{j+1}}{4 C_1 (1+\beta)}
    = \tfrac{2^{j+1}}{4(1+\beta)},
  \)
  Lemma~\ref{lem:SawtoothHardToApproximate} yields a constant $C_3 = C_3 (p) > 0$ such that
  $C_3 \leq \|\Delta_j - F_{x'}\|_{L_p ( (0,1) )}$.
  For $p < \infty$, Fubini's theorem shows that
  \begin{equation*}
    \begin{split}
      \| \Delta_{j,d} - F \|_{L_p(\Omega)}^p
      & \geq \int_{[0,1]^{d-1}}
               \int_{[0,1]}
                 \Big|
                   \Delta_{j}(x_{1})
                   - F((x_1, x'))
                 \Big|^p
               \, d x_1
             \, d x' \\
      & = \int_{[0,1]^{d-1}}
              \| \Delta_j - F_{x'} \|_{L_p ( (0,1) )}^p
            \, d x'
        \geq C_3^p \cdot \int_{[0,1]^{d-1}} \, d x'
        =    C_3^p.
    \end{split}
  \end{equation*}
  Therefore,
  \begin{equation}
    \AppErr (\Delta_{j,d}, \NNreal^{\varrho_r, d, 1}_{W,L,\infty})_{L_p (\Omega)}
    \geq C_3 > 0
    \quad \forall \, W \in \N_0 \text{ satisfying } W \leq C_2 \cdot 2^{(j+1) / \lfloor L / 2 \rfloor} .
    \label{eq:SawtoothApproximationErrorLowerBound}
  \end{equation}
  Since $\|\bullet\|_{L^\infty (\Omega)} \geq \|\bullet\|_{L^1(\Omega)}$,
  this also holds for $p = \infty$.
  In light of the embedding \eqref{eq:ApproximationSpaceEmbedding}, it is sufficient to lower bound
  $\|\Delta_{j,d}\|_{\WASpace[\StandardXSpace[](\Omega)][\varrho_r][q][\alpha][L]}$
  when $q = \infty$.
  In this case, we have
  \begin{align*}
    \|\Delta_{j,d}\|_{\WASpace[\StandardXSpace[]][\varrho_r][\infty][\alpha][L]}
    & = \max \Big\{
               \|\Delta_{j,d}\|_{L_p} \,\,, \quad
               \sup_{W \in \N}
                 \big[
                   (1 + W)^\alpha \cdot
                   \AppErr(\Delta_{j,d}, \NNreal_{W,L,\infty}^{\varrho_r, d, 1})
                 \big]
             \Big\} \\
    ({\scriptstyle{\text{Equation}~\eqref{eq:SawtoothApproximationErrorLowerBound}}})
    & 
      \geq
      \begin{cases}
        \|\Delta_{j,d}\|_{L_p} \geq C_3 ,
        & \text{if } C_2 \cdot 2^{(j+1) / \lfloor L/2 \rfloor} < 1 \\
        C_3 \cdot (1 + \lfloor C_2 \cdot 2^{(j+1) / \lfloor L/2 \rfloor} \rfloor)^\alpha,
        & \text{if } C_2 \cdot 2^{(j+1) / \lfloor L/2 \rfloor} \geq 1
      \end{cases} \\
    & \geq C_3 \, C_2^\alpha \cdot 2^{j \alpha / \lfloor L/2 \rfloor} ,
  \end{align*}
  as desired.
  This completes the proof of the lower bound of
  $\|\Delta_{j,d}\|_{\WASpace[\StandardXSpace[]][\varrho_r][q][\alpha][L]}$.

  \medskip{}

  The lower bound for
  $\|\Delta_{j,d}\|_{\NASpace[\StandardXSpace[]][\varrho_r][q][\alpha][L]}$
  can be derived similarly.
  First, in the same way that we proved Equation~\eqref{eq:SawtoothApproximationErrorLowerBound},
  one can show that
  \[
    \AppErr (\Delta_{j,d}, \NNreal^{\varrho_r, d, 1}_{\infty,L,N})_{L_p (\Omega)}
    \geq C_3 > 0
    \quad \forall \, N \in \N_0 \text{ satisfying } N \leq C'_2 \cdot 2^{(j+1) / (L-1)} ,
  \]
  for a suitable constant $C'_2 = C'_2 (L,r) > 0$.
  The remainder of the argument is then almost identical to that for estimating
  $\|\Delta_{j,d}\|_{\WASpace[\StandardXSpace[](\Omega)][\varrho_r][q][\alpha][L]}$,
  and is thus omitted.
\end{proof}

As our final preparation for showing that the spaces
$\WASpace[\StandardXSpace[](\Omega)][\varrho_r][q][\alpha][L]$ and
$\NASpace[\StandardXSpace[](\Omega)][\varrho_r][q][\alpha][L]$ are distinct for $L \geq 3$
(Lemma~\ref{lem:WeightAndNeuronSpacesDistinct}), we will show that the lower bound
derived in Proposition~\ref{prop:SawtoothNormLowerBound} is sharp
and extends to arbitrary measurable $\Omega$ with nonempty interior.

\begin{thm}\label{thm:SawtoothNormAsymptotic}
 Let $p,q \in (0,\infty]$, $\alpha > 0$, $r \in \N$, $L \in \N_{\geq 2}$,
 and let $\Omega \subset \R^d$ be a bounded admissible domain with non-empty interior.
 Consider  $y \in \R^d$ and $s > 0$ satisfying $y + [0,s]^d \subset \Omega$
 and define
  \[
    \Delta_j^{(y,s)} : \R^d \to [0,1], x \mapsto \Delta_{j,d}\Big(\frac{x-y}{s}\Big)
    \quad \text{for } j \in \N.
  \]
  Then there are $C_1, C_2 > 0$ such that for all $j \in \N$ the function $\Delta_j^{(y,s)}$
  satisfies
  \begin{eqnarray*}
    C_1 \cdot 2^{j \alpha / \lfloor L / 2 \rfloor}
    & \leq & \|\Delta_j^{(y,s)}\|_{\WASpace[\StandardXSpace[](\Omega)][\varrho_r][q][\alpha][L]}
    \leq C_2 \cdot 2^{j \alpha / \lfloor L / 2 \rfloor}\\
 \text{and}\quad
    C_1 \cdot 2^{j \alpha / (L-1)}
    & \leq & \|\Delta_j^{(y,s)}\|_{\NASpace[\StandardXSpace[](\Omega)][\varrho_r][q][\alpha][L]}
    \leq C_2 \cdot 2^{j \alpha / (L-1)}.
    \qedhere
  \end{eqnarray*}
\end{thm}

%

\begin{proof}
  For the upper bound, since $\Omega$ is bounded, Theorem~\ref{thm:ReLUPowersApproxSpaces}
  (Equation~\eqref{eq:Nesting1}, which also holds for $N_q^\alpha$ instead of $W_q^\alpha$)
  shows that it suffices to prove the claim for $r = 1$.
  Since $T_{y,s} : \R^d \to \R^{d}, x \mapsto s^{-1} (x - y)$
  satisfies $\|T_{y,s}\|_{\ell^{0,\infty}_\ast} = 1$,
  a combination of Lemmas~\ref{lem:SawtoothImplementation} and
  \ref{lem:NetworkCalculus}-(\ref{enu:PrePostAffine}) shows that there is a constant $C_L > 0$
  satisfying
 \ifarxiv
  \[
    \Delta_j^{(y,s)} \in \NNreal^{\varrho_1,
                                  d,
                                  1}_{\infty,
                                      L,
                                      \lfloor C_L \cdot 2^{j / (L-1)} \rfloor}
    \qquad \text{and} \qquad
    \Delta_j^{(y,s)} \in \NNreal^{\varrho_1,
                                  d,
                                  1}_{\lfloor C_L \cdot 2^{j / \lfloor L/2 \rfloor} \rfloor,
                                      L,
                                      \infty}
    \qquad \forall \, j \in \N .
  \]
  \fi
  Furthermore, $\Delta_j^{(y,s)} \in X_p (\Omega)$ since $\Omega$ is bounded and $\Delta_j^{(y,s)}$
  is bounded and continuous.
  Thus, the Bernstein inequality \eqref{eq:TrivialBernsteinInequality} yields a constant
  $K_1 > 0$ such that
  \[
    \|\Delta_j^{(y,s)}\|_{\NASpace[\StandardXSpace[](\Omega)][\varrho_1][q][\alpha][L]}
    \leq K_1 \cdot \lfloor C_L \cdot 2^{j / (L-1)} \rfloor^\alpha
    \leq K_1 C_{L}^\alpha \cdot 2^{j \alpha / (L-1)}
  \]
  for all $j \in \N$; similarly, we get a constant $K_2 > 0$ such that
  \[
    \|\Delta_j^{(y,s)}\|_{\WASpace[\StandardXSpace[](\Omega)][\varrho_1][q][\alpha][L]}
    \leq K_2 \cdot \lfloor C_L \cdot 2^{j / \lfloor L/2 \rfloor } \rfloor^\alpha
    \leq K_2 C_{L}^\alpha \cdot 2^{j \alpha / \lfloor L/2 \rfloor}
  \]
  for all $j \in \N$.
  Considering $C_2 := \max \{ K_1,K_2 \} \cdot C_L^\alpha$ establishes the desired upper bound.

  For the lower bound, consider arbitrary $W,N \in \N_{0}$, $F \in \NNreal^{\varrho_{r},d,1}_{W,L,N}$,
  and observe that by Lemma~\ref{lem:NetworkCalculus}-(\ref{enu:PrePostAffine})
  we have $F' := F \circ T_{y,s}^{-1}  \in \NNreal^{\varrho_{r},d,1}_{W,L,N}$.
  In view of Proposition~\ref{prop:SawtoothNormLowerBound},
  the lower bound follows from the inequality
  \[
    \|\Delta_{j}^{(y,s)} \!-\! F\|_{L_{p}(\Omega)}
    \!\geq\! \|\Delta_{j}^{(y,s)} \!- F\|_{L_{p}(y+[0,s]^{d})}
    \!=\!    \|\Delta_{j,d} \circ T_{y,s} -F' \circ T_{y,s}\|_{L_{p}(y+[0,s]^{d})}
    \!=\!    s^{d/p} \, \|\Delta_{j,d} - F'\|_{L_{p}([0,1]^{d})}.
    \qedhere
  \]
\end{proof}

We can now prove Lemma~\ref{lem:WeightAndNeuronSpacesDistinct}.

\begin{proof}[Proof of Lemma~\ref{lem:WeightAndNeuronSpacesDistinct}]
  \textbf{Ad (1)} If
  \(
    \WASpace[X_{p_1}(\Omega)][\varrho_{r_1}][q_1][\alpha][L]
    \subset
    \NASpace[X_{p_2}(\Omega)][\varrho_{r_2}][q_2][\beta][L']
  \),
  then the linear map
  \[
    \iota : \WASpace[X_{p_1}(\Omega)][\varrho_{r_1}][q_1][\alpha][L]
            \to \NASpace[X_{p_2}(\Omega)][\varrho_{r_2}][q_2][\beta][L'],
            f \mapsto f
  \]
  is well-defined.
  Furthermore, this map has a closed graph.
  Indeed, if $f_n \to f$ in $\WASpace[X_{p_1}(\Omega)][\varrho_{r_1}][q_1][\alpha][L]$
  and $f_n = \iota f_n \to g$ in $\NASpace[X_{p_2}(\Omega)][\varrho_{r_2}][q_2][\beta][L']$,
  then the embeddings
  \(
    \WASpace[X_{p_1}(\Omega)][\varrho_{r_1}][q_1][\alpha][L]
    \hookrightarrow X_{p_1}(\Omega) \hookrightarrow L_{p_1}(\Omega)
  \)
  and
  \(
    \NASpace[X_{p_2}(\Omega)][\varrho_{r_2}][q_2][\beta][L']
    \hookrightarrow X_{p_2}(\Omega) \hookrightarrow L_{p_2}(\Omega)
  \)
  (see Proposition~\ref{prop:ApproxSpaceWellDefined} and Theorem~\ref{thm:ReLUPowersApproxSpaces})
  imply that $f_n \to f$ in $L_{p_1}$ and $f_n \to g$ in $L_{p_2}$.
  But $L_p$-convergence implies convergence in measure, so that we get $f = g$.

  Now, the closed graph theorem (which applies to $F$-spaces (see \cite[Theorem 2.15]{RudinFA}),
  hence to quasi-Banach spaces, since these are $F$-spaces
  (see \cite[Remark after Lemma 2.1.5]{VoigtlaenderPhDThesis})) shows that $\iota$ is continuous.
  Here, we used that the approximation classes
  $\WASpace[X_{p_1}(\Omega)][\varrho_{r_1}][q_1][\alpha][L]$
  and $\NASpace[X_{p_2}(\Omega)][\varrho_{r_2}][q_2][\beta][L']$
  are quasi-Banach spaces; this is proved independently in
  Theorem~\ref{th:DNNApproxSpaceWellDefined}.

  Since $\Omega$ has nonempty interior, there are $y \in \R^d$ and $s > 0$
  such that $y + [0,s]^d \subset \Omega$.
  The continuity of $\iota$, combined with Theorem~\ref{thm:SawtoothNormAsymptotic},
  implies for the functions $\Delta_j^{(y,s)}$ from Theorem~\ref{thm:SawtoothNormAsymptotic}
  for all $j \in \N$ that
  \[
    2^{j \beta / (L' - 1)}
    \lesssim \|\Delta_j^{(y,s)}\|_{\NASpace[X_{p_2}(\Omega)][\varrho_{r_2}][q_2][\beta][L']}
    \lesssim \|\Delta_j^{(y,s)}\|_{\WASpace[X_{p_1}(\Omega)][\varrho_{r_1}][q_1][\alpha][L]}
    \lesssim 2^{j\alpha / \lfloor L/2 \rfloor},
  \]
  where the implicit constants are independent of $j$.
  Hence, $\beta / (L' - 1) \leq \alpha / \lfloor L/2 \rfloor$;
  that is, $L' - 1 \geq \tfrac{\beta}{\alpha} \cdot \lfloor L/2 \rfloor$.

  \medskip{}

  \textbf{Ad (2)} Exactly as in the argument above, we get for all $j \in \N$ that
  \[
    2^{j\alpha / \lfloor L/2 \rfloor}
    \lesssim \|\Delta_j^{(y,s)}\|_{\WASpace[X_{p_1}(\Omega)][\varrho_{r_1}][q_1][\alpha][L]}
    \lesssim \|\Delta_j^{(y,s)}\|_{\NASpace[X_{p_2}(\Omega)][\varrho_{r_2}][q_2][\beta][L']}
    \lesssim 2^{j \beta / (L' - 1)}
  \]
  with implied constants independent of $j$.
  Hence, $\alpha / \lfloor L/2 \rfloor \leq \beta / (L' - 1)$;
  that is, $\lfloor L/2 \rfloor \geq \frac{\alpha}{\beta} \cdot (L' - 1)$.

  \medskip{}

  \textbf{Proof of the ``in particular'' part:}
  If
  \(
    \WASpace[X_{p_1}(\Omega)][\varrho_{r_1}][q_1][\alpha][L]
    = \NASpace[X_{p_2}(\Omega)][\varrho_{r_2}][q_2][\alpha][L]
  \),
  then Parts (1) and (2) show (because of $\alpha = \beta$) that
  $L - 1 = \lfloor L /2 \rfloor$.
  Since $L \in \N_{\geq 2}$, this is only possible for $L = 2$.
\end{proof}

As a further consequence of Lemma~\ref{lem:SawtoothHardToApproximate}, we can now prove
the non-triviality of the neural network approximation spaces,
as formalized in Theorem~\ref{thm:ApproximationSpacesNonTrivial}.

\begin{proof}[Proof of Theorem~\ref{thm:ApproximationSpacesNonTrivial}]
  In view of the embedding $\WASpace \hookrightarrow \NASpace$
  (see Lemma~\ref{lem:weightsvsneurons}), it suffices to prove the claim for $\NASpace$.
  Furthermore, it is enough to consider the case $q = \infty$,
  since Equation~\eqref{eq:ApproximationSpaceEmbedding} shows that
  $\NASpace \hookrightarrow \NASpace[\StandardXSpace(\Omega)][\varrho][\infty]$.
  Next, in view of Remark~\ref{sec:RestrictionCartesian},
  it suffices to consider the case $k=1$.
  Finally, thanks to Theorem~\ref{thm:ReLUPowersApproxSpaces},
  it is enough to prove the claim for the special case $\varrho = \varrho_r$
  (for fixed but arbitrary $r \in \N$).

  Since $\Omega$ has nonempty interior, there are $y \in \R^d$ and $s > 0$ such that
  $y + [0,s]^d \subset \Omega$.
  Let us fix $\varphi \in C_c (\R^d)$ satisfying $0 \leq \varphi \leq 1$
  and $\varphi |_{y + [0,s]^d} \equiv 1$.
  With $\Delta_{j}^{(y,s)}$ as in Theorem~\ref{thm:SawtoothNormAsymptotic}, define for $j \in \N$
  \[
    g_j : \R^d \to \R,
          x \mapsto \Delta_j^{(y,s)} (x) \cdot \varphi(x) .
  \]
  Note that $g_j \in C_c (\R^d)$, and hence $g_j |_{\Omega} \in X$.
  Furthermore, since $0 \leq \Delta_j^{(y,s)} \leq 1$, it is easy to see
  $\|g_j|_{\Omega}\|_{X} \leq \|g_j\|_{L_p(\R^d)} \leq \|\varphi\|_{L_p(\R^d)} =: C$
  for all $j \in \N$.

  By Theorem~\ref{th:ReLUDNNApproxSpaceWellDefined} and Proposition~\ref{prop:ApproxSpaceWellDefined},
  we know that $\NASpace[\StandardXSpace[](\Omega)][\varrho_r][\infty]$ is a well-defined
  quasi-Banach space satisfying
  $\NASpace[\StandardXSpace[](\Omega)][\varrho_r][\infty] \hookrightarrow \StandardXSpace[](\Omega)$.
  Let us assume towards a contradiction that the claim of Theorem~\ref{thm:ApproximationSpacesNonTrivial}
  fails; this means $\NASpace[\StandardXSpace[](\Omega)][\varrho_r][\infty] = \StandardXSpace[](\Omega)$.
  Using the same ``closed graph theorem arguments'' as in the proof of
  Lemma~\ref{lem:WeightAndNeuronSpacesDistinct}, we see that this implies
  \(
    \|f\|_{\NASpace[\StandardXSpace[](\Omega)][\varrho_r][\infty]}
    \leq C' \cdot \|f\|_{\StandardXSpace[](\Omega)}
  \)
  for all $f \in \StandardXSpace[](\Omega)$ and a fixed constant $C' > 0$.
  In particular, this implies
  $\|g_j|_{\Omega}\|_{\NASpace[\StandardXSpace[](\Omega)][\varrho_r][\infty]} \leq C'C$
  for all $j \in \N$.
  In the remainder of the proof, we will show that
  $\|g_j|_{\Omega}\|_{\NASpace[\StandardXSpace[](\Omega)][\varrho_r][\infty]} \to \infty$
  as $j \to \infty$, which then provides the desired contradiction.

  To prove $\|g_j|_{\Omega}\|_{\NASpace[\StandardXSpace[](\Omega)][\varrho_r][\infty]} \to \infty$,
  choose $N_0 \in \N$ satisfying $\mathscr{L}(N_0) \geq 2$,
  and let $N \in \N_{\geq N_0}$ and
  $f \in \NNreal^{\varrho_r,d,1}_{\infty,\mathscr{L}(N),N}$ be arbitrary.
  Reasoning as in the proof of Theorem~\ref{thm:SawtoothNormAsymptotic},
  since $\varphi \equiv 1$ on $y + [0,s]^{d}$, we see that
  if we set $T_{y,s} : \R^d \to \R^d, x \mapsto s^{-1} (y-x)$, then
  \[
    \|g_{j} - f\|_{L_{p}(\Omega)}
    \geq \|g_{j} - f\|_{L_{p}(y+[0,s]^{d})}
    = \|\Delta^{(y,s)}_{j} - f\|_{L_{p}(y+[0,s]^{d})}
    = s^{d/p} \cdot \|\Delta_{j,d} - f \circ T^{-1}_{y,s}\|_{L_{p}([0,1]^{d})}.
  \]
  Now, given any $x' \in \R^{d-1}$, let us set
  \(
    f_{x'} : \R \to \R, t \mapsto (f \circ T^{-1}_{y,s}) ((t, x'))
  \).
  As a consequence of Lemma~\ref{lem:NetworkCalculus}-(1), we see
  $f_{x'} \in \NNreal^{\varrho_r,1,1}_{\infty,\mathscr{L}(N), N}$.
  According to Part~2 of Lemma~\ref{lem:MinGrowthRateNbPiecesSufficient},
  there is a constant $K_N \in \N$ such that $f_{x'} \in \PPoly_{K_N}^{r^{\mathscr{L}(N) - 1}}$
  Hence, Lemma~\ref{lem:SawtoothHardToApproximate} yields a constant $C_2 = C_2(p) > 0$
  such that $\|\Delta_j - f_{x'}\|_{L_p ( (0,1) )} \geq C_2$ as soon as
  $2^{j} + 1 \geq 4 \, K_N \cdot (1 + r^{\mathscr{L}(N) - 1}) =: K_N '$.
  Because of $2^{j} + 1 \geq j$, this is satisfied if $j \geq K_N '$.
  In case of $p < \infty$, Fubini's theorem shows
  \begin{align*}
    \|\Delta_{j,d} - f \circ T^{-1}_{y,s}\|_{L_p([0,1]^{d})}^p
    & \geq \int_{[0,1]^{d-1}}
             \int_{[0,1]}
               \Big|
                 \Delta_j (t)
                 - f_{x'}(t)
               \Big|^p
             \, d t
           \, d x'
      =    \int_{[0,1]^{d-1}}
             \|\Delta_j - f_{x'}\|_{L_p ( (0,1) )}^p
           \, d x'
      \geq C_2^p,
  \end{align*}
  whence
  \(
    \|g_{j}-f\|_{L_{p}(\Omega)}
    \geq s^{d/p} \|\Delta_{j,d} - f \circ T^{-1}_{y,s}\|_{L_p ([0,1]^{d})}
    \geq C_2 \cdot s^{d/p}
  \).
  For $p = \infty$, the same estimate remains true because
  $\|\bullet\|_{L_p([0,1]^d)} \leq \|\bullet\|_{L_\infty ([0,1]^d)}$.
  Since $f \in \NNreal^{\varrho_r,d,1}_{\infty,\mathscr{L}(N),N}$ was arbitrary, we have shown
  \[
    \AppErr(g_j, \NNreal^{\varrho_r,d,1}_{\infty,\mathscr{L}(N),N})_{L_p(\Omega)}
    \geq C_2 \cdot s^{d/p} =: C_3
    \qquad \forall \, N \in \N_{\geq N_0} \text{ and } j \geq K_N '.
  \]
  Directly from the definition of the norm
  $\|g_j|_{\Omega}\|_{\NASpace[\StandardXSpace[](\Omega)][\varrho_r][\infty]}$,
  this implies that for arbitrary $N \in \N_{\geq N_0}$
  \[
    \|g_j|_{\Omega}\|_{\NASpace[\StandardXSpace[](\Omega)][\varrho_r][\infty]}
    \geq (1 + N)^\alpha
         \cdot \AppErr(g_j, \NNreal^{\varrho_r,d,1}_{\infty,\mathscr{L}(N),N})_{L_p(\Omega)}
    \geq C_3 \cdot (1 + N)^\alpha
    \quad \forall \, j \geq K_N '.
  \]
  This proves $\|g_j|_{\Omega}\|_{\NASpace[\StandardXSpace[](\Omega)][\varrho_r][\infty]} \to \infty$
  as $j \to \infty$, and thus completes the proof.
\end{proof}

\end{document}